\documentclass[a4paper,12pt,reqno]{amsart}
\usepackage{amsmath, amsfonts, amssymb, amsthm, mathrsfs, mathtools, mathalfa}
\usepackage{times}
\usepackage{color}
\usepackage{graphicx}
\usepackage{comment}
\usepackage{yfonts}
\usepackage{enumerate}
\usepackage{enumitem}
\usepackage{stmaryrd}
\usepackage{fancyhdr}
\usepackage{bm}
\usepackage[utf8]{inputenc}
\usepackage[pdfencoding=auto]{hyperref}
\usepackage{tikz,tikz-cd}
\usepackage[hang,flushmargin]{footmisc}
\usepackage{xr}
\usepackage{extarrows}
\usepackage[textwidth=2cm]{todonotes}
\usepackage{faktor}
\usepackage{amsmath}
\usepackage{yhmath} 
\usepackage{tikz , graphicx}
\usepackage{color}
\usetikzlibrary{arrows}
\usepackage{relsize}

\usepackage[linesnumbered]{algorithm2e}

\usepackage[left=15mm,right=15mm,top=30mm,bottom=25mm]{geometry}
\hypersetup{bookmarksdepth=3}

\newcommand{\xdashrightarrow}[2][]{\ext@arrow 0359\rightarrowfill@@{#1}{#2}}

\newtheorem{theorem}{Theorem}[section]
\newtheorem{lemma}[theorem]{Lemma}

\newtheorem{corollary}[theorem]{Corollary}
\newtheorem{proposition}[theorem]{Proposition}

\theoremstyle{definition}
\newtheorem{defn}[theorem]{Definition}
\newtheorem{remark}[theorem]{Remark}
\newtheorem{example}[theorem]{Example}

\newcommand*{\sbr}[1]{\scalebox{0.8}{$(#1)$}}
\newcommand*{\db}[1]{\llbracket #1\rrbracket}

\newcommand{\mk}{\mathfrak}
\newcommand{\mc}{\mathcal}

\newcommand{\mb}{\mathbb}

\newcommand{\wh}{\widehat}
\newcommand{\wt}{\widetilde}

\newcommand{\ud}{\,\mathrm{d}}
\newcommand{\id}{\mathrm{id}}

\DeclareMathOperator{\diagonal}{diag}
\newcommand*{\diag}[1]{\diagonal(#1)}

\DeclareMathOperator{\ab}{Z}
\DeclareMathOperator{\nsR}{Q}

\DeclareMathOperator{\tran}{\Theta}

\DeclareMathOperator{\diam}{diam}

\DeclareMathOperator{\Span}{Span}
\DeclareMathOperator{\Supp}{Supp}
\DeclareMathOperator{\Lip}{Lip}

\DeclareMathOperator{\q}{c}
\DeclareMathOperator{\ns}{X}
\DeclareMathOperator{\nss}{Y}
\DeclareMathOperator{\co}{\circ\hspace{-0.02 cm}}
\DeclareMathOperator{\cu}{C}
\DeclareMathOperator{\cor}{Cor}

\DeclareMathOperator{\ut}{UT}

\DeclareMathOperator{\tRe}{Re}
\DeclareMathOperator{\tIm}{Im}
\definecolor{dbl}{RGB}{60,185,240}

\makeatletter
\def\l@subsection{\@tocline{2}{0pt}{2pc}{1pc}{}}
\makeatother

\begin{document}

\vspace*{-1cm}

\title{Spectral algorithms in higher-order Fourier analysis}

\author{Pablo Candela}
\address{Instituto de Ciencias Matem\'aticas\\ 
Calle Nicol\'as Cabrera 13-15\\ Madrid 28049\\ Spain}
\email{pablo.candela@icmat.es}

\author{Diego Gonz\'alez-S\'anchez}
\address{HUN-REN Alfr\'ed R\'enyi Institute of Mathematics, Re\'altanoda u. 13-15.\\
Budapest, Hungary, H-1053}
\email{diegogs@renyi.hu}

\author{Bal\'azs Szegedy}
\address{HUN-REN Alfr\'ed R\'enyi Institute of Mathematics, Re\'altanoda u. 13-15.\\
Budapest, Hungary, H-1053}
\email{szegedyb@gmail.com}

\begin{abstract}
Our goal is to provide simple and practical algorithms in higher-order Fourier analysis which are based on spectral decompositions of operators. We propose a general framework for such algorithms and provide a detailed analysis of the quadratic case. Our results reveal new spectral aspects of the theory underlying higher-order Fourier analysis. Along these lines, we prove new inverse and regularity theorems for the Gowers norms based on higher-order character decompositions. Using these results, we prove a spectral inverse theorem and a spectral regularity theorem in quadratic Fourier analysis.
\end{abstract}

\maketitle
\vspace{-0.5cm}
\section{Introduction}

\noindent The theory of higher-order Fourier analysis, initiated by the work of Gowers in arithmetic combinatorics \cite{GSz}, is part of a revolutionary development in mathematics from the early 2000s, which has led to a much deeper understanding of Szemerédi's famous theorem on arithmetic progressions \cite{Szemeredi1}, and of additive structures in general. Other parts of this development occurred in  hypergraph theory (via refinements of Szemer\'edi's regularity lemma and its extensions for hypergraphs \cite{G-hyper, RNSSK, Tao-hyper}), and in ergodic theory (in directions stemming from Furstenberg's approach to Szemer\'edi's theorem \cite{FurstSzem}, especially the structure theory of characteristic factors \cite{HK,Z}). A common goal of these topics is to understand higher-order correlations and interactions in complex structures. This development has had a profound impact in several areas in addition to the aforementioned ones, particularly in number theory (a landmark being the Green-Tao theorem \cite{GT-primes}) and in theoretical computer science \cite{HHLbook,Trevisan}. 

Progress in higher-order Fourier analysis in the last two decades includes the discovery and advancement of general foundations for this theory, which involve nilpotent structures \cite{CamSzeg,CScouplings,GTorb,GTZ,HKparas}. However, on one hand the proofs are often long and intricate, which can make the field less accessible to non-experts, and on the other hand, the development of these foundations has so far focused on the theory rather than on potential practical applications. Yet, classical Fourier analysis is one of the most important tools in modern science, and there are good reasons to believe that its higher-order extensions can also become powerful practical tools. Therefore, there is a strong motivation for finding new approaches that can advance this field, including its foundations, in more elementary and applicable directions. In this paper we introduce such an approach, with the goal of bridging the gap between theory and practice. 

The theoretical part of this approach consists in connecting higher-order components of functions on abelian groups (generalizations of Fourier characters, detailed below) with eigenvectors of certain operators constructed from such functions. The construction of these operators is elementary and is outlined in Subsection \ref{subsec:outline} of this introduction. In the first part of the paper, we provide initial results establishing the aforementioned connection at a relatively elementary level. The second part of the paper (starting in Section \ref{sec:nschars}) delves into deeper reasons for this connection, which involve what can be viewed as a  constituent of the foundations of higher-order Fourier analysis, namely \emph{nilspace theory} (discussed further below). Our results in this part include a new structure (or regularity) theorem for Gowers norms on finite abelian groups, which decomposes functions into sums of certain generalizations of Fourier characters that we call \emph{nilspace characters}, which are nearly orthogonal (Theorem \ref{thm:UpgradedReg}). Consequences of this result include an inverse theorem involving such nilspace characters (Theorem \ref{thm:correlestim}), and new approximate diagonalizations and Parseval identities for Gowers norms (Theorems \ref{thm:ApproxBessel} and \ref{thm:UdDiag}). These theoretical results 
are discussed in more detail in Subsection \ref{subsec:introNilspaceResults} of this introduction. 

To demonstrate the applicability of these results, we provide simple algorithms which use the above-mentioned spectral data to recover the higher-order decompositions, and which do not require detailed knowledge of the deeper aspects of the theory. While the general pattern for these algorithms is easily formulated for any higher order (see Remark \ref{rem:higheralg}), in this paper we focus on fully demonstrating such algorithms in the case of \emph{quadratic Fourier analysis}. This is outlined in Subsection \ref{subsec:introMainResults} in this introduction, where we present our main theorems in this case (Theorems \ref{thm:reg-intro} and \ref{thm:HiSpecBiject-intro}) and formulate the resulting algorithms  (Algorithms \ref{alg:reg} and \ref{alg:indi-qua-char}). These algorithms can be implemented straightforwardly using the spectral decomposition of self-adjoint matrices and the (fast) Fourier transform. Algorithm \ref{alg:reg} could be used in the prediction of missing data (or to carry out what could be called \emph{quadratic denoising}) in ordered data sets and thus in time series prediction, analogously to how the usual Fourier transform is currently used for such tasks (a basic illustration is given in Figure \ref{fig:experiment}). The decomposition  into \emph{quadratic characters} given by Algorithm \ref{alg:indi-qua-char} could be used to identify important components of a function, similarly to principal component analysis (see also Remark \ref{rem:PCAlink}). These aspects will be further explored in a more applied follow-up work, which will focus on the practical aspects of our methods, including experiments and refinements.

To begin explaining and motivating our spectral approach in more detail, let us first recall here some basic ideas of higher-order Fourier analysis. The Gowers norms (or uniformity norms) are among the most important notions in this theory. For each integer $k\geq 2$, the $k$-th Gowers norm (or $U^k$-norm) is defined on the space of complex-valued functions $f$ on a finite abelian group,\footnote{These norms can be defined more generally for bounded measurable functions on  compact abelian groups} the norm $\|f\|_{U^k}$ consisting of an average over configurations known as $k$-dimensional \emph{cubes} on the group (the formula is recalled in equation \eqref{eq:Uknorm} below). The $U^2$-norm of $f$ is equal to the $\ell^4$-norm of the Fourier transform of $f$, thereby connecting the $U^2$-norm with classical (first-order) Fourier analysis. The role of the $U^{k+1}$-norm in $k$-th order Fourier analysis centers on providing a useful concept of a function being \emph{noise of order $k$} (or \emph{quasirandom of degree $k$}), namely, a function having small $U^{k+1}$-norm. This in turn yields, in a dual way, a notion of a function being  \emph{structured of order $k$}, namely, a function being nearly orthogonal to any noise of order $k$. The Gowers norms form an increasing sequence (meaning that $\|f\|_{U^k}\leq \|f\|_{U^{k+1}}$ for every $k$ and $f$), and this implies that, as $k$ increases, fewer functions are classified as noise, and the notion of structured function becomes more inclusive and subtle. Major efforts in this field have gone into describing, as precisely as possible, these higher-order structured functions, in particular by seeking \emph{fundamental} $k$-th order structured functions that could act adequately as generalizations of Fourier characters. As we shall see in this paper, the above-mentioned nilspace characters are functions of this type. A fascinating aspect of this subject is that, whereas the geometric object underlying classical Fourier characters is the circle group, the higher-order structured functions can involve non-abelian nilpotent structures such as the Heisenberg nilmanifold. In these directions, two central and interrelated themes have developed: on one hand, the so-called \emph{inverse theorems for Gowers norms}, which establish that functions with non-negligible $U^{k+1}$-norm correlate non-trivially with some $k$-th order generalization of a Fourier character; on the other hand, \emph{decomposition theorems} (or \emph{regularity lemmas}), which express any bounded function essentially as a sum of a structured part and a noise part of order $k$. 

There is by now extensive literature in higher-order Fourier analysis proving inverse theorems for Gowers norms in various families of abelian groups, often with the remarkable additional feature of giving effective bounds for the result (see for instance 
\cite{CSinverse, GSz, GM1, GM2, GTinv, GTZ, J&T, JST-tot-dis, LSS, MannersQIinv, M1, M2, TZ1, TZ2}). There are also many works proving higher-order decompositions of functions and regularity lemmas for Gowers norms (e.g.\ \cite{GStruct,GW-LinFormsFn,GWcomp,GTarith,GT-primes,Taoerg}). However, works providing algorithmic implementations of the above results are fewer, with notable examples being the paper by Tulsiani and Wolf  introducing quadratic analogues of the Goldreich--Levin algorithm \cite{TW}, and the more recent paper of Kim, Li, and Tidor  on the cubic case \cite{KLT}. 
The algorithms in these works are probabilistic, involve iterative processes \cite[See \S3]{TW}, and are specialized to the finite-field setting (of central interest in theoretical computer science). The new algorithms that we begin to explore in this paper are based on spectral decompositions of operators associated with functions, are more direct (in particular, essentially non-iterative), and applicable on any finite abelian group. 

This new spectral aspect of higher-order Fourier analysis has quite natural motivations. One such motivation is the well-known connection between classical Fourier analysis and spectral decompositions. Fourier characters on abelian groups are eigenvectors of shift operators. More generally, if a linear operator is shift-invariant (i.e.\ it commutes with every shift), then every Fourier character is an eigenvector of that operator. This property enables Fourier analysis to be recovered from spectral theory (see Subsection \ref{subsec:classFourier}). The results in this paper describe a similar connection between higher-order Fourier analysis and spectral theory through adequate generalizations of shift-invariant operators and Fourier characters. Another motivation for using spectral methods comes from Szemer\'edi’s regularity lemma for graphs \cite{Szemeredi2}, which is a precursor of the intricate regularity lemmas in higher-order Fourier analysis. The spectral nature of graph regularization was first highlighted by Frieze and Kannan in \cite{FK2}. Subsequently, the third author showed in \cite{S-spectral} that Szemer\'edi-type regular partitions (even in their stronger forms \cite{AFKS}) can be obtained from the dominant eigenvectors of adjacency matrices. The deep connection between eigenvectors and regularization in graph theory is part of a rich algorithmic framework that includes principal component analysis, low-rank approximations, spectral clustering, and dimensionality reduction more broadly. An important aim of this paper is to position higher-order Fourier analysis within this framework by describing its spectral aspects, with emphasis on regularization.

\subsection{Outline of the spectral approach}\label{subsec:outline}\hfill\smallskip\\
Given a finite abelian group $\ab$ and a function $f:\ab\to \mb{C}$, a first key step in this approach consists in turning $f$ into a matrix in $\mb{C}^{\ab\times\ab}$ (or a \emph{$\ab$-matrix}, as we shall call it; see Definition \ref{def:Zmatrix}). We consider various choices for how to do this, all of which proceed by applying a (typically non-linear) transformation $K:\mb{C}^{\ab}\to\mb{C}^{\ab}$ to every so-called \emph{$\ab$-diagonal function} of the rank-1 matrix
\[
f\otimes\overline{f}:\ab\times \ab\to\mb{C},\, (x,y)\mapsto f(x)\overline{f(y)},
\]
a $\ab$-diagonal being a set of entries whose indices have fixed difference in $\ab$ (see Definition \ref{def:Zmatrix}). 

The matrix $f\otimes\overline{f}$ is a natural and convenient object to consider in higher-order Fourier analysis, since its $\ab$-diagonal function with entry-difference $t$ is the multiplicative derivative 
\[
\Delta_t f(x):=f(x+t)\overline{f(x)},
\]
which plays a key role in this theory (notably in proofs of inverse theorems). The matrix $f\otimes \overline{f}$ thus  encapsulates in a useful way all the multiplicative derivatives of $f$. 

In general, given a map (or operator) $K:\mb{C}^{\ab}\to\mb{C}^{\ab}$ and a matrix $M\in\mb{C}^{\ab\times \ab}$, we denote by $\mathcal{K}(M)$ the matrix obtained by applying $K$ to every $\ab$-diagonal function of $M$. To ensure that $\mathcal{K}$ is well-behaved, we require the operator $K$ to commute with complex conjugation and the shift action of $\ab$ on functions in $\mb{C}^{\ab}$ (in which case we call $K$ an \emph{invariant} operator; see Definition \ref{def:invop}). If these conditions are satisfied, then $\mathcal{K}$ preserves the property of being a \emph{self-adjoint} (or Hermitian) matrix. In particular $\mathcal{K}(f\otimes\overline{f})$ is self-adjoint, so it has real eigenvalues and orthogonal eigenvectors. (In this paper, the linear-algebraic notions pertaining to $\ab$-matrices, including their eigenvalues, are normalized in coherence with the probability Haar measure on $\ab$ and the corresponding inner-product on the space $\mb{C}^{\ab}$, a space on which $\ab$-matrices act as kernels of linear integral operators; see Definition \ref{def:ZmatOp} and the discussion preceding it.) 

It turns out that the spectral data associated with $\mc{K}(f\otimes\overline{f})$ carries important information to perform $k$-th order Fourier analysis on $f$. Let us formulate this as the following principle, which captures in general terms a phenomenon that is central to this approach. 

\medskip
\noindent{\bf Order increment principle:}~{\it If $K:\mb{C}^{\ab}\to\mb{C}^{\ab}$ maps functions in $\mb{C}^{\ab}$ to their $k$-th order structured parts, then for $f\in \mb{C}^{\ab}$ the spectral decomposition of $\mathcal{K}(f\otimes\overline{f})$ can be used to obtain the $(k+1)$-th order structured part \textup{(}and corresponding useful decomposition\textup{)} of $f$.}

\medskip

\noindent This principle can be used to turn $k$-th order Fourier analysis into $(k+1)$-th order Fourier analysis. A precise form of this principle was observed in the ultralimit setting in \cite{S-hofa2}. 

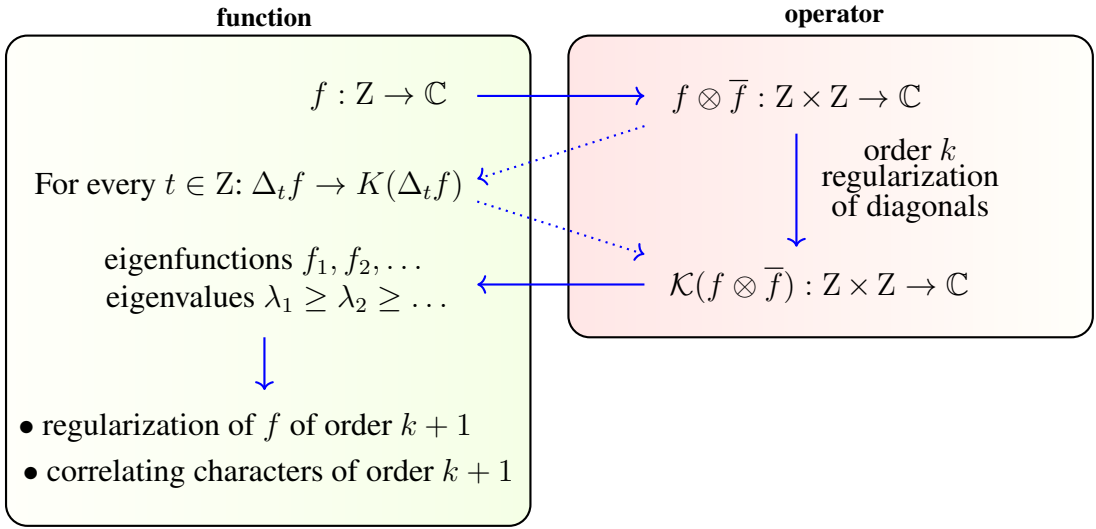
\begin{figure}[htbp]
 \begin{tikzpicture}[thick]
        \draw [help lines, black!10] (-.5,-.5); 
        
        \colorlet{lightbluel}{blue!1}
        \colorlet{lightbluer}{blue!9}
        \colorlet{lightbluee}{blue!13}
        \colorlet{MyColorOnel}{brown!2}
        \colorlet{MyColorOner}{brown!22}
        \colorlet{MyColorTwo}{brown!8}
        \colorlet{MyColorThree}{brown!18}
        \colorlet{MyColorFourl}{yellow!6}
        \colorlet{MyColorFourr}{brown!35}
        \definecolor{tempcolorl}{RGB}{255,250,235}
        \definecolor{tempcolorr}{RGB}{225,225,190}
        \definecolor{dgr}{RGB}{20,125,30} 
        \definecolor{PCA}{RGB}{245,255,230} 
        \definecolor{PCAl}{RGB}{255,255,250} 
        \definecolor{PCB}{RGB}{255,230,230}                 
        \definecolor{PCBl}{RGB}{255,255,250}       
        
        \filldraw[left color = PCAl, right color= PCA, rounded corners=8] (-0.4,0.5) rectangle ++(6.9,6.5) ;
        \filldraw[left color = PCB, right color = PCBl ,  rounded corners=8] (7,3) rectangle ++(6.9,4);

        \node at(3,7.25) {\footnotesize\text{\bf function}};
        \node at(10.5,7.25) {\footnotesize\text{\bf operator}};
          
        \node at(11.5,5.5) {\text{order $k$}};  
        \node at(11.5,5.1) {\text{regularization}};             
        \node at(11.5,4.7) {\text{of diagonals}};                     
        \draw[->, blue] (5.8,6.2) -- (8,6.2);   
        \draw[->, blue, dotted] (8,5.8) -- (5.8,5.1);      
        
        \draw[->, blue, dotted] (5.8,4.8) -- (8,4.1);                     
        \draw[->, blue] (8,3.7) -- (5.8,3.7);      
        \draw[->, blue] (10,5.7) -- (10,4.2);   
        \draw[->, blue] (3,3) -- (3,2.3);

        \node at(4.5,6.2){\text{$f:\ab\to\mathbb{C}$}};
        \node at(2.8,5) {\text{For every $t\in \ab$:~~$\Delta_t f\to K(\Delta_t f)$}};            
        \node at(3.2,4) {\text{eigenfunctions $f_1,f_2,\dots$\;\;\;}};           
        \node at(3.2,3.5) {\text{eigenvalues $\lambda_1\geq\lambda_2\geq\dots$}};         
        \node at(10,6.2){\text{$f\otimes\overline{f}:\ab\times \ab\to\mathbb{C}$}};
        
       \node at(2.8,1.8) {\text{$\bullet$ regularization of $f$ of order $k+1$ }};     
       \node at(3.05,1.2) {\text{$\bullet$ correlating characters of order $k+1$}};                       
                           
        \node at(10.3,3.7){\text{$\mathcal{K}(f\otimes\overline{f}):\ab\times \ab\to\mathbb{C}$}};

        \tikzstyle{every node}=[circle, draw, fill=black!50,
                        inner sep=0pt, minimum width=4pt] ;

    \end{tikzpicture}\vspace{-1cm}
  \caption{Sketch of the spectral approach to higher-order Fourier analysis.}
  \label{fig:sketch-sprec-approac}
\end{figure}    

\noindent A simple choice for $K$ is the averaging operator (mapping $f$ to the constant function\footnote{Recall the standard notation for averaging, whereby given a finite set $X$ and a function $f$ from $X$ to a real or complex vector space, we denote by $\mb{E}_{x\in X} f(x)$ the average $\frac{1}{|X|}\sum_{x\in X} f(x)$.} $\mb{E}_{x\in\ab} f(x)$), which pertains to $0$-th order Fourier analysis. In this case the eigenvectors of $\mathcal{K}(f\otimes\overline{f})$ are the Fourier characters of $f$, and the eigenvalues are the squares of the absolute values of the corresponding Fourier coefficients. This illustrates how $0$-th order Fourier analysis is turned into first order, classical Fourier analysis (we discuss this in more detail in Subsection \ref{subsec:classFourier}). 

The next logical step is to progress from first-order to quadratic Fourier analysis using the above principle. This step requires an operator $K$ that extracts the Fourier-structured component of a function $f$ in an adequate way. Finding such an operator is challenging because of a lack of uniqueness. Indeed, while in an appropriate limit setting (such as the setting of \cite{S-hofa2}), there exists a unique, and even linear,  operator that isolates the structured part (consisting in conditional expectation relative to an appropriate $\sigma$-algebra), the counterparts in  finite settings exist only in an  approximate and non-linear way. There are various possible choices, many of which depend on additional parameters. In this work, we introduce such an operator with especially useful analytic properties, which we call the \emph{Fourier denoising operator}, and we develop several tools for its application. 

To define this operator, recall first that for a character $\chi$ in the dual group $\wh{\ab}$ (where $\ab$ is a finite abelian group), we define the corresponding Fourier coefficient of $f$ by
\begin{equation}
\wh{f}(\chi)=\mb{E}_{x\in \ab} f(x)\overline{\chi(x)}.
\end{equation}
Then, for a fixed (typically small) constant $\varepsilon>0$, our denoising operator is defined by
\begin{equation}\label{eq:Kdefintro}
K_{\varepsilon}(f)=\sum_{\chi\in \wh{\ab}~:~|\wh{f}(\chi)|\ge \varepsilon}\tfrac{|\wh{f}(\chi)|-\varepsilon}{|\wh{f}(\chi)|}\, \wh{f}(\chi)\, \chi.
\end{equation}
Note that there is a simpler operator (that we can call the \emph{Fourier cut-off operator}) which sets Fourier coefficients with absolute value smaller than $\varepsilon$ to $0$ and keeps the remaining terms in the Fourier decomposition unchanged. The main problem with the Fourier cut-off operator is that it is not continuous and has inconvenient properties for calculations. Instead, the operator $K_\varepsilon$ applies the continuous function $x\mapsto {\rm ReLU}(x-\varepsilon)$ to the magnitudes of the Fourier coefficients while keeping their phases. Here, ReLU is the function $x\mapsto (x+|x|)/2$ known from machine learning. It turns out that $K_\varepsilon$ is a contraction in $L^2$ and has other pleasant properties while being sufficiently similar to the cut-off operator. Related to the operator $K_\varepsilon$ we introduce the operator $\mc{K}_\varepsilon:\mb{C}^{\ab\times \ab}\to\mb{C}^{\ab\times \ab}$ which applies $K_\varepsilon$ to all $\ab$-diagonals of a matrix in  $\mb{C}^{\ab\times \ab}$. A more detailed discussion of the Fourier denoising operator begins in Subsection \ref{subsec:Keps}.

\subsection{Main results with algorithmic consequences}\label{subsec:introMainResults}\hfill\smallskip\\
Let us begin by recalling the formula for the Gowers norms from  \cite[Lemma 3.9]{GSz} (see also \cite[Definition 11.2]{T-V}). For each integer $k\geq 2$, the $U^k$-norm of a function $f\in \mb{C}^{\ab}$ is defined by
\begin{eqnarray}\label{eq:Uknorm}
&\|f\|_{U^k}:=\Bigl(\mb{E}_{x,t_1,\dots,t_k\in \ab}~ \prod_{v\in\{0,1\}^k}\mc{C}^{|v|}f(x+v\sbr{1}\, t_1+v\sbr{2}\,t_2+\cdots+v\sbr{k}\,t_k)\Bigr)^{1/2^k},&
\end{eqnarray}
where $\mc{C}$ denotes the complex conjugation operator and $|v|:=\sum_{i=1}^k v\sbr{i}$.

The dual of the $U^k$-norm, defined by $\|f\|_{U^k}^*:=\sup_{g\in \mb{C}^{\ab}, \|g\|_{U^k}\leq 1} \mb{E}_{x\in \ab} f(x)\overline{g(x)}$, is a useful tool to measure the extent to which a function is structured in the sense of $(k-1)$-th order Fourier analysis, as we will explain in more detail later (see Subsection \ref{subsec:kstrucfns}). In particular, we shall use this dual norm to define the quantitative and relatively elementary concept of  \emph{structured function of order $k-1$} (see Definition \ref{def:kstruct}).

The next theorem is one of the main results in this paper. It provides a regularity (or decomposition) result for functions in quadratic Fourier analysis, expressing the quadratically structured part of the function in terms of the dominant eigenvalues and corresponding eigenvectors of the appropriate $\ab$-matrix. Thus it connects regularity with spectral methods in the quadratic setting. This provides a theoretical basis for Algorithm \ref{alg:reg} below.

\begin{theorem}[Spectral $U^3$-regularization]\label{thm:reg-intro}
For every $\rho_0\in[0,1)$ there exists $\varepsilon_0>0$ such that the following holds. Let $\ab$ be a finite abelian group and let $f:\ab\to \mb{C}$ be a 1-bounded function. Then there exists $\rho\in[\rho_0/2,\rho_0]$ and $\varepsilon\in[\varepsilon_0,1]$ with the following property. Let $f_{\text{reg}}$ be the projection of $f$ to the linear span of the eigenspaces of $\mc{K}_\varepsilon(f\otimes\overline{f})$ with corresponding eigenvalues at least $\rho$. Then $\|f-f_{\text{reg}}\|_{U^3}\le 2\rho^{3/8}$ and there exists $h:\ab\to\mb{C}$ such that $\|f_{\text{reg}}-h\|_2\le \rho$ and $\|h\|_{U^3}^*=O_{\rho}(1)$.
\end{theorem}
\noindent The conclusion involving the approximating function $h$, of bounded $U^3$-dual-norm, can be summarized as saying that $f_{\text{reg}}$ is a \emph{structured function of order 2} (see Definition \ref{def:kstruct}). Note also that, as mentioned earlier, eigenvalues here are taken in a normalized form (see Definition \ref{def:ZmatOp}).

\begin{remark}
Analogous to Szemerédi's regularity lemma for graphs \cite{Szemeredi2}, arithmetic regularity lemmas (such as Theorem \ref{thm:reg-intro}) exist in multiple versions \cite{GStruct, GTarith}, each balancing the trade-off between the strength of the statement and the quality of the bounds. For instance, the weak regularity lemma by Frieze and Kannan \cite{FK} provides a more limited approximation for graphs but achieves significantly better bounds, making it suitable for practical applications. Szemer\'edi's original formulation \cite{Szemeredi1} occupies a middle ground in this trade-off: it does not provide effective bounds, but it has  stronger theoretical implications. 
The primary goal of Theorem \ref{thm:reg-intro} is to present a concise, representative member of a broader family of regularity lemmas that connect higher-order Fourier analysis with spectral theory. In fact, the version of Theorem \ref{thm:reg-intro} that we actually prove, namely Theorem \ref{thm:algoregul}, is already finer and more flexible (though also more technical). In future papers, we will explore other versions of the result, including stronger formulations, to extend its applicability and impact. 
\end{remark}
\noindent The calculation of $f_{\text{reg}}$ in Theorem \ref{thm:reg-intro} is efficient in practice. It involves $|{\ab}|$ independent applications of the operator $K_\varepsilon$ (which can be executed in parallel) and the computation of dominant eigenvectors and eigenvalues of a single ${\ab}\times {\ab}$ matrix. Each application of $K_\varepsilon$ can be done using a fast Fourier transform, a simple coefficient truncation, and an inverse fast Fourier transform. The whole process depends on the original function $f$ defined on ${\ab}$, as well as on two positive parameters $\varepsilon$ and $\rho$. The details of this calculation are outlined in the following pseudocode.
\begin{algorithm}
\SetKwInput{KwInput}{Input}                
\SetKwInput{KwOutput}{Output}              
\DontPrintSemicolon
  \KwInput{$f:\ab\to \mb{C}$, $\rho,\varepsilon\in \mb{R}_{>0}$}
  
  $M \leftarrow f(x)\otimes \overline{f(y)}\in \mb{C}^{\ab\times \ab}$

  \For{$t\in \ab$}    
        {         
         $M(\cdot+t,\cdot) \leftarrow K_\varepsilon(M(\cdot+t,\cdot))$
        }
    $(\mu_1,v_1),\ldots,(\mu_{|\ab|},v_{|\ab|})\leftarrow \mathbf{Eigendecomposition}(M)$
  
  \KwOutput{  $f_{\text{reg}}\leftarrow\sum_{\mu_i\ge \rho} \langle f,v_i\rangle v_i$ and $\{(\mu_1,v_1),\ldots,(\mu_{|\ab|},v_{|\ab|})\}$\protect\footnotemark}\vspace{5pt}
  \caption{$U^3$-regularization algorithm.}
  \label{alg:reg}
\end{algorithm}

\footnotetext{We assume that the eigenvalues are ordered decreasingly, i.e., $\mu_i\ge \mu_{i+1}$ for all $i$.}

\begin{remark}[Choice of parameters] In practice, there is no definitive choice for the parameters $\rho$ and $\varepsilon$ in Algorithm \ref{alg:reg}. The optimal values depend heavily on the nature of the dataset the algorithm is applied to. Adjusting these parameters allows for analyzing the quadratic structure of a function at different levels of resolution. Practical aspects of how to choose these parameters will be investigated in a paper focusing on applications. 
\end{remark}

\begin{remark}[Higher-order versions]\label{rem:higheralg} Let us denote the outcome of Algorithm \ref{alg:reg} by $H_{\varepsilon,\rho}(f)$. It is easy to see that for fixed $\varepsilon>0,\rho>0$ the operator $f\mapsto H_{\varepsilon,\rho}(f)$ commutes with shifts and conjugation. Thus, by applying it to the $\ab$-diagonals of $f\otimes\overline{f}$, we obtain a self-adjoint matrix. By choosing a new value $\kappa>0$ and projecting $f$ onto the space spanned by the eigenvectors of this matrix with eigenvalue at least $\kappa$, we obtain a new operator $f\mapsto H_{\varepsilon,\rho,\kappa}(f)$ which approximates the structured part of $f$ in cubic Fourier analysis. By further iterating this process, we obtain a spectral approach to $k$-th order Fourier analysis based on a regularization operator $f\mapsto H_{\varepsilon_1,\varepsilon_2,\dots,\varepsilon_k}(f)$ with $k$ parameters. We leave the analysis of this algorithm to a subsequent paper.
\end{remark}

\begin{remark}[Continuous versions] Both computer experiments and the underlying theory suggest continuous variants of Algorithm \ref{alg:reg}. Note that the output $H_{\varepsilon,\rho}(f)$ of the algorithm is not continuous in its input $f$. An example for a continuous version is given by $H'_{\varepsilon,\rho}(f):=\mb{E}_{\rho'\in[\rho/2,\rho]}H_{\varepsilon,\rho'}(f)$. It is easy to see that $H'_{\varepsilon,\rho}$ can also be computed efficiently, with complexity similar to $H_{\varepsilon,\rho}$. To achieve this, the formula for the projection in $H_{\varepsilon,\rho}$ must be augmented with certain weight terms that depend on the eigenvalues. Recall that the Fourier cut-off operator mentioned earlier suffered from a similar continuity problem, and was replaced by the \emph{continuous} denoising operator $K_\varepsilon$ to obtain our proofs in the quadratic setting. Besides producing more stable and precise outcomes, continuous versions also appear to be more suitable when moving to higher orders via the order increment principle. These and other refinements of the algorithm lie outside the scope of this paper and will be explored in future work.
\end{remark}

\medskip

\begin{figure}[h]
  \centering
  \includegraphics[scale=0.34]{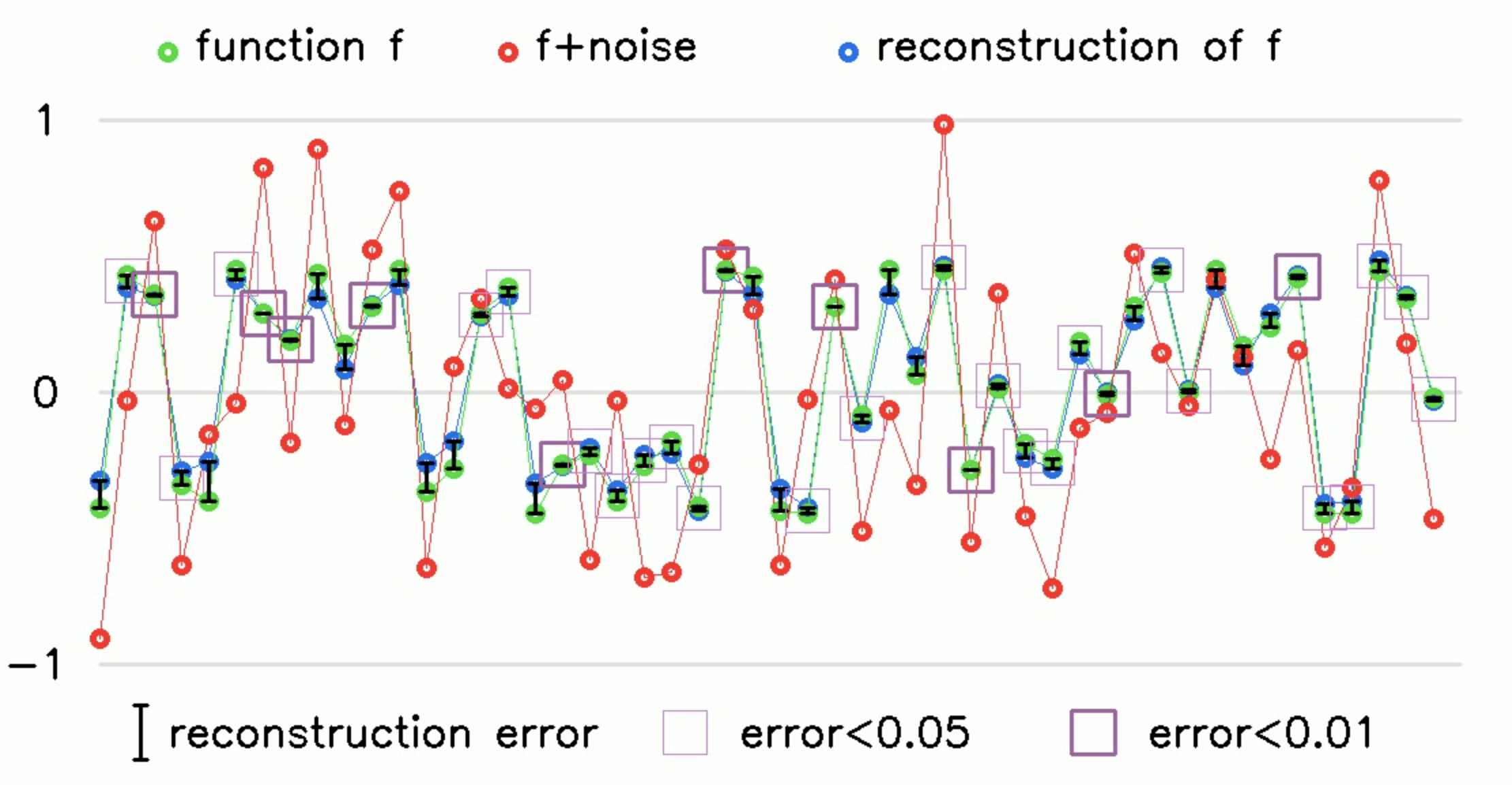}
  \caption{\small{Removing added random noise from a quadratic structured function by a version of Algorithm 1 on the cyclic group $\mathbb{Z}_{500}$. A window of length $50$ is plotted for illustration. In this example, the function $f(i):=\sin(8i^2+3i+1)$, $i\in [500]$ (green graph) is perturbed by random noise, resulting in the function $g=f+r$ (red graph). The spectral algorithm is applied to $g$ and the reconstruction $f_2$ of $f$ (blue graph) is obtained by the projection of $g$ to the space spanned by the $6$ leading eigenvectors of the operator constructed from $g$. The plot highlights the reconstruction error $|f(i)-f_2(i)|$.}}
  \label{fig:experiment}
\end{figure}
\noindent A deeper and more sophisticated version of our spectral algorithm focuses on the meaning of the individual eigenvectors of the operator $\mathcal{K}_\varepsilon(f\otimes\overline{f})$. It turns out that these eigenvectors can give further information on the quadratic structured part of $f$. In particular, if the corresponding eigenvalue is large enough and well-separated from other eigenvalues, such an eigenvector behaves as a quadratic generalization of a Fourier character, and has large inner product with $f$. This generalization is an interesting notion in itself and can be defined in a basic way. The initial observation for this is that a classical Fourier character $\chi$ (or a constant multiple of it) is defined by the property that for every $t$ the multiplicative derivatives $\Delta_t\chi$ is a constant function, and that constant functions are the structured functions in $0$-th order Fourier analysis (this is proved, in greater generality, in Lemma \ref{lem:struct0}). Going one degree higher, we obtain the following basic notion: a function $f\in \mb{C}^{\ab}$ is a \emph{quadratic character} if for every $t\in\ab$ the multiplicative derivative $\Delta_t f$ is \emph{Fourier structured} (of first order) in the sense of Definition \ref{def:kstruct}. This notion of quadratic character depends on two parameters that measure the extent to which the multiplicative derivatives are Fourier structured; more precisely $f$ is a quadratic character of \emph{complexity} $R$ and \emph{precision} $\delta$ if for every $t$ there is $g_t\in\mb{C}^{\ab}$ such that $\|\Delta_t f -g_t\|_{L^2}\leq \delta$ and $\|g_t\|_{U^2}^*\leq R$. Let us mention straightaway that a central result in this paper connects this basic notion of a quadratic character with the 2-step case of the deeper notion of \emph{nilspace characters} mentioned earlier (we discuss this central result in the next subsection). We also introduce basic notions of \emph{characters of order $k$} for each $k\geq 2$, studied in Subsection \ref{subsec:chark} (see Definition \ref{def:k-char}). 

It turns out that the quadratically structured part of $f$ given by Theorem \ref{thm:reg-intro} can be decomposed in terms of  quadratic characters of bounded complexity, and that our second algorithm is able to find (good approximations of) those characters. To state our main result validating this second algorithm, we use the following terms. For a self-adjoint $\ab$-matrix $A$, we denote by $\textup{Spec}_{\rho}(A)$ the set of eigenvalues of $A$ larger than $\rho$, and we denote by $\textup{Eigen}_\rho(A)\subset \mb{C}^{\ab}$ the subspace spanned by eigenvectors of $A$ corresponding to eigenvalues in $\textup{Spec}_{\rho}(A)$. We also say that a set $X\subset \mb{C}$ is \emph{$\delta$-separated} if for every $x,y\in X$ with $x\not=y$ we have $|x-y|\ge \delta$. We can now state the main result in question.
\begin{theorem}\label{thm:HiSpecBiject-intro}
For every $\rho_0\in [0,1/10]$, there exists $\varepsilon_0>0$ such that for any finite abelian group $\ab$ and any 1-bounded function $f:\ab\to \mb{C}$, there exists $\rho\in [\rho_0/2,\rho_0]$ and $\varepsilon\in[\varepsilon_0,1]$ satisfying  the following property. Let $h$ be a function in $\textup{Eigen}_\rho\big(\mc{K}_\varepsilon(f\otimes \overline{f})\big)\subset \mb{C}^{\ab}$ with $\|h\|_2\le 1$ such that $\textup{Spec}_{\rho^7}\big(\mc{K}_\varepsilon(h\otimes \overline{h})\big)$ is $\rho^7$-separated and has cardinality equal to $|\textup{Spec}_{\rho}\big(\mc{K}_\varepsilon(f\otimes \overline{f})\big)|$. Then for every unit eigenvector $v$ of $\mc{K}_\varepsilon(h\otimes \overline{h})$ corresponding to an eigenvalue in $\textup{Spec}_{\rho^7}\big(\mc{K}_\varepsilon(h\otimes \overline{h})\big)$, there is a quadratic character $g$ of complexity $O_{\rho}(1)$ and precision $\rho$ satisfying $\|v-g\|_{L^2}\leq 56\rho^{7/2}$ and $|\langle f,g\rangle|\ge \sqrt{\rho/4}$.
\end{theorem}

\begin{remark}
Theorem \ref{thm:HiSpecBiject-intro} is a simpler version of the result that we actually prove, i.e.\  Theorem \ref{thm:HiSpecBiject}. In this introductory setting, let us mention informally just one of the additional features of the latter theorem: there is in fact a \emph{bijection} between the  eigenvectors of $\mc{K}_\varepsilon(h\otimes \overline{h})$ with large eigenvalues and the dominant 2-step nilspace characters of $f$ given by the new decomposition result proved in this paper (Theorem \ref{thm:UpgradedReg}, discussed in the next subsection). In particular, every such nilspace character can be recovered from these eigenvectors (up to a small error).
\end{remark}

\noindent A typical situation enabling a simple application of Theorem \ref{thm:HiSpecBiject-intro} is when the 1-bounded function $f:\ab\to\mb{C}$ has the largest eigenvalues of $\mc{K}_\varepsilon(f\otimes \overline{f})$ sufficiently separated. Then, taking $h$ to be the structured part $f_{\textrm{reg}}$ given by\footnote{More precisely, we use the finer version of Theorem \ref{thm:reg-intro} that we actually prove, namely Theorem \ref{thm:algoregul}.} Theorem \ref{thm:reg-intro}, the spectral assumptions in Theorem \ref{thm:HiSpecBiject-intro} are satisfied, and we can then directly obtain the dominant quadratic characters of $f$ from the eigenvectors associated with those large eigenvalues (see Lemma \ref{lem:proxispec}). This situation is actually generic, in a sense that we can use in order to handle the cases where those spectral assumptions might fail. 
Indeed, in such cases, we show that one can use a randomized procedure, consisting essentially in calculating the dominant eigenvectors and eigenvalues of $\mathcal{K}_\varepsilon(h\otimes\overline{h})$ where $h$ is a random combination of the dominant eigenvectors of $\mathcal{K}_\varepsilon(f\otimes\overline{f})$. We show that, with high probability, this procedure yields a function $h$ to which Theorem \ref{thm:HiSpecBiject-intro} applies, thus resulting in a decomposition of the structured part of $f$ into quadratic characters (see Subsection \ref{subsubsec:cluster}). The formalization of this probabilistic procedure is given in Theorem \ref{thm:main-random} and Remark \ref{rem:distr-unit-sphere}. Theorems  \ref{thm:HiSpecBiject-intro} and \ref{thm:main-random} together validate Algorithm \ref{alg:indi-qua-char} below.
\begin{algorithm}
\SetKwInput{KwInput}{Input}                
\SetKwInput{KwOutput}{Output}              
\DontPrintSemicolon
  \KwInput{$f:\ab\to \mb{C}$, $\rho,\varepsilon,\delta\in \mb{R}_{>0}$}

  $f_{\text{reg}},(\mu_1,v_1),\ldots,(\mu_{|\ab|},v_{|\ab|})\leftarrow \mathbf{Algorithm \;\ref{alg:reg}}(f,\rho,\varepsilon)$

  $S\leftarrow \max(i\;:\; \mu_i\ge \rho)$

  \eIf{$\mathbf{IsSeparated}(\delta,\{\mu_1\ge\cdots\ge\mu_S\})$
  }{
    $h\leftarrow f_{\text{reg}}$
  }{
    $h\leftarrow \mathbf{RandomUnitVector}(v_1,\ldots,v_S)$
  }

  $f_{\text{reg}}',(\mu_1',v_1'),\ldots,(\mu_{|\ab|}',v_{|\ab|}')\leftarrow \mathbf{Algorithm \;\ref{alg:reg}}(h,\delta,\varepsilon)$

  $S'\leftarrow \max(i\;:\; \mu_i'\ge \delta)$

  \KwOutput{ $\{ v_i'\; : \; i\in[S]\}$, $\big(\mathbf{IsSeparated}(\delta,\{\mu_1'\ge\cdots\ge\mu_S'\}) \textbf{ and }|S|=|S'|\big)$}\vspace{5pt}
  \caption{Quadratic character decomposition.
  }
  \label{alg:indi-qua-char}
\end{algorithm}

\noindent For a 1-bounded function $f:\ab\to \mb{C}$, Algorithm \ref{alg:indi-qua-char} finds a decomposition of $f$ into quadratic characters (the $ v_i'$ in the first output) if the second output equals \textbf{True} (which happens with high probability under appropriate choices of parameters, e.g.\ $\delta=\rho^7$ as in Theorem \ref{thm:HiSpecBiject-intro}). The following auxiliary routines are used: $\mathbf{IsSeparated}(\delta,\{\mu_1\ge\cdots\ge\mu_S\})$, which returns true only if $|\mu_i-\mu_{i+1}|\ge\delta$ for all $i$, and $\mathbf{RandomUnitVector}(v_1,\ldots,v_S)$, which returns a uniformly chosen random unit vector in the subspace $\langle v_1,\ldots,v_S\rangle$ with respect to the $L^2$ norm.

\begin{remark}
We believe that, instead of the  \emph{random} choice made in Theorem \ref{thm:main-random}, there is a \emph{deterministic} procedure to find a function $h\in \textup{Eigen}_\rho\big(\mc{K}_\varepsilon(f\otimes\overline{f})\big)$ satisfying the properties required for Theorem \ref{thm:HiSpecBiject-intro} to be applicable. Investigating this lies outside the scope of this paper.
\end{remark}

\subsection{Nilspace theoretic foundations for the spectral method}\label{subsec:introNilspaceResults}\hfill\smallskip\\
From Section \ref{sec:nschars} onwards, our main proofs rely strongly on the nilspace approach to higher-order Fourier analysis, which is summarized in Figure \ref{fig:k-th-order}. Nilspaces are algebraic structures that emerge naturally from the large scale behavior of functions on abelian groups relative to the Gowers norms. Despite having been introduced relatively recently \cite{CamSzeg}, nilspaces are treated in various works (see in particular \cite{Cand:Notes1, Cand:Notes2, CSinverse, GMV1,GMV2,GMV3}). We shall recall some basic background on nilspace theory in Section \ref{sec:nschars}. 

The nilspace approach yields a regularity theorem for the Gowers norms which holds in particular on any finite abelian group, and which is a key ingredient in this paper \cite[Theorem 1.5]{CSinverse}. According to this theorem, a 1-bounded function $f$ on a finite abelian group $\ab$ has a structured part of order $k$ that is a so-called $k$-step \emph{nilspace polynomial}, i.e.\ a function of the form  $F\co\phi$, where $\phi:\ab\to \ns$ is a nilspace morphism from $\ab$ (viewed as a 1-step nilspace) into a $k$-step nilspace $\ns$ which is \emph{compact and of finite rank} (in particular $\ns$ is a finite-dimensional manifold), and $F$ is a 1-bounded Lipschitz complex-valued function on $\ns$ (see Definition \ref{def:nil-poly}). The pair $(\ns,F)$ is in some sense a compressed version of $(\ab,f)$ which contains information on the $k$-th order structure of $f$ at a given ``resolution''. 

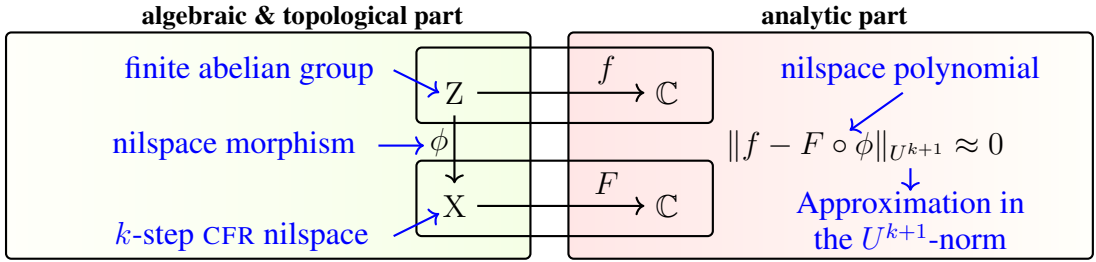
\begin{figure}[htbp]
 \begin{tikzpicture}[thick]
        \draw [help lines, black!10] (-.5,-.5); 
        
        \colorlet{lightbluel}{blue!1}
        \colorlet{lightbluer}{blue!9}
        \colorlet{lightbluee}{blue!13}
        \colorlet{MyColorOnel}{brown!2}
        \colorlet{MyColorOner}{brown!22}
        \colorlet{MyColorTwo}{brown!8}
        \colorlet{MyColorThree}{brown!18}
        \colorlet{MyColorFourl}{yellow!6}
        \colorlet{MyColorFourr}{brown!35}
        \definecolor{tempcolorl}{RGB}{255,250,235}
        \definecolor{tempcolorr}{RGB}{225,225,190}
        \definecolor{dgr}{RGB}{20,125,30} 
        \definecolor{PCA}{RGB}{245,255,230} 
        \definecolor{PCAl}{RGB}{255,255,250} 
        \definecolor{PCB}{RGB}{255,230,230}                 
        \definecolor{PCBl}{RGB}{255,255,250}       
        
        \filldraw[left color = PCAl, right color= PCA, rounded corners=2] (-0.4,1) rectangle ++(6.9,3) ;
        \filldraw[left color = PCB, right color = PCBl ,  rounded corners=2] (7,1) rectangle ++(6.9,3); 
        \draw[draw=black, rounded corners=2] (5,2.8) rectangle ++(3.9,1);         
        \draw[draw=black, rounded corners=2] (5,1.3) rectangle ++(3.9,1);

        \node at(3.5,4.2) {\footnotesize\text{\bf algebraic \& topological part}};
        \node at(10.5,4.2) {\footnotesize\text{\bf analytic part}};
              
        \node at(2.8,3.5) {\textcolor{blue}{finite abelian group}};              
        \node at(2.6,2.5) {\textcolor{blue}{nilspace morphism}};      
        \node at(2.7,1.3) {\textcolor{blue}{$k$-step \textsc{cfr} nilspace}};   
        \node at(11.5,1.7) {\textcolor{blue}{Approximation in}};             
        \node at(11.5,1.3) {\textcolor{blue}{the  $U^{k+1}$-norm}};           \node at(11.5,3.5) {\textcolor{blue}{nilspace polynomial}};                      
        \draw[->] (5.8,3.2) -- (8,3.2);     
        \draw[->] (5.8,1.7) -- (8,1.7);      
        \draw[->] (5.5,2.9) -- (5.5,2);

        \draw[->, blue] (4.7,3.5) -- (5.3,3.2);    
        \draw[->, blue] (4.7,1.3) -- (5.3,1.6);    
        \draw[->, blue] (4.5,2.5) -- (5.1,2.5);    
        \draw[->, blue] (11.4,3.2) -- (10.7,2.7);    
        \draw[->, blue] (11.5,2.2) -- (11.5,1.9);    
        
        \node at(5.5,3.2){\text{$\ab$}};
        \node at(5.5,1.7) {\text{$\ns$}};    
        \node at(8.3,3.2){\text{$\mathbb{C}$}};
        \node at(8.3,1.7) {\text{$\mathbb{C}$}};      
        \node at(7.5,3.5){\text{$f$}};
        \node at(7.5,2) {\text{$F$}};  
        \node at(5.3,2.55){\text{$\phi$}};                            
        \node at(10.9,2.5){\text{$\|f-F\circ\phi\|_{U^{k+1}}\approx 0$}};

        \tikzstyle{every node}=[circle, draw, fill=black!50,
                        inner sep=0pt, minimum width=4pt] ;                                  
    \end{tikzpicture}\vspace{-1cm}
  \caption{The nilspace approach to $k$-th order Fourier analysis}
  \label{fig:k-th-order}
\end{figure} 

In this paper we refine this regularization technique by combining it with a generalized Fourier decomposition of $F$. To this end, we use the fact that the compact and finite-rank (\textsc{cfr}) nilspace $\ns$ has the structure of an iterated \emph{principal abelian bundle}, with \emph{structure groups} $\ab_1,\ab_2,\dots,\ab_k$ which are compact abelian Lie groups (see \cite[Proposition 2.1.9]{Cand:Notes2}). The last structure group $\ab_k$ acts on $\ns$ freely by nilspace automorphisms. This can be used to decompose a Lipschitz function $F:\ns\to\mb{C}$ uniquely as a series $F=\sum_{\chi\in\wh{\ab}_k}F_\chi$ converging uniformly. This yields the decomposition of the structured part of $f$ as
\begin{equation}\label{eq:StructDecompIntro}
F\co\phi = \sum_{\chi\in\wh{\ab}_k}F_\chi\co\phi.
\end{equation}
Our main theoretical results describe how this additional decomposition of the structured part further enhances the framework provided by the nilspace regularity theorem \cite[Theorem 1.5]{CSinverse}, especially when this decomposition is combined with the strong equidistribution property of $\phi$ (known as \emph{balance}) given by this theorem. Below, we summarize the main results that we obtain in this direction, and their connection with our spectral approach.
\medskip

\noindent{\bf Nilspace characters.}~~In Section \ref{sec:nschars} we introduce the generalizations of Fourier characters mentioned previously in this introduction, namely \emph{nilspace characters}, and we describe their fundamental properties. Nilspace characters are basically the functions of the form $F_\chi\co\phi$ obtained in the decomposition \eqref{eq:StructDecompIntro} (for their formal description see Definition \ref{def:nilspace-char}). These functions provide a nilspace theoretic generalization of Fourier characters, valid on any finite abelian group. In the specific setting of cyclic groups, they are closely related (though not identical) to the notion of \emph{nilcharacter} from \cite{GTZ}. A central task in our study of nilspace characters is to connect this notion with the more elementary notion of a \emph{character of order $k$} mentioned in the previous subsection (and formalized in Definition \ref{def:k-char}). Note that both notions are \emph{parametric} (i.e.\ they depend on additional parameters, such as a \emph{complexity} parameter). This is analogous to other concepts, such as being a ``structured'' or ``quasirandom'' function, which are also parametric, and become exact (i.e.\ parameter-free) only in a limit setting (see for instance \cite{S-hofa1,S-hofa2,S-hofa3}). A central result in Section \ref{sec:nschars} is Theorem \ref{thm:nscharsquadchars}, which focuses on the quadratic case and establishes, informally speaking, that 2-step nilspace characters are quadratic characters under an appropriate transformation of their parameters.

\medskip

\noindent{\bf Decompositions with nilspace characters.}~~ In Section \ref{sec:regularity-of-f} we prove a refined version of the nilspace regularity theorem which implements the additional decomposition \eqref{eq:StructDecompIntro} of the structured part as a sum of nilspace characters; see Theorem \ref{thm:UpgradedReg}. Moreover, using the balance property of the underlying morphism $\phi$, this theorem guarantees that the nilspace characters satisfy approximate orthogonality properties which are crucial for our purposes. This new formulation of the nilspace regularity theorem also brings higher-order Fourier analysis more in line with classical Fourier analysis, in particular by yielding a new result constituting a generalization of Parseval's identity (see Theorem \ref{thm:ApproxBessel}). Note that a result of this kind has been called for as a desirable tool for higher-order Fourier analysis (see \cite[Section 16]{GowersBAMS}). Another new result of this type is an approximate diagonalization formula for the Gowers norms in terms of nilspace characters (see Theorem \ref{thm:UdDiag}).

\medskip

\noindent{\bf The bridge between spectral decompositions and higher-order characters.} 
In Section \ref{sec:Strucop}, we prove Theorem \ref{thm:main-algo-pf}, which links the spectral properties of the operators $\mc{K}_\varepsilon(f\otimes\overline{f})$ with the quadratic nilspace character components of $f$ given by Theorem \ref{thm:UpgradedReg}. An informal version of this theorem asserts that
$$\mathcal{K}_\varepsilon(f\otimes\overline{f})=\sum_{\chi\in S}g_\chi\otimes\overline{g_\chi}+E$$
where $\varepsilon$ is a small, appropriately chosen number, the functions $g_\chi:=F_\chi\co\phi$ are the dominant 2-step nilspace character components of $f$ for some finite set $S$, and the matrix $E$ is a small error. In particular, by the orthogonality properties of the functions $g_\chi$, this theorem implies that each $g_\chi$ is a so-called \emph{pseudoeigenvector} of $\mc{K}_\varepsilon(f\otimes\overline{f})$ (see Definition \ref{def:pseudo}). As a consequence, one obtains that the corresponding pseudoeigenvalue is close to some proper eigenvalue $\lambda$. Furthermore, if $\lambda$ is well separated from other eigenvalues, then the corresponding eigenvector $v$ is close to $g_\chi$. This phenomenon yields a direct connection between the eigenvectors of $\mathcal{K}_\varepsilon(f\otimes\overline{f})$ and the nilspace characters $g_\chi$. 

\subsection{Connection with the ultralimit setting}\hfill\\
To close this introduction, let us mention that the spectral approach to higher-order Fourier analysis is deeply motivated by an analogous  construction in the ultralimit setting \cite{S-hofa1,S-hofa2,S-hofa3}, which, on one hand, is more exact (not depending on various parameters), and on the other hand is intrinsically qualitative (in particular not yielding effective bounds). This construction involves measurable functions on ultraproducts of abelian groups. In particular, the Gowers norms become \emph{semi}norms on such groups. In general, a major benefit of this ultralimit approach is its simpler and more algebraic language. This is illustrated by the following facts that hold in this setting: the notion of noise of order $k-1$ becomes exact, namely, it becomes the property of having $U^k$-seminorm equal to 0; then, every bounded measurable function $f$ can be \emph{uniquely} decomposed as $f = f_s + f_r$, where $f_r$ is noise of order $k-1$ and $f_s$ is a structured function of order $k-1$, i.e., is orthogonal to any noise function of order $k-1$. Moreover, it can be proved that there is a sub-$\sigma$-algebra $\mc{F}_{k-1}$ on any such group such that a function $g$ is $\mc{F}_{k-1}$-measurable if and only if $g$ is structured of order $k-1$. Hence, the structured part $f_s$ can be obtained as the conditional expectation $\mb{E}(f|\mc{F}_{k-1})$ relative to the $\sigma$-algebra $\mc{F}_{k-1}$. In this setting, the spectral approach simplifies to applying the map $g \mapsto \mb{E}(g|\mc{F}_{k-1})$ to each $\ab$-diagonal function of $f \otimes \overline{f}$, yielding the unique self-adjoint operator $M(f,k)$.

Quite surprisingly, and non-trivially, the eigenvectors of $M(f,k)$ are structured functions of order $k$ and their multiplicative derivatives are structured of order $k-1$ (see \cite[Theorem 2]{S-hofa1} and \cite{S-hofa2}). This constitutes a limit-setting analogue of the notion of $k$-th order character introduced in this paper (Definition \ref{def:k-char}). Another useful property of $M(f,k)$ is that the $k$-th order structured part of $f$ is equal to the projection of $f$ onto the space spanned by the eigenvectors of $M(f,k)$ with non-zero eigenvalues. The resulting decomposition of the structured part of $f$ into these eigenvectors is a higher-order generalization of the ordinary Fourier transform. 

The ultralimit setting serves as a valuable framework for envisioning further results in higher-order Fourier analysis and deriving them at a qualitative level. This will be explored in future work in connection with the spectral approach, especially in the general $k$-th order case. On the other hand, for the aim of making higher-order Fourier analysis practical, it becomes vital to introduce useful notions of a finite (usually parametric) type, and in this effort one is also aided by the deeper insights provided by the structures underlying higher-order decompositions, such as nilmanifolds and, more broadly, nilspaces. Hence this paper's focus on developing new connections between these structures, spectral properties, and higher-order decompositions.

\medskip

\noindent \textbf{Acknowledgements.} 
The authors thank R. Gim\'enez Conejero and Alex Hof for valuable discussions regarding the appendix of this paper. All authors received funding from Spain's MICINN project PID2020-113350GB-I00. The second and third-named authors received funding from the projects KPP 133921, Momentum (Lend\"ulet) 30003, and the Artificial Intelligence National Laboratory (MILAB, RRF-2.3.1-21-2022-00004), NKFIH, Hungary.

\section{Basic notions}\label{sec:basics}

\noindent Throughout this paper, by $\ab$ we will denote a finite abelian group, unless we explicitly say otherwise. Much of our work will involve the following special type of square matrices.
\begin{defn}[$\ab$-matrices]\label{def:Zmatrix}
Let $\ab$ be a finite abelian group. A \emph{$\ab$-matrix} is a matrix in $\mb{C}^{\ab\times\ab}$. Given such a matrix $M$ and an element $t\in\ab$, we call the set of pairs $\{(z+t,z): z\in \ab\}$, viewed as a set of indices of entries of $M$, the \emph{$\ab$-diagonal} of $M$ at level $t$ (or the $t$-th $\ab$-diagonal of $M$), and we denote it by $\diag{t}$. The $\ab$-\emph{diagonal function} of $M$ at level $t$ is the function
\begin{equation}\label{eq:Zdiag}
\mk{D}_{M,t}:\ab \to\mb{C},\;   z\mapsto M(z+t,z).
\end{equation}
We then define the map $\mk{D}_M:\ab\to L^2(\ab)$ as $t\mapsto \mk{D}_{M,t}$. 
\end{defn}
\noindent An important type of $\ab$-matrix for us will be the rank-1 examples of the form $M=f\otimes \overline{f}$ for $f\in \mb{C}^{\ab}$, i.e.\ $M(x,y)=f(x)\overline{f(y)}$ for $x,y\in\ab$. In this case, note that $\mk{D}_{f\otimes \overline{f},t}(z)=\Delta_tf(z)$.
\begin{remark} Thus $\ab$-matrices have an abelian group structure on the index set of their rows and columns. The term ``diagonal'' is inspired by the case where $\ab$ is a finite cyclic group $\mb{Z}_N$, identified with the set of integers $\{0,1,\ldots, N-1\}$ with addition modulo $N$. In this case, labeling the rows and columns of $M$ in order by $0,1,\ldots, N-1$, the $\ab$-diagonals are indeed diagonal subsets of the matrix. For a general finite abelian group $\ab$, we can fix an arbitrary ordering of the elements of $\ab$ (with 0 as the first element) and use this ordering for rows and columns. However, in general the $\ab$-diagonals no longer look like diagonal subsets of the matrix (take, for instance, $\ab=\mb{Z}_2^2$).
\end{remark}
\noindent We equip $\ab$ with its probability Haar measure (i.e., normalized counting measure). This induces the standard inner-product operation $\langle f,g\rangle:= \mb{E}_{x\in \ab} f(x)\overline{g(x)}$ for any functions $f,g\in \mb{C}^{\ab}$, with the associated $L^2$-norm $\|f\|_{L^2}=\langle f,f\rangle^{1/2}$. In what follows, the notation $\|\cdot\|_2$ will always indicate this normalized $L^2$-norm when applied to functions in $\mb{C}^{\ab}$. When we wish to use instead the euclidean norm of a vector without normalization, then we will signal this explicitly by writing $\|\cdot\|_{\ell^2}$ instead of $\|\cdot\|_2$ (thus $\|f\|_{\ell^2}=(\sum_{x\in \ab} |f(x)|^2)^{1/2}$). 

It will be useful to view $\ab$-matrices as kernels of integral linear operators on $L^2(\ab)$ relative to the normalized Haar measure (in particular, in this sense the concept of $\ab$-matrix extends naturally to compact abelian groups). However, this view entails an associated normalization of certain operations, which we shall use throughout the paper and which we underline as follows.
\begin{defn}\label{def:ZmatOp}
For any $\ab$-matrix $M\in \mb{C}^{\ab\times\ab}$ and any function $f\in \mb{C}^{\ab}$, we define the function $Mf\in\mb{C}^{\ab}$ by $Mf(x)= \mb{E}_{y\in\ab} M(x,y)f(y)$. Accordingly, the product of two $\ab$-matrices $M$, $M'$ is the $\ab$-matrix $MM'$ defined by $(MM')(x,y)=\mb{E}_{z\in \ab} M(x,z)M'(z,y)$ (thus $MM'$ is the kernel of the composition of the linear integral operators with kernels $M,M'$). The notions of eigenvector and eigenvalue of a $\ab$-matrix are also normalized accordingly, thus a vector $v\in\mb{C}^{\ab}$ is an \emph{eigenvector of $M$ with eigenvalue $\lambda\in\mb{C}$} if $\mb{E}_{y\in \ab} M(x,y)v(y)=\lambda v(x)$ for every $x\in \ab$. Finally, note that when we use the notation $\|\cdot\|_2$ with a $\ab$-matrix, we always mean the version of the Euclidean norm (or Hilbert-Schmidt norm, or Schatten 2-norm) compatible with the normalized inner-product on $\mb{C}^{\ab}$ specified above, so $\|M\|_2:=(\mb{E}_{x,y\in \ab} |M(x,y)|^2)^{1/2}$.
\end{defn}
\noindent We have seen above that from a $\ab$-matrix we obtain a map $\mk{D}_M:\ab\to L^2(\ab)$ sending each element of $\ab$ to its corresponding $\ab$-diagonal function in $M$. We can invert this map.
\begin{defn}
Given $F:\ab\to L^2(\ab)$, we define the $\ab$-matrix $\wt{M}(F)(x,y):=[F(x-y)](y)$.
\end{defn}
 
\noindent In this notation $[F(x-y)](y)$, the square brackets enclose the \emph{function} $F(x-y)\in L^2(\ab)$, and the subsequent term $(y)$ indicates the argument at which the function $F(x-y)$ is evaluated.
\begin{lemma}
For any $\ab$-matrix $M$ we have $\wt{M}(\mk{D}_M) = M$, and for any map $F:\ab\to L^2(\ab)$ we have $\mk{D}_{\wt{M}(F)} = F$.
\end{lemma}
\begin{proof}
We have $\wt{M}(\mk{D}_M)(x,y) = [\mk{D}_M(x-y)](y) = [z\mapsto M(z+x-y,z)](y) = M(x,y)$. Conversely, for any $F:\ab\to L^2(\ab)$ and $x,y\in \ab$, we have $[\mk{D}_{\wt{M}(F)}(x)](y) = [z\mapsto\wt{M}(F)(z+x,z)](y) = \wt{M}(F)(x+y,y) = [F(x)](y)$.\end{proof}

\noindent Next, we give an alternative description of the product of two $\ab$-matrices $M,{M'}$, consisting in writing the product in terms of the $\ab$-diagonal functions of the matrices (rather than in terms of rows and columns as usual). To do so, let us view the $\ab$-matrix ${M}$ as the sum of the restrictions of ${M}$ to its $\ab$-diagonals: ${M}(x,y)=\sum_{t\in\ab} 1_{\diag{t}}(x,y) {M}(x,y)$.\footnote{The function $1_{\diag{t}}(x,y)$ is the characteristic function of the $t$-th diagonal, i.e., $1_{\diag{t}}(x,y)=1$ if $x-y=t$ and $0$ otherwise.} Note that $1_{\diag{t}}(x,y)$ for $x,y\in \ab$ is a permutation matrix, so $1_{\diag{t}}(x,y) {M}(x,y)$ can be regarded as a weighted permutation $\ab$-matrix, denoted ${M}_t$. We then have the following formula.

\begin{lemma}[Product of weighted permutation $\ab$-matrices]Given two weighted permutation $\ab$-matrices $M_t$ and $ M'_{t'}$, their product is a weighted permutation $\ab$-matrix, supported on the diagonal at level $t+t'$, with $(x,y)$-entry given by the following formula:\begin{equation}\label{eq:wpermatprod}M_tM'_{t'}(x,y) = \tfrac{1}{|\ab|}1_{\diag{t+t'}}(x,y) M(x,x-t) M'(y+t',y).\end{equation}\end{lemma}

\begin{proof} We have
\[
M_tM'_{t'}(x,y) =\mb{E}_{v\in\ab}  M_t(x,v) M'_{t'}(v,y)
= \mb{E}_{v\in\ab}  1_{\diag{t}}(x,v) M(x,v)
1_{\diag{t'}}(v,y) M'(v,y).
\]
Here $1_{\diag{t}}(x,v) 
1_{\diag{t'}}(v,y)=1(x=v+t)1(v=y+t'=x-t)$, so the last average above equals
\begin{multline*}
1_{\diag{t+t'}}(x,y) \mb{E}_{v\in \ab} 1(x=v+t)M(x,v) M'(v,y) \\ = \tfrac{1_{\diag{t+t'}}(x,y)}{|\ab|}  M(x,x-t) M'(x-t,y) =  \tfrac{1_{\diag{t+t'}}(x,y)}{|\ab|} M(x,x-t) M'(y+t',y),
\end{multline*}
as claimed in \eqref{eq:wpermatprod}.     
\end{proof}
Now we can give the formula for multiplication of $\ab$-matrices via $\ab$-diagonals.
\begin{proposition}\label{prop:diagprod}
Let $\ab$ be a finite abelian group and let $M, M'\in \mb{C}^{\ab\times\ab}$. Then for every $w\in\ab$ we have
\begin{equation}\label{eq:Zdiagofprod}
\mk{D}_{MM',w}(z) = \mb{E}_{t \in\ab} \;\mk{D}_{M,t}(z+w-t)\; \mk{D}_{M',w-t}(z).
\end{equation}
\end{proposition}
\begin{proof} Using \eqref{eq:wpermatprod} we have that
\begin{align*}
MM'(z+w,z) & =  \big(\textstyle\sum_{t\in\ab} M_t(z+w,z) \big)\big(\textstyle\sum_{t'\in\ab} M'_{t'}(z+w,z)\big) = \textstyle\sum_{t,t'\in\ab} M_t M'_{t'}(z+w,z)\\ 
& =  \textstyle\sum_{t,t'\in\ab} \tfrac{1}{|\ab|}1_{\diag{t+t'}}(z+w,z) M(z+w,z+w-t) M'(z+t',z)\\
& =  \tfrac{1}{|\ab|} \textstyle\sum_{t,t'\in\ab:t+t'=w}  M(z+w,z+w-t) M'(z+t',z)\\
& =  \mb{E}_{t \in\ab}  M(z+w,z+w-t) M'(z+w-t,z),
\end{align*}
and this equals $\mb{E}_{t \in\ab} \;\mk{D}_{M,t}(z+w-t)\; \mk{D}_{M',w-t}(z)$, as claimed.
\end{proof}

\begin{remark}\label{rem:diagalg}
Proposition \ref{prop:diagprod} is useful in particular because it implies that if $\mc{A}$ is a shift-invariant algebra of functions (i.e.\ closed under addition, multiplication, and shifting) and every $\ab$-diagonal of $M$ and $M'$ is a function in $\mc{A}$, then this also holds for the $\ab$-matrix product $MM'$. (This observation is used for instance in Lemma \ref{lem:MWnorm}.)
\end{remark}

\begin{remark}\label{rem:finabgprestrict}
Among the notions introduced above and in what follows, many are easily extended to all compact abelian groups, and some even to locally-compact abelian groups. Similarly, most proofs extend to compact abelian groups. We have nevertheless restricted our notation and treatment to finite abelian groups $\ab$, because the main applications of algorithmic type concern primarily such groups, but also because this restriction simplifies the treatment markedly, avoiding various analytic issues related to measurability or convergence (especially once we start using Fourier analysis below). 
\end{remark}

\medskip

\subsection{Invariant operators}\hfill\smallskip\\
In what follows we shall often refer to maps $\mb{C}^{\ab} \to \mb{C}^{\ab}$ as \emph{operators} (which will not necessarily be linear). We will apply such operators to $\ab$-matrices along the $\ab$-diagonals, and the main operators that we shall apply have a specific invariance property, defined as follows.
\begin{defn}\label{def:invop}
Let $\ab$ be a finite abelian group, let $K$ be any operator $\mb{C}^{\ab} \to \mb{C}^{\ab}$, and let $M$ be a $\ab$-matrix. We denote by $\mc{K}(M)$ the $\ab$-matrix $\wt{M}(K\co \mk{D}_M) = \wt{M}(K( \mk{D}_M))$. For a function $f:\ab\to\mb{C}$ and $h\in \ab$, let $T^hf$ denote the function $x\mapsto f(x+h)$. We say that the operator $K$ is \emph{invariant} if $K(T^hf)=T^hK(f)$ and $K(\overline{f})=\overline{K(f)}$, and we then say that the associated matrix operator $\mc{K}$ is also \emph{invariant}.
\end{defn}
\begin{remark}\label{rem:calnote}
Here and throughout the paper, note that given a map $K: \mb{C}^{\ab}\to\mb{C}^{\ab}$, for any $\ab$-matrix $M$ we use the \emph{calligraphic} notation $\mc{K}(M)$ to denote the matrix obtained by applying $K$ to each $\ab$-diagonal of $M$.
\end{remark}

\noindent We shall apply such invariant operators to self-adjoint matrices $M$, such as rank-1 matrices of the form $f\otimes \overline{f}$ for $f\in\mb{C}^{\ab}$. The following lemma shows that invariance of $K$ suffices to ensure that $\mc{K}(M)$ is also self-adjoint.

\begin{lemma}\label{lem:sapreserve}
If $M\in\mb{C}^{\ab\times\ab}$ is self-adjoint and $K$ is invariant, then $\mc{K}(M)$ is self-adjoint. 
\end{lemma}

\begin{proof}
We need to show that for all $x,y\in \ab$ we have $\mc{K}(M)(x,y) = \overline{\mc{K}(M)(y,x)}$. By definition we have $\mc{K}(M)(x,y) = \wt{M}(K(\mk{D}_M))(x,y) = [K(\mk{D}_M)(x-y)](y)= [K(\mk{D}_{M,x-y})](y)$. Recall that $\mk{D}_{M,x-y}(z)=M(z+x-y,z) = \overline{M(z,z+x-y)}$, where in the last equality we used that $M$ is self-adjoint. Let $t\in \ab$. Applying the shift operator $T^t$ to this function, we have $T^t\mk{D}_{M,x-y}(z) = \overline{M(z+t,z+t+x-y)}$. Letting $t:=y-x$, we further obtain $T^{y-x}\mk{D}_{M,x-y}(z) = \overline{\mk{D}_{M,y-x}}$.

Since $K$ is invariant, for any $t\in \ab$ and $f\in L^2(\ab)$, we have $T^{-t}K(T^t f) = K(f)$. In particular, $\mc{K}(M)(x,y)=[K(\mk{D}_{M,x-y})(z)] (y)= [T^{-t}K(T^t \mk{D}_{M,x-y}(z))](y) = [K(T^t \mk{D}_{M,x-y}(z))](y-t)$. Letting $t:=y-x$, we have $\mc{K}(M)(x,y) = [K(\overline{\mk{D}_{M,y-x}})(z)](x)$. As $K$ commutes with complex conjugation, we have $[K(\mk{D}_{M,x-y})(z)] (x)=\overline{[K(\mk{D}_{M,y-x})(z)](x)}$, and this is by definition equal to $\overline{\mc{K}(M)(y,x)}$.
\end{proof}

\noindent Using various invariant operators $K:\mb{C}^{\ab}\to\mb{C}^{\ab}$ on the diagonals, and then taking the spectral decomposition of the resulting matrices, we obtain eigenvectors with various properties. Let us fix such an operator $K$ and let $f:\ab\to\mb{C}$. As noted earlier, the matrix $f\otimes \overline{f}$ is a rank-1 self-adjoint matrix. Let $Q:=\mathcal{K}(f\otimes \overline{f})$. Since $Q$ is Hermitian by Lemma \ref{lem:sapreserve}, we can apply the classical spectral theorem, obtaining the decomposition $Q=\sum_{i=1}^n \lambda_i\, v_i\otimes \overline{v_i}$, where $n=|\ab|$, the $v_i$ are eigenvectors of $Q$ forming an orthonormal basis of $\mb{C}^{\ab}$, and for each $i\in [n]$ the real scalar $\lambda_i$ is the eigenvalue of $Q$ corresponding to $v_i$. We can also assume (by permuting if necessary) that $\lambda_1\geq\lambda_2\geq\dots\geq\lambda_n$. The general idea of our spectral approach is that eigenvectors of $Q$ corresponding to the largest eigenvalues are interesting ``components'' of the function $f$. The coefficients $c_i:=\langle f,v_i \rangle$ are also important, and their relation with the eigenvalues $\lambda_i$ is often nontrivial.

In this section, we give various more detailed illustrations of this phenomenon, while keeping the involved machinery at a relatively simple level (in particular, not yet requiring deeper background such as nilspace theory).

\subsection{A simple instance of the spectral approach: recovering classical Fourier analysis}\label{subsec:classFourier}\hfill\smallskip\\
As a first illustration of our use of invariant operators, let us detail here the remark made in the introduction, namely that if we apply to $f\otimes \overline{f}$ the operator consisting in simply averaging out each $\ab$-diagonal, then the spectral decomposition of the resulting matrix (over $\mb{C}$) yields essentially the Fourier decomposition of $f$. 

To see this in more detail, let $K$ be the invariant operator consisting in averaging the function $f:\ab\to\mb{C}$ over the group $\ab$, that is $K(f)=\mb{E}_{\ab}(f)$. Note that $K$ is the orthogonal projection to the subspace of constant functions in $\mb{C}^{\ab}$. In this case, one can easily show that the classical Fourier characters $\chi\in \wh{\ab}$ are eigenvectors of $\mc{K}(f\otimes \overline{f})$, and the eigenvalue corresponding to a character $\chi$ is the squared magnitude of the corresponding Fourier coefficient, i.e.\ $|\wh{f}(\chi)|^2$. Indeed, using the standard Fourier expansion $f = \sum_{\chi\in \widehat{\ab}} \langle f,\chi\rangle \chi$, we have 
\begin{align*}
\mc{K}(f\otimes \overline{f})(x,y) & =  \mb{E}_{z\in \ab} f(x+z)\overline{f(y+z)} = \mb{E}_{z\in \ab} \textstyle\sum_{\chi\in \wh{\ab}} \langle f,\chi\rangle \chi(x+z) \overline{\textstyle\sum_{\chi'\in \wh{\ab}} \langle f,\chi'\rangle \chi'(y+z)}\\
& =  \textstyle\sum_{\chi,\chi'\in \wh{\ab}}  \langle f,\chi\rangle \overline{\langle f,\chi'\rangle } \mb{E}_{z\in \ab} \chi(x+z) \overline{\chi'(y+z)} = \textstyle\sum_{\chi\in \wh{\ab}} |\langle f,\chi\rangle|^2 \chi(x)\overline{\chi(y)}.
\end{align*}

\begin{remark}\label{rem:similar-eigen-fourier}
Note that even though the previous calculations tell us that Fourier characters are eigenvectors, we may not be able to recover all of them from the previous decomposition using only spectral analysis. To illustrate this, consider $f=\chi+\chi'$ for two Fourier characters $\chi\not=\chi'$ on $\ab$. Then, looking at the spectrum of $Q$ as described above, we will just find an eigenspace generated by $\chi$ and $\chi'$ for the eigenvalue 1, and another eigenspace of eigenvalue 0 generated by all other characters. Thus, if we want to recover the Fourier decomposition using this procedure, then we need a way of finding individual characters (such as $\chi,\chi'$ here) inside eigenspaces of dimension greater than 1. The same phenomenon can occur in the higher-order versions of this spectral approach, and will be addressed in Subsection \ref{subsec:specsep}.
\end{remark}

\begin{remark}\label{rem:PCAlink}
This example, involving the averaging operator $K$, enables us to highlight the connection between our approach and principal component analysis. More precisely, note that the operator $\mc{K}(f\otimes \overline{f})$ in this case is equal to the Gram matrix of the shifts of $f$, which can also be viewed as a covariance matrix. Thus, this operator is positive semidefinite, and the eigenvectors corresponding to dominant eigenvalues of this operator are the principal components of the set of shifts of $f$.
\end{remark}

\noindent In what follows, we shall apply more subtle invariant operators than the averaging $K$ in the above example. In particular, by applying an operator which keeps only the large Fourier coefficients of each $\ab$-diagonal of $f\otimes \overline{f}$ (roughly speaking), we will see that the resulting dominant eigenvectors are important \emph{quadratic}-Fourier-analytic components of $f$. This is the first non-classical level of the general \emph{order-increment principle} mentioned in the introduction. To detail all of this, we first need to discuss more precisely the notions of higher-order structures and characters that we will be able to capture in the spectral approach. We do this in the next two subsections.

\subsection{Structured functions of order $k$}\label{subsec:kstrucfns}\hfill\smallskip\\
An important principle in higher-order Fourier analysis is the dichotomy between structure and randomness relative to the Gowers norms. A 1-bounded function $f:\ab\to\mb{C}$ is said to be \emph{quasirandom of order $k$} if $\|f\|_{U^{k+1}}\le \delta$ for some small $\delta>0$, and a more precise aspect of the dichotomy is the decomposition of a function into a sum of a quasirandom part and a structured part, where \emph{structure of order $k$} can be defined in various interrelated ways based on the idea of being orthogonal to all quasirandom functions of order $k$. In infinite settings, such as the ergodic theoretic one  \cite{HK}, or the ultraproduct setting \cite{S-hofa1}, this orthogonality can be defined exactly, and the notion of structure is captured in an exact way by certain subspaces of a Hilbert space, corresponding to certain $\sigma$-algebras or \emph{factors} (see also Remark \ref{rem:FSA}). However, in the setting of finite abelian groups, which we focus on in this paper, the notion of higher-order structure takes \emph{approximate} and \emph{quantitative} (or \emph{parametric}) forms. In this subsection we give the precise definition of \emph{structured function of order $k$} that we shall use in the rest of the paper. We discuss this first for order $1$, and then generalize to higher orders. 

For order 1 (i.e., in the setting of classical Fourier analysis), we use the following notion.
\begin{defn}[Fourier-structured functions]\label{def:fsfunc}
Let $\ab$ be a finite abelian group. A function $f:\ab\to\mb{C}$ is \emph{$(R,\delta)$-Fourier-structured} if there is a function $g:\ab\to\mb{C}$ that is a linear combination of at most $R$ Fourier characters on $\ab$ and that satisfies $\|f-g\|_2\leq\delta$.
\end{defn}
\begin{remark}
This property is almost identical to the notion of quantitative \emph{almost-periodicity} from \cite[Definition 10.34]{T-V}, but note that we do not require the coefficients in the linear combination to be 1-bounded. For functions $f$ satisfying $\|f\|_2\leq 1$, the two properties are equal. Indeed, if $f$ is $(R,\delta)$-Fourier-structured and $\|f\|_2\leq 1$, then we can see that $f$ is  $(R,\delta)$-almost-periodic as follows. Letting $g=\sum_{i=1}^R c_i\chi_i$ be the linear combination of characters $\chi_i$ with $\|f-g\|_2\leq \delta$, we know that the orthogonal projection to the subspace spanned by these characters, namely $g'=\sum_{i=1}^R \wh{f}(\chi_i)\chi_i$, satisfies $\|f-g'\|_2\le \delta$. Moreover, it follows from Parseval's identity that $|\wh{f}(\chi_i)|\leq \|f\|_2\leq 1$ for each $i$, and $f$ is thus $(R,\delta)$-almost-periodic as per \cite{T-V}. Naming this property in terms of \emph{structure} rather than almost-periodicity, as we do, is motivated by the developments below, which will introduce higher-order generalizations of this property that have a more direct connection with the notion of structure (as in the structure-randomness dichotomy) than with almost-periodicity.
\end{remark}
\noindent Before we motivate Definition \ref{def:fsfunc} further, let us recall other known tools to measure the Fourier structure of a function. These involve certain norms. A central example is the Wiener norm (also known as the algebra norm, or the spectral norm; see \cite{Green&Sanders, Taoerg}), which we recall here.
\begin{defn}
The \emph{Wiener norm} of a function $f:\ab\to\mb{C}$ is $\|f\|_{A(\ab)}:=\|\wh{f}\|_{\ell^1}$.
\end{defn}
\noindent Being $(R,\delta)$-Fourier-structured is closely related to being bounded in the Wiener norm. We first show that the latter property implies the former, using the next lemma. 

Recall that for a metric space $X$ and a function $f:X\to \mb{C}$, the \emph{support} of $f$, denoted $\Supp(f)$, is the closure of the set $\{x\in X:f(x)\not=0\}$.
\begin{lemma}\label{lem:abstineq}
Let $S$ be a finite set, let $f:S\to\mb{C}$, and let $\delta>0$. Let $g:S\to\mb{C}$ be the function $g(s)=f(s)\,\mathbf{1}(|f(s)|\geq\delta)$. Then $|\Supp(g)|\leq\delta^{-1}\|f\|_{\ell^1(S)}$ and $\|f-g\|_{\ell^2(S)}\leq\delta\|f\|_{\ell^1(S)}$.
\end{lemma}

\begin{proof}
It is clear that $\delta\,|\Supp(g)|\leq\|g\|_{\ell^1}\leq\|f\|_{\ell^1}$, which proves the first claim. On the other hand, by H\"older's inequality we have $\|f-g\|_{\ell^2}\leq\|f-g\|_{\ell^\infty}\|f-g\|_{\ell^1}\leq\delta\|f\|_{\ell^1}$. 
\end{proof}

\begin{corollary}\label{cor:strucbound}
For any $\delta>0$, any function $f:\ab\to\mb{C}$ is $(\delta^{-1}\|f\|_{A(\ab)}^2,\delta)$-Fourier-structured.
\end{corollary}

\begin{proof} Applying Lemma \ref{lem:abstineq} with $\wh{f}$ and $\delta>0$, we obtain that there is a function $\wh{g}:\wh{\ab}\to\mb{C}$ whose support has size at most $\delta^{-1}\|f\|_{A(\ab)}$ and such that $\|\wh{f}-\wh{g}\|_{\ell^2}\leq\delta$. The reverse Fourier transform $g$ of $\wh{g}$ is a linear combination of at most $\delta^{-1}\|f\|_{A(\ab)}$ characters, and satisfies $\|f-g\|_2=\|\wh{f}-\wh{g}\|_{\ell^2}\leq\delta$. Thus $f$ is $(\delta^{-1}\|f\|_{A(\ab)},\delta\|f\|_{A(\ab)})$-Fourier-structured, whence, changing $\delta\|f\|_{A(\ab)}$ to $\delta$, the result follows.
\end{proof}

\begin{remark}\label{rem:U2dualbound}
More generally, if for some $\alpha\in (0,1]$ the norm $\|\wh{f}\|_{\ell^{2-\alpha}}$ is bounded above by some constant $C$, then for every $\delta>0$ it follows similarly that $f$ is $(C/\delta^{2-\alpha},\delta^\alpha C)$-Fourier-structured. In particular, this is the case for the $U^2$-dual norm, since we have the well-known formula $\|f\|_{U^2}^*=\|\wh{f}\|_{\ell^{4/3}}$.
\end{remark}
\noindent Let us now consider the converse direction, i.e., whether $(R,\delta)$-Fourier-structure implies being bounded in the Wiener norm or in the $U^2$-dual norm. We first recall the following example, which rules out a naive converse to Corollary \ref{cor:strucbound}.
\begin{example}
Let $\mb{Z}_N$ be the finite cyclic group of order $N$, and let  $f:\mb{Z}_N\to \{0,1\}$ be the indicator function of the interval $[1,\lfloor N/2\rfloor]$. We have the well-known fact that $\|f\|_{A(\ab)}=\Omega(\log N)$, whereas any norm $\|\wh{f}\|_{\ell^p}$ with $p>1$ is bounded independently of $N$. In particular $\|f\|_{U^2}^*$ is bounded, so by Remark \ref{rem:U2dualbound} the function $f$ is $(O(\delta^{-2/3},O(\delta^{4/3}))$-Fourier-structured for every $\delta$, independently of $N$.
\end{example}
\noindent Notwithstanding this example, we have the following equivalence result, providing an alternative expression for Fourier structure, using the $U^2$-dual norm instead of the Fourier transform.
\begin{lemma}\label{lem:order1equiv}
Consider the following properties of functions $f:\ab\to\mb{C}$.
\setlength{\leftmargini}{0.8cm}
\begin{enumerate}
    \item For some $R> 0$ and $\delta\in [0,1]$, the function $f$ is $(R,\delta)$-Fourier-structured.
    \item For some $R'> 0$ and $\delta'\in [0,1]$, there exists $g:\ab\to\mb{C}$ with $\|g\|_{U^2}^*\leq R'$ and $\|f-g\|_2\leq \delta'$.
\end{enumerate}
These properties are equivalent in the following sense: if $(i)$ holds with $(R, \delta)$ then $(ii)$ holds with $\delta'=\delta$ and $R'=R^{3/4}\|f\|_2$; conversely, if $(ii)$ holds with $(R', \delta')$ then, for every $\alpha>0$,  $(i)$ holds with $R=\alpha^{-6}{R'}^4(\|f\|_2+1)^2$ and $\delta=\alpha+\delta'$.
\end{lemma}
\begin{proof}
If $(i)$ holds, then, letting $g$ be the projection of $f$ to the linear span of the given $R$ characters, we know that $\|f-g\|_2\leq \delta$, and $g$ has Fourier coefficients of modulus at most $\|f\|_2$, so $\|g\|_{U^2}^*\leq R^{3/4} \|f\|_2$, whence $(ii)$ holds as claimed.

If $(ii)$ holds, then we claim that for any $\gamma>0$ the function $g$ is $(\gamma^{-2}(\|f\|_2+1)^2,\gamma^{1/3}{R'}^{2/3})$-Fourier structured. Indeed, 
let $S=\{\chi\in \wh{\ab}: |\wh{g}(\chi)|>\gamma\}$, and define $h(x):=\sum_{\chi\in S} \wh{g}(\chi) \chi(x)$. By Plancherel's theorem we have $\|g-h\|_2^2=\sum_{\chi\not\in S} |\wh{g}(\chi)|^2\leq \gamma^{2/3} \,\|\wh{g}\|_{\ell^{4/3}}^{4/3}\leq \gamma^{2/3}{R'}^{4/3}$. We also have $\gamma^2|S|\leq \|g\|_2^2\leq (\|f\|_2+\delta')^2$. Our claim follows. Then the original function $f$ is $(\gamma^{-2}(\|f\|_2+1)^2,\gamma^{1/3}{R'}^{2/3}+\delta')$-Fourier structured. Choosing $\gamma$ so that $\gamma^{1/3}{R'}^{2/3}=\alpha$, we obtain that $(i)$ holds as claimed.
\end{proof}
\noindent The statement of property $(ii)$ in Lemma \ref{lem:order1equiv} clearly suggests the following natural generalization to higher orders. This is the main notion of higher-order structure used in this paper.
\begin{defn}[Structured functions of order $k$]\label{def:kstruct}
Let $k$ be a positive integer. We say that a function $f:\ab\to\mb{C}$ is \emph{$(R,\delta)$-structured of order $k$} if there is a function $g:\ab\to\mb{C}$ satisfying $\|g\|_{U^{k+1}}^*\leq R$ and $\|f-g\|_2\leq \delta$.  
\end{defn}

\begin{remark}\label{rem:dualdecrease}
Since the $U^k$ norms form an increasing sequence (i.e.\ $\|f\|_{U^k}\leq \|f\|_{U^{k+1}}$ for all $k$), the $U^k$-dual norms form a decreasing sequence. It follows that $(R,\delta)$-structure of order $k$ becomes a weaker (more inclusive) property as $k$ increases.
\end{remark}
\begin{remark}
There are other higher-order parametric notions of structure, for instance the one based on the \emph{uniform almost-periodicity norms} introduced by Tao in \cite[Definition 5.2]{Taoerg}, which proceeds by generalizing the Wiener norm instead of the $U^2$-dual norm. The notion in Definition \ref{def:kstruct} turned out to be sufficiently natural and technically convenient for us.
\end{remark}
\noindent We know, by Lemma \ref{lem:order1equiv}, that structure of order 1 means proximity in $L^2$ to a combination of a few Fourier characters. Before we develop further aspects of structure of order $k$ for general $k>1$, let us observe that structure of order $0$ also has a useful meaning, even though this case is peculiar in the sense that $\|f\|_{U^1}:=|\mb{E}_{x\in\ab} f(x)|$ is a seminorm and not a norm. Indeed, let us say that a function $f:\ab\to\mb{C}$ is \emph{$(R,\delta)$-structured of order 0} if there exists $g:\ab\to\mb{C}$ such that $\|f-g\|_2\leq\delta$ and such that the seminorm $\|g\|_{U^1}^*:=\sup_{h:\ab\to\mb{C},\,\|h\|_{U^1}\leq 1} |\langle g,h\rangle|$ is at most $R$.
\begin{lemma}\label{lem:struct0}
A function $f:\ab\to\mb{C}$ is $(R,\delta)$-structured of order 0 for some $R$ if and only if there is a constant function $g$ such that $\|f-g\|_2\leq \delta$.
\end{lemma}
\begin{proof}
For the forward implication, let $g$ be such that $\|f-g\|_2\leq\delta$ and $\|g\|_{U^1}^*\leq R$. Note that, for each non-trivial character $\chi\in \wh{\ab}\setminus\{\id\}$, we have $\|\chi\|_{U^1}=0$. Hence, for every integer $n\ge 1$, we have $\|n\chi\|_{U^1}\leq 1$ and therefore $n|\wh{g}(\chi)|=|\langle g,n\chi\rangle|\le \|g\|_{U^1}^*\le R$, whence $|\wh{g}(\chi)|\le R/n$. Taking the limit as $n\to\infty$ we deduce that $\wh{g}(\chi)=0$ for all $\chi\not=\id$, so $g$ is constant.

For the converse, note that if $g$ is constant and $\|f-g\|_2\leq\delta$, then $\|g\|_{U^1}^*=|g|\leq \|f\|_2+\delta$, so $f$ is $(R,\delta)$-structured of order $0$ for $R=\|f\|_2+\delta$.
\end{proof}

\noindent To elucidate further the \emph{structure of order $k$} property from Definition \ref{def:kstruct}, and for other uses later on, let us establish the following equivalent formulation.
\begin{proposition}\label{prop:kstructequiv}
Consider the following two properties of functions $f:\ab\to\mb{C}$.
\setlength{\leftmargini}{0.8cm}
\begin{enumerate}
    \item The function $f$ is $(R,\delta)$-structured of order $k$.
    \item For every 1-bounded function $h:\ab\to\mb{C}$ satisfying $\|h\|_{U^{k+1}}\leq a\leq 1$, we have $|\langle f,h\rangle| \leq b$.
\end{enumerate}
For 1-bounded functions $f:\ab\to\mb{C}$, these properties are equivalent in the following sense: if $(i)$ holds with $(R, \delta)$ then $(ii)$ holds with any $a$ and $b=Ra+\delta$; conversely, if $(ii)$ holds with $(a,b)$, then $(i)$ holds with some $R=O_{a}(1)$ and $\delta=2(a^{1/2}+b^{1/2})$.
\end{proposition}
\begin{proof}
If $(i)$ holds with $(R, \delta)$, then $|\langle f,h\rangle|\leq |\langle g,h\rangle| + |\langle f-g,h\rangle|\leq \|g\|_{U^{k+1}}^*\|h\|_{U^{k+1}} + \|f-g\|_2 \|h\|_2\leq Ra+\delta$, so $(ii)$ holds as claimed.

For the converse, suppose that $(ii)$ holds with $a,b$. Let $
\nu>0$ (to be fixed later) and $\mc{G}:\mb{R}_+\to\mb{R}_+$ be the function $\mc{G}(m)=\frac{\nu}{2(m+1)}$. By a standard regularity lemma, for example \cite[Corollary 5.2]{GStruct},\footnote{Note that this result is stated for $\mb{R}$-valued functions. For $\mb{C}$-valued ones we can simply separate the real and imaginary parts and apply \cite[Corollary 5.2]{GStruct} to each.}  there is $R=R(\nu,\mc{G})>0$ such that we have a decomposition $f=f_s+f_e+f_r$ where $\|f_s\|_{U^{k+1}}^*\leq R$, $\|f_e\|_2\leq\nu/2^{3/2}$,  $\|f_r\|_{U^{k+1}}\leq \mc{G}(\|f_s\|_{U^{k+1}}^*)$, and $|f_r|\leq 2$. Note that $|\langle f,f_r\rangle|=|\langle f_s,f_r\rangle+\langle f_e,f_r\rangle+\|f_r\|_2^2|\geq\|f_r\|_2^2-|\langle f_s,f_r\rangle |-|\langle f_e,f_r\rangle |\geq \|f_r\|_2^2-\|f_s\|_{U^{k+1}}^*\,\frac{\nu} {2(\|f_s\|_{U^{k+1}}^*+1)}-2^{1/2}\nu/2^{3/2}\ge \|f_r\|_2^2-\nu$. Hence $\|f_r\|_2\leq (|\langle f,f_r\rangle|+\nu)^{1/2}$, so
\begin{equation}\label{eq:l2dest2}
\|f-f_s\|_2\leq\nu+(|\langle f,f_r\rangle|+\nu)^{1/2}.
\end{equation}
Let $\nu = a$ (and so, $R=O_a(1)$) and note that, if $h=f_r/2$, we have $\|h\|_\infty\le 1$ and $\|h\|_{U^{k+1}}\le  \frac{\nu}{4(\|f_s\|_{U^{k+1}}^*+1)}\le a$. Hence, applying $(ii)$ with this function $h$, we have $|\langle f,f_r/2\rangle|\leq b$, so by \eqref{eq:l2dest2} we deduce that $\|f-f_s\|_2\leq a+(2b+a)^{1/2}$. Hence $(i)$ holds with $R$ and $\delta$ as claimed.
\end{proof}
\begin{remark}\label{rem:FSA}
Property $(ii)$ in Proposition \ref{prop:kstructequiv} is a formulation of $k$-th order structure which is strongly motivated from several viewpoints on higher-order Fourier analysis in infinite settings. Specifically, in the ergodic-theory setting from \cite{HK}, as well as in the ultralimit setting of \cite{S-hofa1}, or in the setting of cubic couplings \cite{CScouplings}, the notion of higher-order structure is captured by certain $\sigma$-algebras (the Host--Kra characteristic factors in \cite{HK} and their generalizations in \cite{CScouplings}, or the higher-order Fourier $\sigma$-algebras in \cite{S-hofa1}), where measurability relative to such a $\sigma$-algebra of order $k$ can be defined as orthogonality to every function with vanishing $U^{k+1}$-seminorm (see \cite[Theorem 1]{S-hofa1} or \cite[Lemma 5.8]{CScouplings}). Property $(ii)$ above is a finite analogue of this.
\end{remark}

\medskip

\subsection{Characters of order $k$}\label{subsec:chark}\hfill\smallskip\\
Here we use the notion of higher-order structured function from the previous subsection to define higher-order analogues of Fourier characters on abelian groups. In particular, quadratic (i.e.\ order 2) characters will play a key role in our algorithmic applications. By $\ab$ we continue to denote a finite abelian group, but note that many of the following notions still hold for more general compact abelian groups.
\begin{defn}\label{def:k-char}
Let $k$ be a positive integer. An \emph{$(R,\delta)$-character of order $k$} on a finite abelian group $\ab$ is a function $f:\ab\to\mb{C}$ with $\|f\|_2\leq 1$ such that for every $t\in \ab$, the multiplicative derivative $\Delta_t f$ is $(R,\delta)$-structured of order $k-1$. We refer to $R$ as the \emph{complexity parameter}, and refer to $\delta$ as the \emph{precision parameter}. When it is unnecessary to specify these parameters, we simply call $f$ a \emph{character of order $k$} on $\ab$.
\end{defn}
\noindent Many central examples of such characters are actually 1-bounded functions. The weaker bound $\|f\|_2\leq 1$ in the definition is useful from linear-algebraic and functional-analytic viewpoints.
\begin{remark}
Using that the $U^k$-dual norms decrease with $k$ (see Remark \ref{rem:dualdecrease}), we see that for every function $f:\ab\to\mb{C}$ with $\|f\|_2\leq 1$, we have $\|f\|_{U^{k+1}}^*\leq \|f\|_{U^2}^*\leq |\ab|^{3/4}$, so any such function is a $(|\ab|^{3/4},0)$-character of order $k$. The notion  of $(R,\delta)$-character of order $k$ is interesting only when $\ab$ is large, $R$ is bounded and $\delta$ is small. In what follows, we will see examples and theorems involving such characters where $R$ depends on $\delta$ but not on $|\ab|$. 
\end{remark}
\noindent Before we look at some basic examples of higher-order characters, let us briefly recall the notion of Gowers $U^k$-products \cite[p. 419]{T-V} and the Gowers-Cauchy-Schwarz inequality \cite[(11.6)]{T-V}.

\begin{defn}[Gowers $U^k$-product]\label{def:gower-u-k-prod}
Let $\ab$ be a finite abelian group, $k\ge 1$, and let $\{f_v\in \mb{C}^{\ab}\}_{v\in \db{k}}$ be a set of functions (where $\db{k}$ denotes $\{0,1\}^k$). The \emph{Gowers $U^k$-product} $\langle (f_v)_{v\in \db{k}}\rangle_{U^k}$ of these functions is defined as $\mb{E}_{x,t_1,\dots,t_k\in \ab}~ \prod_{v\in\{0,1\}^k}\mc{C}^{|v|}f_v(x+v\sbr{1}\, t_1+\cdots+v\sbr{k}\,t_k)$. The \emph{Gowers-Cauchy-Schwarz inequality} states that $|\langle (f_v)_{v\in \db{k}}\rangle_{U^k}|\le \prod_{v\in\db{k}}\|f_v\|_{U^k}$.
\end{defn}

\begin{example}\label{ex:globalpolyphases}
The most basic example of a character of order $k$ is a global polynomial phase function of degree-$k$ on $\ab$, i.e.\ a function $\phi:\ab\to\mb{C}$ such that $\Delta_{t_1}\cdots \Delta_{t_{k+1}}\phi(x)=1$ for all $x,t_1,\ldots,t_{k+1}\in \ab$. Indeed, such a function is in fact a $(1,0)$-character of order $k$. To see this, note that it follows from the definition of $\phi$ that $|\phi(x)|=1$ for all $x\in\ab$. Then, for any fixed $t\in\ab$, note that $\Delta_t \phi$ is a polynomial phase function of degree $k-1$, so for every $x,t_1,\ldots,t_k\in\ab$ we have $\overline{\Delta_t \phi(x)}=\prod_{v\in \{0,1\}^k\setminus 0^k} \mc{C}^{|v|} \Delta_t \phi(x+v\sbr{1}t_1+\cdots+v\sbr{k}t_k)$ (where $\mc{C}$ denotes the complex-conjugation operator and $0^k=(0,\ldots,0)$). It follows that for every function $g:\ab\to\mb{C}$ with $\|g\|_{U^k}\leq 1$, we have $\langle g, \Delta_t \phi\rangle = \langle g, \Delta_t \phi,\Delta_t \phi,\ldots, \Delta_t \phi\rangle_{U^k}\leq \|g\|_{U^k} \|\Delta_t \phi\|_{U^k}^{2^k-1}\leq 1$, whence $\|\Delta_t \phi\|_{U^k}^*\leq 1$, which proves that $\Delta_t \phi$ is $(1,0)$-structured of order $k-1$, as required. 

More complicated examples of higher-order characters involve nilpotent structures. In particular, nilsequences with so-called \emph{vertical frequency}, called \emph{nilcharacters} in \cite[Definition 6.1]{GTZ}, and their generalizations introduced in the present paper, called \emph{nilspace characters}, provide important examples of higher-order characters (see Section \ref{sec:nschars}).
\end{example}

To illustrate Definition \ref{def:k-char} further, let us consider in more detail its special case for $k=1$. This tells us that $f$ is an $(R,\delta)$-character of order 1 if for every $t\in \ab$ we have that $\Delta_t f$ is $(R,\delta)$-structured of order $0$. By Lemma \ref{lem:struct0}, this implies that there is a constant function $g_t$ such that $\|g_t-\Delta_t f\|_2\leq \delta$. If $\delta$ is sufficiently small, then this in turn implies indeed, as it should, that such a function $f$ is close in $L^2$ to a scalar multiple of a Fourier character. 
\begin{proposition}\label{prop:ord1char}
If $f:\ab\to\mb{C}$ is an $(R,\delta)$-character of order 1 with $\delta\in (0,\|f\|_2/20]$, then there exists $\chi\in \wh{\ab}$ such that $\|f-\wh{f}(\chi)\chi\|_2 \le 4\delta^{1/2}\|f\|_2^{1/2}$.
\end{proposition}
\noindent There are various ways to prove this by reducing to a simple instance of a so-called $99\%$ \emph{inverse theorem} for the $U^2$-norm (for more on this topic, see \cite{E&T}). Let us give a short proof here. 
\begin{proof}
By assumption, for every $t\in \ab$ there is a constant $g_t$ such that $\|\Delta_tf-g_t\|_2\leq \delta$. Hence $\big|\|\Delta_t f\|_{U^1}-|g_t|\big|\leq |\mb{E} ( \Delta_t f-g_t)|\leq \delta$. Then $\mb{E}_t |g_t|\leq \|f\|_1^2+\delta\leq \|f\|_2^2+\delta$, and, writing $h_t:=\Delta_t f-g_t$, we have $\mb{E}_t |g_t|^2 = \mb{E}_t\mb{E}_x |g_t|^2 \geq \mb{E}_t\mb{E}_x (|\Delta_t f(x)|^2-2|h_t| |\Delta_t f|)\geq \|f\|_2^4-2\delta \|f\|_2^2$. Hence $\|f\|_{U^2}^4=\mb{E}_t \|\Delta_t f\|_{U^1}^2 \geq \mb{E}_t |g_t|^2 -2\delta \mb{E}_t |g_t| +\delta^2\geq \|f\|_2^4-4\delta \|f\|_2^2-\delta^2$. Completing squares and rearranging, we deduce that $\|f\|_{U^2}^2\geq \|f\|_2^2-5\delta$ (which is an assumption for a $99\%$-inverse-theorem). Let $x=\|\wh{f}\|_\infty$. Using that $\|f\|_2\geq x \geq \|f\|_{U^2}^2/\|f\|_2$, we have $x (\|f\|_2-x)\leq \|f\|_2^2-\|f\|_{U^2}^2\leq 5\delta$, so either $x\leq \|f\|_2 \frac{1-\sqrt{1-20\delta/\|f\|_2}}{2}$, or $x\geq \|f\|_2 \frac{1+\sqrt{1-20\delta/\|f\|_2}}{2}$. The former case combined with $x\geq \|f\|_{U^2}^2 / \|f\|_2$ implies $\|f\|_{U^2}^2\leq 10\delta\|f\|_2$, and this combined with $\|f\|_2^2\le \|f\|_{U^2}^2+5\delta$ implies $\|f\|_2\leq 5(\delta+\sqrt{\delta^2+\delta/5})\leq 4\sqrt{\delta}$ (using that $\delta\leq 1/20$), so the conclusion holds with any $\chi$. The latter case implies the conclusion with $\chi$ such that $|\wh{f}(\chi)|=\|\wh{f}\|_\infty$, as then $\|f-\wh{f}(\chi)\chi\|_2^2 =\|f\|_2^2-\|\wh{f}\|_\infty^2\leq 15\delta\|f\|_2$.
\end{proof}
\noindent This last result can be proved with better constants by a finer (but longer) argument using the Fourier coefficients of $f$. We omit this, as such an improvement is not needed in this paper.

It will be useful to introduce the following generalization of higher-order characters.

\begin{defn}[Weak character of order $k$]\label{def:wqc}
An \emph{$(R,\delta_1,\delta_2)$-weak character of order $k$} on a finite abelian group $\ab$ is a function $f:\ab\to\mb{C}$ with $\|f\|_2\leq 1$ such that there is a set $S\subset \ab$ with $|S|\geq|\ab|(1-\delta_2)$ and such that for every $t\in S$ the function $\Delta_t f$ is $(R,\delta_1)$-structured of order $k-1$. We refer to $R$ as the \emph{complexity parameter} and to $\delta_1,\delta_2$ as the \emph{precision parameters} of $f$.
\end{defn}
\noindent There is a fact that we must now establish, in order to connect properly the notions of higher-order structure and higher-order characters. Namely, we need to show that characters of order $k$ are structured functions of order $k$. In fact, this holds more generally for \emph{weak} characters of order $k$.

\begin{lemma}\label{lem:wqc-abprop}
Let $g$ be an $(R,\delta_1,\delta_2)$-weak-character of order $k$ on a finite abelian group $\ab$. Then for every 1-bounded function $f:\ab\to\mb{C}$ satisfying $\|f\|_{U^{k+1}}\leq a\leq 1$, we have $|\langle f,g\rangle|\leq (a^2 R+\delta_1+\delta_2^{1/2})^{1/2}$. In particular $g$ is an $(R',\delta')$-structured function of order $k$, for $R'=O_a(1)$ and $\delta'=2(a^{1/2}(R^{1/4}+1)+\delta_1^{1/4}+\delta_2^{1/8})$.
\end{lemma}
\begin{proof}
By Definition \ref{def:wqc}, for a set $S\subset \ab$ of cardinality at least $(1-\delta_2)|\ab|$, for every $t\in S$ there is a function $\mc{E}_t:\ab\to\mb{C}$ with $\|\mc{E}_t\|_2\leq \delta_1$ and a function $h_t:\ab\to\mb{C}$ with $\|h_t\|_{U^k}^*\leq R$, such that $\Delta_t g = h_t+\mc{E}_t$. Then, as
$\|\Delta_t f\|_2\leq 1$ and $(\mb{E}_t \|\Delta_t g\|_2^2)^{1/2}= \|g\|_2^2\leq 1$, we have
\begin{align*}
|\langle f,g\rangle|^2 & =  \mb{E}_t \langle \Delta_t f,\Delta_t g\rangle  =  \mb{E}_t 1_S(t)\, \langle\Delta_t f, h_t \rangle + \mb{E}_t  1_S(t)\, \langle \Delta_t f, \mc{E}_t\rangle + \mb{E}_t 1_{S^c}(t)\, \langle \Delta_t f,  \Delta_t g\rangle \\
& \leq  \mb{E}_t 1_S(t)\,  \|\Delta_t f\|_{U^k}\,  \|h_t\|_{U^k}^* + \mb{E}_t 1_S(t)\, \|\Delta_t f\|_2\, \|\mc{E}_t\|_2 + \mb{E}_t 1_{S^c}(t) \|\Delta_t f\|_2\|\Delta_t g\|_2\\
& \leq \mb{E}_t 1_S(t)\, \|\Delta_t f\|_{U^k} \,  \|h_t\|_{U^k}^* + \delta_1 \mb{E}_t \|\Delta_t f\|_2+ (\mb{E}_t 1_{S^c}(t))^{1/2} (\mb{E}_t\|\Delta_t g\|_2^2)^{1/2}.\\
& \leq \mb{E}_t 1_S(t)\, \|\Delta_t f\|_{U^k} \,  \|h_t\|_{U^k}^* + \delta_1+ \delta_2^{1/2}.
\end{align*}
By H\"older's inequality, letting $q=2^k/(2^k-1)$, we have that $\mb{E}_t 1_S(t)\, \|\Delta_t f\|_{U^k} \, \|h_t\|_{U^k}^*$ is at most $(\mb{E}_t  \|\Delta_t f\|_{U^k}^{2^k})^{1/2^k} \,  (\mb{E}_t 1_S(t)\, {\|h_t\|_{U^k}^*}^q)^{1/q}$  $=\|f\|_{U^{k+1}}^2\, (\mb{E}_t  1_S(t)R^q)^{1/q}$. Hence $|\langle f,g\rangle|^2\leq a^2 R +\delta_1+\delta_2^{1/2}$, as claimed. By Proposition \ref{prop:kstructequiv}, it follows that $g$ is $(R',\delta')$-structured of order $k$ with $R'=O_a(1)$ and $\delta'=2(a^{\frac{1}{2}}+b^{\frac{1}{2}})$, where $b=(a^2 R +\delta_1+\delta_2^{\frac{1}{2}})^{\frac{1}{2}}\leq aR^{\frac{1}{2}}+\delta_1^{\frac{1}{2}}+\delta_2^{\frac{1}{4}}$.
\end{proof}
\begin{remark}
It is now natural to wonder how special the characters of order $k$ are among the structured functions of order $k$. We have seen in Proposition \ref{prop:ord1char} that characters of order 1 are essentially scalar multiples of Fourier characters (up to small $L^2$-errors), and these order-1 characters are indeed special among structured functions of order 1, we may even say that they are \emph{fundamental} in the sense that every structured function of order 1 is a combination of a bounded number of characters of order 1 (up to a small $L^2$-error), by Fourier analysis. This picture generalizes to higher orders; we shall discuss this in Section \ref{sec:nschars} using the concept of \emph{nilspace character}. Specifically, we will show that every structured function of order $k$ is, up to a small $L^2$-error, a sum of a bounded number of $k$-step nilspace characters (this will follow from the regularity lemma Theorem \ref{thm:UpgradedReg}). Moreover, focusing on the quadratic case, we will also show that 2-step nilspace characters are characters of order $2$ as per Definition \ref{def:k-char}. We expect this latter fact to extend to general order $k$, but as explained in later sections, proving this is outside the scope of this paper.
\end{remark}

\noindent As mentioned in the introduction, we aim to provide algorithms to obtain decompositions of functions into structured and quasirandom parts or order 2 (i.e.\ relative to the $U^3$-norm). A key step in our approach consists in modifying $\ab$-matrices by applying an operator to the $\ab$-diagonals which performs a \emph{Fourier regularization} of these diagonals, i.e., which isolates a suitable Fourier-structured part of these functions. We shall now introduce this operator and develop some of its main properties.

\subsection{The Fourier denoising operator $K_\varepsilon$}\label{subsec:Keps}\hfill\smallskip\\
In this subsection, we define the main notion of Fourier regularization that we shall use to develop our spectral approach in quadratic Fourier analysis. 

Informally speaking, the goal is to define a notion on finite abelian groups satisfying properties similar to those of the expectation operator corresponding to the first-order Fourier $\sigma$-algebra $\mc{F}_1$ in the ultraproduct setting, or the Kronecker factor in ergodic theory (see Remark \ref{rem:FSA}). In other words, we want a tool that decomposes a function on a finite abelian group into a structured and a quasirandom part of order 1 (i.e.\ relatively to classical Fourier analysis), and which satisfies certain additional properties that are useful from algorithmic viewpoints. In the infinite setting of ultraproducts (see \cite{S-hofa1}, or later \cite{CSinverse}), the regularization has the form $f=f_s+f_r$ where $f_s=\mb{E}(f|\mc{F}_1)$ while $f_r=f-f_s$. The function $f_r$ is also the projection of $f$ to the orthogonal complement of $L^2(\mc{F}_1)$. Thus, both operators $f\mapsto f_s$ and $f\mapsto f_r$ are projections in the infinite setting, so both operators are contractions relative to the $L^2$-norm. In particular, we have the continuity properties
\[
\|f-f'\|_2\geq\|f_s-f'_s\|_2\;\textrm{ and }\; \|f-f'\|_2\geq\|f_r-f'_r\|_2.
\]
In the finite setting, a first candidate for the desired tool could be the Fourier cut-off operator mentioned in Subsection \ref{subsec:outline}, namely the operator which, for a fixed $\varepsilon>0$, simply eliminates the Fourier coefficients of $f$ that have modulus less than $\varepsilon$. Unfortunately, this operator is too blunt and does not have suitable features such as the above contraction property. However, there is a variant of this operator that does have such properties. Before we describe it, let us mention another very useful fact that holds in infinite settings: the \emph{translation invariance} whereby we have $(T^hf)_s=T^h(f_s)$ and $(T^hf)_r=T^h(f_r)$ for every group element $h$. To guarantee this property, we consider operators of the following form.
\begin{defn}
Let $r:\mb{R}_{\geq 0}\to\mb{R}_{\geq 0}$ be an arbitrary function with $r(0)=0$. For $z\in\mb{C}$  we define $q_r(z):=z\, r(|z|)/|z|$ for $z\neq 0$ and $q_r(0)=0$. We then define an operator $K_r$ on $\mb{C}^{\ab}$ as follows: for any function $f\in \mb{C}^{\ab}$, letting $f=\sum_{\chi\in\wh{\ab}} \wh{f}(\chi) \chi$ be its Fourier expansion, we define 
\begin{align}\label{eq:genFourierop} K_r(f):=\textstyle\sum_{\chi\in\wh{\ab}}q_r(\wh{f}(\chi))\; \chi. 
\end{align}
\end{defn}
\noindent Since translation only affects the phase of the Fourier coefficients, we have $T^h K_r(f)=K_r(T_hf)$. It is also clear that $\overline{K_r(f)}=K_r(\overline{f})$. Hence, $K_r$ is invariant as per Definition \ref{def:invop}.

We now want to identify a simple property of functions $r$ which ensures that the associated operator $K_r$ has the desired continuity property.

Recall that a function $g:\mb{C}\to\mb{C}$ is said to be \emph{$C$-Lipschitz} if for all $x$ and $y$ in the domain of $g$ we have $|g(x)-g(y)|\leq C |x-y|$. Similarly, given a map $K:L^2(X)\to L^2(X)$ (for some topological space $X$ with a Borel probability measure), we define the Lipschitz constant
\begin{equation}\label{eq:defLip2norm}
\|K\|_{\Lip}:=\inf\{C\geq 0: \textrm{for all }f,g \in L^2(X),\textrm{ we have }\|K(f)-K(g)\|_2\leq C\|f-g\|_2\}.
\end{equation}
When we need to specify the domain of $K$ we write $\|K\|_{\Lip(L^2(X))}$.
The following lemma identifies a useful property of the kind alluded to above.
\begin{lemma}\label{lem:cont-lip-op}
Let $r:\mb{R}_{\geq 0}\to \mb{R}_{\geq 0}$ with $r(0)=0$ be such that $q_r$ is $C$-Lipschitz for some constant $C$. Then for any finite abelian group $\ab$ we have  $\|K_r\|_{\Lip(\mb{C}^{\ab})}\leq C$.
\end{lemma}

\begin{proof}
By Plancherel's theorem we have
$\|K_r(f)-K_r(g)\|_2^2\!=\!\sum_{\chi\in\wh{G}} |q_r(\wh{f}(\chi))-q_r(\wh{g}(\chi))|^2\!\leq\!\sum_{\chi} C^2 |\wh{f}(\chi)-\wh{g}(\chi)|^2 =  C^2\|f-g\|_2^2$.
\end{proof}

Now we focus on a specific function $r$ and define the principal operator that we shall use.
\begin{defn}[Fourier denoising operator]\label{def:cont-cutoff}
For $\varepsilon>0$, we denote by $K_\varepsilon$ the operator defined by applying \eqref{eq:genFourierop} with the specific function $r=r_\varepsilon:\mb{R}_{\geq 0}\to \mb{R}_{\geq 0}$, $x\mapsto \max(x-\varepsilon,0)$.
\end{defn}

\begin{lemma}\label{lem:pointcontraction}
The function $q_\varepsilon:=q_{r_\varepsilon}$ is a $1$-Lipschitz function on $\mb{C}$, whence the operator $K_\varepsilon$ is $1$-Lipschitz on $\mb{C}^{\ab}$.
\end{lemma}

\begin{proof}
We prove the following inequality by a simple case distinction:
\begin{equation}\label{eq:qepsineq}
\forall\,x,y\in\mb{C},\; |q_\varepsilon(x)-q_\varepsilon(y)|\leq |x-y|.
\end{equation}
The first case is $\min(|x|,|y|)\geq \varepsilon$. In this case, let $u:=x/|x|$ and $v=y/|y|$. Then $q_\varepsilon(x)=u(|x|-\varepsilon)$ and $q_\varepsilon(y)=u(|y|-\varepsilon)$. Note that $\tRe(u\overline{v})\leq |u\overline{v}| \leq 1$. Using that $u,v$ have modulus 1, we have that $|u(|x|-\varepsilon)-v(|y|-\varepsilon)|^2$ equals
\begin{align*}
  & (|x|-\varepsilon)^2+(|y|-\varepsilon)^2 -2(|x|-\varepsilon)(|y|-\varepsilon)\tRe(u\overline{v})\\
& =  |x|^2+|y|^2-2\varepsilon (|x|+|y|)+2\varepsilon^2 - 2\tRe(x\overline{y}) +2\varepsilon (|x|+|y|) \tRe(u\overline{v}) -2\varepsilon^2\tRe(u\overline{v})\\
& =  |x-y|^2-2\varepsilon\,(|x|+|y|-\varepsilon)\, (1-\tRe(u\overline{v}))\; \leq \; |x-y|^2.
\end{align*}
The second case is $\min(|x|,|y|)<\varepsilon\leq \max(|x|,|y|)$. In this case, suppose without loss of generality that $|x|<\varepsilon, |y|\geq\varepsilon$, and then observe that  $q_\varepsilon(x)=0$, so $|q_\varepsilon(x)-q_\varepsilon(y)|=|y|-\varepsilon$. The triangle inequality then yields $|x-y|\geq |y|-|x|\geq |y|-\varepsilon = |q_\varepsilon(x)-q_\varepsilon(y)|$, as required.

The third case is $\max(|x|,|y|)\leq \varepsilon$. Here \eqref{eq:qepsineq} holds trivially, since $q_\varepsilon(x)=q_\varepsilon(y)=0$.  
\end{proof} 

\begin{remark}
Letting $q'_\varepsilon(z):=z-q_\varepsilon(z)$, a similar argument shows that $q'_\varepsilon$ is also a contraction on $\mb{C}$, i.e., we have $|x-y|\geq |q'_\varepsilon(x)-q'_\varepsilon(y)|$ for all $x,y\in\mb{C}$. This is less important for our applications, so we omit the proof. We do however record the following result, which will be used later in this paper.
\end{remark}
\begin{lemma}\label{lem:q'}
We have $|q'_\varepsilon(z)|=\min(|z|,\varepsilon)$ for every $z\in \mb{C}$, and thus
\begin{equation}\label{eq:q'tineq}
    \forall\,x,y\in \mb{C},\; |q'_\varepsilon(x+y)|\leq |q'_\varepsilon(x)|+|q'_\varepsilon(y)|.
\end{equation}
\end{lemma}
\begin{proof}
We see that $|q'_\varepsilon(z)|=\min(|z|,\varepsilon)$ by a simple case distinction. Indeed for $|z|\leq \varepsilon$ we have $|q'_\varepsilon(z)|=|z-q_\varepsilon(z)|=|z|=\min(|z|,\varepsilon)$, and for $|z|>\varepsilon$ we have $|q'_\varepsilon(z)|=|z-q_\varepsilon(z)|=|z(1-\frac{|z|-\varepsilon}{|z|})|= \varepsilon=\min(|z|,\varepsilon)$.   

To see that \eqref{eq:q'tineq} holds, note first that $|q'_\varepsilon(x+y)|=\min(|x+y|,\varepsilon)\leq \min(|x|+|y|,\varepsilon)$, so it suffices to prove that
\begin{equation}\label{eq:q'tineq2}
\forall\,a,b\in\mb{R}_{\geq 0},\; \min(a+b,\varepsilon)\leq \min(a,\varepsilon)+\min(b,\varepsilon).
\end{equation}
This can also be proved by a case distinction.

The first case is $a+b\leq \varepsilon$. Then $\min(a+b,\varepsilon)=a+b=\min(a,\varepsilon)+\min(b,\varepsilon)$ (where the last equality is since $a\leq a+b\leq \varepsilon$ so that $a=\min(a,\varepsilon)$, and similarly $b=\min(b,\varepsilon)$). So in this case we have equality.

The second case is $a+b> \varepsilon$. Then the left side of \eqref{eq:q'tineq2} is $\varepsilon$. There is then a first sub-case in which both $a$ and $b$ are less than $\varepsilon$, in which case the right side of \eqref{eq:q'tineq2} is $a+b$, which by assumption in this second case is $>\varepsilon$, so \eqref{eq:q'tineq2} holds in this sub-case. In the second sub-case at least one of $a,b$ is at least $\varepsilon$, say wlog it is $a$. Then the right side of \eqref{eq:q'tineq2} is $\varepsilon+\min(b,\varepsilon)\geq \varepsilon$, so \eqref{eq:q'tineq2} holds. This completes this second case.
\end{proof}

We have thus proved that the operator $K_\varepsilon$ has the desired continuity property mentioned at the beginning of this section.

It remains to verify that, for functions $f:\ab\to\mb{C}$, the operator $K_\varepsilon$ yields a decomposition of $f$ into structured and quasirandom parts, in which $K_\varepsilon(f)$ is a valid Fourier-structured part. Note that if $\|f\|_2\leq 1$, then letting $K_\varepsilon'(f)=f-K_\varepsilon(f)$, we have the decomposition $f=K_\varepsilon(f)+K_\varepsilon'(f)$ where $\|K'_\varepsilon(f)\|_{U^2}\leq\varepsilon^{1/2}$. We now prove that $K_\varepsilon(f)$ is indeed Fourier structured.

\begin{lemma}\label{lem:Wupp}
For every function $f:\ab\to \mb{C}$, the function $K_\varepsilon(f)$ is $(\varepsilon^{-2}\|f\|_2^2,0)$-Fourier structured. In particular we have
\begin{equation}\label{eq:KepsWiener}
\|K_\varepsilon(f)\|_{A(\ab)}\leq\varepsilon^{-1}\|f\|_2^2.
\end{equation}
\end{lemma}
\begin{proof}
By Parseval's identity we have $|\{\chi\in \wh{\ab}:|\wh{f}(\chi)|\ge \varepsilon\}|\le \|f\|_2^2/\varepsilon^2$, so $K_\varepsilon(f)$ is $(\varepsilon^{-2}\|f\|_2^2,0)$-Fourier structured. To obtain \eqref{eq:KepsWiener}, we use the Cauchy-Schwarz inequality and the fact that $\|K_\varepsilon(f)\|_2\le\|f\|_2$.
\end{proof}
\noindent Thus, if $\|f\|_2\leq 1$ then $K_\varepsilon(f)$ is $(\varepsilon^{-2},0)$-Fourier structured. This completes the verification that  $f=K_\varepsilon(f)+K'_\varepsilon(f)$ is a decomposition into Fourier-structured and quasirandom parts.

\begin{remark}\label{rem:OtherLipCutoffs}
There are many possible choices of Fourier regularizing operators other than $K_\varepsilon$. In general, we want such an operator to leave the large Fourier coefficients roughly unchanged, to reduce the small ones to values close to 0, and to be invariant. For our proofs in this paper, it turned out that $K_\varepsilon$ presented particularly convenient properties, such as Lemmas \ref{lem:pointcontraction} and \ref{lem:q'}. Within the family of operators $K_r$ defined in \eqref{eq:genFourierop}, some of these convenient properties of $K_\varepsilon$ are available with other choices of $r$. For instance, it is not hard to show that if $r$ is $C$-Lipschitz, then the operator $K_r$ is $3C$-Lipschitz. One can then choose functions $q_r$ that still reduce to 0 the small Fourier coefficients (like $q_\varepsilon$) while leaving the large Fourier coefficients unchanged (unlike $q_\varepsilon$). A specific example is $r(x)=x$ if $x\geq\varepsilon$, $r(x)=2x-\varepsilon$ for $x\in [\varepsilon/2,\varepsilon]$, and $r(x)=0$ otherwise. Looking into other possible choices of Fourier regularizing operators may be an interesting research direction, in particular regarding quantitative improvements of the resulting algorithms. In this direction, let us emphasize that the property of $r$ being Lipschitz is quite convenient to avoid certain technical issues, as shown in the next example.
\end{remark}

\begin{example}
Suppose that $r$ is the blunt cut-off function $r(x)=x\,\textbf{1}(x\geq \varepsilon)$ for a fixed $\varepsilon>0$, and thus $K_r$ is precisely the \emph{Fourier cut-off operator} from Subsection \ref{subsec:outline}. Let $p$ be a prime number such that $\varepsilon>1/\sqrt{p}$, and let $\delta>\varepsilon-1/\sqrt{p}$. Then one can find a 1-bounded function $f:\mb{Z}_p\to \mb{C}$ and a $\delta$-bounded function $g:\mb{Z}_p\to\mb{C}$ such that, while $\|(f+g)-f\|_2=\delta$, we have $\|K_r(f+g)-K_r(f)\|_2\ge 1$. Choosing carefully $\varepsilon$ and $p$, we can make $\delta$ arbitrarily small, whence this choice of $K_r$ fails to be a Lipschitz continuous operator in general, which is problematic for later calculations in this paper and also in specific applications. Let us give an explicit example of the latter. Let $f(x):=\exp(2\pi i x^2/p)$ and $g:=\delta f$. Then, by Gauss-sum estimates, we know that all Fourier coefficients of $f$ have absolute value $1/\sqrt{p}$. Thus $K_r(f)=0$, whereas $K_r(f+g)=f+g$. Hence $\|K_r(f+g)-K_r(f)\|_2 = \|f+g\|_2=1+\delta\ge 1$. This devolves into a non-continuous behaviour of $\mc{K}_r$ when we vary $\varepsilon$. For instance, let $h(x)=\exp(2\pi i x^3/p)$ for a large prime $p$. By similar arguments as above, it is easy to see that if we consider $\mc{K}_r(h\otimes \overline{h})$ for different $\varepsilon$, we get the following behaviour: if $\varepsilon<1/\sqrt{p}$ then $\mc{K}_r(h\otimes \overline{h})=h\otimes \overline{h}$ and thus we see only one non-zero eigenvalue, namely 1, with corresponding eigenvector $h$. On the other hand, when $\varepsilon\ge 1/\sqrt{p}$, all diagonals $\mc{D}_{\mc{K}_r(h\otimes \overline{h}),t}$ for $t\in \ab\setminus\{0\}$ are 0, and our operator becomes the linear operator that simply multiplies by $1/p$. Hence, $\mc{K}_r$ exhibits a jump between two drastically different behaviours at $\varepsilon=1/\sqrt{p}$.
\end{example}

\begin{remark} 
In higher orders $k\geq 2$, an interesting choice of operator $K$ to apply to the $\ab$-diagonals is the dual-function operator $f\mapsto [f]_{U^k}$, where for $f:\ab\to\mb{C}$ we define the dual function $[f]_k$ (also denoted by $\mc{D}_k(f)$; see \cite[Definition 11.13]{T-V}) as follows:
\begin{equation}\label{eq:dualfn}
[f]_k(x):=\mb{E}_{x,t_1,\ldots,t_k\in \ab}\prod_{v\in \{0,1\}^k\setminus\{0^k\}} \mc{C}^{|v|+1}f(x+v\sbr{1} t_1+\cdots +v\sbr{k} t_k).
\end{equation}
We call the resulting $\ab$-matrix $\mathcal{K}(f\otimes\overline{f})$ \emph{the $U^{k+1}$-dual operator}, and denote it by $[[f]]_{U^{k+1}}$. It turns out that $[[f]]_{U^{k+1}}$ is a positive semi-definite operator, which has another definition very similar to the formula for the dual function \eqref{eq:dualfn} given by an average over cubes, but using an edge of the cube instead of a vertex. A useful property of dual functions is that $[f]_{U^k}$ is a structured function of order $k-1$ (in the sense of Definition \ref{def:kstruct}) which satisfies $\langle f,[f]_{U^k}\rangle=\|f\|_{U^k}^{2^k}$, a correlation which indicates that $[f]_{U^k}$ is non-trivially related to the structured part of $f$. Initial experiments with the eigenvectors of $[[f]]_{U^{k+1}}$ show promise in describing the quadratic structure of $f$, but further investigation is required.
\end{remark}

\noindent To close this section, let us introduce some additional tools that will enable us, among other things, to detail more precisely the observation made in Remark \ref{rem:diagalg}, i.e.\ the convenient interaction of products of $\ab$-matrices with the Fourier structure of their $\ab$-diagonals. To do so, we introduce the following matrix norm which quantifies this Fourier structure.

\begin{defn}\label{def:MAnorm}
For any matrix $M\in \mb{C}^{\ab\times \ab}$ we define
\begin{equation}
\|M\|_{MA}:=\max_{t\in \ab}\|\mc{D}_{M,t}\|_{A(\ab)}.
\end{equation}
\end{defn}
\begin{lemma}\label{lem:MWnorm}
The function $\|\cdot \|_{MA}$ is a sub-multiplicative norm on the vector space $\mb{C}^{\ab\times \ab}$ with respect to matrix multiplication. 
\end{lemma}
\begin{proof}
The fact that $\|\cdot\|_{MA}$ is a norm is easily checked. The sub-multiplicativity (namely that $\|XY\|_{MA}\leq\|X\|_{MA}\|Y\|_{MA}$ for all $X,Y\in \mb{C}^{\ab\times \ab}$) follows directly from Proposition \ref{prop:diagprod}.
\end{proof}
\noindent This matrix norm immediately enables the control of the norm of \emph{polynomials} evaluated on such matrices, which will be very useful to isolate certain eigenvalues later on. Let us record this.
\begin{defn}\label{def:pluspoly}
Given any polynomial $P(x)=\sum_{i=0}^n a_ix^i$ with coefficients $a_i\in \mb{C}$, let $P^+$ denote the polynomial defined by
\begin{align}\label{eq:pluspoly}
P^+(x):=\textstyle\sum_{i=0}^n |a_i|x^i.
\end{align}
\end{defn}

\begin{remark}
Note that for any polynomials $P,Q$ with coefficients in $\mb{C}$, for any $x\in \mb{R}_{\geq 0}$ we have $(PQ)^+(x)\le P^+(x)Q^+(x)$. In particular, for any integer $m\ge 1$ and $x\in \mb{R}_{\geq 0}$, we have $(P(x)^m)^+\le P^+(x)^m$.
\end{remark}

Lemma \ref{lem:MWnorm} has the following direct consequence.
\begin{lemma}\label{lem:polsub}
For any matrix $M\in\mb{C}^{\ab\times \ab}$ and polynomial $P\in\mb{C}[x]$, we have $\|P(M)\|_{MA}\leq P^+(\|M\|_{MA})$.
\end{lemma}

\begin{proof}
Letting $P(x)=\sum_{i=0}^n a_ix^i$ we have $P(M)=\textstyle\sum_{i=0}^n a_iM^i$, and by Lemma \ref{lem:MWnorm}, we have
\[
\|P(M)\|_{MA}\leq \textstyle\sum_{i=0}^n |a_i|\|M^i\|_{MA}\leq \textstyle\sum_{i=0}^n |a_i|\|M\|_{MA}^i=P^+(\|M\|_{MA}).\qedhere
\]
\end{proof}
\noindent Let us record an important fact about weak quadratic characters, namely that these functions are precisely those which are approximately invariant under the operator $\mc{K}_\varepsilon$.
\begin{proposition}\label{prop:weak-qua-char-stable-under-operator}
Let $f:\ab\to \mb{C}$ be a 1-bounded function. Then the following holds:
\setlength{\leftmargini}{0.6cm}
\begin{enumerate}
\item For any $\gamma,\delta,\varepsilon\in (0,1]$, if $\|\mc{K}_\varepsilon(f\otimes \overline{f})-f\otimes \overline{f}\|_2\leq \gamma$, then $f$ is a $(\frac{4}{\varepsilon^2\delta^2},\,(\frac{2\gamma}{\delta})^{\frac{1}{2}},\,\delta)$-weak quadratic character.
\item For any $R>0$ and $\delta_1,\delta_2\in [0,1]$, if $f$ is an $(R,\delta_1,\delta_2)$-weak quadratic character, then $\|\mc{K}_\varepsilon(f\otimes \overline{f})-f\otimes \overline{f}\|_2\leq 6\varepsilon^{1/4}R^{1/2}+ 4\delta_1+2\delta_2^{1/2} $.
\end{enumerate}
\end{proposition}
The proof will use the following basic property of $K_\varepsilon$.
\begin{lemma}\label{lem:KepsStruct1}
Let $R>0$, $\delta\in [0,1]$, and $g:\ab\to\mb{C}$ be $(R,\delta)$-structured of order 1. Then $\|K_\varepsilon(g)-g\|_2\leq 2\delta+4\varepsilon^{1/4}R^{1/2}(1+\|g\|_2)^{1/4}$.
\end{lemma}
\begin{proof}
By Lemma \ref{lem:order1equiv}, for any $\alpha>0$ there is a function $h=\sum_{i\in [M]} c_i\chi_i$ with $M\leq \alpha^{-6}R^4(1+\|g\|_2)^2$ and $\|g-h\|_2\leq \alpha+\delta$. Then, using Lemma \ref{lem:pointcontraction}, we have $\|K_\varepsilon(g)-g\|_2\leq \|K_\varepsilon(g)-K_\varepsilon(h)\|_2+\|K_\varepsilon(h)- h\|_2+\|h-g\|_2\leq \|K_\varepsilon(h)- h\|_2+2\|h-g\|_2\leq \|K_\varepsilon(h)- h\|_2+2(\alpha+\delta)$. By \eqref{eq:genFourierop} and Lemma \ref{lem:q'}, we have $\|K_\varepsilon(h)- h\|_2 = (\sum_{i\in [M]}|q_\varepsilon'(c_i)|^2)^{1/2}\leq \varepsilon M^{1/2}$. Thus $\|K_\varepsilon(g)-g\|_2\leq 2(\alpha+\delta)+\varepsilon\alpha^{-3}R^2(1+\|g\|_2)$, a quantity which we minimize by choosing $\alpha=(3/2)^{-1/4}\varepsilon^{1/4}R^{1/2}(1+\|g\|_2)^{1/4}$. The result follows.
\end{proof}
\begin{proof}[Proof of Proposition \ref{prop:weak-qua-char-stable-under-operator}]
Since $\mb{E}_t\|K_\varepsilon(\Delta_t f)-\Delta_t f\|_2^2\le \gamma^2$, it follows from Markov's inequality that $|\{t\in \ab:\|K_\varepsilon(\Delta_t f)-\Delta_t f\|_2 < \frac{\gamma}{(\delta/2)^{1/2}}\}|\ge |\ab|(1-\delta/2)$. By Lemma \ref{lem:Wupp}, we have that $K_\varepsilon(\Delta_t f)$ is $(\varepsilon^{-2}\|\Delta_tf\|_2^2,0)$-Fourier structured. Note that $\mb{E}_t\|\Delta_tf\|_2^2=\|f\|_2^4$, so $|\{t\in \ab:\|\Delta_tf\|_2^2< \frac{\|f\|_2^4}{\delta/2}\}|\ge |\ab|(1-\delta/2)$. Hence, by the union bound, for at least $(1-\delta)|\ab|$ elements $t\in \ab$, we have that $\Delta_tf$ is $(\frac{2\|f\|_2^4}{\varepsilon^2\delta},\frac{2^{1/2}\gamma}{\delta^{1/2}})$-Fourier structured, hence $(R',\frac{2^{1/2}\gamma}{\delta^{1/2}})$-structured of order 1 where, by Lemma \ref{lem:order1equiv} and what we know about $t$, we have $R'=\frac{2^{3/4}\|f\|_2^3}{\varepsilon^{3/2}\delta^{3/4}}\|\Delta_t f\|_2\leq \frac{2^{5/4}\|f\|_2^5}{\varepsilon^{3/2}\delta^{5/4}}\leq \frac{4}{\varepsilon^2\delta^2}$ . This proves $(i)$.

To prove $(ii)$, let $S\subset \ab$ with $|S|\ge (1-\delta_2)|\ab|$ and such that for every $t\in S$ we have that $\Delta_tf$ is $(R,\delta_1)$-structured of order 1. Note that
\begin{align*}
\mb{E}_t\|K_\varepsilon(\Delta_t f)-\Delta_t f\|_2^2=\mb{E}_t 1_S(t)\|K_\varepsilon(\Delta_t f)-\Delta_t f\|_2^2+\mb{E}_t 1_{S^c}(t)\|K_\varepsilon(\Delta_t f)-\Delta_t f\|_2^2.
\end{align*}
By the triangle inequality and Lemma \ref{lem:pointcontraction}, we have $\|K_\varepsilon(\Delta_t f)-\Delta_t f\|_2\le  2\|\Delta_t f\|_2$, so $\mb{E}_t 1_{S^c}(t) \|K_\varepsilon(\Delta_t f)-\Delta_t f\|_2^2\le 4\mb{E}_t 1_{S^c}(t) \|\Delta_t f\|_2^2\leq  4\delta_2$ (using that $\|\Delta_t f\|_2\leq \|f\|_\infty\leq 1$). On the other hand, for every $t\in S$ we have $\Delta_tf=g_t+\mc{E}_t$ where $\|\mc{E}_t\|_2\le \delta_1$ (in particular $\|g_t\|_2\le 1+\delta_1$). Hence 
\[\|K_\varepsilon(\Delta_t f)-\Delta_t f\|_2= \|K_\varepsilon(g_t+\mc{E}_t)-K_\varepsilon(g_t)+K_\varepsilon(g_t)-g_t-\mc{E}_t\|_2.
\]
By Lemma \ref{lem:pointcontraction} we have $\|K_\varepsilon(g_t+\mc{E}_t)-K_\varepsilon(g_t)\|_2\le \|\mc{E}_t\|_2\le \delta_1$, and by Lemma \ref{lem:KepsStruct1} we have $\|K_\varepsilon(g_t)-g_t\|_2\le 2\delta_1+4\varepsilon^{1/4}R^{1/2}(2+\delta_1)^{1/4}\le 2\delta_1+6\varepsilon^{1/4}R^{1/2}$. We conclude that for $t\in S$ we have $\|K_\varepsilon(\Delta_t f)-\Delta_t f\|_2\le 4\delta_1+6\varepsilon^{1/4}R^{1/2}$, so $\mb{E}_t 1_S(t)\|K_\varepsilon(\Delta_t f)-\Delta_t f\|_2^2\le (4\delta_1+6\varepsilon^{1/4}R^{1/2})^2$.

Hence $\|\mc{K}_\varepsilon(f\otimes \overline{f})-f\otimes \overline{f}\|_2\leq (4\delta_2 + [4\delta_1+6\varepsilon^{1/4}R^{1/2}]^2)^{1/2}$ and the result follows.
\end{proof}

\begin{remark}\label{rem:qua-char-stable} Note that, to prove $(i)$, we did not need the assumption $\|f\|_\infty\leq 1$. Indeed, a bound $\|f\|_2\leq 1$ was sufficient. Moreover, the proof of statement $(ii)$ used the bound $\|f\|_\infty\leq 1$ only to control the term involving $S^c$. Since for proper (i.e.\ non-weak) quadratic characters, there is no such term (as $S^c$ is empty), it follows that for quadratic characters $f$, Proposition \ref{prop:weak-qua-char-stable-under-operator} holds in general (i.e.\ without assuming that $f$ is 1-bounded). In particular, if $f$ is an $(R,\delta)$-quadratic character then $\|\mc{K}_\varepsilon(f\otimes \overline{f})-f\otimes \overline{f}\|_2\leq 6\varepsilon^{1/4}R^{1/2} +4\delta$.
\end{remark}

Finally, we record the following fact, which is useful to control the behaviour of weak quadratic characters under scalar multiplication.

\begin{lemma}\label{lem:KepsStructHom}
For any function $f:\ab\to\mb{C}$ and $c\in\mb{C}$, we have
\begin{equation}\label{eq:KepsHom1}
\|\mc{K}_\varepsilon\big((cf)\otimes(\overline{cf})\big)-(cf)\otimes(\overline{cf})\|_2 = |c|^2\; \|\mc{K}_{\varepsilon/|c|^2}(f\otimes\overline{f})-f\otimes\overline{f}\|_2.
\end{equation}
In particular,
\begin{equation}\label{eq:KepsHom2}
\|\mc{K}_\varepsilon((cf)\otimes(\overline{cf}))-(cf)\otimes(\overline{cf})\|_2 \leq \max\big(|c|^2\,\|\mc{K}_{\varepsilon^{1/2}}(f\otimes\overline{f})-f\otimes\overline{f}\|_2,2\varepsilon^{1/2}\|f\|_2^2\big).
\end{equation}
\end{lemma}
\begin{proof}
By \eqref{eq:genFourierop} and Lemma \ref{lem:q'}, we have that $\|\mc{K}_\varepsilon((cf)\otimes(\overline{cf}))-(cf)\otimes(\overline{cf})\|_2^2$ equals
\begin{multline*}
\mb{E}_t \|K_\varepsilon(|c|^2\Delta_t f)- |c|^2\Delta_t f\|_2^2
= \mb{E}_t \textstyle\sum_{\chi\in \wh{\ab}}|\wh{K_\varepsilon(|c|^2\Delta_t f)}(\chi)- |c|^2\wh{\Delta_t f}(\chi)|^2 \\ = \mb{E}_t \textstyle\sum_{\chi}\big|q_\varepsilon'\big(\,|c|^2\wh{\Delta_t f}(\chi)\big)\big|^2
= \mb{E}_t \textstyle\sum_{\chi}\min\big(|c|^2|\wh{\Delta_t f}(\chi)|,\varepsilon\big)^2.
\end{multline*}
Using that $\min(\lambda a,\lambda b)=\lambda \min( a, b)$ for any $a,b,\lambda\geq 0$, we see that the last average above equals
$|c|^4 \, \mb{E}_t \sum_{\chi}\min\big(|\wh{\Delta_t f}(\chi)|,\varepsilon/|c|^2\big)^2 =  |c|^4\, \|\mc{K}_{\varepsilon/|c|^2}(f\otimes\overline{f})-f\otimes\overline{f}\|_2^2$, and \eqref{eq:KepsHom1} follows.

To see \eqref{eq:KepsHom2}, note first that the argument above yields the equalities
\begin{eqnarray}
\forall\, c\in\mb{C},\quad \|\mc{K}_\varepsilon((cf)\otimes(\overline{cf}))-(cf)\otimes(\overline{cf})\|_2^2   & = & \mb{E}_t \textstyle\sum_{\chi}\min\big(|c|^2|\wh{\Delta_t f}(\chi)|,\varepsilon\big)^2,\label{eq:Kepsinc}\\
\forall\, c\in\mb{C}\setminus\{0\},\quad \|\mc{K}_{\varepsilon/|c|^2}(f\otimes\overline{f})-f\otimes\overline{f}\|_2^2 & = &\mb{E}_t \textstyle\sum_{\chi}\min\big(|\wh{\Delta_t f}(\chi)|,\varepsilon/|c|^2\big)^2,\label{eq:Kepsdec}
\end{eqnarray}
and that \eqref{eq:Kepsinc} implies that $\|\mc{K}_\varepsilon((cf)\otimes(\overline{cf}))-(cf)\otimes(\overline{cf})\|_2$ increases as $|c|$ increases, while  \eqref{eq:Kepsdec} implies that $\|\mc{K}_{\varepsilon/|c|^2}(f\otimes\overline{f})-f\otimes\overline{f}\|_2$  decreases as $|c|$ increases.

Now $\|\mc{K}_{\varepsilon/|c|^2}(f\otimes\overline{f})-f\otimes\overline{f}\|_2\leq 2\|f\otimes\overline{f}\|_2=2\|f\|_2^2$. Hence, if $|c|^2\leq \varepsilon^{1/2}$, then, by the increasing property, the left side of \eqref{eq:KepsHom2} is at most $2\varepsilon^{1/2}\|f\|_2^2$, while if $|c|^2 > \varepsilon^{1/2}$, then by the decreasing property, the left side of \eqref{eq:KepsHom2} is at most $|c|^2\,\|\mc{K}_{\varepsilon^{1/2}}(f\otimes\overline{f})-f\otimes\overline{f}\|_2$, and \eqref{eq:KepsHom2} follows.
\end{proof}

\section{Elementary motivation of the spectral approach in quadratic Fourier analysis}\label{sec:elem-analysis}

\noindent As discussed in the introduction, the spectral approach in quadratic Fourier analysis is based on the idea that, given a function $f:\ab\to\mb{C}$, the dominant eigenvectors of the Fourier-regularized matrix $\mc{K}_\varepsilon(f\otimes \overline{f})$ are interesting quadratic Fourier components of $f$. In this section, we present some initial non-trivial results illustrating this idea more precisely. 

The main result is that certain eigenvectors of the matrix $\mc{K}_\varepsilon(f\otimes \overline{f})$, corresponding to large eigenvalues, are weak quadratic characters. We establish this formally in the first subsection below. In the second subsection, we  gather further useful properties of the operator $K_\varepsilon$.
 
\subsection{Spectrally isolated eigenvectors of $\mc{K}_\varepsilon(f\otimes\overline{f})$ are (weak) quadratic characters}\hfill
\begin{defn}
Let $M$ be a self-adjoint linear operator on a finite-dimensional Hilbert space,\footnote{For us, $M$ will always be a $\ab$-matrix and thus its eigenvalues are defined as per Definition \ref{def:ZmatOp}.} let $\lambda$ be an eigenvalue of $M$ and let $\theta>0$. We say that $\lambda$ is \emph{$\theta$-isolated} if the multiplicity of $\lambda$ is $1$ and every other eigenvalue $\lambda'$ of $M$ satisfies $|\lambda-\lambda'|>\theta$.
\end{defn}
\noindent The main result in this subsection, Proposition \ref{prop:sepevquadchar}, tells us that eigenvectors of the matrix $\mc{K}_\varepsilon(f\otimes \overline{f})$ corresponding to isolated large eigenvalues are weak quadratic characters, with parameters depending on how structured the diagonals of the matrix are. To establish this result, we will use the interaction of the matrix norm $\|\cdot \|_{MA}$ (recall Definition \ref{def:MAnorm}) with polynomial matrix expressions, described in Lemma \ref{lem:polsub}. To apply this effectively on isolated eigenvalues, we shall use the following polynomial approximations of indicator functions.
\begin{defn}
For any $n\in\mb{N}$ and $\lambda\in\mathbb{R}\setminus\{0\}$, we define the polynomial
\[
p_n(x,\lambda):=\lambda^{-1}x\big(1-(x-\lambda)^2/4\big)^n.
\]
\end{defn}
Note that $p_n(\lambda,\lambda)=1$. The idea is that if $n$ is sufficiently large and $x$ sufficiently separated from $\lambda$, then $p_n(x,\lambda)$ is small. We make this precise as follows.
\begin{lemma}\label{lem:goodpoly}
For any $\theta>0$ and $\lambda\in[-1,1]\setminus\{0\}$, if $n\geq 4\theta^{-2}\ln(\theta^{-1}\lambda^{-1})$, then $|p_n(x,\lambda)|\leq \theta|x|$ for every $x\in[-1,1]$ satisfying $|x-\lambda|\geq\theta.$
\end{lemma}

\begin{proof} The inequality holds for $x=0$, as $p_n(0,\lambda)=0$, so we can assume that $x\neq 0$. Then \begin{align*}|p_n(x,\lambda)|/|x|&=|\lambda ^{-1}\big(1-(x-\lambda)^2/4\big)^n|=\lambda ^{-1}\big(1-(x-\lambda)^2/4\big)^n \\ &\leq\lambda^{-1}\exp(-n(x-\lambda)^2/4)\leq\lambda^{-1}\exp(-n\theta^2/4)\leq\theta\end{align*} where the last inequality follows from the assumption on $n$.
\end{proof}
\noindent Recall from \eqref{eq:pluspoly} the definition of $p_n^+$. Since $p_n(x,\lambda)=\lambda^{-1} x (- \frac{x^2}{4} +\lambda \frac{x}{2} +(1-\frac{\lambda^2}{4}))^n$, it follows that for $0<|\lambda|\leq 1$ and $x\geq 0$ we have
\begin{equation}\label{eq:pbound}
p_n^+(x,\lambda)=|\lambda|^{-1} x\big( \tfrac{x^2}{4} +\tfrac{|\lambda| x}{2} +(1-\tfrac{\lambda^2}{4})\big)^n \leq |\lambda|^{-1} x\big( \tfrac{x^2}{4} + x +1\big)^n = |\lambda|^{-1} x\big(\tfrac{x}{2}+1\big)^{2n}.
\end{equation}
Recall also that we normalize eigenvectors $v$ of a matrix to have $\|v\|_2=1$. We can now present the main result of this subsection.
\begin{proposition}\label{prop:sepevquadchar}
Let $M\in\mb{C}^{\ab\times \ab}$ be a self-adjoint $\ab$-matrix with $\|M\|_2\leq 1$, and let $\theta\in (0,1]$. Suppose that $\lambda\neq 0$ is a $\theta$-isolated eigenvalue of $M$, and let $g$ be a corresponding eigenvector of $M$. Then, letting $n:=\lceil 4\theta^{-2}\ln(\theta^{-1}\lambda^{-1})\rceil$, we have that for at least $(1-\theta^{1/2})|\ab|$ values of $t\in\ab$, the function $\Delta_t g$ is $(p_n^+(\|M\|_{MA},\lambda)^2\theta^{-1},2\theta^{1/2})$-Fourier structured, so $g$ is a weak quadratic character with parameters $\big(2 p_n^+(\|M\|_{MA},\lambda)^{3/2}\theta^{-5/4},3\theta^{1/2},\theta\big)$. In particular, if $|\lambda|\geq \theta$, then $g$ is a weak quadratic character with parameters
\begin{equation}\label{eq:wquadcharparams}
\big(2\theta^{-3} \|M\|_{MA} (\|M\|_{MA}/2+1)^{2\lceil 8\theta^{-2}\ln(\theta^{-1})\rceil},\, 3\theta^{1/2} ,\, \theta\big).
\end{equation}
\end{proposition}
\begin{proof}
As $M$ is self-adjoint, let $\{g_i\}_{i=1}^{|\ab|-1}$ and $\{\lambda_i\}_{i=1}^{|\ab|-1}$ be eigenvectors and corresponding eigenvalues of $M$ in a way that $\{g,g_1,\ldots,g_{|\ab|-1}\}$ forms an orthonormal basis. Since $\|M\|_2\leq 1$, we have $\lambda^2+\sum_{i=1}^{|\ab|-1}\lambda_i^2\leq 1$. In particular, all eigenvalues of $M$ are of absolute value at most $1$. Let $M':=p_n(M,\lambda)$. We have (using $p_n(\lambda,\lambda)=1$) that $M'=g\otimes\overline{g}+\sum_{i=1}^{|\ab|-1} p_n(\lambda_i,\lambda) g_i\otimes\overline{g_i}$.

By Lemma \ref{lem:goodpoly}, $|p_n(\lambda_i,\lambda)|\leq\theta|\lambda_i|$ holds for $1\leq i\leq {|\ab|-1}$, so $\sum_{i=1}^{|\ab|-1} p_n(\lambda_i,\lambda)^2\leq\theta^2\sum_{i=1}^{|\ab|-1}\lambda_i^2\leq\theta^2$. Therefore $\mb{E}_t\|\mc{D}_{M',t}-\Delta_t g\|_2^2=\|g\otimes\overline{g}-M'\|_2^2\leq\theta^2$. This means that for a subset $S\in \ab$ with $|S|\geq(1-\theta/2)|\ab|$ we have that, if $t\in S$, then $
\|\mc{D}_{M',t}-\Delta_t g\|_2\leq(2\theta)^{1/2}$.

As $\mb{E}_t\|\Delta_t g\|_2^2=\|g\|_2^4\leq 1$, the set $S':=\{t\in\ab: \|\Delta_t g\|_2^2\leq \frac{2}{\theta}\}$ satisfies $|S'|/|\ab|\geq 1-\theta/2$. 

By Lemma \ref{lem:polsub}, $\|M'\|_{MA}\leq p_n^+(\|M\|_{MA},\lambda)$. By Corollary \ref{cor:strucbound}, for every $t\in S\cap S'$, $\mc{D}_{M',t}$ is $(p_n^+(\|M\|_{MA},\lambda)^2\theta^{-1},\theta)$-Fourier structured, so $\Delta_t g$ is $(p_n^+(\|M\|_{MA},\lambda)^2\theta^{-1},3\theta^{1/2})$-Fourier structured and $\|\Delta_t g\|_2\leq (2/\theta)^{1/2}$. The parameters for $g$ involving $\lambda$ follow by Lemma \ref{lem:order1equiv}, and the final parameters in \eqref{eq:wquadcharparams} follow from the latter and \eqref{eq:pbound}.
\end{proof}
\noindent Proposition \ref{prop:sepevquadchar} is just one member of a family of similar results; we chose this one for purposes of illustration and motivation. In other variants one could want for instance the precision parameter of the character $g$ to be independent of the isolation $\theta$. Proving such variants can involve different choices of the polynomial $p_n$. Proposition \ref{prop:sepevquadchar} implies the next theorem.

\begin{theorem}\label{thm:sepeveciswqchar}
Let $\ab$ be a finite abelian group, let $f:\ab\to\mb{C}$ be a 1-bounded function, and let $\theta,\varepsilon\in (0,1]$. Suppose that there exists a $\theta$-isolated eigenvalue $\lambda$ of $\mc{K}_\varepsilon\big(f\otimes \overline{f}\big)$ with $|\lambda|\geq \theta$, and let $g$ be a corresponding eigenvector with $\|g\|_2= 1$. Then $g$ is a weak quadratic character with parameters $
\bigl(2\theta^{-3}\varepsilon^{-1}(1/(2\varepsilon)+1)^{2n},3\theta^{1/2},\theta\bigr)$, where $n=\lceil 8\theta^{-2}\ln(\theta^{-1})\rceil$.
\end{theorem}

\begin{proof}
Let $M$ be the self-adjoint matrix $\mc{K}_\varepsilon\big(f\otimes \overline{f}\big)$ and $M':=f\otimes \overline{f}$. Note that $\|\mc{D}_{M',t}\|_2\leq\|\mc{D}_{M',t}\|_\infty\leq 1$ holds for every $t\in \ab$,  so $\|\mc{D}_{M,t}\|_2\leq 1$ holds for every $t\in \ab$ by Lemma \ref{lem:pointcontraction}. By \eqref{eq:KepsWiener} we have $\|\mc{D}_{M,t}\|_{A(\ab)}\leq \varepsilon^{-1}$ for every $t\in \ab$. Note also that $\|M\|_2=(\mb{E}_t\|\mc{D}_{M,t}\|_2^2)^{1/2}\leq 1$. The result now follows by applying Proposition \ref{prop:sepevquadchar}.
\end{proof}
\noindent There are various directions in which one can want to strengthen Theorem \ref{thm:sepeveciswqchar}, and which provide good motivation for the upcoming sections of the paper.

One such direction seeks to ensure that the eigenvector $g$ is not just a \emph{weak} quadratic character, but rather a proper quadratic character. Another direction, which is important in connection with inverse theorems, is to add a non-trivial lower bound for the correlation between the original function $f$ and the quadratic character $g$. In yet another direction, the goal is to obtain a correlating quadratic character $g$ even when none of the largest eigenvalues is isolated.

In the upcoming sections we will make progress in these directions. To this end, we prepare in the next subsection by obtaining further useful properties of the operator $K_\varepsilon$. Then, in Section \ref{sec:nschars}, we will start using nilspace theory, and especially certain decompositions in terms of \emph{nilspace characters}, which will be key ingredients for our results in the above directions. We will then be able to say more about the relations between weak quadratic characters and 2-step nilspace characters (in particular with Theorem \ref{thm:nscharsquadchars}; see also Remark \ref{rem:eq-def-qua-char}).\\

\subsection{Further properties of the operator $K_\varepsilon$}\hfill\smallskip\\
Here, we gather several inequalities involving the operator $K_\varepsilon$ from Definition \ref{def:cont-cutoff}.

Throughout this subsection we maintain the notation $\ab$ for a finite abelian group. Recall also the distinction between $K_\varepsilon$ and $\mc{K}_\varepsilon$ from Remark \ref{rem:calnote}.

The first two results capture continuity of $\mc{K}_\varepsilon$ relative to the $L^2$-norm.

\begin{lemma}\label{lem:KepsL2cont}
Let $M,N\in\mb{C}^{\ab\times\ab}$ and let $\varepsilon\in (0,1]$. Then
\begin{equation}\label{eq:L2matKcont}
\|\mc{K}_\varepsilon(M)-\mc{K}_\varepsilon(N)\|_2\leq \|M-N\|_2.
\end{equation}
\end{lemma}
\begin{proof}
Computing the $L^2$-norm via diagonals, note that the left side of \eqref{eq:L2matKcont} squared equals $\mb{E}_t \|K_\varepsilon(\mc{D}_{M,t})-K_\varepsilon(\mc{D}_{N,t})\|_2^2$. For every $t\in\ab$, by Lemma \ref{lem:pointcontraction}, $\|K_\varepsilon(\mc{D}_{M,t})-K_\varepsilon(\mc{D}_{N,t})\|_2^2\leq \|\mc{D}_{M,t}-\mc{D}_{N,t}\|_2^2=\|\mc{D}_{M-N,t}\|_2^2$. Since $\mb{E}_t\|\mc{D}_{M-N,t}\|_2^2=\|M-N\|_2^2$, the result follows.
\end{proof}

\begin{corollary}\label{cor:cont-cutoff}
For any functions $f,g:\ab\to \mb{C}$ and $\varepsilon\in (0,1]$, we have
\begin{equation}
\|\mc{K}_\varepsilon(f\otimes \overline{f})-\mc{K}_\varepsilon(g\otimes \overline{g})\|_2\le \|f-g\|_2\; (\|f\|_2+\|g\|_2).
\end{equation}
\end{corollary}

\begin{proof}
Using \eqref{eq:L2matKcont} we see that $\|\mc{K}_\varepsilon(f\otimes \overline{f})-\mc{K}_\varepsilon(g\otimes \overline{g})\|_2$ is at most
\begin{multline*}
\|\mc{K}_\varepsilon(f\otimes \overline{f})-\mc{K}_\varepsilon(f\otimes \overline{g})\|_2+\|\mc{K}_\varepsilon(f\otimes \overline{g})-\mc{K}_\varepsilon(g\otimes \overline{g})\|_2  \\
 \leq \|f\otimes \overline{f}-f\otimes \overline{g}\|_2+\|f\otimes \overline{g}-g\otimes \overline{g}\|_2 = \|f\otimes (\overline{f}-\overline{g})\|_2+\|(f-g)\otimes \overline{g}\|_2,
\end{multline*}
and this equals $\|f-g\|_2\, (\|f\|_2+\|g\|_2)$ as required.
\end{proof}

The next results build up towards a continuity property of the map $f\mapsto \mc{K}_\varepsilon(f\otimes\overline{f})$ relative to the $U^3$-norm. This property will eventually be established in Corollary \ref{cor:Fourier-cutoff-U3-perturb} below. We begin with the following simple lemma.

\begin{lemma}\label{lem:cutvsu2} For any function $f:\ab\to\mb{C}$ and any $\varepsilon\in (0,1]$, we have $\|K_\varepsilon(f)\|_2\leq\varepsilon^{-1}\|f\|_{U^2}^2$.
\end{lemma}
\begin{proof}
Let $g=K_\varepsilon(f)$ and let $S:=\Supp(\wh{g})=\{\chi\in\wh{\ab}: |\wh{f}(\chi)|>\varepsilon\}$. By Markov's inequality we then have $|S|\varepsilon^4\leq\|\wh{f}\|_{\ell^4}^4=\|f\|_{U^2}^4$. Furthermore, we have $|\wh{f}|\geq |\wh{g}|$. By Parseval's identity and the Cauchy-Schwarz inequality, we then have
\[
\|g\|_2^2=\textstyle\sum_{\chi\in S} |\wh{g}(\chi)|^2\leq \textstyle\sum_\chi |\wh{f}(\chi)|^2 1_S(\chi)\leq |S|^{1/2}\||\wh{f}|^2\|_{\ell^2}=|S|^{1/2}\|f\|_{U^2}^2\leq\varepsilon^{-2}\|f\|_{U^2}^4.\qedhere
\]
\end{proof}

A similar argument yields the following result.

\begin{lemma}\label{lem:KepsU2control}
For any functions $f,g:\ab\to\mb{C}$ and $\varepsilon\in (0,1]$, we have
\begin{equation}
\|K_\varepsilon(f)-K_\varepsilon(g)\|_2\leq \varepsilon^{-1/2}\; \|f-g\|_{U^2}\; (\|f\|_2^2+\|g\|_2^2)^{1/4}.
\end{equation}
\end{lemma}
\begin{proof}
We have
\[
\|K_\varepsilon(f)-K_\varepsilon(g)\|_2^2=\textstyle\sum_{\chi\in \wh{\ab}}|\wh{K_\varepsilon(f)}(\chi)-\wh{K_\varepsilon(g)}(\chi)|^2=\textstyle\sum_{\chi\in \wh{\ab}}|q_\varepsilon(\wh{f}(\chi))-q_\varepsilon(\wh{g}(\chi))|^2.
\]
Since $q_\varepsilon$ eliminates Fourier coefficients of modulus at most $\varepsilon$, the last sum above gets non-zero contributions only from characters $\chi$ in the set
$S:=\{\chi:|\wh{f}(\chi)|\geq \varepsilon\}\cup  \{\chi:|\wh{g}(\chi)|\geq\varepsilon\}$. We have 
$\varepsilon^2\, |\{\chi:|\wh{f}(\chi)|\geq \varepsilon\}|\leq \sum_\chi | \wh{f}(\chi)|^2 = \|f\|_2^2$ and similarly $|\{\chi:|\wh{g}(\chi)|\geq \varepsilon\}|\leq \varepsilon^{-2}  \|g\|_2^2$. Hence $
|S|\leq \varepsilon^{-2}\, (\|f\|_2^2+\|g\|_2^2)$.

By Lemma \ref{lem:pointcontraction} the function $q_\varepsilon$  is a contraction, whence (using Cauchy-Schwarz) we have
\[
\textstyle\sum_{\chi\in S}|q_\varepsilon(\wh{f}(\chi))-q_\varepsilon(\wh{g}(\chi))|^2\leq \textstyle\sum_{\chi\in S}|\wh{f}(\chi)-\wh{g}(\chi)|^2 = \textstyle\sum_{\chi\in S}|(\wh{f-g})(\chi)|^2 
 \leq  |S|^{1/2} \|\wh{f-g}\|_{\ell^4}^2.
\]
Since $\|\wh{f-g}\|_{\ell^4}^2=\|f-g\|_{U^2}^2$,  the result follows.
\end{proof}
Applying Lemma \ref{lem:KepsU2control} to diagonal functions of $\ab$-matrices (recall  \eqref{eq:Zdiag}), we shall obtain below a result (Lemma \ref{lem:DUcontrol}) enabling us to control $L^2$-norms of the form $\|\mc{K}_\varepsilon(M)-\mc{K}_\varepsilon(N)\|_2$, for $\ab$-matrices $M$ and $N$, in terms of a norm which is the case $k=2$ of the following notion.
\begin{defn}
For any integer $k\geq 2$, we define the following norm on $\mb{C}^{\ab\times\ab}$:
\begin{equation}
\forall\,M\in\mb{C}^{\ab\times \ab},\quad \|M\|_{DU^k}:=\big(\mb{E}_{t\in\ab}\|\mk{D}_{M,t}\|_{U^k}^{2^k}\big)^{1/2^k}.
\end{equation}
\end{defn}

\begin{lemma}\label{lem:DUnorm}
For every $k\geq 2$ we have that $\|\cdot\|_{DU^k}$ is a norm on the vector space $\mb{C}^{\ab\times \ab}$.   
\end{lemma}

\begin{proof}
The only non-trivial norm axiom is the sub-additivity. This follows from the sub-additivity of the $U^k$-norm and the $L^{2^k}$-norm, indeed
\begin{align*}
\|M+N\|_{DU^k}  = & \big(\mb{E}_t\|\mk{D}_{M,t}+\mk{D}_{N,t}\|_{U^k}^{2^k}\big)^{1/2^k}\leq \big(\mb{E}_t(\|\mk{D}_{M,t}\|_{U^k}+\|\mk{D}_{N,t}\|_{U^k})^{2^k}\big)^{1/2^k}\\
\leq & \big(\mb{E}_t \|\mk{D}_{M,t}\|_{U^k}^{2^k}\big)^{1/2^k}+\big(\mb{E}_t\|\mk{D}_{N,t}\|_{U^k}^{2^k}\big)^{1/2^k}=\|M\|_{DU^k}+\|N\|_{DU^k}.\qedhere
\end{align*}
\end{proof}
We can now obtain the lemma announced above.
\begin{lemma}\label{lem:DUcontrol}
Let $M,N\in\mb{C}^{\ab\times\ab}$ and $\varepsilon\in (0,1]$. Then
\begin{equation}\label{eq:DUcontrol}
\|\mc{K}_\varepsilon(M)-\mc{K}_\varepsilon(N)\|_2\leq \varepsilon^{-1/2}\; \|M-N\|_{DU^2}\; (\|M\|_2^2+\|N\|_2^2)^{1/4}.
\end{equation}
\end{lemma}
\begin{proof}
The left side of \eqref{eq:DUcontrol} squared equals $\mb{E}_t\|K_\varepsilon(\mk{D}_{M,t})-K_\varepsilon(\mk{D}_{N,t})\|_2^2$. Applying Lemma \ref{lem:KepsU2control} for every $t\in\ab$, this is seen to be at most $ 
\varepsilon^{-1}\mb{E}_t\|\mk{D}_{M,t}-\mk{D}_{N,t}\|_{U^2}^2\,(\|\mk{D}_{M,t})\|_2^2+\|\mk{D}_{N,t}\|_2^2)^{1/2}$. Using that $\mk{D}_{M,t}-\mk{D}_{N,t}=\mk{D}_{M-N,t}$ and Cauchy-Schwarz, the last quantity is seen to be at most $ 
\varepsilon^{-1}(\mb{E}_t\|\mk{D}_{M-N,t}\|_{U^2}^4)^{1/2}\,(\mb{E}_t\|\mk{D}_{M,t})\|_2^2+\mb{E}_t\|\mk{D}_{N,t}\|_2^2)^{1/2}$ and \eqref{eq:DUcontrol} follows.
\end{proof}
Similarly to how we deduced Corollary \ref{cor:cont-cutoff} from Lemma \ref{lem:KepsL2cont}, we shall deduce the following from Lemma \ref{lem:DUcontrol}.

\begin{corollary}\label{cor:Fourier-cutoff-U3-perturb}
Let $\varepsilon\in(0,1]$ and let $f,g:\ab\to \mb{C}$ be functions. Then
\begin{equation}\label{eq:FCU3P}
\big\|\mc{K}_\varepsilon(f\otimes \overline{f})-\mc{K}_\varepsilon (g\otimes \overline{g})\big\|_2\leq \varepsilon^{-1/2}\, \|f-g\|_{U^3}\, (\|f\|_{U^3}^{4/3}+\|g\|_{U^3}^{4/3})^{3/4}(\|f\|_2^2+\|g\|_2^2)^{1/2}.
\end{equation}
\end{corollary}
\noindent To prove Corollary \ref{cor:Fourier-cutoff-U3-perturb}, we introduce the following special case of a Gowers $U^k$-product.
\begin{defn}
For any integer $k\geq 2$ and any functions $f,g:\ab\to \mb{C}$, we denote by $\langle f,g\rangle_{U^k}$ the special case of the Gowers $U^k$-product\footnote{See Definition \ref{def:gower-u-k-prod}.} $\langle (f_v)_{v\in \db{k}}\rangle_{U^k}$ in which $f_v=f$ for $v\sbr{k}=0$ and $f_v=g$ for $v\sbr{k}=1$. Note that 
\begin{equation}\label{eq:mUkprod}
\langle f,g\rangle_{U^k}=\mb{E}_{t\in\ab} \|fT^t \overline{g}\|_{U^{k-1}}^{2^{k-1}}.   
\end{equation}
In particular, we always have $\langle f,g\rangle_{U^k}\geq 0$. Note also that from \eqref{eq:mUkprod} it follows that
\begin{equation}\label{eq:mUkprodDUk}
\langle f,g\rangle_{U^k}=\| f\otimes \overline{g}\|_{DU^{k-1}}^{2^{k-1}}.
\end{equation}
\end{defn}

\begin{proof}[Proof of Corollary \ref{cor:Fourier-cutoff-U3-perturb}]
We have 
\[
\big\|\mc{K}_\varepsilon(f\otimes \overline{f})-\mc{K}_\varepsilon (g\otimes \overline{g})\big\|_2\leq \big\|\mc{K}_\varepsilon(f\otimes \overline{f})-\mc{K}_\varepsilon (f\otimes \overline{g})\big\|_2+\big\|\mc{K}_\varepsilon(f\otimes \overline{g})-\mc{K}_\varepsilon (g\otimes \overline{g})\big\|_2.
\]
To bound the first term on the right hand side, note that, by Lemma \ref{lem:DUcontrol} and \eqref{eq:mUkprodDUk}, we have 
\begin{align*}
\big\|\mc{K}_\varepsilon(f\otimes \overline{f})-\mc{K}_\varepsilon (f\otimes \overline{g})\big\|_2^2 & \leq  \varepsilon^{-1}\, \|f\otimes(\overline{f}-\overline{g})\|_{DU^2}^2\, (\|f\otimes \overline{f}\|_2^2+\|f\otimes \overline{g}\|_2^2)^{1/2}\\
& =  \varepsilon^{-1}\;\langle f,f-g\rangle_{U^3}^{1/2}\;\|f\|_2\; (\|f\|_2^2+\|g\|_2^2)^{1/2}.
\end{align*}
Similarly, $\big\|\mc{K}_\varepsilon(f\otimes \overline{g})-\mc{K}_\varepsilon (g\otimes \overline{g})\big\|_2^2 \leq \varepsilon^{-1}\;\langle f-g,g\rangle_{U^3}^{1/2}\;\|g\|_2\;(\|f\|_2^2+\|g\|_2^2)^{1/2}$.
In particular, the product $\big\|\mc{K}_\varepsilon(f\otimes \overline{f})-\mc{K}_\varepsilon (f\otimes \overline{g})\big\|_2 \big\|\mc{K}_\varepsilon(f\otimes \overline{g})-\mc{K}_\varepsilon (g\otimes \overline{g})\big\|_2 $ is at most
\[
\varepsilon^{-1}\; \langle f,f-g\rangle_{U^3}^{1/4}\;\langle f-g, g\rangle_{U^3}^{1/4}\; \|f\|_2^{1/2}\; \|g\|_2^{1/2}\; (\|f\|_2^2+\|g\|_2^2)^{1/2}.
\]
Hence $
\big\|\mc{K}_\varepsilon(f\otimes \overline{f})-\mc{K}_\varepsilon (g\otimes \overline{g})\big\|_2^2\leq \varepsilon^{-1}\, (\|f\|_2^2+\|g\|_2^2)^{\frac{1}{2}} (\langle f,f-g\rangle_{U^3}^{\frac{1}{4}}\|f\|_2^{\frac{1}{2}}+\langle f-g,g\rangle_{U^3}^{\frac{1}{4}}\|g\|_2^{\frac{1}{2}})^2$. 
By the Gowers-Cauchy-Schwarz inequality, we have $\langle f,f-g\rangle_{U^3}^{1/4}\leq \|f\|_{U^3} \|f-g\|_{U^3}$ and $\langle f-g,g\rangle_{U^3}^{1/4}\leq \|g\|_{U^3} \|f-g\|_{U^3}$. Hence
\[
\big\|\mc{K}_\varepsilon(f\otimes \overline{f})-\mc{K}_\varepsilon (g\otimes \overline{g})\big\|_2^2\leq \varepsilon^{-1}\; \|f-g\|_{U^3}^2\; (\|f\|_2^2+\|g\|_2^2)^{\frac{1}{2}}\; (\|f\|_{U^3}\|f\|_2^{\frac{1}{2}}+\|g\|_{U^3}\|g\|_2^{\frac{1}{2}})^2.
\]
By H\"older's inequality with indices $4$ and $4/3$ applied to the last factor, the last line is at most
$\varepsilon^{-1}\; \|f-g\|_{U^3}^2\; (\|f\|_2^2+\|g\|_2^2)^{\frac{1}{2}}\; (\|f\|_{U^3}^{\frac{4}{3}}+\|g\|_{U^3}^{\frac{4}{3}})^{\frac{3}{2}}\; (\|f\|_2^2+\|g\|_2^2)^{\frac{1}{2}}$. The result follows.
\end{proof}

\noindent The third and final type of bounds concerns functions and $\ab$-matrices that are stable under the Fourier denoising operator,  particularly functions $f$ such that $\|K_\varepsilon(f)-f\|_2$ is small. We now prove various versions of the fact that the sum of stable functions (or $\ab$-matrices) is also stable.

\begin{lemma}\label{lem:cutdiff} For any two functions $f,g:\ab\to\mb{C}$ and any $\varepsilon>0$ we have
\begin{equation}\label{eq:cutdiff}
\|K_\varepsilon(f+g)-(f+g)\|_2\leq \|K_\varepsilon(f)-f\|_2+\|K_\varepsilon(g)-g\|_2.
\end{equation}
\end{lemma}
\begin{proof}
By Parseval and the definitions of $K_\varepsilon$, $q_\varepsilon$, and $q_\varepsilon'$ (recall Definition \ref{def:cont-cutoff}), we have
\begin{multline*}
 \|K_\varepsilon(f+g)-(f+g)\|_2  = \big(\textstyle\sum_{\chi\in \wh{\ab}} |\wh{K_\varepsilon(f+g)}(\chi)-(\wh{f(\chi)}+\wh{g(\chi)})|^2\big)^{1/2}\\
 =  \big(\textstyle\sum_{\chi\in \wh{\ab}} |q_\varepsilon(\wh{f}(\chi)+\wh{g}(\chi))-(\wh{f(\chi)}+\wh{g(\chi)})|^2\big)^{1/2}
 = \big(\textstyle\sum_{\chi\in \wh{\ab}} \big|q'_\varepsilon\big( \wh{f}(\chi)+\wh{g}(\chi) \big)\big|^2\big)^{1/2}.
\end{multline*}
By Lemma \ref{lem:q'}, this is at most $\big(\sum_{\chi\in \wh{\ab}} \big(\big|q'_\varepsilon\big( \wh{f}(\chi)\big)\big| +\big|q'_\varepsilon\big(\wh{g}(\chi) \big)\big|\big)^2\big)^{\frac{1}{2}}\leq \big(\sum_{\chi\in \wh{\ab}} \big|q'_\varepsilon\big( \wh{f}(\chi)\big)\big|^2\big)^{\frac{1}{2}}$ $+ \big(\sum_{\chi\in \wh{\ab}} \big|q'_\varepsilon\big( \wh{g}(\chi)\big)\big|^2\big)^{1/2} = \|K_\varepsilon(f)-f\|_2+\|K_\varepsilon(g)-g\|_2$.
\end{proof}

\begin{corollary}\label{cor:cutdiff3}
For any $\ab$-matrices $M_1,M_2,\ldots,M_n$, and any $\varepsilon\in (0,1]$, we have
\begin{equation}\label{eq:cutdiff3}
\Big\|\mc{K}_\varepsilon\Big(\textstyle\sum_{i\in [n]} M_i\Big)-\textstyle\sum_{i\in [n]} M_i\Big\|_2\leq \textstyle\sum_{i\in [n]} \|\mc{K}_\varepsilon(M_i)-M_i\|_2.
\end{equation}
\end{corollary}
\begin{proof}
First let us note that for any two such $\ab$-matrices $M$ and $N$ we have
\begin{equation}\label{eq:cutdiff2}
\|\mc{K}_\varepsilon(M+N)-M-N\|_2\leq \|\mc{K}_\varepsilon(M)-M\|_2+\|\mc{K}_\varepsilon(N)-N\|_2.
\end{equation}
Indeed, this follows from Lemma \ref{lem:cutdiff} applied to the diagonals (calculating the $L^2$-norms via diagonals as in previous proofs above). Now \eqref{eq:cutdiff3} follows from \eqref{eq:cutdiff2} by a simple induction.
\end{proof}
\noindent We now combine the above three types of bounds to give estimates that will be useful to handle various error terms that occur in applications of the regularity lemma proved in the next section.

\begin{proposition}\label{prop:reglemtool}
Let $\varepsilon\in(0,1]$, $n\in \mb{N}$, and  $\alpha_i \ge 0$ for $i\in [5]$. Let $f:\ab\to\mb{C}$ be a function satisfying $\|f\|_2\leq 1$, and suppose that we have a decomposition $f=\sum_{i=1}^n f_i+f_e+f_r$ where 
\begin{enumerate}
\item $\|f_e\|_2\leq\alpha_1\leq 1$,
\item $\|f_r\|_2\leq 1$ and $\|f_r\|_{U^3}\leq\alpha_2$,
\item $\|f_i\|_2\leq\alpha_3$ for every $i\in [n]$,
\item $\langle f_i,f_j\rangle_{U^3}\leq\alpha_4$ for all $i\neq j$ in $[n]$,
\item $\|\mathcal{K}_\varepsilon(f_i\otimes\overline{f_i})-f_i\otimes \overline{f}_i\|_2\leq\alpha_5$ for every $i\in [n]$.
\end{enumerate}
 Then
\begin{equation}\label{eq:reglemtool}
\big\|\textstyle\sum_{i=1}^n f_i\otimes\overline{f}_i-\mc{K}_\varepsilon(f\otimes\overline{f})\big\|_2\leq  5\alpha_1 + 16 \varepsilon^{-1/2}\alpha_2+ \varepsilon^{-1/2} n^2(n^{1/2}\alpha_3+3) \alpha_4 + n\alpha_5.
\end{equation}
\end{proposition}
\begin{proof}
Let $F=\sum_{i=1}^n f_i$. We begin by eliminating $f_e$ using the first kind of continuity (relative to $L^2$ errors). We have $f\otimes \overline{f} = (F+f_r)\otimes \overline{(F+f_r)}+(F+f_r)\otimes \overline{f_e}+f_e\otimes \overline{(F+f_r)} + f_e\otimes \overline{f_e}$.

By Lemma \ref{lem:KepsL2cont}, we have that $\|\mc{K}_\varepsilon(f\otimes \overline{f})-\mc{K}_\varepsilon((F+f_r)\otimes \overline{(F+f_r)})\|_2$ is at most
\begin{multline*}
\|(F+f_r)\otimes \overline{f_e}+f_e\otimes \overline{(F+f_r)} + f_e\otimes \overline{f_e}\|_2\\
 \leq\;  2\|F+f_r\|_2\,\|f_e\|_2+ \|f_e\|_2^2 \; \leq \;  2(1+\alpha_1)\alpha_1 +\alpha_1^2 \; \leq \;  2\alpha_1+3\alpha_1^2 \; \leq \; 5\alpha_1.
\end{multline*}

Next, we eliminate $f_r$ using the second kind of continuity (relative to $U^3$ errors). By the triangle inequality, we have that $\|\mc{K}_\varepsilon(f\otimes \overline{f})-\mc{K}_\varepsilon(F\otimes \overline{F})\|_2$ is at most
\begin{multline*}
\|\mc{K}_\varepsilon(f\otimes \overline{f})-\mc{K}_\varepsilon((F+f_r)\otimes \overline{(F+f_r)})\|_2+ \|\mc{K}_\varepsilon((F+f_r)\otimes \overline{(F+f_r)})-\mc{K}_\varepsilon(F\otimes \overline{F})\|_2\\
 \leq  5\alpha_1 + \|\mc{K}_\varepsilon((F+f_r)\otimes \overline{(F+f_r)})-\mc{K}_\varepsilon(F\otimes \overline{F})\|_2.
\end{multline*}
By Corollary \ref{cor:Fourier-cutoff-U3-perturb} with $f=F+f_r$ and $g=F$, and using $\|F+f_r\|_2\leq 1+\alpha_1$, $\|F\|_2\leq 2+\alpha_1$, and $\|\cdot\|_{U^3}\leq \|\cdot\|_{L^2}$, we see that $\|\mc{K}_\varepsilon((F+f_r)\otimes \overline{(F+f_r)})-\mc{K}_\varepsilon(F\otimes \overline{F})\|_2$ is at most
\begin{align*}
 &   \varepsilon^{-1/2} \|f_r\|_{U^3}(\|F+f_r\|_{U^3}^{4/3}+\|F\|_{U^3}^{4/3})^{3/4}\; (\|F+f_r\|_2^2+\|F\|_2^2)^{1/2}\\
 & \leq  \varepsilon^{-1/2}\alpha_2 ((1+\alpha_1)^{4/3}+(2+\alpha_1)^{4/3})^{3/4}\; ((1+\alpha_1)^2+(2+\alpha_1)^2)^{1/2}\\
 & \leq  \varepsilon^{-1/2}\alpha_2 (2^{4/3}+3^{4/3})^{3/4} (4+9)^{1/2} \leq 16 \varepsilon^{-1/2}\alpha_2.
\end{align*}
Thus, so far we have proved that 
\begin{equation}\label{eq:sofar1}
\|\mc{K}_\varepsilon(f\otimes \overline{f})-\mc{K}_\varepsilon(F\otimes \overline{F})\|_2\leq 5\alpha_1 + 16 \varepsilon^{-1/2}\alpha_2.
\end{equation}
Next, from the expansion $F\otimes \overline{F}= \sum_{i\in [n]} f_i\otimes\overline{f_i}+\sum_{i\neq j\in [n]} f_i\otimes\overline{f_j}$, we eliminate the latter cross-terms. By Lemma \ref{lem:DUcontrol}, we have that $\|\mc{K}_\varepsilon(F\otimes \overline{F})-\mc{K}_\varepsilon(\sum_{i\in [n]} f_i\otimes\overline{f_i})\|_2 $ is at most
\begin{multline*}
\varepsilon^{-1/2} \big\|\textstyle\sum_{i\neq j\in [n]} f_i\otimes\overline{f_j}\big\|_{DU^2}\big(\|F\otimes \overline{F} \|_2^2+\big\| \textstyle\sum_{i\in [n]} f_i\otimes\overline{f_i}\big\|_2^2\big)^{1/4}
\\ \leq  \varepsilon^{-1/2} \textstyle\sum_{i\neq j\in [n]} \langle f_i, f_j\rangle_{U^3}\,\big((2+\alpha_1)^4+n^2\alpha_3^4\big)^{1/4} 
\\ \leq \tfrac{\alpha_4}{\varepsilon^{1/2}} n^2 (2+\alpha_1+n^{1/2}\alpha_3) \leq \tfrac{\alpha_4}{\varepsilon^{1/2}}  n^2(3+n^{1/2}\alpha_3).
\end{multline*}
Combining this with \eqref{eq:sofar1} and the triangle inequality, we deduce that
\begin{eqnarray}\label{eq:sofar2}
& \big\|\mc{K}_\varepsilon(f\otimes \overline{f})-\mc{K}_\varepsilon(\sum_{i\in [n]} f_i\otimes\overline{f_i})\big\|_2 \leq 5\alpha_1 + 16 \varepsilon^{-1/2}\alpha_2 + \varepsilon^{-1/2} n^2(n^{1/2}\alpha_3+3)\alpha_4 .&
\end{eqnarray}
Now, by Corollary \ref{cor:cutdiff3} and property $(iv)$, we have $\|\sum_{i\in [n]} f_i\otimes\overline{f_i}-\mc{K}_\varepsilon(\sum_{i\in [n]} f_i\otimes\overline{f_i})\|_2\leq n\alpha_5$. Then, combining this with \eqref{eq:sofar2} via the triangle inequality, we deduce \eqref{eq:reglemtool}.
\end{proof}

\section{Nilspace characters}\label{sec:nschars}

\noindent In this section, we begin to use nilspace theory to progress towards a central goal of this paper, i.e., proving the validity of our spectral algorithms in quadratic Fourier analysis (Theorem \ref{thm:reg-intro}). To do so, in this section, we introduce and study what will be the main nilspace-theoretic notion of a higher-order Fourier analytic component of a function: the notion of a \emph{nilspace character}. In particular, we shall establish in Theorem \ref{thm:nscharsquadchars} a fact that will be important for us, namely that a 2-step nilspace character of bounded complexity on a finite abelian group $\ab$ is a \emph{character of order 2} on $\ab$ (as per Definition \ref{def:k-char}), with a bounded complexity parameter depending on the precision parameter but not on the cardinality of $\ab$. To explain this in more detail, we need to recall certain basic notions from nilspace theory. 

In addition to the original paper \cite{CamSzeg}, there are various subsequent sources treating nilspace theory in depth, including  \cite{Cand:Notes1,Cand:Notes2,GMV1,GMV2,GMV3}. Therefore, rather than stating the full definitions here, let us mention the basic concepts of this theory and give references for their definitions. The basic concepts that we shall use include that of a \emph{nilspace} \cite[Definition 1.2.1]{Cand:Notes1}, a \emph{compact nilspace} \cite[Definition 1.0.2]{Cand:Notes2} and \emph{compact finite-rank \textup{(}\textsc{cfr}\textup{)} nilspaces} \cite[Definition 1.0.2]{Cand:Notes2}, as well as the \emph{morphisms of nilspaces} \cite[Definition 2.2.11]{Cand:Notes1}. We shall also use the \emph{structure groups} and \emph{bundle structure} of nilspaces \cite[Definitions 3.2.17 and 3.2.18, and Theorem 3.2.19]{Cand:Notes1}, as well as the existence of a unique probability Haar measure on each compact nilspace \cite[Proposition 2.2.5]{Cand:Notes2}. It is worth recalling the notion of \emph{nilspace polynomial}, see \cite[Definition 1.3]{CSinverse}.

\begin{defn}\label{def:nil-poly}
A $k$-step \emph{nilspace polynomial} on a finite abelian group $\ab$ is a function of the form $F\co\phi:\ab\to\mb{C}$ constructed as follows: there is a $k$-step \textsc{cfr} nilspace $\ns$ such that the map $\phi$ is a nilspace morphism from $\ab$ (viewed as a 1-step nilspace) into $\ns$ and $F:\ns\to\mb{C}$ is a 1-bounded continuous function. Usually, additional parameters are specified, such as a parameter bounding the complexity of the  nilspace $\ns$ as well as the Lipschitz norm of $F$ (see \cite{CSinverse}).
\end{defn}
\noindent Nilspace characters are a special type of nilspace polynomials. To define nilspace characters, we first need to extend, to the setting of compact nilspaces, the notion (well-known in the setting of nilmanifolds) of a function having a \emph{vertical frequency} (see \cite[Definition 6.1]{GTZ} and more recently \cite[Definition 2.15]{LSS}).

\begin{defn}[Functions with vertical frequency on a compact nilspace]
Let $\ns$ be a $k$-step compact nilspace, let  $\ab_k$ be the $k$-th structure group of $\ns$ (a compact abelian group), and let $\chi$ be a character in $\wh{\ab_k}$. We say that a function $F:\ns\to\mb{C}$ has \emph{vertical frequency} $\chi$ if it satisfies the following equation for every $x\in \ns$ and every $z\in \ab_k$: $F(x+z)=\chi(z)F(x)$. Given a bounded Borel function $F:\ns\to\mb{C}$ and $\chi\in\wh{\ab}_k$, we denote by $F_\chi$ the \emph{component of $F$ with vertical frequency $\chi$}, defined as follows (where $\mu_{\ab_k}$ is the probability Haar measure on $\ab_k$):
\begin{equation}\label{eq:F_chi}
F_\chi(x)=\int_{\ab_k}F(x+z)\overline{\chi(z)}\ud\mu_{\ab_k}(z).
\end{equation}
\end{defn}
It is easy to see that functions with vertical frequency $\chi$ form a closed subspace of the Hilbert space $L^2(\ns)$. We thus obtain the following notion (see \cite[Definition 3.50]{CScouplings}).

\begin{defn}
Let $\ns$ be a $k$-step compact nilspace, let $\ab_k$ be the $k$-th structure group of $\ns$, and let $\chi$ be a character in $\wh{\ab_k}$. We denote by $W(\chi,\ns)$ the Hilbert space of functions in $L^2(\ns)$ that have vertical frequency $\chi$.
\end{defn}

\begin{remark}\label{rem:WX}
Let $\ns$ be a $k$-step \textsc{cfr} nilspace, let $\chi\in \wh{\ab_k(\ns)}$, and let $p\in[1,\infty]$. Then equation \eqref{eq:F_chi} defines a continuous linear operator $(\cdot)_\chi:L^p(\ns)\to L^p(\ns)$ of norm 1. The space $W^p(\chi,\ns):=\{f\in L^p(\ns):f\text{ has vertical frequency }\chi\}$ is a closed subspace of $L^p(\ns)$.  Moreover, for $p=2$, we have the decomposition $L^2(\ns)=\bigoplus_{\chi\in \wh{\ab_k(\ns)}} W^2(\chi,\ns)$ where the subspaces $W^2(\chi,\ns)$ are pairwise orthogonal. We will not need these results in this paper, as we will use mainly the components $F_\chi$ of much more specific functions $F$ (Lipschitz continuous functions), so we omit the proofs.
\end{remark}
We can now formulate the main concept in this section.
\begin{defn}[Nilspace character]\label{def:nilspace-char}
A $k$-step \emph{nilspace character} on a finite abelian group $\ab$ is a 1-bounded function $g_\chi$ of the following form: there is a $k$-step \textsc{cfr} nilspace $\ns$, a nilspace morphism $\phi:\mc{D}_1(\ab)\to\ns$, and a 1-bounded continuous function $F_\chi:\ns\to\mb{C}$ having vertical frequency $\chi$, such that $g_\chi=F_\chi\co\phi$.  
\end{defn}
\noindent Various additional properties can be required for such  nilspace characters (as for nilspace polynomials more generally): we can require $F_\chi$ to have bounded Lipschitz constant $C$ (we then say that $g_\chi$ is a \emph{$C$-Lipschitz nilspace character}); we can require the complexity\footnote{See \cite[Definition 1.2]{CSinverse}.} of the underlying nilspace to be at most $m$ (we then say that $g_\chi$ is a \emph{complexity-$m$ nilspace character}); we can also add an  equidistribution requirement, namely that the  morphism $\phi$ should be $b$-balanced\footnote{See \cite[Definition 5.1]{CSinverse}.} (we then say that $g_\chi$ is a \emph{$b$-balanced nilspace character}). Furthermore, one can require that the modulus of the character be equal to 1 everywhere. Note that with the latter requirement, when $\ns$ is a nilmanifold, we recover the notion of a \emph{nilcharacter} introduced by Green, Tao and Ziegler in \cite[Definition 6.1]{GTZ}. Note also that, on the other hand, the modulus-1 requirement can often clash with the continuity requirement for $F_\chi$ (see the discussion of topological obstructions in \cite[p.\ 1254]{GTZ}), in which case one can either weaken the continuity to piecewise continuity, or work with multidimensional generalizations (as in \cite{GTZ}). Finally, note that the concept of nilspace character can be extended to compact abelian groups, by requiring that $\phi$ be a \emph{continuous} morphism. As usual in this paper, we focus on finite abelian groups $\ab$.

In the following subsections, we develop various properties of nilspace characters, especially to prove that they are adequate higher-order generalizations of classical Fourier characters.

\subsection{Uniform approximation of Lipschitz functions by vertical Fourier sums}\hfill\smallskip\\
The structure theorem (or regularity lemma) from \cite[Theorem 1.5]{CSinverse} yields, as a key ingredient of the structured part in the resulting decomposition, a function on a compact nilspace that has a bounded Lipschitz norm. It will be very useful for later calculations to be able to approximate such a function by a sum of boundedly many functions with vertical frequency. That is the purpose of the main result in this subsection. To obtain this, we first gather some tools related to Lipschitz functions. We shall recall some basic definitions and notation following \cite[\S 8.4.2]{CMN}. 

Given a metric space $X$, we say that a metric $d$ on $X$ is a \emph{compatible metric} if $d$ generates the given topology on $X$.
\begin{defn}
Let $X$ be a compact metric space with a compatible metric $d$. For any continuous function $f:X\to \mb{C}$, we write $\|f\|_L$ for the quantity $\sup_{x\not=y}\frac{|f(x)-f(y)|}{d(x,y)}$. We say that $f$ is \emph{Lipschitz} if $\|f\|_L$ is finite, and we say that $f$ is \emph{$C$-Lipschitz} if $\|f\|_L\le C$. We define $\|f\|_{\textup{sum}}:=\|f\|_L+\|f\|_\infty$.
\end{defn}
\noindent Note that $\|\cdot\|_{\textup{sum}}$ is a norm whereas $\|\cdot\|_L$ is only a seminorm.

Given a $k$-step compact nilspace $\ns$, we say a compatible metric $d$ on $\ns$ is \emph{$\ab_k$-invariant} (where $\ab_k$ is the $k$-th structure group of $\ns$) if $d(x+z,y+z)=d(x,y)$ for all $x,y\in \ns$ and $z\in\ab_k$. 

\begin{lemma}\label{lem:proj-Lip}
Let $\ns$ be a $k$-step compact nilspace and let $d$ be a $\ab_k$-invariant compatible metric on $\ns$. Then, for every Lipschitz function $F:\ns\to \mb{C}$ \textup{(}relative to $d$\textup{)} and every $\chi\in \widehat{\ab_k}$, the function $F_\chi:\ns\to\mb{C}$, $x\mapsto \int_{\ab_k}F(x+z)\overline{\chi(z)}\ud z$ satisfies $\|F_\chi\|_L\leq \|F\|_L$ and $\|F_\chi\|_\infty\leq \|F\|_\infty$.
\end{lemma}

\begin{proof} For all $x,y\in \ns$, using the $\ab_k$-invariance we have
\[
|F_\chi(x)-F_\chi(y)|\le \textstyle\int_{\ab_k} |F(x+z)-F(y+z)|\ud z \le \textstyle\int_{\ab_k} C d(x+z,y+z)\ud z=C d(x,y),
\]
so $\|F_\chi\|_L\leq \|F\|_L$. We also have $|F_\chi(x)|\le \int_{\ab_k} |F(x+z)|\ud z \leq \|F\|_\infty$.
\end{proof}

\begin{remark}
Note that for Lemma \ref{lem:proj-Lip} to work, we needed the metric $d$ on $\ns$ to be $\ab_k$-invariant. By \cite[Lemma 2.1.11]{Cand:Notes2}, one can equip every compact $k$-step nilspace with such a metric.\footnote{Later, we will equip $k$-step \textsc{cfr} nilspaces with metrics coming from Riemannian tensors, see Appendix \ref{app:metrics}. For those we can also guarantee $\ab_k$-invariance, see Lemma \ref{lem:inv-riem-metric}. }
\end{remark}

\begin{remark}
The Lipschitz constant of a function is not a topological property of the space, in the sense that given a compact metric space $S$ and a continuous function $F:S\to \mb{C}$, different compatible metrics can make $F$ be Lipschitz or not. This holds even if we assume that the space is a $k$-step compact nilspace and all the considered metrics are $\ab_k$-invariant. For instance, let $\mb{T}=\mb{R}/\mb{Z}$ and let $|x|_{\mb{T}}$ be the usual distance from $x$ to the nearest integer, under the usual identification of $x$ with a real number in $[0,1)$. Clearly $d(x,y):=|x-y|_{\mb{T}}$ is a $\mb{T}$-invariant metric on $\mb{T}$. Moreover, so is $d_\alpha(x,y):=|x-y|_{\mb{T}}^\alpha$ for any $\alpha\in(0,1]$. Let $F:\mb{T}\to \mb{R}$ be the function defined by $F(x)=\sqrt{x}$ for $x\in[0,1/2)$ and $F(x)=\sqrt{1-x}$ otherwise. It can be checked that $F$ is a 1-Lipschitz function relative to the metric $d_{1/2}$ on $\mb{T}$, but it is not Lipschitz relative to $d$. 
\end{remark}
\noindent When working on \textsc{cfr} nilspaces, problematic metrics, such as in the above example, can be conveniently ruled out by focusing on metrics induced by the manifold structure of these spaces, i.e.\ by \emph{Riemannian metrics}. This generalizes a standard approach in the setting of nilmanifolds, e.g., in \cite{GTZ}. We gather some basic facts on Riemannian metrics in Appendix \ref{app:metrics} (see in particular Definition \ref{def:metric-cfr-nil}) and use these to get the following approximation result announced above. For simplicity, compact abelian Lie groups will always be assumed to be endowed with the \emph{flat metric}, see Definition \ref{def:metric-cpct-lie-group}.\footnote{Modulo minor changes in the constants, all results work the same way if we endow compact abelian Lie groups with any Riemannian metric. This is because by Proposition \ref{prop:eq-metrics-cfr-nil}, the identity map $\id:\ns\to\ns$ is bi-Lipschitz when we consider any two Riemannian metrics on $\ns$. Hence, if a map $F:\ns\to\mb{C}$ is Lipschitz with respect to a Riemannian metric, then it is also Lipschitz relative to any other Riemannian metric. We fix one such metric in Definition \ref{def:metric-cpct-lie-group} and use this one unless otherwise stated.}

\begin{lemma}\label{lem:unif-approx-CFRns}
Let $\ns$ be a $k$-step \textsc{cfr} nilspace endowed with a Riemannian metric. Let $\ab_k$ be the $k$-th structure group of $\ns$. Then there is a function $\nsR_{\ns}:\mb{R}_{>0}\times \mb{R}_{\ge 0}\to \mb{N}$ such that the following holds. For any $C,\delta>0$, there is a set $S\subset \wh{\ab_k}$ with $|S|\leq \nsR_{\ns}(\delta,C)$ such that, for any $C$-Lipschitz function $F:\ns\to \mb{C}$ \textup{(}relative to the given metric\textup{)}, we have $\|F-\sum_{\chi\in S} F_\chi\|_\infty \le \delta$.
\end{lemma}
This includes, as a special case, a uniform approximation result for Lipschitz functions on nilmanifolds. Uniform approximation results by sums of functions with vertical frequency have been used in this setting previously (e.g.\ \cite[Lemma 6.4]{GTZ}). Note, however, that Lemma \ref{lem:unif-approx-CFRns} yields an approximation by a partial Fourier series of $F$ relative to the last structure group, which is more precise than an approximation in terms of \emph{some} functions with vertical frequencies.

The proof of Lemma \ref{lem:unif-approx-CFRns} will use the following special case for compact abelian Lie groups.
\begin{lemma}\label{lem:unif-app-lie-gr}
Let $G$ be a compact abelian Lie group. Then there exists $\nsR_G:\mb{R}_{> 0}\times \mb{R}_{\ge 0}\to \mb{N}$ such that the following holds. For any $C,\delta>0$, there is a set $S\subset \wh{G}$ with $|S|\leq \nsR_G(\delta,C)$ such that for any $C$-Lipschitz function $f:G\to \mb{C}$ we have $\|f-\sum_{\chi\in S}\wh{f}(\chi)\chi\|_\infty<\delta$.
\end{lemma}

\begin{proof}
We know that $G\cong \ab\times \mb{T}^n$ for some finite abelian group $\ab$ and some integer $n\geq 0$.

First, suppose that $\ab$ is trivial, so that $G\cong  \mb{T}^n$. Then, it follows from known results (due to Lebesgue for $n=1$, and to Zhizhiashvili for $n>1$, see \cite[p.\ 201]{Lebesgue} and \cite[p.\ 642]{Golubov} respectively) that if $f$ is $C$-Lipschitz on $\mb{T}^n$ then the Fourier sums $S_M(f):=\sum_{r\in\mb{Z}^n:\|r\|_{\ell^\infty}\leq M} \wh{f}(\chi_r) \chi_r$ on $\mb{T}^n$ satisfy $\|S_M(f)-f\|_\infty\leq K_{C,n}\frac{\log^n M}{M}$, where the constant $K_{C,n}$ depends only on $C$ and $n$.\footnote{In \cite[p.\ 642]{Golubov} the metric used on $\mb{T}^n$ is the Euclidean one; but as noted earlier, by Proposition \ref{prop:eq-metrics-cfr-nil}, this is equivalent to the fixed Riemannian metric on $\mb{T}^n$, and the associated constants are taken into account in $K_{C,n}$.} Hence the claim holds for $M$ sufficiently large in terms of $\delta$, $C$ and $n$.

Now if $G=\ab\times \mb{T}^n$ for a general finite abelian group $\ab$, then for every $a\in \ab$ and $x\in \mb{T}^n$ we have $f(a,x)=\sum_{\xi\in \wh{\ab}} \big(\mb{E}_{z\in \ab} f(z,x)\overline{\xi(z)}\big)\,\xi(a)$. For every $z\in \ab$ the function $f(z,\cdot)$ is $C$-Lipschitz on $\mb{T}^n$, so we can apply the previous paragraph with $M$ such that $\|S_M(f(z,\cdot))-f(z,\cdot)\|_\infty\leq \delta/|\ab|$. The result follows with $S=\{(\xi,\chi_r)\in\wh{\ab}\times \wh{\mb{T}^n} = \wh{G}: \|r\|_{\ell^\infty}\leq M\}$.
\end{proof}

\begin{proof}[Proof of Lemma \ref{lem:unif-approx-CFRns}]
Let $\ns\times \ab_k(\ns)\to \ns$, $(x,z)\mapsto x+z$ be the addition map. On $\ns\times \ab_k(\ns)$, we consider the usual Riemannian metric tensor which is the product of the Riemannian tensor of $\ns$ with the usual Riemannian tensor on the compact abelian group $\ab_k(\ns)$. Note that, in particular, $(x,z)\mapsto x+z$ is a morphism between $k$-step \textsc{cfr} nilspaces (regarding $\ab_k(\ns)$ as the group nilspace $\mc{D}_k(\ab_k(\ns))$) and by Proposition \ref{prop:eq-metrics-cfr-nil} we know that it is Lipschitz. Hence the map $(x,z)\mapsto F(x+z)$ (being the composition of two Lipschitz maps) is $C'$-Lipschitz with $C'$ depending only on $\ns$ and $C$. We claim that, for any fixed $x_0\in \ns$, the map $z\mapsto F(x_0+z)$ is also $C'$-Lipschitz.

Indeed, we have $|F(x_0+z)-F(x_0+z')|\le C'd_{\ns\times \ab_k(\ns)}((x_0,z),(x_0,z'))$. We claim that, for all $z,z'\in \ab_k(\ns)$ and $x_0\in \ns$, we have $d_{\ab_k(\ns)}(z,z')=d_{\ns\times \ab_k(\ns)}((x_0,z),(x_0,z'))$. Note that $(x_0,z)$ and $(x_0,z')$ are in different connected components if and only if $z$ and $z'$ are in different connected components. Hence, in this case we have $d_{\ab_k(\ns)}(z,z')=1=d_{\ns\times \ab_k(\ns)}((x_0,z),(x_0,z'))$ by definition. If $z$ and $z'$ lie in the same connected component, let $\gamma:[0,1]\to \ns\times \ab_k(\ns)$, $t\mapsto(x(t),z(t))$ be a path connecting $(x_0,z)$ with $(x_0,z')$. Then $L_\gamma=\int_0^1 \sqrt{g_{\ns\times \ab_k(\ns)}(\nabla\gamma,\nabla\gamma)}$ where $g_{\ns\times \ab_k(\ns)}$ is the Riemannian metric tensor on $\ns\times \ab_k(\ns)$. As $g_{\ns\times \ab_k(\ns)}(\nabla\gamma,\nabla\gamma) = g_{\ns}(\nabla x(t),\nabla x(t))+g_{\ab_k(\ns)}(\nabla z(t),\nabla z(t))$, if we let $\gamma':[0,1]\to\ab_k(\ns)$, $t\mapsto z(t)$, we have that $L_\gamma \ge L_{\gamma'}$. Thus, $d_{\ns\times\ab_k(\ns)}((x_0,z),(x_0,z'))\ge d_{\ab_k(\ns)}(z,z')$. By a similar argument we obtain the reversed inequality, and the claim follows.

Therefore, since the map $z\mapsto F(x_0+z)$ is $C'$-Lipschitz, we can apply Lemma \ref{lem:unif-app-lie-gr} to it and complete the proof (noting that the Lipschitz constant $C'$ does not depend on $x_0\in \ns$). 
\end{proof}

\begin{remark}\label{rem:explicit-set-S}
Note that the proof of Lemma \ref{lem:unif-app-lie-gr} shows that, in Lemma \ref{lem:unif-approx-CFRns}, letting $\ab$ be a finite abelian group and $n\geq 0$ be an integer such that $\ab_k(\ns)\cong \ab\times \mb{T}^n$, we can add the information that the set $S$ is of the concrete form $S=\{(\xi,\chi_r)\in \wh{\ab}\times \wh{\mb{T}^n}: r\in \mb{Z}^n, \|r\|_{\ell^\infty}\leq M\}$ for some sufficiently large integer $M$.
\end{remark}

We close this section by discussing another advantage that functions with a uniform bound on their Lipschitz constant present for our later calculations, namely, that we can then have uniform control on the convergence of integrals. Before we go into this, let us recall a standard example showing that such uniform control fails without a uniform Lipschitz bound. 

\begin{example}
Let $I=[0,1]$ with the usual topology and for each positive integer $n$ consider the Dirac measure $\delta_{1/n}$. These measures converge weakly to the Dirac measure $\delta_0$, i.e., for every continuous function $f:I\to\mb{C}$ we have $\int f\ud\delta_{1/n}\to \int f\ud\delta_0$ as $n\to \infty$. However,  this convergence is not uniform in $f$, as can be seen with the sequence of continuous functions $f_n(x)=(1-nx)\mathbf{1}_{[0,1/n]}(x)$, $n\in\mb{N}$, which satisfies $\int f_n \ud\delta_{1/n}-\int f_n \ud\delta_0=-1$, and so $\sup_{\|f\|_\infty\le 1} \left|\int f\ud\delta_{1/n}-\int f\ud\delta_0\right|\not\to 0 \text{ as }n\to\infty$.
\end{example}
\noindent Such examples could a priori be problematic for our proofs later, because the \emph{balance} property (recalled below) is defined in terms of weak convergence of measures, and we then do not have uniform convergence of integrals. However, restricting our attention to Lipschitz functions with uniform Lipschitz constant will enable us to avoid such issues. A convenient way to implement such a uniform control is to use the Kantorovich-Rubinstein (or Wasserstein) norm.
\begin{defn}[Kantorovich-Rubinstein norm]
Let $X$ be a compact metric space, and let $\mc{P}(X)$ denote the set of Borel probability measures on $X$. The \emph{Kantorovich—Rubinstein norm} on $\mc{P}(X)$ is defined as follows:
\[
\|\mu\|_{KR}:=\sup_{f:X\to\mb{C}\textrm{ Lipschitz with }\|f\|_{\textup{sum}}\le 1}\left| \int f\ud\mu \right|.
\]
\end{defn}
\noindent The full definition of $\|\mu\|_{KR}$ is as a norm on the vector space of Borel measures with bounded variation (see \cite[Theorem 8.4.5]{CMN}). A convenient fact about this norm is its direct relation with weak convergence of measures, as the following result indicates (this is a direct consequence of \cite[Theorem 8.4.13]{CMN}).

\begin{theorem}\label{thm:weak-kr-eq}
Let $X$ be a compact metric space and let $\mu,\mu_1,\mu_2,\ldots$ be measures in $\mc{P}(X)$. Then $\mu_j\to \mu$ weakly if and only if $\|\mu_j-\mu\|_{KR} \to 0$.
\end{theorem}

\noindent We will use Theorem \ref{thm:weak-kr-eq} to transform the \emph{balance} property of the regularity lemma \cite[Theorem 1.5]{CSinverse} into a bound for the Kantorovich-Rubinstein norm. Let us recall the definition of balance (see \cite[Definition 5.1]{CSinverse}).

\begin{defn}
Let $\nss$ be a $k$-step \textsc{cfr} nilspace. For each $n\in \mb{N}$ fix a metric $d_n$ on $\mc{P}(\cu^n(\nss))$. Let $\ns$ be a $k$-step compact nilspace and $\phi:\ns\to \nss$ be a continuous morphism. For $b>0$, we say that $\phi$ is $b$-\emph{balanced} if for every $n\le 1/b$ we have $d_n(\mu_{\cu^n(\ns)}\co (\phi^{\db{n}})^{-1},\mu_{\cu^n(\nss)})\le b$.
\end{defn}

\begin{corollary}\label{cor:equiv-weak-kr}
Let $\nss$ be a $k$-step \textsc{cfr} nilspace. For each $n\in \mb{N}$ fix a metric $d_n$ on $\mc{P}(\cu^n(\nss))$ and a metric $d_n'$ on $\cu^n(\nss)$. Then the following properties hold:
\setlength{\leftmargini}{0.5cm}
\begin{itemize}
    \item For every $\delta>0$ there exists $b=b(\delta)>0$ such that if $\phi:\ns\to \nss$ is a $b$-balanced morphism, then for all $n\leq 1/\delta$ we have $\|\mu_{\cu^n(\ns)}\co (\phi^{\db{n}})^{-1}-\mu_{\cu^n(\nss)}\|_{KR}\le \delta$.
    \item For every $b>0$ there exists $\delta=\delta(b)>0$ such that if for every  $n\leq 1/\delta$ we have $\|\mu_{\cu^n(\ns)}\co (\phi^{\db{n}})^{-1}-\mu_{\cu^n(\nss)}\|_{KR}\le \delta$, then $\phi$ is $b$-balanced.
\end{itemize}
\end{corollary}

\begin{proof}
For any fixed $n\ge 0$, the space $\mc{P}(\cu^n(\nss))$ is endowed with the weak topology, which is metrizable. The balance property is using by definition one such possible metric $d_n$, see \cite[Definition 5.1]{CSinverse}. By Theorem \ref{thm:weak-kr-eq}, the Kantorovich-Rubinstein norm also defines a metric for the weak topology on $\mc{P}(\cu^n(\nss))$. The result follows, as by Theorem \ref{thm:weak-kr-eq} convergence of a sequence $\mu_i\in \mc{P}(\cu^n(\nss))$ relative to $d_n$ is equivalent to convergence relative to $\|\cdot\|_{KR}$.
\end{proof}

\subsection{2-step nilspace characters are quadratic characters}\hfill\smallskip\\
In this section we prove the following result. 
\begin{theorem}\label{thm:nscharsquadchars}
Let $\ns$ be a 2-step \textsc{cfr} nilspace and let $\sigma\in(0,1/2)$. Let $\ab$ be a finite abelian group, let $\phi:\ab\to\ns$ be a morphism, and let $F_\chi\co\phi$ be a 2-step nilspace character on $\ab$ with vertical frequency $\chi$. Then, for every $h\in \ab$, we have $\Delta_h(F_\chi\co\phi)=\sum_{\gamma\in S_h} \lambda_{h,\gamma}\gamma(z)+\mc{E}_h(z)$ where $S_h\subset \wh{\ab}$ satisfies $|S_h|= O_{\ns,\|F\|_{\textup{sum}}}(\sigma^{-O_{\ns}(1)})$, $\mc{E}_h:\ab\to \mb{C}$ satisfies $\|\mc{E}_h\|_\infty\le \sigma$, and for all $h\in \ab$ and $\gamma\in S_h$ we have $|\lambda_{h,\gamma}|\le \|F_\chi\|_\infty^2$.
\end{theorem}
\begin{remark}\label{rem:qcharparams}
In particular, this establishes that any such 2-step nilspace character $F_\chi\co\phi$ is a 1-bounded $(R,\sigma)$-character of order 2 on $\ab$, in the sense of Definition \ref{def:k-char}, with precision parameter $\sigma$ and complexity parameter $R=O_{\ns,\|F\|_{\textup{sum}}}(\sigma^{-O_{\ns}(1)})$. Note the additional strength (compared with Definition \ref{def:k-char}) in that the error $\mc{E}_h$ here is small in $\|\cdot\|_\infty$, not just in $\|\cdot\|_2$.
\end{remark} 
\begin{proof}
Since $\ns$ is 2-step and of finite rank, the 1-step nilspace factor $\ns_1$ is isomorphic to a compact abelian Lie group, which is therefore of the form $A\times \mb{T}^\ell$ for some finite abelian group $A$ and some $\ell\in \mb{Z}_{\ge 0}$. In particular, $\ns_1$ has exactly $|A|$ connected components.

Now observe that $\pi_1\co\phi$ is a morphism between 1-step nilspaces (from $\mc{D}_1(\ab)$ to $\ns_1$), so it is an affine homomorphism of the corresponding abelian groups, and we therefore have $\pi_1\co \phi=\psi+\pi_1(\phi(0))$ where $\psi:\ab\to A\times \mb{T}^\ell$ is a continuous group homomorphism. 

Let $\tran(\ns_1)^0$ be the connected component of the identity of the translation group of $\ns_1$. Note that this is isomorphic to the subgroup $\{0\}\times \mb{T}^\ell$ of $A\times \mb{T}^\ell$ (since $\tran(\ns_1)\cong A\times \mb{T}^\ell=\ab_1(\ns_1)$). Thus, for every element $(0,t)\in A\times \mb{T}^\ell$ there is an element $\beta^{(1)}_t=\beta^{(1)}_{(0,t)}\in \tran(\ns_1)^0$ such that, for every $x\in \ns$, $\beta^{(1)}_t\pi_1(x) = \pi_1(x)+(0,t)$. Let $h:\tran(\ns)\to \tran(\ns_1)$ be the continuous homomorphism defined as in \cite[Lemma 2.9.3]{Cand:Notes2}, i.e., for $\alpha\in \tran(\ns)$ we let $h(\alpha)$ be the unique map such that, for every $x\in \ns$, we have $h(\alpha)(\pi_1(x))=\pi_1(\alpha(x))$. By \cite[Theorem 2.9.10 $(ii)$]{Cand:Notes2}, the restriction of $h$ to $\tran(\ns)^0$ defines a surjective homomorphism between Lie groups $h|_{\tran(\ns)^0}:\tran(\ns)^0\to \tran(\ns_1)^0$. By the Open Mapping Theorem for topological groups, $h|_{\tran(\ns)^0}$ is an open map and in particular $\tran(\ns_1)^0\cong \tran(\ns)^0/\ker(h|_{\tran(\ns)^0})$. Hence, the map $h|_{\tran(\ns)^0}$ induces an open Hausdorff equivalence relation, see \cite[Chapter 1, \S 5, 2 and \S 8, 3]{bourbaki-gen-top}. By \cite[p.\ 107, Proposition 10]{bourbaki-gen-top} there exists a compact $B\subset \tran(\ns)^0$ such that $h|_{\tran(\ns)^0}(B)=\tran(\ns_1)^0$. Thus, for every $t\in \mb{T}^\ell$, we let $\beta_t\in B$ be such that $h|_{\tran(\ns)^0}(\beta_t)=\beta_1^{(1)}$.\footnote{There may be different elements $\beta_t,\beta_t'\in B$ such that $h|_{\tran(\ns)^0}(\beta_t)=h|_{\tran(\ns)^0}(\beta_t')=\beta_1^{(1)}$. For each $t$ we simply choose one such element. This choice can be made in a measurable way, but we omit the proof as we shall not use this fact.} Note that $B$ can be chosen depending only on $\ns$. Thus, in what follows, we assume that for each $\ns$, we make one such choice of a set $B$. Hence, if some quantity or object depends on $B$, we will simply say that it depends on $\ns$.

Now, for each $a\in A$, consider the \emph{translation bundle} corresponding to $\beta^{(1)}_a$, i.e., the following standard construction in nilspace theory (see \cite[Definition 3.3.34]{Cand:Notes1}):
\[
\nss_a := \mc{T}(\ns,\beta^{(1)}_a,1) = \{(x_0,x_1)\in \ns^2: \pi_1(x_0)+(a,0)=\pi_1(x_1)\}.
\]
By \cite[Lemma 3.3.35]{Cand:Notes1} we know that $\nss_a$ is a 2-step compact nilspace that is a subnilspace of the arrow space $\ns\Join_1 \ns$. Abusing the notation we will identify $\nss_{(a,0)}=\nss_a$.

By \cite[Lemma 3.3.38]{Cand:Notes1}, the 1-step canonical factor $\pi_1(\nss_a)$ (denoted $\mc{T}^*$ in \cite{Cand:Notes1}) is, from the purely algebraic viewpoint, a degree-$1$ extension of $\ns_1$ by $\ab_2=\ab_2(\ns)$ and, by the results from \cite[Section 2.1]{Cand:Notes2}, this extension is a continuous $\ab_2$-bundle over $\ns_1$. In particular $\nss_1$ is a compact 1-step nilspace, i.e., an affine \emph{compact abelian group}, which is an extension of the affine compact abelian group $\ns_1$ by $\ab_2$.

Recall also that, by \cite[Definition 3.3.34 and Proposition 3.3.36]{Cand:Notes1}, the nilspace $\nss_a$ is a continuous $\ab$-bundle over $\pi_1(\nss_a)$, for the Polish group $\ab=\ab_2^\Delta=\{(z,z):z\in \ab_2\}\leq \ab_2\times \ab_2$. In other words, the nilspace factor map $\pi_1:\nss_a\to\pi_1(\nss_a)$ is the map sending any $(x_0,x_1)\in \nss_a$ to its orbit under the $\ab$-action, i.e., the action of the diagonal subgroup $\ab_2^\Delta$.

Let $p$ denote the projection homomorphism $(a,t)\mapsto (a,0)$ on $A\times \mb{T}^\ell$. 

Our aim is to show that the multiplicative derivative  $\Delta_h(F_\chi\co\phi):z\mapsto \overline{F_\chi\co\phi(z)} F_\chi(\phi(z+h))$ factors through $\nss_{p(\psi(h))}$. To this end, we note the decomposition $\Delta_h(F_\chi\co\phi)(z) = V_h \co \delta_h(z)$, where
\[
\begin{array}{cccc}
\delta_h: & \mc{D}_1(\ab) & \to & \nss_{p(\psi(h))} \\
& z & \mapsto& \big(\phi(z),\beta_{p(\psi(h))-\psi(h)} \,\phi(z+h)\big)
\end{array}
\]
and
\[
\begin{array}{cccc}
V_h: & \nss_{p(\psi(h))} & \to & \mb{C} \\
& (x_0,x_1) & \mapsto & \overline{F_\chi(x_0)} F_\chi(\beta_{p(\psi(h))-\psi(h)}^{-1} \,x_1).
\end{array}
\]
Let us check first that these maps are well-defined. The first thing to note is that $p(\psi(h))-\psi(h)$ is in $\{0\}\times \mb{T}^\ell$, so the lift $\beta_{p(\psi(h))-\psi(h)}$ indeed exists as defined above. Next, we need to prove that $\delta_h(z)\in \nss_{p(\psi(h))}$. That is, we need to check that
\[
\pi_1(\phi(z))+p(\psi(h))=\pi_1\big(\beta_{p(\psi(h))-\psi(h)}\, \phi(z+h)\big).
\]
But the right hand side here equals by definition $p(\psi(h))-\psi(h)+\pi_1(\phi(z+h)) = p(\psi(h))-\psi(h)+\pi_1(\phi(0))+\psi(z+h) = \pi_1(\phi(z))+p(\psi(h))$, as required.

We claim that $\delta_h$ is a nilspace morphism. To see this, note that, given an $n$-cube $\q$ on $\ab$, it suffices to check that $(\phi\co\q,\phi\co(\q+h))$ is an $n$-cube on $\ns\Join_1 \ns$, i.e., that the 1-arrow $\langle\phi\co\q,\phi\co(\q+h)\rangle_1$ is in $\cu^{n+1}(\ns)$. Indeed, if we show this then it will follow that $\delta_h\co\q$ is also a cube, since the additional application of the translation $\beta_{p(\psi(h))-\psi(h)}$ in the second coordinate preserves cubes by definition of translations. But the 1-arrow $\langle\phi\co\q,\phi\co(\q+h)\rangle_1$ is clearly a cube because it is just $\phi\co \langle\q,\q+h\rangle_1$, and this is indeed an $(n+1)$-cube on $\ns$ since $\langle\q,\q+h\rangle_1$ is an $(n+1)$-cube on $\ab$ and $\phi$ is a morphism.

Now note that the space $\nss_{p(\psi(h))}$ is among the finitely many spaces $\nss_a$, $a\in A$. On the other hand, by Corollary \ref{cor:cpct-trans-unif-lip}, $V_h$ is a $C$-Lipschitz function with $C$ bounded in terms of $\|F_\chi\|_L$, $\ns$, and $\nss_{p(\psi(h))}$.\footnote{To be more precise, note that by standard arguments and Corollary \ref{cor:cpct-trans-unif-lip} we have that there exists some $L>0$ depending only on $\|F_\chi\|_{\textup{sum}}$ and on the compact set $B$ such that for any $(x_0,x_1),(y_0,y_1)\in \nss_{p(\psi(h))}$ we have $|V_h(x_0,x_1)-V_h(y_0,y_1)|\le L(d_{\ns}(x_0,y_0)+d_{\ns}(y_0,y_1))$. But the map $\nss_{p(\psi(h))}\to \ns$ defined as $(x_0,x_1)\mapsto x_0$ is a continuous morphism (and similarly the projection to $x_1$). Thus, by Proposition \ref{prop:eq-metrics-cfr-nil}, it is Lipscshitz, and thus $d(x_0,y_0)\le L'd_{\nss_{p(\psi(h))}}((x_0,x_1),(y_0,y_1))$ for some $L'\ge 0$. The result follows.} Hence, we can take this bound $C$ to be uniform in $h$, since $\nss_{p(\psi(h))}$ can take at most $|A|$ different values, and this number $|A|$ depends only on $\ns_1$. Therefore $\|V_h\|_L=O_{\ns,\|F_\chi\|_{\textup{sum}}}(1)$. Moreover $V_h$ is invariant under the action of the structure group $\ab_2(\ns)$ acting diagonally on $\nss_{p(\psi(h))}$, so there exists $W_h:\pi_1(\nss_{p(\psi(h))})\to\mb{C}$ such that $W_h\co \pi_{1} = V_h$. Hence, by Lemma \ref{lem:lip-bnd-on-factor} we have $\|W_h\|_L = O_{\ns,\|F_\chi\|_{\textup{sum}}}(1)$ as well.

Finally, by Proposition \ref{prop:eq-metrics-cfr-nil} the metric on $\pi_1(\nss_{p(\psi(h))})$ is Lipschitz equivalent to the usual metric on any compact abelian Lie group. Hence, by Lemma \ref{lem:unif-app-lie-gr} we can approximate the function $W_h$ in $L^\infty$ by a linear combination of boundedly many characters:
\[
\forall\,(x_0,x_1)\in \nss_{p(\psi(h))},\quad W_h(\pi_1(x_0,x_1)) = \textstyle\sum_{\gamma\in S}\lambda_{h,\gamma} \gamma(\pi_1(x_0,x_1))+\mc{E}_h(\pi_1(x_0,x_1))
\]
where $\|\mc{E}_h\|_\infty\le \sigma$, $\lambda_{h,\gamma}=\wh{W_h}(\gamma)$ (so, in particular $|\lambda_{h,\gamma}|\le \|F_\chi\|_\infty^2$), and $|S| = O_{\sigma,\ns,\|F_\chi\|_{\textup{sum}}}(1)$. Moreover, from the proof of Lemma \ref{lem:unif-app-lie-gr} we can give a more precise estimate of the size of $S$. Assume that the 1-step nilspace $\nss_{p(\psi(h))}\cong \mc{D}_1(A'\times \mb{T}^{\ell'})$ where $A'$ and $\ell'$ depend solely on $\ns$ (and $p(\psi(h))$, but as this is a finite set depending only on $\ns$ we regard such a dependence as a dependence in $\ns$). By Remark \ref{rem:explicit-set-S}, note that letting $S=\{(\wt{\chi},\chi_r)\in \wh{A'}\times \wh{\mb{T}^{\ell'}}:r\in \mb{Z}^{\ell'}, \|r\|_{\infty}\le M\}$ then for the given $\sigma$ we need to ensure that the upper bound $O_{\ns,\|F_\chi\|_{\textup{sum}}}(\frac{\log^{\ell'}M}{M})$ is at most $\sigma$. Since$\frac{\log^{\ell'}M}{M}\le\frac{O_{\ell'}(1)}{M^{1/2}}$, it suffices to assume that $M\ge \Omega_{\ns,\|F_\chi\|_{\textup{sum}}}(\sigma^{-2})$. The estimate on the size of $S$ follows by taking the $\ell'$-th power of such a quantity and multiplying by $|A'|$.

Hence, we conclude that 
$\overline{F_\chi\co\phi(z)} \; T^h (F_{\chi}\co\phi)(z) = \textstyle\sum_{\gamma\in S}\lambda_{h,\gamma} \gamma(\pi_1(\delta_h(z)))+\mc{E}_h(\pi_1(\delta_h(z)))$. Now note that each function $ \gamma(\pi_1(\delta_h(z)))$ is a Fourier character on $\ab$ multiplied by a complex number of modulus 1, so the result follow by relabelling said character as $\gamma$.
\end{proof}

\begin{remark}\label{rem:eq-def-qua-char}
Theorem \ref{thm:nscharsquadchars} tells us that 2-step nilspace characters are quadratic characters as per Definition \ref{def:k-char}, and these are clearly weak quadratic characters as per Definition \ref{def:wqc}. We strongly believe in the further claim that a 1-bounded weak quadratic character is close to a 2-step nilspace character, thus establishing that these three notions of quadratic character are essentially (or approximately) equivalent up to small additive errors. In fact, we believe that this claim can be established by applying to this weak quadratic character $f$ our upcoming Theorem \ref{thm:main-algo-pf}, combining it with the stability property of $f$ given by Proposition \ref{prop:weak-qua-char-stable-under-operator}, and applying additional linear-algebraic tools. We chose not to pursue this in detail here so as not to lengthen the paper further, as the resulting claim is not crucial for our present purposes. On the other hand, the approximate equivalence between the elementary notion of a quadratic character and the deeper notion of a 2-step nilspace character, which the claim would establish, may be an interesting direction to pursue in future work  on a purely theoretical level. In particular, this direction may suggest new more efficient ways to prove inverse theorems for Gowers norms.
\end{remark}

\subsection{Approximate orthogonality of balanced nilspace characters}\hfill\smallskip\\
In the structure theorem (or regularity lemma) for the $U^{k+1}$-norm obtained in \cite[Theorem 1.5]{CSinverse}, the structured part of the original function $f:\ab\to\mb{C}$ is of the form $F\co \phi$ for some highly-balanced morphism $\phi$ from $\ab$ into some $k$-step \textsc{cfr} nilspace $\ns$, and some Lipschitz function $F:\ns\to\mb{C}$. It is then natural to apply Lemma \ref{lem:unif-approx-CFRns} to the function $F$ and examine the properties of the resulting Fourier components $F_\chi$.

The first main result of this subsection establishes the following: the more balanced the morphism $\phi$ is, the smaller the $U^{k+1}$-products $\langle F_\chi\co\phi,F_{\chi'}\co\phi\rangle_{U^{k+1}}$ will be for distinct characters $\chi, \chi'$. This will be used in the proof of the validity of our algorithms.

\begin{proposition}\label{prop:smallU3prod}
Let $\delta,C>0$, and let $\ns$ be a $k$-step compact nilspace endowed with a $\ab_k$-invariant metric. There exists $b=b(\delta,C,\ns)>0$ such that if $\phi:\ab\to\ns$ is a $b$-balanced morphism, then for every 1-bounded $C$-Lipschitz function $F:\ns\to \mb{C}$ and every $\chi\neq \chi'$ in $\wh{\ab_k}$, we have $\langle F_{\chi}\co\phi, F_{\chi'}\co\phi\rangle_{U^{k+1}}\leq \delta$.
\end{proposition}

\begin{proof}
Let $\textbf{F}$ denote the function on $\ns^{\db{k+1}}$ sending $x=(x_v)_{v\in\db{k+1}}$ to $\prod_{v\in\db{k+1}}y_v\in \mb{C}$ with $y_v=F_{\chi}(x_v)$ if $v_{k+1}=0$ and  $y_v=F_{\chi'}(x_v)$ otherwise. Note that
\[
\langle F_{\chi}\co\phi, F_{\chi'}\co\phi\rangle_{U^{k+1}}= \mb{E}_{\q\in\cu^{k+1}(\ab)} \textbf{F} \co \phi^{\db{k+1}}(\q).
\]
By Lemma \ref{lem:proj-Lip}, the functions $F_\chi$ and $F_{\chi'}$ are both $C$-Lipschitz. For every $n\in \mb{N}$ we can endow $\ns^{\db{k+1}}$ with the metric $d_n'(\q,\q'):=\sum_{v\in \db{n}}d_{\ns}(\q(v),\q'(v))$ (and $\cu^n(\ns)$ with the restriction of this metric) which ensures that $\textbf{F}$ has $\|\textbf{F}\|_{\textup{sum}} \le 1+C$ relative to that metric.\footnote{To prove this, note that as $\|F\|_\infty\le 1$ we have $\|\textbf{F}\|_\infty\le 1$. For the Lipschitz constant, note that, by our choice of metric on $\cu^n(\ns)$, we can use a standard telescoping-sum argument. More precisely, let $(v_i)_{i=1,\ldots,2^{k+1}} = \db{k+1}$ be an arbitrary enumeration of the elements of $\db{k+1}$. Then, for $i=1,\ldots,2^{k+1}$, let $\q''_i\in \ns^{\db{k+1}}$ be defined as $\q''_i(v_j)=\q(v_j)$ for $j\le i$ and $\q'(v)$ otherwise. Then $|\textbf{F}(\q)-\textbf{F}(\q')|\le \sum_{i=1}^{2^{k+1}} |\textbf{F}(\q''_i)-\textbf{F}(\q''_{i+1})|\le \sum_{i=1}^{2^{k+1}} Cd_{\ns}(\q(v_i),\q'(v_i))$. Hence $\|\textbf{F}\|_L\le C$.}

Let $\eta>0$ be a parameter to be fixed later. By Corollary \ref{cor:equiv-weak-kr} (applied with any fixed metrics $d_n$ on $\mc{P}(\cu^n(\ns))$ and the metrics $d_n'$ on $\cu^n(\ns)$ for $n\in \mb{N}$ that we have just defined), we know that there exists $b=b(\eta,\ns)>0$ sufficiently small such that, if $\phi:\ab\to \ns$ is $b$-balanced, then $\|\mu_{\cu^{k+1}(\ab)}\co (\phi^{\db{{k+1}}})^{-1}-\mu_{\cu^{k+1}(\ns)}\|_{KR}\le \eta$ (here we assume that the metrics $d,d_n,d_n'$ are chosen \emph{a priori} depending only on $\ns$, hence we only record such a dependence). By definition of $\|\cdot\|_{KR}$, we then have
\[
\left|\textstyle\int_{\cu^{k+1}(\ns)} \textbf{F}(\q)\ud\mu_{\cu^{k+1}(\ns)}(\q)-\textstyle\int_{\cu^{k+1}(\ab)} \textbf{F}\co \phi^{\db{k+1}}(\q)\ud\mu_{\cu^{k+1}(\ab)}(\q) \right| \le \eta\|\textbf{F}\|_{\textup{sum}}\le \delta
\]
where $\eta$ is chosen to be $\delta/(C+1)$.

We now claim that $\int_{\cu^{k+1}(\ns)} \textbf{F}(\q)\ud\mu_{\cu^{k+1}(\ns)}(\q)=0$. Confirming this will complete the proof. Recall that the Haar measure on $\cu^{k+1}(\ns)$ is preserved by the action of $\cu^{k+1}(\ab_k)$, by \cite[Lemma 2.2.6]{Cand:Notes2}. Hence,
\[
\textstyle\int_{\cu^{k+1}(\ns)} \textbf{F}(\q)\ud\mu_{\cu^{k+1}(\ns)}(\q)=\textstyle\int_{\cu^{k+1}(\ns)}\textstyle\int_{\cu^{k+1}(\mc{D}_{k}(\ab_k))} \textbf{F}(\q+t)\ud\mu_{\cu^{k+1}(\mc{D}_{k}(\ab_k))}(t)\ud\mu_{\cu^{k+1}(\ns)}(\q).
\]
But by definition of $\textbf{F}$, for every $\q$ we have
\[
\textstyle\int_{\cu^{k+1}(\mc{D}_{k}(\ab_k))} \textbf{F}(\q+t)\ud\mu_{\cu^{k+1}(\mc{D}_{k}(\ab_k))}(t) = \textbf{F}(\q) \textstyle\int_{\cu^{k+1}(\mc{D}_{k}(\ab_k))} \textstyle\prod_{v\in \db{n}}\chi_v(t(v))\ud\mu_{\cu^{k+1}(\mc{D}_{k}(\ab_k))}(t)
\]
where $\chi_v=\chi$ if $v_{k+1}=0$ and $\chi_v=\chi'$ otherwise. By \cite[Lemma 3.52]{CScouplings} the last integral above  is 0 and the claim follows.
\end{proof}

\noindent The second main result is that $\phi$ being highly balanced also ensures that distinct nilspace characters $F_\chi\co\phi$ and $ F_{\chi'}\co\phi$ are quasiorthogonal in the following sense.
\begin{defn}\label{def:qo}
Given an inner-product space $H$ and $\delta>0$, we say that two elements $f,g\in H$ are \emph{$\delta$-quasiorthogonal} if $|\langle f,g\rangle|\leq \delta$.
\end{defn}
\begin{theorem}\label{thm:quasiorthog}
Let $\delta,C>0$ and let $\ns$ be a $k$-step compact nilspace endowed with a $\ab_k$-invariant metric. There exists $b=b(\delta,C,\ns)>0$ such that if $\phi:\ab\to\ns$ is a $b$-balanced morphism, then for every 1-bounded $C$-Lipschitz function $F:\ns\to \mb{C}$ and $\chi\neq \chi'$ in $\wh{\ab_k}$, the functions $F_\chi\co\phi$, $F_{\chi'}\co\phi$ are $\delta$-quasiorthogonal.
\end{theorem}
\begin{proof}
Let $\textbf{F}:\ns\to \mb{C}$ be the function $x\mapsto \overline{F_\chi(x)}F_{\chi'}(x)$ and note that $\langle F_{\chi'},F_{\chi}\rangle = \mb{E}_{z\in \ab } \textbf{F}\co \phi(z)$. By Lemma \ref{lem:proj-Lip}, the functions $F_\chi$ and $F_{\chi'}$ are both $C$-Lipschitz and 1-bounded. Hence $\|\textbf{F}\|_\infty\le 1$. For $x,y\in\ns$, $|\overline{F_\chi(x)}F_{\chi'}(x)-\overline{F_\chi(y)}F_{\chi'}(y)|\le |\overline{F_\chi(x)}F_{\chi'}(x)-\overline{F_\chi(x)}F_{\chi'}(y)|+|\overline{F_\chi(x)}F_{\chi'}(y)-\overline{F_\chi(y)}F_{\chi'}(y)|\le 2\|F\|_\infty Cd_{\ns}(x,y)$, so $\|\textbf{F}\|_{\textup{sum}}\le 1+2C$.

Let $\eta>0$ be a parameter to be fixed later. By Corollary \ref{cor:equiv-weak-kr}, we know that there exists $b=b(\eta,\ns)>0$ sufficiently small such that, if $\phi:\ab\to \ns$ is $b$-balanced, then $\|\mu_{\ab}\co \phi^{-1}-\mu_{\ns}\|_{KR}\le \eta$. By definition of $\|\cdot\|_{KR}$, we then have
\[
\left|\textstyle\int_{\ns} \textbf{F}(x)\ud\mu_{\ns}(x)-\textstyle\int_{\ab} \textbf{F}\co \phi(z)\ud\mu_{\ab}(z) \right| \le \eta\|\textbf{F}\|_{\textup{sum}}\le \delta
\]
where $\eta$ is chosen to be $\delta/(2C+1)$.

We now claim that $\int_{\ns} \textbf{F}(x)\ud\mu_{\ns}(x)=0$. Confirming this will complete the proof. Since the Haar measure on $\ns$ is preserved by the action of $\ab_k$ (by \cite[Lemma 2.2.6]{Cand:Notes2}), we have
\[
\textstyle\int_{\ns} \textbf{F}(x)\ud\mu_{\ns}(x)=\textstyle\int_{\ns}\textstyle\int_{\ab_k} \textbf{F}(x+t)\ud\mu_{\ab_k}(t)\ud\mu_{\ns}(x).
\]
By definition of $\textbf{F}$, for every $x$, using that $\chi\not=\chi'$, we have
\[
\textstyle\int_{\ab_k} \textbf{F}(x+t)\ud\mu_{\ab_k}(t) = \textbf{F}(x) \textstyle\int_{\ab_k} (\overline{\chi}\chi')(t)\ud\mu_{\ab_k}(t) = 0,
\]
which concludes the proof.
\end{proof}

\begin{remark}
We strongly believe that the two main results of this subsection are in fact consequences of a single stronger result, which would constitute in various ways a more natural form of approximate orthogonality. More precisely, we believe that, rather than the smallness of $\langle F_\chi\co\phi,F_{\chi'}\co\phi\rangle_{U^{k+1}}$ and of $\langle F_\chi\co\phi,F_{\chi'}\co\phi\rangle$ (established in these results), what holds is the following property, which implies the previous two: the smallness of $\sup_{h\in \ab}\|F_\chi\co\phi \,\overline{T^hF_{\chi'}\co\phi}\|_{U^k}$. The latter stronger property is likely to follow from a finer analysis of the underlying nilspace structures, but as we shall not need this result per se in this paper, we do not pursue it here. Let us at least give some examples clarifying the relations between these three properties.
\end{remark}    

\begin{example}
Let us illustrate the aforementioned relations in the case of the $U^2$-norm.

First we give examples of functions $f,g$ showing that the product $\langle f,g\rangle_{U^3} =\mb{E}_h \|f\overline{T^h g}\|_{U^2}^4$ can be arbitrarily small while $\sup_{h\in \ab} \|f\overline{T^h g}\|_{U^2}$ is bounded away from 0. Let $f(x)=e(x^3/p)$ and $g(x)=e((x^3+bx^2)/p)$ for some $b\neq 0$ in $\mb{Z}_p$, with $p\ge 5$ prime. Then $\mb{E}_h \|f\overline{T^h g}\|_{U^2}^4=o(1)_{p\to\infty}$, because $\|f\overline{T^h g}\|_{U^2}=o(1)_{p\to\infty}$ when $3h+b\neq 0\mod p$. However, note that
\[
\|f\overline{T^h g}\|_{U^2} = \|e(-(3h+b)x^2-(3h^2+2bh)x-(h^3+bh^2))/p)\|_{U^2},
\]
so for $h=-3^{-1}b\mod p$, the quadratic term above cancels and we obtain
\[
\|f\overline{T^h g}\|_{U^2} = \|e(-(3^{-1}b^2-3^{-1}2b^2)x-(h^3+bh^2))/p)\|_{U^2} = \|e(3^{-1}b^2 x)/p)\|_{U^2}
\]
which equals 1 since $e(3^{-1}b^2 x)/p)$ is a Fourier character. 

Next let us give examples showing that $\langle f,g\rangle_{U^3}$ can be arbitrarily small while the usual inner product $\langle f,g\rangle$ is large (so that, in particular, the smallness of $\langle f,g\rangle_{U^3}$ obtained in Proposition \ref{prop:smallU3prod} does not imply the quasiorthogonality obtained in Theorem \ref{thm:quasiorthog}). Taking simply $g=f=e(x^3/p)$, we have $\langle f,g\rangle_{U^3}=\|f\|_{U^3}^8=o(1)_{p\to\infty}$, yet $\langle f,g\rangle = 1$.
\end{example}

\subsection{$k$-step nilspace polynomials are structured functions of order $k$}\hfill\\
For general theoretical reasons in the nilspace approach to higher-order Fourier analysis, and especially for our main proofs in Section \ref{sec:regularity-of-f}, it is useful to establish that $k$-step nilspace polynomials of bounded complexity are structured functions of order $k$ in the sense of Definition \ref{def:kstruct}. We prove this with the following result, with the added strength that the $L^2$-error in the latter definition is here reduced to an $L^\infty$-error.

\begin{theorem}\label{thm:nil-poly-are-close-to-bnd-Uk-dual}
Let $m\in \mb{N}$ and $C\ge 0$. Then for any $\delta>0$ there exists a constant $N=N_{\delta,m,C}\ge 0$ such that the following holds. Let $\ns$ be a $k$-step \textsc{cfr} nilspace of complexity at most $m$, let $\ab$ be a finite abelian group, let $\phi:\ab\to\ns$ be a nilspace morphism, and let $F:\ns\to \mb{C}$ be a Lipschitz function with $\|F\|_{\textup{sum}}\le C$. Then there is a function $h:\ab\to\mb{C}$ such that $\|F\co\phi-h\|_\infty\leq\delta$ and $\|h\|_{U^{k+1}}^*\le N$.
\end{theorem}
\noindent To prove this, we shall use a quantitative form of the Stone-Weierstrass theorem for functions $F:\ns\to \mb{C}$ such that the complexity of $\ns$ and $\|F\|_{\textup{sum}}$ are bounded (see Proposition \ref{prop:quant-stone-weierstrass}). As an ingredient we use the following result on $[0,1/2]^n$ (equipped with the Euclidean metric).

\begin{lemma}\label{lem:decomp-of-function-in torus}
Let $n\in \mb{N}$ and $C\ge 0$. Then, for any $\delta>0$ there exists $N=N_{\delta,n,C}\ge 0$ such that the following holds. Let $g:[0,1/2]^n\to \mb{C}$ be a Lipschitz function such that $\|g\|_{\textup{sum}}\le C$. Then, for every $i\in[N]$ and $j\in[n]$, there exist a complex number $\lambda_i=O_{C}(1)$ and a Lipschitz function $h_{i,j}:[0,1/2]\to \mb{C}$  with $\|h_{i,j}\|_{\textup{sum}}=O_{C}(1)$, and there exists $\mc{E}:[0,1/2]^n\to\mb{C}$ with $\|\mc{E}\|_\infty\le \delta$, such that $g(x_1,\ldots,x_n)=\sum_{i=1}^{N}\lambda_{i}\prod_{j=1}^n h_{i,j}(x_j)+\mc{E}(x_1,\ldots,x_n)$.
\end{lemma}

\begin{proof}
We embed isometrically $[0,1/2]^n$ into $\mb{T}^n$ by identifying $\mb{T}^n\cong [0,1)^n$ (using the Euclidean metric on $[0,1/2]^n$ and the flat metric on $\mb{T}^n$). By the Tietze extension theorem, we extend $g$ to a continuous function $g':\mb{T}^n\to\mb{C}$ with $\|g'\|_{\textup{sum}}=O(C)$. The partial Fourier series  $S_M(g'):=\sum_{r\in\mb{Z}^n:\|r\|_{\ell^\infty}\leq M} \wh{g'}(\chi_r) \chi_r$ on $\mb{T}^n$ satisfy $\|S_M(g')-g'\|_\infty\leq K_{C,n}\frac{\log^n M}{M}$, where  $K_{C,n}$ depends only on $C$ and $n$ (see \cite[p.\ 201]{Lebesgue} for $n=1$  and \cite[p.\ 642]{Golubov} for $n>1$ ). Letting $M$ be large enough so that $K_{C,n}\frac{\log^n M}{M}\le \delta$, we have $\|S_M(g')-g'\|_\infty\leq \delta$. Note that each of the terms $\wh{g'}(\chi_r) \chi_r$ has the desired form. The result follows.
\end{proof}
\begin{proposition}\label{prop:quant-stone-weierstrass}
Let $\ns$ be a $k$-step \textsc{cfr} nilspace and let $C\ge 0$. Let $K:\cor^{k+1}(\ns)\to\ns$ be the corner-completion function (see \cite[Lemma 2.1.12]{Cand:Notes1}). Then, for any $\delta>0$, there exists $N=N_{\delta,\ns,C}\ge 0$ such that the following holds. For any $F:\ns\to\mb{C}$ with $\|F\|_{\textup{sum}}\le C$ there exist Lipschitz functions $(h_{i,v}:\ns\to \mb{C})_{i\in[N],v\in \db{k+1}\backslash\{0^{k+1}\}}$ with $\|h_{i,v}\|_{\textup{sum}}=O_{\ns,C}(1)$ such that $\|F\co K-\sum_{i\in[N]}\prod_{v\in \db{k+1}\backslash\{0^{k+1}\}}h_{i,v}(x_v)\|_\infty \le \delta$.
\end{proposition}

\begin{proof}
By Lemma \ref{lem:cor-compl-is-diff} and the fact that $\ns$ is compact, the function $F\co K$ is $O_{C,\ns}(1)$-Lipschitz. By the Tietze extension theorem, we can extend $F\co K$ to a continuous function on $\ns^{\db{k+1}\backslash\{0^{k+1}\}}$, which we denote  by $F'$, and for which we can assume that $\|F'\|_{\textup{sum}}=O_{C,\ns}(1)$. As $\ns$ is a compact manifold by Theorem \ref{thm:lie-fib-nil-are-manifold}, we can find a finite $C^\infty$ partition of unity $(\rho_i)_{i\in \mc{I}}$ such that, for all $i\in\mc{I}$, $\Supp(\rho_i)\subset U_i$ is compact and there exists a diffeomorphism $\phi_i:U_i\to (0,1/2)^\ell$ (for some $\ell=\ell_{\ns}$). This then induces a partition of unity in $\ns^{\db{k+1}\backslash\{0^{k+1}\}}$ simply by considering products of different $\rho_i$, i.e., $\{\prod_{v\in\db{k+1}\backslash\{0^{k+1}\}}\rho_{I_v}:I\in\mc{I}^{\db{k+1}\backslash\{0^{k+1}\}}\}$. For any $I\in \mc{I}^{\db{k+1}\backslash\{0^{k+1}\}}$, let  $\rho^*_I$ be the function mapping $\q^*\in \ns^{\db{k+1}\backslash\{0^{k+1}\}}$  to $\prod_{v\in\db{k+1}\backslash\{0^{k+1}\}}\rho_{I_v}(\q^*(x_v))$.

Hence $F' = \sum_{I\in \mc{I}^{\db{k+1}\backslash\{0^{k+1}\}}} F'\rho_I$. Note that $|\mc{I}^{\db{k+1}\backslash\{0^{k+1}\}}|=O_{\ns}(1)$ so if we prove the desired decomposition for each function $F'\rho_I$ then the result will follow. Note also that by construction $\|F'\rho_I\|_{\textup{sum}}=O_{\ns,C}(1)$. For $I\in \mc{I}^{\db{k+1}\backslash\{0^{k+1}\}}$, let $\phi_I$ denote the map $\prod_{v\in \db{k+1}\backslash\{0^{k+1}\}}\phi_{I_v}:\prod_{v\in \db{k+1}\backslash\{0^{k+1}\}}U_{I_v}\to (0,1/2)^{\ell(2^{k+1}-1)}$. Then $F'\rho_I = (F'\rho_I)\co \phi_I^{-1}\co\phi_I$. Note that $\|\phi_I\|_{\textup{Lip}}$ and $\|\phi_I^{-1}\|_{\textup{Lip}}$ are both $O_{\ns}(1)$, which implies that $\|(F'\rho_I)\co \phi_I^{-1}\|_{\textup{sum}}=O_{C,\ns}(1)$. Thus, we can apply Lemma \ref{lem:decomp-of-function-in torus} to $(F'\rho_I)\co \phi_I^{-1}$ with $\delta_I=\delta/|\mc{I}^{\db{k+1}\backslash\{0^{k+1}\}}|$. Hence $
\big\|(F'\rho_I)\co \phi_I^{-1}-\textstyle\sum_{i\in[N]}\textstyle\prod_{v\in \db{k+1}\backslash\{0^{k+1}\}}h_{i,v}^{I}\big\|_\infty \le \delta_I$, 
where $h_{i,v}^{I}:(0,1/2)^\ell\to \mb{C}$. \footnote{Lemma \ref{lem:decomp-of-function-in torus} requires $(F'\rho_I)\co \phi_I^{-1}$ to be defined and Lipschitz on $[0,1/2]^n$ and we have it defined on $(0,1/2)^n$ where $n=\ell(2^{k+1}-1)$. But since $\Supp(\rho_i)\subset U_i$ is compact, we can define $(F'\rho_I)\co \phi_I^{-1}$ to be 0 on $[0,1/2]^n\setminus (0,1/2)^n$, and the resulting function is still Lipschitz as required, by \cite[Theorem 2.5.6 and Remark 2.5.5]{CMN}.} The desired decomposition follows simply by considering $\prod_{v\in \db{k+1}\backslash\{0^{k+1}\}}h_{i,v}^{I}\co\phi_I$ for $i\in[N]$ and $I\in \mc{I}^{\db{k+1}\backslash\{0^{k+1}\}}$.\end{proof}

\begin{proof}[Proof of Theorem  \ref{thm:nil-poly-are-close-to-bnd-Uk-dual}]
Let $K_{k+1}:=\{0,1\}^{k+1}\setminus\{0^{k+1}\}$. For every $x,t_1,\ldots,t_{k+1}\in\ab$ we have
\[
F\co\phi(x) = F\co K\big( (\phi(x+v_1t_1+\cdots+v_{k+1}t_{k+1}))_{v\in K_{k+1}}\big)
\]
where $K:\cor^{k+1}(\ns)\to\ns$ is the continuous corner completion function.

By Proposition \ref{prop:quant-stone-weierstrass} applied to $F$ and $K$ there exists $N=N_{\delta,m,C}\ge 0$ (recall that $\ns$ has complexity at most $m$) and functions $(h_{i,v}:\ns\to \mb{C})_{i\in[N],v\in K_{k+1}}$ with $\|h_{i,v}\|_{\textup{sum}}=O_{m,C}(1)$ such that $\|F\co K-\sum_{i\in[N]}\prod_{v\in K_{k+1}}h_{i,v}(x_v)\|_\infty \le \delta$.

Let $h:\ab\to\mb{C}$, $x\mapsto \mb{E}_{t_1,\ldots,t_{k+1}\in\ab} \sum_{i\in[N]}\prod_{v\in K_{k+1}}h_{i,v}(\phi(x+v\cdot t))$, where $t=(t_1,\ldots,t_{k+1})$. We have for every $x\in \ab$,
\begin{align*}
|F\co\phi(x)-h(x)| & =  \big|\mb{E}_{t_1,\ldots,t_{k+1}\in\ab} (F\co K-\textstyle\sum_{i\in[N]}\textstyle\prod_{v\in K_{k+1}}h_{i,v})\big( (\phi(x+v\cdot t)_{v\in K_{k+1}}\big)\big| \\
& \leq \big\|F\co K - \textstyle\sum_{i\in[N]}\textstyle\prod_{v\in K_{k+1}}h_{i,v})\big\|_\infty\; \leq\; \delta.
\end{align*}
To bound $\|h\|_{U^{k+1}}^*$, let $g:\ab\to\mb{C}$ with $\|g\|_{U^{k+1}}\leq 1$, and note that $\mb{E}_x h(x) \overline{g(x)}$ equals
\begin{multline*}
= \mb{E}_{x,t_1,\ldots,t_{k+1}\in \ab}\, \overline{g(x)} \textstyle\sum_{i\in[N]}\prod_{v\in K_{k+1}}\!\!h_{i,v}(\phi(x+v\cdot t))= \textstyle\sum_{i\in [N]} \langle g, (h_{i,v}\co\phi)_{v\in K_{k+1}}\rangle_{U^{k+1}}\\ 
\leq\textstyle\sum_{i\in [N]}  \|g\|_{U^{k+1}} \textstyle\prod_{v\in K_{k+1}} \|h_{i,v}\co\phi\|_{U^{k+1}}  \leq \textstyle\sum_{i\in [N]} \textstyle\prod_{v\in K_{k+1}} \|h_{i,v}\|_\infty,
\end{multline*}
where we used the Gowers-Cauchy-Schwarz inequality. Hence $\|h\|_{U^{k+1}}^*=O_{\delta,m,C}(1)$.
\end{proof}

\section{Refined regularity in terms of nilspace characters}\label{sec:regularity-of-f}

\subsection{The refined regularity lemma}\hfill\smallskip\\
Recall that a complexity notion for compact nilspaces is a bijection between the countable set of isomorphism classes of \textsc{cfr} nilspaces and the natural numbers (see \cite[Definition 1.2]{CSinverse}). Throughout this section, we assume that we have fixed an arbitrary complexity notion, denoting the resulting list of nilspaces by $\{\nss_1,\nss_2,\ldots\}$. For each $\nss_i$ in this fixed complexity notion, let $\nsR_{\nss_i}(\eta,C)$ be the function given by Lemma \ref{lem:unif-approx-CFRns}. For each $m\in \mb{N}$ and $\eta>0$, we then define 
\begin{equation}
\nsR(\eta,m):=\max_{i\le m}\nsR_{\nss_i}(\eta,m).
\end{equation}
We can now prove the following regularity theorem, refining  \cite[Theorem 1.5]{CSinverse} for finite abelian groups. (The proof could be extended to compact abelian groups as enabled by \cite[Theorem 1.5]{CSinverse}, but we do not need this.)
\begin{theorem}[Regularity]\label{thm:UpgradedReg} Let $k\in \mb{N}$ and let $\mc{D}:\mb{R}_{>0}\times \mb{N}\to \mb{R}_{>0}$ be an arbitrary function. For every $\eta>0$ there exists $N=N(\eta,\mc{D})>0$ such that the following holds. For every finite abelian group $\ab$ and every 1-bounded function $f:\ab\to \mb{C}$, for some $m\in [1,N]$ there exists a $k$-step \textsc{cfr} nilspace $\nss$ of complexity at most $m$, a $\mc{D}(\eta,m)$-balanced morphism $\phi:\ab\to\nss$, and a 1-bounded $m$-Lipschitz function $F:\nss\to\mb{C}$, such that
\begin{equation}
 f = f_s + f_e + f_r,
\end{equation}
where $f_s=F\co\phi$, $\|f_e\|_1\le \eta$, $\|f_r\|_\infty\le 1$, and $\max(\|f_r\|_{U^{k+1}},|\langle f_r,f_s\rangle|,|\langle f_r,f_e \rangle|)\le \mc{D}(\eta,m)$.

Moreover, there is a set $S\subset \widehat{\ab_k(\nss)} $ with $|S|\leq \nsR(\eta,m)$, such that, letting $g_\chi=F_\chi\co\phi$ for each $\chi\in S$, we have 
\begin{equation}\label{eq:fsdecomp}
f_s = \textstyle\sum_{\chi\in S}g_\chi\;+\;g_e,
\end{equation}
where $\|g_e\|_\infty\le \eta$, and for all $\chi\neq \chi'$ in $S$, we have $\max\big(|\langle g_{\chi},g_{\chi'}\rangle|,\, \langle g_{\chi},g_{\chi'}\rangle_{U^{k+1}}\big)\le \mc{D}(\eta,m)$.
\end{theorem}
\begin{proof}
Having fixed $\mc{D}$ and $\eta$, we apply \cite[Theorem 1.5]{CSinverse} with a function $\mc{D}'$ for which we shall specify constraints throughout the proof. Firstly we require that $\mc{D}'(\eta,m) \le \mc{D}(\eta,m)$. For this function $\mc{D}'$ (to be fully specified later), let $f=f_s+f_e+f_r$ be the decomposition provided by \cite[Theorem 1.5]{CSinverse}. It then remains to decompose $f_s = F\co \phi$ further as claimed in \eqref{eq:fsdecomp}. By \cite[Theorem 1.5]{CSinverse} we know that $F:\nss\to \mb{C}$ is $m$-Lipschitz, 1-bounded, and the complexity of $\nss$ is at most $m$. Also, the morphism $\phi:\ab\to \nss$ is $\mc{D}'(\eta,m)$-balanced. 

As $\nss$ has complexity at most $m$, Lemma \ref{lem:unif-approx-CFRns} applied with $\delta=\eta$ yields a set $S\subset\wh{\ab_k(\nss)}$ with $|S|\leq \nsR(\eta,m)$ such that $F = \sum_{\chi\in S}F_\chi + F_e$ where $\|F_e\|_\infty\le \eta$. Letting $g_e=F_e\co\phi$, we obtain \eqref{eq:fsdecomp}. By Lemma \ref{lem:proj-Lip}, we have $\|F_\chi\|_L\le \|F\|_L$ and $\|F_\chi\|_\infty\le \|F\|_\infty$, so $\|F_\chi\|_{\textrm{sum}}\leq m+1$.

Finally, we apply Proposition \ref{prop:smallU3prod} and Theorem \ref{thm:quasiorthog} with $\delta=\mc{D}(\eta,m)$, thus obtaining that there exists $b=b(\mc{D}(\eta,m),m,\nss)>0$ such that if $\phi$ is $b$-balanced, then the conclusions of Proposition \ref{prop:smallU3prod} and Theorem \ref{thm:quasiorthog} hold, giving us the claimed smallness of $U^{k+1}$-products and quasiorthogonality, namely $\max\big(|\langle g_{\chi},g_{\chi'}\rangle|,\, \langle g_{\chi},g_{\chi'}\rangle_{U^{k+1}}\big)\le \mc{D}(\eta,m)$ for all $\chi\neq \chi'$ in $S$. Since $\nss$ has complexity at most $m$, this function $b(\mc{D}(\eta,m),m,\nss)$ can be bounded by a function $b'(\mc{D}(\eta,m),m)$. We require that $\mc{D}'(\eta,m)<b'(\mc{D}(\eta,m),m)$.
\end{proof}

\begin{remark}\label{rem:improvement-5-1}
Theorem \ref{thm:UpgradedReg} can be slightly improved at the cost of making its formulation more complicated. We could let $\mc{D}_2:\mb{R}_{>0}\times \mb{N}\to \mb{R}_{>0}$ be an additional arbitrary function and require that $\|g_e\|_\infty\le \mc{D}_2(\eta,m)$. Then, note that the size of $S$ would be bounded by $\nsR(\mc{D}_2(\eta,m),m)$. In turn, the function $\mc{D}(\eta,m)$ may also depend on $\mc{D}_2$. Therefore, we would have two nested levels of freedom, namely, a first free choice of $\mc{D}_2$ and then, depending on this, a choice of $\mc{D}$.
\end{remark}

\subsection{Correlation, approximate diagonalization of the $U^k$ norm, and regularization}\hfill\smallskip\\
This subsection presents three main consequences of the refined regularity theorem that follow by elementary arguments once we have Theorem \ref{thm:UpgradedReg}.

We begin with a simple refinement of the inverse theorem \cite[Theorem 5.2]{CSinverse} (in the case of finite abelian groups), which shows that a 1-bounded function $f:\ab\to\mb{C}$ with large $U^{k+1}$-norm correlates non-trivially not just with the structured part $f_s$ (which is a nilspace polynomial in the sense of \cite{CSinverse}) but rather with one of the nilspace characters in the decomposition of $f_s$.

\begin{theorem}[Inverse theorem with nilspace characters]\label{thm:correlestim}
Under the assumptions and notation of Theorem \ref{thm:UpgradedReg}, suppose in addition that $\|f\|_{U^{k+1}}\geq \delta>0$. Then 
\begin{equation}\label{eq:inv-thm-nil-poly}
\max_{\chi\in S} |\langle f,g_\chi\rangle|\geq  \tfrac{\delta^{2^{k+1}}}{4|S|} = \Omega_\delta(1).
\end{equation}
\end{theorem}
\begin{proof}
Following the proof of \cite[Theorem 5.2]{CSinverse}, we have that if $\eta$ and $\mc{D}(\eta,m)$ are sufficiently small in terms of $\delta$, then $|\langle f,f_s\rangle|\ge \delta^{2^{k+1}}/2$. Substituting here $f_s=\sum_{\chi\in S}g_\chi+g_e$, and noting that we can also take $\eta$ to be at most $\delta^{2^{k+1}}/4$, we see that there exists $\chi\in S$ such that $|\langle f,g_\chi\rangle|\ge \delta^{2^{k+1}}/(4|S|)$ where $|S|=O_\delta(1)$.
\end{proof}
\noindent We can now prove a new, more detailed correlation estimate, which gives information not just on the maximal correlation but also on the correlation of $f$ with each nilspace character in  \eqref{eq:fsdecomp} having non-negligible $L^2$-norm. The information in question is that every such inner product is close to the squared $L^2$-norm of the nilspace character.

\begin{lemma}\label{lem:correlations}
Under the assumptions and notation of Theorem \ref{thm:UpgradedReg}, for each $\chi\in S$ we have
\begin{equation}\label{eq:correlations}
\big| \langle f,g_\chi \rangle - \|g_\chi\|_2^2\,\big| \leq \mc{D}(\eta,m)\big(|S|+O_{\eta,m}(1)\big)+3\eta.
\end{equation}
\end{lemma}
\begin{proof}
We have $\langle f,g_\chi\rangle=\sum_{\chi'\in S} \langle g_{\chi'} ,g_\chi\rangle+\langle g_e +f_e,g_\chi\rangle+\langle f_r,g_\chi\rangle$. 
Hence
\begin{align*}
\big| \langle f,g_\chi\rangle - \|g_\chi\|_2^2\,\big|\leq \textstyle\sum_{\chi'\in S\setminus\{\chi\}} |\langle g_{\chi'},g_\chi\rangle|+|\langle g_e+f_e,g_\chi\rangle|+|\langle f_r,g_\chi\rangle|.
\end{align*}
Using the quasiorthogonality, the sum over $S\setminus\{\chi\}$  here is at most $|S|\mc{D}(\eta,m)$. 

Next, using that $\|g_\chi\|_\infty\leq 1$, we have $|\langle g_e+f_e,g_\chi\rangle|\leq\|g_e\|_\infty+\|f_e\|_1\leq 2\eta$.

Finally, we bound $|\langle f_r,g_\chi\rangle|$. We have $g_\chi=F_\chi\co\phi$ where, as explained in the proof of Theorem \ref{thm:UpgradedReg}, we have $\|F_\chi\|_{\textrm{sum}}\leq m+1$. By Theorem \ref{thm:nil-poly-are-close-to-bnd-Uk-dual} with $\delta=\eta$, there is $h':\ab\to\mb{C}$ with $\|h'\|_{U^{k+1}}^*=O_{\eta,m}(1)$ such that $\| g_\chi -h'\|_\infty\leq \eta$. Hence $g_\chi$ is an $(R,\eta)$-structured function of order $k$ with $R=O_{\eta,m}(1)$. Since $f_r$ is 1-bounded and $\|f_r\|_{U^{k+1}}\leq \mc{D}(\eta,m)$, by Proposition \ref{prop:kstructequiv} with $h=f_r$, we have  $|\langle f_r,g_\chi\rangle|\leq \mc{D}(\eta,m) O_{\eta,m}(1)+\eta$. The result follows.
\end{proof}
\begin{remark}\label{rem:improvement-5-4}
If necessary, in the above proof the estimate $\|g_\chi-h'\|_\infty\leq \eta$ can be strengthened to $\|g_\chi-h'\|_\infty\leq \delta$, letting $\delta$ be a new independent parameter, as enabled by Theorem \ref{thm:nil-poly-are-close-to-bnd-Uk-dual}. This would be at the cost of increasing the $O_{\eta,m}(1)$ term to a $O_{\delta,m}(1)$ quantity which will increase as $\delta$ decreases. This variation is similar to the one mentioned in Remark \ref{rem:improvement-5-1}.
\end{remark}
\noindent The next main result in this subsection is an approximate diagonalization of the $U^{k+1}$-norm. To motivate this, recall the well-known formula for the $U^2$-norm in terms of Fourier coefficients:
\begin{align}\label{eq:U2diag}
\|f\|_{U^2}^4=\textstyle\sum_{\chi\in \wh{\ab}} |\wh{f}(\chi)|^4.
\end{align}
This formula can be viewed as a diagonalization of the $U^2$-norm in the sense that when we write $\|f\|_{U^2}^4$ as $\langle f,f,f,f\rangle_{U^2}$, substitute here $f$ by its Fourier expansion $f=\sum_{\chi\in\wh{\ab}} \wh{f}(\chi)\chi$, and expand the $U^2$-product, the orthonormality of characters yields cancellation of all non-diagonal terms in the expansion, resulting in the right side of  \eqref{eq:U2diag}. Note that this right side can be rewritten as $\sum_{\chi\in \wh{\ab}} \|\wh{f}(\chi)\chi\|_{U^2}^4$.

As a consequence of Theorem \ref{thm:UpgradedReg}, we can obtain the following generalization of \eqref{eq:U2diag}. This type of result will help in the final proof of this section, but it is also of independent interest. Indeed, formulas of this type, which can be viewed as approximate ``Pythagorean theorems" for the Gowers norms, have been sought as a desirable feature of higher-order Fourier analysis; see for instance a special case of such a result for the $U^3$-norm in \cite[Theorem 4.1]{GWcomp}.

\begin{theorem}[Approximate diagonalization of $\|\cdot\|_{U^{k+1}}$]\label{thm:UdDiag}
Let $k\in \mb{N}$, let $\mc{D}:\mb{R}_{>0}\times \mb{N}\to \mb{R}_{>0}$ be an arbitrary function, and let $\eta>0$. Let $\ab$ be a finite abelian group and let $f:\ab\to\mb{C}$ be a 1-bounded function. Let $f=\sum_{\chi\in S} g_\chi + g_e+f_e+f_r$ and $m\leq N(\eta,\mc{D})$ be the decomposition and complexity bound given by Theorem \ref{thm:UpgradedReg}. Then
\begin{equation}\label{eq:UdDiag}
\|f\|_{U^{k+1}}^{2^{k+1}} = \sum_{\chi\in S} \|g_\chi\|_{U^{k+1}}^{2^{k+1}} + \mc{E},
\end{equation}
where $|\mc{E}|\leq 2^{2^{k+1}}\big( 4\eta^{\frac{k+2}{2^{k+1}}}+\mc{D}(\eta,m)\big)+(|S|^{2^{k+1}}-1)\mc{D}(\eta,m)^{1/2^k}$.
\end{theorem}
\noindent We extract the final part of the proof of this theorem as the following separate lemma, which will be used in other results below. This lemma gives more properly an approximate ``Pythagorean formula" for Gowers norms (as it holds not just for the nilspace characters in Theorem \ref{thm:UdDiag}, but for any system of 1-bounded functions with pairwise small Gowers inner-products).
\begin{lemma}[Approximate Pythagorean formula for $\|\cdot\|_{U^s}$]\label{lem:diagUdSpec}
Let $s\geq 2$ be an integer, let $G$ be a compact abelian group, let $B$ be a finite set, and let $(g_i)_{i\in B}$ be a collection of 1-bounded Borel functions on $G$ such that for any $i\neq j$ in $B$ we have $\langle g_i,g_j\rangle_{U^{s}}\leq \delta$. Then
\begin{equation}\label{eq:diagUdSpec}
\Big|\,\big\|\sum_{i\in B} g_i\big\|_{U^s}^{2^s} - \sum_{i\in B} \|g_i\|_{U^{s}}^{2^{s}}\, \Big|\leq (|B|^{2^s}-1)\delta^{2^{1-s}}.
\end{equation}
\end{lemma}
\begin{proof}
We expand $\|\sum_{i\in B} g_i\|_{U^{s}}^{2^{s}}$ into $|B|^{2^{s}}$ Gowers $U^{s}$-products, obtaining
\begin{eqnarray*}
&\big\|\sum_{i\in B} g_i\big\|_{U^{s}}^{2^{s}} = \sum_{(i_v)_{v\in \db{{s}}}\in B^{\db{{s}}}} \langle (g_{i_v})_{v\in \db{{s}}}\rangle_{U^{s}}.&
\end{eqnarray*}
In the right side here, the summands for which the corresponding vector $(i_v)_{v\in \db{{s}}}$ is constant add up to $\sum_{i\in B} \|g_i\|_{U^{s}}^{2^{s}}$. The main task is thus to show that every other term is small. Every such term is a product $ \langle (g_{i_v})_{v\in \db{{s}}}\rangle_{U^{s}}$ such that for some pair of vertices $v,w\in \db{{s}}$ we have $i_v\neq i_w$. We claim that we can then assume additionally that these vertices $v,w$ are \emph{adjacent} in the discrete cube $\db{{s}}$ (i.e., only one of their coordinates differ); this follows from the fact that starting from any vertex of $\db{{s}}$ there is a Hamiltonian path in the graph of 1-faces (or ``edges") of the discrete cube $\db{{s}}$, so if our claim was false then starting from any vertex and following such a path, we would deduce that $(i_v)_{v\in \db{{s}}}$ is constant, a contradiction. Assuming thus that $v,w$ are adjacent with $i_v\neq i_w$, now note that by ${s}-1$ applications of the Cauchy-Schwarz inequality, similar to the applications involved in the proof of the $U^{s}$-Gowers-Cauchy-Schwarz inequality (see \cite{T-V}), and using that $\|g_i\|_{U^{s}}\leq 1$ for every $i\in B$, we obtain that the $U^{s}$-product in question is at most $\langle g_{i_v},g_{i_w}\rangle_{U^{s}}^{1/2^{{s}-1}}\leq \delta^{1/2^{{s}-1}}$. The result follows.
\end{proof}
We shall also use the following upper bound for the Gowers norms.
\begin{lemma}\label{lem:udvsl2}
Let $s\geq 2$ be an integer, let $G$ be a compact abelian group, and let $f:G\to\mb{C}$ be a bounded Borel function. Then
\begin{equation}\label{eq:udvsl2}
\|f\|_{U^{s}}^{2^{s}}\leq \min\big(\|f\|_2^{2{s}+2}\|f\|_{L^\infty}^{2^{s}-2{s}-2},\;\|f\|_1^{{s}+1}\|f\|_{L^\infty}^{2^{s}-{s}-1}\big).
\end{equation}
\end{lemma}
\begin{proof}
To see that $\|f\|_{U^{s}}^{2^{s}}\leq \|f\|_2^{2{s}+2}\|f\|_{L^\infty}^{2^{s}-2{s}-2}$, note that the graph of 1-faces on the discrete cube $\db{{s}}$ contains two disjoint stars $S_1,S_2$ of order ${s}+1$, namely the star with center $0^{s}$ and the star with center $1^{s}$ (note that for ${s}=3$ these stars partition $\db{{s}}$, and for ${s}>3$ their union leaves $2^{s}-2{s}-2$ vertices in the complement). Using the Cauchy-Schwarz inequality, we then have
\begin{align*}
\|f\|_{U^{s}}^{2^{s}}  \leq &\|f\|_{L^\infty}^{2^{s}-2{s}-2}\, \textstyle\int_{\q\in\cu^{s}(G)} \textstyle\prod_{v\in S_1} |f(\q(v))| \textstyle\prod_{w\in S_2} |f(\q(w))|\\
\leq &  \|f\|_{L^\infty}^{2^{s}-2{s}-2} \big(\textstyle\int_{\q\in\cu^{s}(G)} \textstyle\prod_{v\in S_1} |f(\q(v))|^2\big)^{1/2} \big(\textstyle\int_{\q\in\cu^{s}(G)}\textstyle\prod_{w\in S_2} |f(\q(w))|^2\big)^{1/2}.
\end{align*}
The variables in each star $S_i$ can be changed into ${s}+1$ independent variables $x_1,\ldots,x_{{s}+1}\in G$, and it follows that the last two integrals both equal $\|f\|_2^{2s+2}$. 

The proof of the other inequality $\|f\|_{U^{s}}^{2^{s}}\leq \|f\|_1^{{s}+1}\|f\|_{L^\infty}^{2^{s}-{s}-1}$ is similar.
\end{proof}
\begin{proof}[Proof of Theorem \ref{thm:UdDiag}]
Letting $f_c=\sum_{\chi\in S} F_\chi\co\phi$, we have that $\big|\|f\|_{U^{k+1}}-\|f_c\|_{U^{k+1}}\big|$ is at most
\[\|g_e+f_e+f_r\|_{U^{k+1}} \leq \|g_e\|_\infty + 3^{1-\frac{k+2}{2^{k+1}}}\|f_e\|_{L^1}^{\frac{k+2}{2^{k+1}}}+\|f_r\|_{U^{k+1}}\leq \eta+3^{1-\frac{k+2}{2^{k+1}}}\eta^{\frac{k+2}{2^{k+1}}}+\mc{D}(\eta,m),
\]
where we used that $\|f_e\|_\infty\le 3$ and thus $\|f_e\|_{U^{k+1}}^{2^{k+1}}\le 3^{2^{k+1}-(k+2)}\|f_e\|_1^{k+2}$ (by Lemma \ref{lem:udvsl2}). Hence, using $\|f\|_{U^{k+1}}\leq \|f\|_{L^\infty}\leq 1$ and $\|f_c\|_{U^{k+1}}\leq \|f_c\|_{L^\infty}\leq 1+\eta$, we obtain that $|\|f\|_{U^{k+1}}^{2^{k+1}}-\|f_c\|_{U^{k+1}}^{2^{k+1}}|\leq \big|\|f\|_{U^{k+1}}-\|f_c\|_{U^{k+1}}\big|\;\big|\|f\|_{U^{k+1}}^{2^{k+1}-1}+\|f\|_{U^{k+1}}^{2^{k+1}-2}\|f_c\|_{U^{k+1}}+\cdots +\|f_c\|_{U^{k+1}}^{2^{k+1}-1})\leq  (\eta+3^{1-\frac{k+2}{2^{k+1}}}\eta^{\frac{k+2}{2^{k+1}}}+\mc{D}(\eta,m)) \;\big(1+(1+\eta)+(1+\eta)^2+\cdots+(1+\eta)^{2^{k+1}-1})$, which is at most $2^{2^{k+1}}(\eta+3^{1-\frac{k+2}{2^{k+1}}}\eta^{\frac{k+2}{2^{k+1}}}+\mc{D}(\eta,m))$.

Hence $\big|\|f\|_{U^{k+1}}^{2^{k+1}}-\|f_c\|_{U^{k+1}}^{2^{k+1}}\big|\leq 2^{2^{k+1}}( 4\eta^{\frac{k+2}{2^{k+1}}}+\mc{D}(\eta,m))$. Combining this with Lemma \ref{lem:diagUdSpec} applied to $f_c$, the result follows.
\end{proof}
\noindent In a similar spirit, we can also prove the following approximate Parseval identity, to be used in the proof of Proposition \ref{prop:regul} below. Higher-order analogues of Parseval's identity have also been sought as desirable features of higher-order Fourier analysis (see \cite[Section 16]{GowersBAMS}).
\begin{theorem}[Approximate higher-order Parseval identity]\label{thm:ApproxBessel}
Let $k\in \mb{N}$, let $\mc{D}:\mb{R}_{>0}\times \mb{N}\to \mb{R}_{>0}$ be an arbitrary function, and let $\eta>0$. Let $\ab$ be a finite abelian group, and let $f:\ab\to\mb{C}$ be a 1-bounded function. Let $f=\sum_{\chi\in S} g_\chi + g_e+f_e+f_r$ and $m\leq N(\eta,\mc{D})$ be the decomposition and complexity bound given by Theorem \ref{thm:UpgradedReg}. Then
\begin{eqnarray}\label{eq:ApproxBessel}
&\big| \|f\|_2^2 -  \sum_{\chi\in S} \|g_\chi\|_2^2 - \|f_r\|_2^2 \big|\leq  (|S|^2+4)\mc{D}(\eta,m)+8\eta.&
\end{eqnarray}
\end{theorem}
\begin{proof}
Letting again $f_c=\sum_{\chi\in S} g_\chi$, and recalling that $f_s=f_c+g_e$, we have, using the estimates given in Theorem \ref{thm:UpgradedReg}, that
\[
|\|f\|_2^2-\|f_s\|_2^2| \leq 2|\langle f_s,f_r\rangle | + 2|\langle f_e, f_r\rangle |+ 2|\langle f_s, f_e\rangle |+\|f_e\|_2^2+\|f_r\|_2^2 \leq 4\mc{D}(\eta,m)+5\eta+\|f_r\|_2^2.
\]
As $\|f_s\|_\infty\leq 1$ and $\|g_e\|_\infty\leq \eta$, we have $|\|f_s\|_2^2-\|f_c\|_2^2|\leq 2|\langle f_c,g_e\rangle|+\|g_e\|_2^2\leq 2\eta +\eta^2\le 3\eta$, and so $|\|f\|_2^2-\|f_c\|_2^2| \leq  |\|f\|_2^2-\|f_s\|_2^2| +|\|f_s\|_2^2-\|f_c\|_2^2| \leq 4\mc{D}(\eta,m)+8\eta+\|f_r\|_2^2$. We finish by using that $|\|f_c\|_2^2-\sum_{\chi\in S} \|g_\chi\|_2^2|\leq |S|^2\max_{\chi\neq\chi'\in S}|\langle g_\chi,g_{\chi'}\rangle| \leq |S|^2\mc{D}(\eta,m)$. 
\end{proof}
\noindent The final result of this subsection is a key ingredient for the proof of our main result Theorem \ref{thm:reg-intro}, to be completed in the last section (see Theorem \ref{thm:algoregul}).
\begin{proposition}\label{prop:regul}
Let $k\in \mb{N}$, let $\mc{D}:\mb{R}_{>0}\times \mb{N}\to \mb{R}_{>0}$ be an arbitrary function, and let $\eta>0$. Let $\ab$ be a finite abelian group, and let $f:\ab\to\mb{C}$ be a 1-bounded function. Let
\begin{align*}
f=\textstyle\sum_{\chi\in S} g_\chi + g_e+f_e+f_r
\end{align*}
be the decomposition given by Theorem \ref{thm:UpgradedReg}, with associated complexity bound $m\leq N(\eta,\mc{D})$. For any $\rho>0$, let $S_\rho:=\{\chi\in S :\|g_\chi\|_2^2\geq\rho\}$. Then $|S_\rho|\leq\rho^{-1} (1+c_0)$ and
\begin{equation}\label{eq:regul}
\big\|f-\textstyle\sum_{\chi\in S_\rho}g_\chi\big\|_{U^{k+1}}\leq\rho^{(k+1)/2^{k+1}} (1+c_0)^{1/2^{k+1}}+2|S|\mc{D}(\eta,m)^{1/4^{k+1}}+4\eta^{(k+2)/2^{k+1}},
\end{equation}
where $c_0=(|S|^2+4)\mc{D}(\eta,m)+8\eta$. Moreover, for any $\sigma>0$, there exists $h:\ab\to\mb{C}$ with $\|h\|_{U^{k+1}}^*=O_{\sigma,\rho,\eta,m}(1)$ and $\|h-\sum_{\chi\in S_\rho}g_\chi\|_\infty\leq \sigma$.
\end{proposition}
\begin{remark}
Note that, by letting $\mc{D}$ be sufficiently small so that $(\nsR(\eta,m)^2+4)\mc{D}(\eta,m)\le \eta$, we have $c_0\le 9\eta$. Hence, if we further assume that $2\nsR(\eta,m)\mc{D}(\eta,m)^{1/4^{k+1}}\le \eta^{(k+2)/2^{k+1}}$, \eqref{eq:regul} becomes $\Big\|f-\sum_{\chi\in S_\rho}g_\chi\Big\|_{U^{k+1}}=O(\rho^{(k+1)/2^{k+1}} +\eta^{(k+2)/2^{k+1}})$.
\end{remark}
\noindent The last sentence of Proposition \ref{prop:regul} tells us that $\sum_{\chi\in S_\rho}g_\chi$ is a valid $k$-th order structured part of $f$, since it is an $(R,\delta)$-structured function of order $k$, with the additional strength that the error is small uniformly (not just in $L^2$ as in Definition \ref{def:kstruct}).
\begin{proof}
By \eqref{eq:ApproxBessel}, and using that $\|f_r\|_2\geq 0$, we have $\sum_{\chi\in S}\|g_\chi\|_2^2\leq \|f\|_2^2+c_0\leq 1+c_0$ where $c_0=(|S|^2+4)\mc{D}(\eta,m)+8\eta$, so $|S_\rho|\leq\rho^{-1}(1+c_0)$. Using that $\|g_e\|_\infty\leq \eta$, $\|f_e\|_\infty\leq 3$, $\|f_e\|_1\leq \eta$, and \eqref{eq:udvsl2} applied with $s=k+1$, we have
\begin{multline*}
\big|\,\big\|f-\textstyle\sum_{\chi\in S_\rho}g_\chi\big\|_{U^{k+1}} - \big\|\textstyle\sum_{\chi\in S\setminus S_\rho}g_\chi\big\|_{U^{k+1}}\,\big| \leq \|g_e\|_{U^{k+1}}+\|f_e\|_{U^{k+1}}+\|f_r\|_{U^{k+1}}\\
 \leq  \|g_e\|_\infty +  (\|f_e\|_\infty^{2^{k+1}-k-2} \|f_e\|_1^{k+2})^{1/2^{k+1}}+\|f_r\|_{U^{k+1}} \leq 4\eta^{(k+2)/2^{k+1}}+\mc{D}(\eta,m)=:c_1.
\end{multline*}
Hence $\|f-\sum_{\chi\in S_\rho}g_\chi\|_{U^{k+1}}\leq \|\sum_{\chi\in S\setminus S_\rho}g_\chi\|_{U^{k+1}} + c_1$. By \eqref{eq:diagUdSpec} applied with $B=S\backslash S_\rho$, we have 
$\|\sum_{\chi\in S\setminus S_\rho}g_\chi\|_{U^{k+1}}^{2^{k+1}}\leq \sum_{\chi\in S\setminus S_\rho} \|g_\chi\|_{U^{k+1}}^{2^{k+1}} + (|S\setminus S_\rho|^{2^{k+1}}-1)\mc{D}(\eta,m)^{1/2^k}$. Moreover, by \eqref{eq:udvsl2} and the fact that each $g_\chi$ is 1-bounded, we have
\[
\textstyle\sum_{\chi\in S\setminus S_\rho} \|g_\chi\|_{U^{k+1}}^{2^{k+1}} \leq \textstyle\sum_{\chi\in S\setminus S_\rho}\|g_\chi\|_2^{2k+4} < \rho^{k+1}\textstyle\sum_{\chi\in S\setminus S_\rho}\|g_\chi\|_2^2\leq \rho^{k+1}(1+c_0).
\]
Putting all this together, we conclude that
\begin{align*}
\big\|f-\textstyle\sum_{\chi\in S_\rho}g_\chi\big\|_{U^{k+1}} & \leq \big(\rho^{k+1}(1+c_0)+|S|^{2^{k+1}} \mc{D}(\eta,m)^{1/2^k}\big)^{1/2^{k+1}}+ c_1\\
& \leq  \rho^{(k+1)/2^{k+1}}(1+c_0)^{1/2^{k+1}}+|S| \mc{D}(\eta,m)^{1/2^{2k+1}}+ c_1,
\end{align*}
where in the last inequality we used that $\|\cdot\|_{\ell^{2^{k+1}}}\leq \|\cdot\|_{\ell^1}$.

To prove the last claim, we apply Theorem \ref{thm:nil-poly-are-close-to-bnd-Uk-dual}. Recall that for each $\chi\in S_\rho$ we have $g_\chi=F_\chi\co\phi$ with $\|F_\chi\|_{\textrm{sum}}\leq m+1$. Thus $\sum_{\chi\in S_\rho}g_\chi=F_\rho\co\phi$ where  $F_\rho:=\sum_{\chi\in S_\rho} F_\chi:\ns\to\mb{C}$ satisfies $\|F_\rho\|_{\textrm{sum}}\leq |S_\rho| (m+1)$. It then follows from Theorem \ref{thm:nil-poly-are-close-to-bnd-Uk-dual}, applied with $\delta=\sigma$, that there is a function $h:\ab\to\mb{C}$ with $\|F_\rho\co\phi-h\|_\infty\leq \sigma$ and $\|h\|_{U^{k+1}}^*=O_{\sigma,\rho,\eta,m}(1)$.
\end{proof}
\section{A structure theorem for $\mc{K}_\varepsilon(f\otimes\overline{f})$ using 2-step nilspace characters}\label{sec:Strucop}
\noindent In this section, we restrict the general $k$-th order treatment of Section \ref{sec:regularity-of-f} and focus on the quadratic case (2-nd order) to obtain a key ingredient for our algorithmic applications: Theorem \ref{thm:main-algo-pf} below. The idea is to apply the Fourier denoising operator $K_{\varepsilon}$ (recall Definition \ref{def:cont-cutoff}) to the decomposition resulting from our upgraded regularity lemma (Theorem \ref{thm:UpgradedReg}) in the case $k=2$. The resulting structure theorem for $\mc{K}_\varepsilon\big(f\otimes \overline{f}\big)$ decomposes this matrix into a sum of a bounded number of rank-1 matrices corresponding to non-negligible 2-step nilspace characters. 

\begin{theorem}\label{thm:main-algo-pf}
For any function $\mc{D}:\mb{R}_{>0}\times \mb{N}\to \mb{R}_{>0}$ and any $\eta\in(0,1)$, there exists $M=M(\eta,\mc{D})>0$ such that the following holds. Let $\ab$ be a finite abelian group, and let $f:\ab\to\mb{C}$ be a 1-bounded function. Then, for some $m\leq M$, there is a 2-step \textsc{cfr} nilspace $\ns$ of complexity at most $m$, a $\mc{D}(\eta,m)$-balanced morphism $\phi:\ab\to\ns$, a 1-bounded $m$-Lipschitz function $F:\ns\to\mb{C}$, and a set $S\subset \wh{\ab_k}(\ns)$ with $|S|\leq \nsR(\eta,m)$, such that for every $\varepsilon>0$ we have
\begin{equation}\label{eq:quadavcut1}
\mc{K}_\varepsilon\big(f\otimes \overline{f}\big) = \sum_{\chi\in S} g_\chi\otimes \overline{g_\chi}\; + \; E,
\end{equation}
where for every $\chi\in S$ we have $g_\chi:=F_\chi\co\phi$ with $\|g_\chi\|_2\geq \eta^2/\nsR(\eta,m)$, and $E\in \mb{C}^{\ab\times\ab}$ satisfies $\|E\|_2\leq 21\sqrt{\eta}+20\mc{D}(\eta,m)\varepsilon^{-1/2}  \nsR(\eta,m)^{5/2}+O_{m,\eta}(\varepsilon^{1/4})$.
\end{theorem}

\begin{remark}\label{rem:fixing-varepsilon}
Note that we can ensure that $\|E\|_2\leq 23\sqrt{\eta}$ by letting $\varepsilon=\varepsilon_{\eta,m}>0$ be such that $O_{m,\eta}(\varepsilon^{1/4})\le \sqrt{\eta}$ and then setting $\mc{D}(\eta,m)$ so that $20\mc{D}(\eta,m)\varepsilon^{-1/2}  \nsR(\eta,m)^{5/2}\le \sqrt{\eta}$.
\end{remark}
\begin{proof} We apply Theorem \ref{thm:UpgradedReg} with $\eta$ and $\mc{D}$. This yields that for some $M=M(\eta,\mc{D})>0$, there exists $m\le M$ such that we have the decomposition $f=\sum_{\chi\in S}g_\chi+g_e+f_e+f_r$, with the properties specified in Theorem \ref{thm:UpgradedReg}. Relabeling this set $S$ (resulting from Theorem \ref{thm:UpgradedReg}) as $S'$, we now restrict the last sum to the following subset of $S'$, to ensure that each $\|g_\chi\|_2$ is large:
\[
S:=\{\chi\in S':\|g_\chi\|_2\geq \eta^2/|S'|\}.
\]
Note that $\|\sum_{\chi\in S'\setminus S} g_\chi\|_2\leq |S'\setminus S| \frac{\eta^2}{|S'|}\leq \eta^2$. Therefore, letting $f_2:=\sum_{\chi\in S'\setminus S} g_\chi+g_e+f_e$, we have $f=\sum_{\chi\in S}g_\chi+f_2+f_r$, where $\|f_2\|_2\leq \eta^2 +\eta+\sqrt{3\eta}\leq 4\sqrt{\eta}$ (where we used that $\eta\in(0,1)$) and $\|f_r\|_{U^3}\leq \mc{D}(\eta,m)$. Note that since $|S'|\le \nsR(\eta,m)$, for every $\chi\in S$ we have $\|g_\chi\|_2\ge \eta^2/\nsR(\eta,m)$.

Now the result will follow from Proposition \ref{prop:reglemtool} by setting the parameters $\alpha_i$ in that proposition according to the above data. More precisely, in our present application of Proposition \ref{prop:reglemtool}, we have that condition $(i)$ in that proposition holds with $\alpha_1=4\sqrt{\eta}$, condition $(ii)$ holds with $\alpha_2=\mc{D}(\eta,m)$, condition $(iii)$ holds with $\alpha_3=1$ and $n=|S|$ (letting the $f_i$ in that proposition be the nilspace characters $g_\chi$). Furthermore, condition $(iv)$ holds with $\alpha_4=\mc{D}(\eta,m)$ (by the last sentence of Theorem \ref{thm:UpgradedReg}). Finally, by Theorem \ref{thm:nscharsquadchars} we have that each $g_\chi$ is a quadratic character on $\ab$ with parameters $(R,\sigma)$ where $R=O_{m}(\sigma^{-O_{m}(1)})$, so by Proposition  \ref{prop:weak-qua-char-stable-under-operator}, we have $\|\mc{K}_\varepsilon(g_\chi\otimes\overline{g_\chi})-g_\chi\otimes\overline{g_\chi}\|_2\leq 4\sigma+ O_{m}(\varepsilon^{1/4}\sigma^{-O_{m}(1)})$. Setting $\sigma=\eta/(4 \nsR(\eta,m))$, we have that condition $(v)$ in Proposition \ref{prop:reglemtool} holds with $\alpha_5=\eta/\nsR(\eta,m)+O_{m,\eta}(\varepsilon^{1/4})$.

Substituting these quantities into the bound $5\alpha_1+ 16 \varepsilon^{-1/2}\alpha_2+\varepsilon^{-1/2}  n^2(n^{1/2}\alpha_3+3)\alpha_4 + n\alpha_5$ from Proposition \ref{prop:reglemtool}, and using that $|S|\leq \nsR(\eta,m)$, the result follows. \end{proof}

\begin{remark}
Similarly as in Remarks \ref{rem:improvement-5-1} and \ref{rem:improvement-5-4}, the choice of $\eta^2/\nsR(\eta,m)$ as the lower bound for all the $\|g_\chi\|_2$ is arbitrary. A more general result can be proved by requiring another arbitrary function $\mc{D}_3(\eta,m)$ as this lower bound, but at the cost of a longer and more complicated formulation.
\end{remark}

\section{Algorithmic consequences of the structure theorem for $\mc{K}_\varepsilon(f\otimes\overline{f})$}\label{sec:algoconseqs}

\noindent In this section we use the notions of pseudoeigenvalues and pseudoeigenvectors, well-known in matrix analysis (see \cite[p.\ 16]{Tref&Embr}). However, we adapt these notions to $\ab$-matrices as follows, coherently with the normalization used in this paper (recall Definition \ref{def:ZmatOp}). We shall use these notions only in this normalized form throughout this section.
\begin{defn}\label{def:pseudo}
Let $\ab$ be a finite abelian group, let $M\in\mb{C}^{\ab\times \ab}$ be a $\ab$-matrix, and let $\beta>0$. A \emph{$\beta$-pseudoeigenvalue} of $M$ is a number $\lambda\in \mb{C}$ such that for some $v\in \mb{C}^{\ab}$ with $\|v\|_2=1$ we have $\|Mv-\lambda v\|_2< \beta$. We then say that $v$ is a \emph{$\beta$-pseudoeigenvector} corresponding to $\lambda$, or \emph{$(\lambda,\beta)$-pseudoeigenvector}. The set of $\beta$-pseudoeigenvalues of $M$ is denoted by $\sigma_\beta(M)$.
\end{defn}
\noindent Given a self-adjoint $\ab$-matrix $M\in\mb{C}^{\ab\times \ab}$ , we denote by $\textup{Spec}(M)$ the spectrum of $M$, i.e.\ the multiset of (real) eigenvalues of $M$, where the multiplicity of each eigenvalue equals the dimension of its eigenspace. For any $\rho>0$ we define the multiset
\begin{equation}\label{eq:specdef}
\textup{Spec}_\rho(M):=\{\lambda\in \textup{Spec}(M):\lambda \geq \rho\}.
\end{equation}
By Theorem \ref{thm:main-algo-pf}, we know that the self-adjoint matrix $\mc{K}_\varepsilon(f\otimes\overline{f})$ is (up to a small $L^2$-error) a sum of rank-1 matrices $g_\chi\otimes \overline{g_\chi}$ where the $g_\chi$ are quasiorthogonal nilspace characters with non-negligible $L^2$-norm. Then every $g_\chi$ yields a pseudoeigenvector of $\mc{K}_\varepsilon(f\otimes\overline{f})$, as follows.
\begin{lemma}\label{lem:pseudoevals}
Under the assumptions of Theorem \ref{thm:main-algo-pf}, let $\mc{K}_\varepsilon(f\otimes \overline{f})= \sum_{\chi\in S} g_\chi\otimes\overline{g_\chi} + E$ be the decomposition, obtained in \eqref{eq:quadavcut1}, for the 1-bounded function $f:\ab\to \mb{C}$. Then for any $\chi\in S$, the function $u_\chi:=g_\chi/\|g_\chi\|_2$ is a $\beta$-pseudoeigenvector of $\mc{K}_\varepsilon(f\otimes \overline{f})$ with pseudoeigenvalue $\lambda_\chi:=\|g_\chi\|_2^2$, where
 $\beta=21\sqrt{\eta}+\mc{D}(\eta,m)(20\varepsilon^{-1/2}  \nsR(\eta,m)^{5/2}+\nsR(\eta,m)^2/\eta^2)+O_{m,\eta}(\varepsilon^{1/4})$.
\end{lemma}
\begin{remark}
Note that, choosing $\varepsilon>0$ sufficiently small so that the term $O_{m,\eta}(\varepsilon^{1/4})$ is at most $\eta^{1/2}$, we can then prescribe the function $\mc{D}$ to be sufficiently fast-decreasing so that $\beta=O(\eta^{1/2})$.
\end{remark}

\begin{proof}
Recall from Theorem \ref{thm:main-algo-pf} that $\|g_\chi\|_2\geq \eta^2/\nsR(\eta,m)$ for each $\chi\in S$. Letting $M=\mc{K}_\varepsilon(f\otimes \overline{f})$, we have $M u_\chi = \lambda_\chi u_\chi + \sum_{\chi'\in S\setminus\{\chi\}}   \lambda_{\chi'} u_{\chi'} \langle u_{\chi'},u_\chi\rangle + E u_\chi$. Hence 
\[
\|M u_\chi - \lambda_\chi u_\chi \|_2 \leq \textstyle\sum_{\chi'\in S\setminus\{\chi\}}  \|g_{\chi'}\|_2^2 \, \tfrac{|\langle g_{\chi'},g_\chi\rangle |}{\|g_{\chi'}\|_2\|g_{\chi}\|_2} + \| Eu_\chi \|_2 \leq \textstyle\sum_{\chi'\in S\setminus\{\chi\}}  \frac{\|g_{\chi'}\|_2}{\|g_{\chi}\|_2} \, |\langle g_{\chi'},g_\chi\rangle |+ \| E\|_2.
\]
Using that $|\langle g_{\chi'},g_\chi\rangle |\leq \mc{D}(\eta,m)$ (by Theorem \ref{thm:UpgradedReg}), we conclude that $
\|M u_\chi - \lambda_\chi u_\chi \|_2\leq |S| \frac{\nsR(\eta,m)}{\eta^2} \mc{D}(\eta,m)+ \| E\|_2$, and the result follows.
\end{proof}
\noindent We now face the question of how to recover, algorithmically, the pseudo-eigenvectors (normalized nilspace characters) $u_\chi$ from the eigenvalues and eigenvectors of $\mc{K}_\varepsilon(f\otimes\overline{f})$. 

Firstly, in Subsection \ref{subsec:reg} we obtain a decomposition of $f$ into a quadratically structured part $\sum_{\chi\in S_\rho} g_\chi$ and a $U^3$-noise part, using the leading eigenvalues and eigenvectors of $\mc{K}_\varepsilon(f\otimes\overline{f})$ (see Theorem \ref{thm:algoregul}). This will enable us to obtain such a decomposition algorithmically.

Secondly, in Subsection \ref{subsec:specsep} we focus on this structured part $\sum_{\chi\in S_\rho} g_\chi$ to describe how one can recover the individual nilspace characters $g_\chi$ using dominant eigenvalues of $\mc{K}_\varepsilon(f\otimes\overline{f})$. 

Before we get started on the first task, let us record the following simple result, that will enable us to express the pseudoeigenvectors $u_\chi$ in terms of the eigenvectors of $\mc{K}_\varepsilon(f\otimes\overline{f})$.
\begin{lemma}\label{lem:clusterapprox}
Let $\ab$ be a finite abelian group and let $M\in \mb{C}^{\ab\times \ab}$ be a self-adjoint $\ab$-matrix, with real eigenvalues $\lambda_i$ for $i\in [|\ab|]$. Let $\lambda$ be a $\beta$-pseudoeigenvalue of $M$ and let $u$ be a corresponding $(\lambda,\beta)$-pseudoeigenvector. Let $\{v_i:i\in [|\ab|]\}$ be an orthonormal basis of eigenvectors of $M$,  let $u=\sum_{i\in[|\ab|]}\mu_i v_i$ be the expansion of $u$ in this basis, and consider the eigenvalue cluster
$C_\delta(\lambda):=\{i\in [|\ab|]:|\lambda_i-\lambda|\leq \delta \}$. Then 
\begin{align}\label{eq:pseudoevec}
\big\|u-\textstyle\sum_{i\in C_\delta(\lambda)} \mu_i v_i\big\|_2\leq\beta/\delta.
\end{align}
\end{lemma}
 \begin{proof}
Suppose that the eigenvalues $\lambda_i$ of $M$ satisfy $\lambda_i\geq \lambda_j$ if $i<j$. By assumption we have $Mu-\lambda u=\sum_{i=1}^{|\ab|} \mu_i(\lambda_i-\lambda)v_i$, so $\|Mu-\lambda u\|_2^2=\sum_{i=1}^{|\ab|}\mu_i^2(\lambda_i-\lambda)^2$. Since $u$ is a $\beta$-pseudoeigenvector of $M$, the left side of the last equation is at most $\beta^2$, so $\sum_{i=1}^{|\ab|}\mu_i^2(\lambda_i-\lambda)^2\leq\beta^2$. Let $I_\delta(\lambda):=\{i\in [|\ab|]:|\lambda_i-\lambda|>\delta \}$ be the complement of $C_\delta(\lambda)$. Then $\delta^2 \sum_{i\in I_\delta}\mu_i^2\leq \sum_{i=1}^{|\ab|}\mu_i^2(\lambda_i-\lambda)^2\leq\beta^2$, whence $\big\|u-\sum_{i\in C_\delta(\lambda)} \mu_i v_i\big\|_2^2=\sum_{i\in I_\delta}\mu_i^2\leq\beta^2/\delta^2$.\end{proof}

\subsection{Obtaining a quadratically structured part of $f$ from the spectrum of $\mc{K}_\varepsilon(f\otimes\overline{f})$}\label{subsec:reg}\hfill\smallskip\\
The case $k=2$ of Proposition \ref{prop:regul} establishes the sum of 2-step nilspace characters $\sum_{\chi\in S_\rho} g_\chi$ as a valid quadratic (or 2-nd order) structured part of $f$. The main result in this subsection, Theorem \ref{thm:algoregul}, tells us that this structured part can be recovered, up to a small error, as the projection of the function $f$ to the linear span of the eigenvectors of $\mc{K}_\varepsilon(f\otimes\overline{f})$ with large eigenvalues. This will prove Theorem \ref{thm:reg-intro}, and thus the validity of our spectral algorithm to obtain the order-2 structured part of a bounded function (i.e.\ Algorithm \ref{alg:reg}).

To prove this result, we shall need to compare two orthogonal projections. On the one hand, the projection of $f$ to the space spanned by its dominant 2-step nilspace characters $g_\chi$ (provided by Proposition \ref{prop:regul}). On the other hand, the projection of $f$ to the linear span of the leading eigenvectors of $\mc{K}_\varepsilon(f\otimes\overline{f})$. We are thus naturally led to standard tools in the theory of orthogonal projections in Hilbert space, which we now gather before proving the main result.

The first such tool is the following metric on the set of orthogonal projections.
\begin{defn}
Let $W$ be a real or complex Hilbert space with inner product $\langle\cdot,\cdot\rangle$ and norm $\|v\|_W:=\langle v,v\rangle^{1/2}$. For any two orthogonal projections $P,Q$ on $W$, we denote by $d(P,Q)$ the operator norm of  $P-Q$, i.e.
\begin{equation}\label{eq:aperture}
    d(P,Q):= \|P-Q\| = \textstyle\sup_{v\in W,\|v\|_W\leq 1} \|Pv-Qv\|_W.
\end{equation}
\end{defn}
\noindent We shall apply this in the Hilbert space $\mb{C}^{\ab}$ for a finite abelian group $\ab$, equipped with the inner product and $L^2$-norm with the normalization used throughout this paper. Recall that any such orthogonal projection $P$ is uniquely associated with the vector subspace that is its image, which we shall denote by $U_P$. Thus, the above metric gives a notion of distance between linear subspaces of $W$, which is well-known in the theory of Grassmannians (see, for instance, \cite[(3)]{Morris} or the book \cite{Akh&Glaz}). Recall the following basic fact (see \cite[\S 39]{Akh&Glaz}, where the metric is called the \emph{aperture} of the subspaces $U_P$, $U_Q$).
\begin{lemma}\label{lem:dim-equal-grass}
For any pair of orthogonal projections $P,Q$ on a Hilbert space $W$, we have $d(P,Q)\leq 1$, and if $d(P,Q)< 1$ then the dimensions of the subspaces $U_P$, $U_Q$ are equal.
\end{lemma}
\begin{proof}
To see that $\|P-Q\|\leq 1$, let us recall here a short standard argument: note that $P-Q=P(1-Q)-(1-P)Q$, so that for every $v$ we have $Pv-Qv=P(1-Q)v-(1-P)Qv$, and since $P(1-Q)v$ and $(1-P)Qv$ are clearly orthogonal, it follows that $\|(P-Q)v\|_W^2=\|P(1-Q)v\|_W^2+\|(1-P)Qv\|_W^2\leq \|(1-Q)v\|_W^2+\|Qv\|_W^2=\|v\|_W^2$, as required.

The second claim in the proof is the main theorem in \cite[\S 39]{Akh&Glaz}.
\end{proof}
\noindent The following alternative expression for $d(P,Q)$ will be useful (see \cite[p. 111, equation (2)]{Akh&Glaz}): 
\begin{equation}\label{eq:altGrassmann}
d(P,Q)=\max\Big(\sup_{v\in U_P:\|v\|_W\leq 1} \textrm{dist}(v,U_Q), \sup_{v\in U_Q:\|v\|_W\leq 1} \textrm{dist}(v,U_P)\Big),
\end{equation}
where $\textrm{dist}(v,U_Q):=\|v-Q(v)\|_W$. We shall also use the following estimate.
\begin{lemma}\label{lem:d'bound} Let $P,Q$ be orthogonal projections on a Hilbert space $W$. Let $b_1,b_2,\dots,b_\ell$ be an orthonormal basis of $U_P$ (the image of $P$). Then
\begin{equation}
\sup_{v\in U_P:\|v\|_W\leq 1} \textup{dist}(v,U_Q) \leq \ell^{1/2}\max_{1\leq i\leq \ell}\|b_i-Q(b_i)\|_W.
\end{equation}
\end{lemma}

\begin{proof}
Let $v=\sum_{i=1}^\ell\lambda_i b_i$ with $\|v\|_W\leq 1$. By the Cauchy-Schwarz inequality, we have that $\|v-Q(v)\|_W$  is at most  $\sum_{i=1}^\ell|\lambda_i|\|b_i-Q(b_i)\|_W\leq \bigl(\sum_{i=1}^\ell|\lambda_i|\bigr)\max_{1\leq i\leq \ell}\|b_i-Q(b_i)\|_W \leq \ell^{1/2}\max_{1\leq i\leq \ell}\|b_i-Q(b_i)\|_W$ (where we used that $\|v\|_W^2=\sum_{i=1}^\ell|\lambda_i|^2$).
\end{proof}
\begin{defn}[Projection to the span of spectrally-dominant eigenvectors]
Let $T$ be a self-adjoint linear operator on a Hilbert space $W$, and let $\rho>0$. We denote by $\textup{Eigen}_\rho(T)$ the subspace of $W$ spanned by eigenvectors with corresponding eigenvalues in $\textup{Spec}_\rho(T)$. We denote by $\mc{P}_{T,\rho}$ the orthogonal projection to $\textup{Eigen}_\rho(T)$ in $W$. 
\end{defn}
\noindent Another ingredient is the following refinement of the Gram-Schmidt process.
\begin{lemma}[Quantitative Gram-Schmidt process]\label{lem:GSapp}
Let $W$ be a Hilbert space. For each positive integer $s\geq 2$ there is a constant $C_s>0$ such that the following holds. Let $u_1,\ldots,u_s\in W$ such that $\|u_i\|_W=1$ and such that for some $\tau\in (0,1)$ we have $|\langle u_i,u_j\rangle|\leq \tau/C_s$ for all pairs $i\neq j$. Then there exist orthonormal vectors $w_1,\ldots,w_s\in W$ such that for every $i\in [s]$ we have $\|u_i-w_i\|_2\leq C_s\tau$ and such that, for any $i\in[s]$, $\Span(w_1,\ldots,w_i)=\Span(u_1,\ldots,u_i)$.
\end{lemma}
\noindent Estimates for $C_s$ can be obtained from the proof, e.g.\ $C_2=1$ and $C_s=13\cdot5^{s-3}-1$ for $s\ge 3$.

\begin{proof}
We apply the Gram-Schmidt orthonormalization to the set $\{u_1,\ldots,u_s\}$, and keep track of the changes performed by this process, to ensure the $L^2$-smallness of these changes.

Following the formulas used in the process (see for instance \cite[\S 0.6.4]{Horn&Johnson}), we have $w_1=u_1$, then $w_2'=u_2-\langle u_2,w_1\rangle w_1$ and $w_2:=w_2'/\|w_2'\|_W$. Hence $\|u_2-w_2'\|_W=|\langle u_2,w_1\rangle|\leq \tau$. We can therefore let $C_2=1$. 

Now, for general $s>2$, the formulae are $w_s' = u_s - \sum_{j=1}^{s-1}\langle u_s,w_j\rangle w_j$ and $w_s =w_s'/\|w_s'\|_W$.  For $j\in [s-1]$, we have $|\langle u_s,w_j\rangle|\leq |\langle u_s,u_j\rangle|+\|u_j-w_j\|_W$ by the triangle and Cauchy-Schwarz inequalities, and the fact that $u_j$ is a unit vector. Hence by induction we can assume that for every $j\in [s-1]$ we have $|\langle u_s,w_j\rangle|\leq (1+ C_j)\tau$. It follows that $\|u_s-w_s'\|_W\leq \sum_{j\in [s-1]}  |\langle u_s,w_j\rangle|\leq \tau \sum_{j\in [s-1]}   (1+ C_j)=\tau \big(s+\sum_{j=3}^{s-1}C_j\big)$. This in turn implies, by the triangle inequality, that $|1-\|w_s'\|_W|\leq \big(s+\sum_{j=3}^{s-1}C_j\big)\tau$. Therefore, letting $w_s:=w_s' /\|w_s'\|_W$, we have $\|u_s-w_s\|_W$ equals
\begin{equation*}
\frac{\|u_s\|w_s'\|_W-w_s'\|_W}{\|w_s'\|_W} \leq \frac{\|u_s(\|w_s'\|_W-1)\|_W+\|u_s-w_s'\|_W}{1-(1-\|w_s'\|_W)} \leq  \frac{2\tau\big(s+\textstyle\sum_{j=3}^{s-1}C_j\big)}{1-\tau \big(s+\textstyle\sum_{j=3}^{s-1}C_j\big)}.
\end{equation*}
If $\tau\leq 1/\big(2 (s+\sum_{j=3}^{s-1}C_j)\big)$, then $\|u_s-w_s\|_W\leq 4(s+\sum_{j=3}^{s-1}C_j) \tau $. Letting $C_s:=4(s+\sum_{j=3}^{s-1}C_j)$, we have  that if $\tau\leq 1/C_s$ then $\|u_s-w_s\|_W\leq C_s\tau$. This completes the induction.
\end{proof}
\noindent We also use the following version of the Hoffman-Wielandt theorem \cite[Corollary 6.3.8]{Horn&Johnson}.
\begin{theorem}\label{thm:H-W}
Let $\ab$ be a finite abelian group and let $A, B\in \mb{C}^{\ab\times \ab}$ be self-adjoint $\ab$-matrices. Let $\lambda_1\geq \cdots\geq \lambda_n$ be the eigenvalues of $A$ arranged in non-increasing order, and similarly let $\lambda_1'\geq \cdots \geq \lambda_n'$ be the eigenvalues of $B$ arranged in non-increasing order. Then 
\begin{equation}\label{eq:H-W}
\textstyle\sum_{i=1}^n |\lambda_i'-\lambda_i|^2\leq \|A-B\|_2^2.
\end{equation}
\end{theorem}
\noindent Note that in the standard matrix-analysis setup (such as in \cite{Horn&Johnson}), the norm on the right side in \eqref{eq:H-W} is the Frobenius norm (involving summation over all entries, rather than normalized summation as here), and the eigenvalues are also defined relative to summation. In \eqref{eq:H-W}, we give the equivalent \emph{normalized} version of the Hoffman-Wielandt inequality, where we divide the original inequality by $n^2=|\ab|^2$. On the right side, this division goes into having the normalized $L^2$-norm, i.e., the Hilbert-Schmidt norm, while on the left, the division is compensated by the fact that our eigenvalues are $1/n$ times the usually-defined eigenvalues (recall Definition \ref{def:ZmatOp}).

The following result enables us, given proximal self-adjoint $\ab$-matrices, to find proximal orthogonal projections onto dominant eigenspaces of these $\ab$-matrices.

\begin{theorem}\label{thm:sim-subsp}
Let $\rho_1,\rho_2\in (0,1)$. Let $T_1,T_2$ be self-adjoint $\ab$-matrices such that $\|T_2\|_2\leq 1$ and $\|T_1-T_2\|_2\leq\frac{\rho_2\rho_1^4}{120}$. Then there is $\rho\in [\rho_1/2,\rho_1]$ such that $d(\mc{P}_{T_1,\rho},\mc{P}_{T_2,\rho})\leq\rho_2$, whence in particular $|\textup{Spec}_\rho(T_1)| = |\textup{Spec}_\rho(T_2)|$. Moreover $T_2$ has no eigenvalue in $[\rho-\frac{\rho_1^3}{30},\rho+\frac{\rho_1^3}{30})$. 
\end{theorem}
\begin{proof}
Let $\delta=\|T_1-T_2\|_2$. Let $S=\{\lambda_1\geq \ldots\geq \lambda_t\}=\textup{Spec}_\rho(T_2)\cap[\rho_1/2,\rho_1]$ be the (multi)set of eigenvalues of $T_2$ which also lie in $[\rho_1/2,\rho_1]$. The assumption $\|T_2\|_2\leq 1$, together with the fact that (by our general choice of normalizations) the norm $\|T_2\|_2$ is the sum of squares of the eigenvalues of $T_2$, implies that $\sum_{i\in[t]}\lambda_i^2\leq \|T_2\|_2^2\leq 1$, so $t\leq 4/\rho_1^2$. Note that there exists $\rho'\in[\rho_1/2+\frac{\rho_1^3}{10},\rho_1]$ such that $S\cap[\rho'-\frac{\rho_1^3}{10},\rho')=\emptyset$. Indeed, otherwise, letting $S'=S\cup\{\rho_1/2,\rho_1\}$, we would have that any two consecutive elements of $S'$ differ by at most $\frac{\rho_1^3}{10}$, implying that we could partition $[\rho_1/2,\rho_1]$ into $t+1$ consecutive intervals of length at most $\frac{\rho_1^3}{10}$, which yields the contradiction that $\frac{\rho_1^3}{10}(t+1)\le \frac{\rho_1^3}{10}(\frac{4}{\rho_1^2}+1)<\frac{\rho_1}{2}$.

We claim that $(\textup{Spec}(T_1)\cup \textup{Spec}(T_2))\cap [\rho'-2\frac{\rho_1^3}{30},\rho'-\frac{\rho_1^3}{30}) =\emptyset$. Indeed, Theorem \ref{thm:H-W} implies that each eigenvalue of $T_1$ is at distance at most $\delta$ from an eigenvalue of $T_2$. As the interval $[\rho'-\frac{\rho_1^3}{10},\rho')$ contains no eigenvalue of $T_2$,  the interval $[\rho'-\frac{\rho_1^3}{10}+\delta,\rho'-\delta)$ cannot contain eigenvalues of $T_1$, and since $\delta=\|T_1-T_2\|_2\le  \frac{\rho_2\rho_1^4}{120}\le \frac{\rho_1^3}{30}$, the claim follows. Let $\rho:=\rho'-\frac{\rho_1^3}{30}$, $P_1:=\mc{P}_{T_1,\rho}$, and $P_2:=\mc{P}_{T_2,\rho}$.

Let $v$ be an eigenvector of $T_1$ with $\|v\|_2=1$ and eigenvalue $\kappa$ in $\textup{Spec}_\rho(T_1)$. We have $\|T_2v-\kappa v\|_2=\|T_1v+(T_2-T_1)v-\kappa v\|_2=\|(T_2-T_1)v\|_2\leq\delta$.

By Lemma \ref{lem:clusterapprox}, there exists $w\in \textup{Eigen}_{\rho}(T_2)$  such that $\|v-w\|_2\leq \delta 30\rho_1^{-3}$ (more precisely, $w$ is the combination of eigenvectors corresponding to the relevant eigenvalue cluster, as given by Lemma \ref{lem:clusterapprox}, noting in addition that, by our choice of $\rho$, there is no such eigenvalue less than $\rho$ in this cluster). It follows that $\|v-P_2(v)\|_2\leq\|v-w\|_2\leq 30 \delta  \rho_1^{-3}$. The image of $P_1$ is $\textup{Eigen}_{\rho}(T_1)$, which has dimension $|\textup{Spec}_{\rho}(T_1)|\leq  {\rho}^{-2} \sum_{\lambda\in \textup{Spec}(T_1)} \lambda^2= {\rho}^{-2} \|T_1\|_2^2\leq (1+\delta)^2/{\rho}^2\le  16/\rho_1^2$. Hence, by Lemma \ref{lem:d'bound} applied with $P=P_1, Q=P_2$, we have $\sup_{v\in U_{P_1}:\|v\|\leq 1} \textup{dist}(v,U_{P_2}) \leq 120 \delta \rho_1^{-4}$. Similarly, we obtain $\sup_{v\in U_{P_2}:\|v\|\leq 1} \textup{dist}(v,U_{P_1})\leq 120 \delta \rho_1^{-4}$. Hence, by \eqref{eq:altGrassmann}, we have $d(P_1,P_2)\leq 120 \delta \rho_1^{-4}$, so if $\delta\leq\rho_2\rho_1^4/120$, then $d(P_1,P_2)\leq\rho_2$. As $\rho_2<1$, by the last sentence of Lemma\ref{lem:dim-equal-grass} we have  $|\textup{Spec}_{\rho}(T_1)| = |\textup{Spec}_{\rho}(T_2)|$.
\end{proof}

We can now prove the main result of this subsection.
\begin{theorem}\label{thm:algoregul}
Let $\mc{H}:\mb{R}_{>0}\to \mb{R}_{>0}$ and $\mc{B}:\mb{R}_{>0}\times \mb{N}\to \mb{R}_{>0}$ be arbitrary functions, and let $\rho_0\in(0,1)$. Then there exists $M=M(\rho_0,\mc{B}, \mc{H})>0$ and $\varepsilon_0=\varepsilon_0(\rho_0,\mc{B}, \mc{H})>0$ such that the following holds. For any finite abelian group $\ab$ and any $1$-bounded function $f:\ab\to\mb{C}$, there exists $m\le M$, $\varepsilon=\varepsilon_{\mc{H},\rho_0,m}\in[\varepsilon_0,1]$, and $\rho=\rho_{\rho_0,f}\in[\rho_0/2,\rho_0]$ satisfying the following.

Letting $P_{\rho}$ denote the orthogonal projection to $\textup{Eigen}_\rho\big(\mc{K}_\varepsilon(f\otimes\overline{f})\big)$ in $\mb{C}^{\ab}$, we have
\begin{equation}\label{eq:algoregul}
\|f-P_{\rho}(f)\|_{U^3}\le  2\rho^{3/8}.
\end{equation}
Moreover, there is a $2$-step \textsc{cfr} nilspace $\ns$ of complexity $\leq m$, a $\mc{B}(\rho_0,m)$-balanced morphism $\phi:\ab\to\ns$, a 1-bounded $m$-Lipschitz function $F:\ns\to\mb{C}$, and a set $S_{\rho}=S_{\rho,m,f}\subset \wh{\ab_k}(\ns)$ with $|S_{\rho}|=|\textup{Spec}_\rho\big(\mc{K}_\varepsilon(f\otimes\overline{f})\big)|\le 10/\rho$, such that for every $\chi\in S_{\rho}$, letting $g_\chi=F_\chi\co\phi$, we have $\|g_\chi\|_2^2\geq \rho$ and
\begin{eqnarray}\label{eq:algoregul2}
& \|P_{\rho}(f)-\sum_{\chi\in S_{\rho}}g_\chi\|_2\le\mc{H}(\rho_0),&
\end{eqnarray}
and there exists $h:\ab\to\mb{C}$ such that $\|h-\sum_{\chi\in S_{\rho}}g_\chi\|_{\infty}\le \mc{H}(\rho_0)$ and $\|h\|_{U^3}^* = O_{\rho_0,\mc{B},\mc{H}}(1)$. 
\end{theorem}
\begin{remark}
This result can be viewed as a spectral structure-randomness decomposition of order 2. In particular, note that by \eqref{eq:algoregul2} and the properties of $h$, we have that $P_\rho(f)$ is a structured function of order 2 in the sense of Definition \ref{def:kstruct}, with parameters $\big(O_{\rho_0,\mc{B},\mc{H}}(1),2\mc{H}(\rho_0)\big)$.
\end{remark}
\begin{remark}\label{rem:separation-rho-from-eigen}
It will follow from the proof that $\rho$ can be chosen so that there are no eigenvalues of $\mc{K}_\varepsilon(f\otimes \overline{f})$ in the interval $[\rho-\rho_0^3/30,\rho+\rho_0^3/30)$.
\end{remark}

\begin{proof}
Let $C_s$ be the constant defined in Lemma \ref{lem:GSapp}. Let 
\begin{equation}\label{eq:upsilon}
\Upsilon(\rho_0):=\min\left( 1, \frac{\mc{H}(\rho_0)}{\rho_0^{5/2}/26+1}, \frac{(\rho_0/2)^{3/8}/2}{\rho_0^4/600+\rho_0^{5/2}/26+1} \right),
\end{equation} 
and let $\eta=\rho_0^{8}\Upsilon(\rho_0)^2/(9\cdot 10^6)\in[0,1]$. Let $\mc{D}:\mb{R}_{>0}\times \mb{N}\to \mb{R}_{>0}$ and $\varepsilon>0$ be chosen so that the following conditions hold:

\begin{enumerate}
\item\label{it:reg1} The number $\tau:=C_{\nsR(\eta,m)}\nsR(\eta,m)^2\eta^{-4}\mc{D}(\eta,m)$ is at most 1,
\item\label{it:reg2} $2\nsR(\eta,m)C_{\nsR(\eta,m)}\tau\le \sqrt{\eta}$,  
    \item\label{it:reg3} $(\nsR(\eta,m)^2+4)\mc{D}(\eta,m)\le \eta$,
    \item\label{it:reg4} $\varepsilon=\varepsilon_{m,\eta}>0$ is chosen following Remark \ref{rem:fixing-varepsilon} (which also imposes another condition on $\mc{D}$ being small enough), so that $\|E\|_2\leq 23\sqrt{\eta}$.
    \item\label{it:reg5} $\mc{D}(\eta,m)(\nsR(\eta,m)+O_{\eta,m}(1))\le \eta^{1/2}$, where the implicit constant is defined as per Lemma \ref{lem:correlations},
    \item\label{it:reg6} $2\nsR(\eta,m)\mc{D}(\eta,m)^{1/4^3}\le \eta^{1/2}$, and
    \item $\mc{D}(\eta,m)\le\mc{B}(\rho_0,m)$.
\end{enumerate}
We apply Theorem \ref{thm:main-algo-pf} to $f$ with parameters $\eta$ and $\mc{D}$. Thus, we obtain the decomposition $\mc{K}_\varepsilon(f\otimes\overline{f})=\sum_{\chi\in S} g_\chi\otimes\overline{g_\chi} + E$. Note that letting $\varepsilon_0:=\min_{m\le M}(\varepsilon_{\eta,m})$ we have that indeed $\varepsilon\ge \varepsilon_0>0$ and $\varepsilon_0=\varepsilon_0(\rho_0,\mc{B})$.

Next, note that by our choices above, the unit vectors $u_\chi:=g_\chi/\|g_\chi\|_2$ satisfy, for any $\chi\neq\chi'$,
\begin{eqnarray}\label{eq:bnd-normalized-u-chi}
&|\langle u_\chi,u_{\chi'}\rangle|\leq |\langle g_\chi,g_{\chi'}\rangle| / (\|g_\chi\|_2\|g_{\chi'}\|_2) \leq \nsR(\eta,m)^2\eta^{-4}\mc{D}(\eta,m)\leq \tau/C_s,&    
\end{eqnarray}
with $s=|S|$ and $\tau$ defined in \eqref{it:reg1} (in particular $\tau\le 1$).
\begin{eqnarray}\label{eq:order-g-chi}
&\text{We apply Lemma \ref{lem:GSapp} to $\{u_\chi\}_{\chi\in S}$ ordered decreasingly according to $\|g_\chi\|_2^2$.}&
\end{eqnarray}
Thus, we obtain the orthonormal vectors $w_\chi$. We then define $g_\chi':=\|g_\chi\|_2 w_\chi$ for each $\chi\in S$. We thus have the orthogonal set $\{g_\chi':\chi\in S\}$ which satisfies
\begin{equation}\label{eq:GSstep}
\|g'_\chi\|_2=\|g_\chi\|_2\textrm{ and }\|g_\chi-g_{\chi}'\|_2=\|g_\chi\|_2 \|u_\chi-w_\chi\|_2\leq C_s\tau.    
\end{equation}
Now we consider the following $\ab$-matrices:
\begin{eqnarray}\label{eq:T1T2}
&T_1:=\sum_{\chi\in S}g'_\chi\otimes\overline{g'_\chi}\quad\text{and}\quad T_2:=\mc{K}_\varepsilon(f\otimes\overline{f})=\sum_{\chi\in S}g_\chi\otimes\overline{g_\chi}+E.&
\end{eqnarray}
Note that by Lemma \ref{lem:KepsL2cont} with $M=\mc{K}_\varepsilon(f\otimes\overline{f})$ and $N=0$, we have $\|T_2\|_2\leq 1$. We also have $\|g'_\chi\otimes\overline{g'_\chi}-g_\chi\otimes\overline{g_\chi}\|_2\leq \|g'_\chi-g_\chi\|_2(\|g'_\chi\|_2+\|g_\chi\|_2)\leq 2 C_s\tau$. Hence $\|T_1-T_2\|_2\leq |S|2 C_s\tau+\|E\|_2$, and by \eqref{it:reg2} and \eqref{it:reg4}, we have $\|T_1-T_2\|_2\leq 24\sqrt{\eta}$.

We want to apply Theorem \ref{thm:sim-subsp} to $T_1$ and $T_2$ with $\rho_1=\rho_0$, and $\rho_2=\Upsilon(\rho_0)$. Hence, we want to ensure that $\|T_1-T_2\|_2\leq 24\sqrt{\eta}\le \rho_0^4\Upsilon(\rho_0)/120$. But note that this holds by our choice of $\eta$. Hence, there exists $\rho\in[\rho_0/2,\rho_0]$ such that $d(\mc{P}_{T_1,\rho},\mc{P}_{T_2,\rho})\leq\Upsilon(\rho_0)$. By definition of the metric $d$ (recall \eqref{eq:aperture}) and the fact that $\|f\|_2\leq 1$, this implies that $\|\mc{P}_{T_2,\rho}(f)-\mc{P}_{T_1,\rho}(f)\|_2\leq \Upsilon(\rho_0)$. 

Note that $\mc{P}_{T_2,\rho}(f)$ is the desired projection $P_{\rho}(f)$. Let $S_{\rho}:=\{\chi\in S:\|g_\chi\|_2^2\geq \rho\}$. Note that by \eqref{eq:GSstep} and Theorem \ref{thm:sim-subsp} we know that $|S_\rho|=|\{\chi\in S:\|g_\chi'\|_2^2\ge \rho\}|=|\textup{Spec}_\rho(T_2)|$ and that $\rho\in[\rho_0/2,\rho_0]$ can be chosen explicitly depending on $T_2$. Note also that by Proposition \ref{prop:regul} and \eqref{it:reg3} we have $|S_\rho|\le 10/\rho$.

We claim that $\mc{P}_{T_1,{\rho}}(f)$ is close in $L^2$ to $\sum_{\chi\in S_{\rho}} g_\chi$. Indeed, the vectors $w_\chi=\frac{g_\chi'}{\|g_\chi'\|_2}$ form an orthonormal set of eigenvectors of $T_1$, with corresponding eigenvalues $\|g_\chi'\|_2^2$. This together with $\|g'_\chi\|_2=\|g_\chi\|_2$ implies the following expression for the projection:
\[
\mc{P}_{T_1,{\rho}}(f)=\textstyle\sum_{\chi\in S_{\rho}}\big\langle f,\tfrac{g_\chi'}{\|g_\chi'\|_2}\big\rangle \tfrac{g_\chi'}{\|g_\chi'\|_2}=\textstyle\sum_{\chi\in S_{\rho}}\big\langle f,\tfrac{g_\chi'}{\|g_\chi\|_2}\big\rangle  \tfrac{g_\chi'}{\|g_\chi\|_2}.
\]
By \eqref{eq:GSstep} we have $\|g_\chi-g_\chi'\|_2\leq C_s\tau$, so $\|\langle f,g_\chi'\rangle g_\chi'-\langle f,g_\chi\rangle g_\chi\|_2\leq 2 C_s\tau$, and so
\begin{align*}
\big\|\mc{P}_{T_1,{\rho}}(f)-\textstyle\sum_{\chi\in S_{\rho}}\big\langle f,\tfrac{g_\chi}{\|g_\chi\|_2}\big\rangle  \tfrac{g_\chi}{\|g_\chi\|_2}\big\|_2  \leq &\textstyle\sum_{\chi\in S_{\rho}} \big\|\big\langle f,\tfrac{g_\chi'}{\|g_\chi\|_2}\big\rangle  \tfrac{g_\chi'}{\|g_\chi\|_2}-\big\langle f,\tfrac{g_\chi}{\|g_\chi\|_2}\big\rangle  \tfrac{g_\chi}{\|g_\chi\|_2}\big\|_2 \\ \leq & 2 {\rho}^{-1} |S_{\rho}| C_s\tau\; \leq \;2 {\rho}^{-1} |S| C_s\tau\; \leq \; \rho_0^3\Upsilon(\rho_0)/1500, 
\end{align*}
where we used that $|S_{\rho}|\le |S|$, that $2|S|C_s\tau\le \sqrt{\eta}$ (by \eqref{it:reg2}), and that $\rho\geq \rho_0/2$.

Finally, by \eqref{eq:correlations}, for every $\chi\in S_{{\rho}}$ we have
\[
|\langle f,g_\chi/\|g_\chi\|_2\rangle-\|g_\chi\|_2|\leq {\rho}^{-1/2}\big(\mc{D}(\eta,m)(\nsR(\eta,m)+O_{\eta,m}(1))+3\eta^{1/2}\big).
\]
Recalling the definition of $\eta$ and using \eqref{it:reg5}, we have $|\langle f,\frac{g_\chi}{\|g_\chi\|_2}\rangle-\|g_\chi\|_2|\leq \rho_0^{7/2}\Upsilon(\rho_0)/530$. Hence, using $|S_{{\rho}}|\le 10/{\rho}$ and $\rho\geq \rho_0/2$, we conclude that
\[
\big\|\textstyle\sum_{\chi\in S_{{\rho}}} g_\chi-\textstyle\sum_{\chi\in S_{{\rho}}}\big\langle f,\tfrac{g_\chi}{\|g_\chi\|_2}\big\rangle  \tfrac{g_\chi}{\|g_\chi\|_2}\big\|_2  \leq |S_{{\rho}}| \rho_0^{7/2}\Upsilon(\rho_0)/530 \le 2\rho_0^{5/2}\Upsilon(\rho_0)/53.
\]
Hence $\|\mc{P}_{T_1,{{\rho}}}(f)-\sum_{\chi\in S_{{\rho}}} g_\chi\|_2\le \rho_0^{5/2}\Upsilon(\rho_0)/26$, which proves our claim.

Using this last estimate we can prove \eqref{eq:algoregul2}, indeed $\|P_{{\rho}}(f)-\sum_{\chi\in S_{{\rho}}} g_\chi\|_2$ is at most
\[\|\textstyle\sum_{\chi\in S_{{\rho}}}g_\chi-\mc{P}_{T_1,{{\rho}}}(f)\|_2 + \|\mc{P}_{T_1,{{\rho}}}(f)-\mc{P}_{T_2,{{\rho}}}(f)\|_2 \le \rho_0^{5/2}\Upsilon(\rho_0)/26+\Upsilon(\rho_0)\le \mc{H}(\rho_0).
\]

Now, using the above estimates, the fact that the $L^2$-norm dominates the $U^3$-norm, and the bound \eqref{eq:regul}, we conclude that
\begin{align*}
\|f-P_{{\rho}}(f)\|_{U^3}  &\leq  \|f-\textstyle\sum_{\chi\in S_{{\rho}}}g_\chi\|_{U^3}+ \|P_{{\rho}}(f)-\textstyle\sum_{\chi\in S_{{\rho}}} g_\chi\|_2\\
& \leq \left[{\rho}^{3/8} (1+c_0)^{1/8}+2|S|\mc{D}(\eta,m)^{1/4^3}+4\eta^{1/2}\right]+\rho_0^{5/2}\Upsilon(\rho_0)/26+\Upsilon(\rho_0).
\end{align*}
Now using that $1+c_0\leq 1+ (|S|^2+4)\mc{D}(\eta,m)+8\eta\le 1+9\eta\le 10$ (by \eqref{it:reg3}), and that $2|S|\mc{D}(\eta,m)^{1/4^3}\leq \eta^{1/2}$ by \eqref{it:reg6}, we deduce that the last line above is at most $10^{1/8}\rho^{3/8}+5\eta^{1/2}+\Upsilon(\rho_0)\left(\rho_0^{5/2}/26+1 \right)\leq 10^{1/8}\rho^{3/8}+\Upsilon(\rho_0)\left( \rho_0^4/600+\rho_0^{5/2}/26+1 \right)$, and this is at most $ 2\rho^{3/8}$ by our choice of $\Upsilon$.

To prove the final statement of the theorem, we apply Proposition \ref{prop:regul} for $\sigma:=\mc{H}(\rho_0)$. Thus, we obtain some $h:\ab\to \mb{C}$ such that $\|h-\sum_{\chi\in S_{\rho}}g_\chi\|_{\infty}\le \mc{H}(\rho_0)$ and $\|h\|_{U^3}^* = O_{\rho_0,\mc{B},\mc{H}}(1)$. 
\end{proof}

We can now complete the proof of the validity of our spectral $U^3$-regularization algorithm.

\begin{proof}[Proof of Theorem \ref{thm:reg-intro}]
Note that this theorem is just a version of Theorem \ref{thm:algoregul} where we fix an arbitrary function $\mc{B}$ and we let $\mc{H}(\rho_0)=\rho_0/2$. 
\end{proof}

\subsection{Obtaining dominant nilspace characters of $f$ from the spectrum of $\mc{K}_\varepsilon(f\otimes \overline{f})$}\label{subsec:specsep}\hfill\medskip\\
Now that we have a method for recovering algorithmically the structured part $\sum_{\chi\in S_\rho}g_\chi$ using the eigenvectors of $\mc{K}_\varepsilon(f\otimes \overline{f})$, it is natural to try to go further and recover as much as possible the individual nilspace characters $g_\chi$, $\chi\in S_\rho$. Lemma \ref{lem:clusterapprox} guarantees that every such function $g_\chi$ is (up to an error with small $L^2$-norm) a linear combination of eigenvectors of $\mc{K}_\varepsilon(f\otimes \overline{f})$ that have corresponding eigenvalues close to the number $\|g_\chi\|_2^2$. The main problem that we may face is that there could be a cluster of many such eigenvalues, with corresponding eigenvectors thus spanning a multidimensional space, in which case we would have to find how to recover $g_\chi$ algorithmically inside this space. 

A simple observation is that if the dominant eigenvalues of $\mc{K}_\varepsilon(f\otimes \overline{f})$ are sufficiently isolated, then the above clustering cannot happen and we are therefore able to recover $g_\chi$ as a scalar multiple of precisely one of the eigenvectors (up to a small $L^2$-error). In fact, this holds more generally in a sense that will be very useful for us. Namely, for any unit vector $h$ in the linear span of the leading eigenvectors of $\mc{K}_\varepsilon(f\otimes\overline{f})$, if the leading eigenvalues of the \emph{new} matrix $\mc{K}_\varepsilon(h\otimes\overline{h})$ are well-separated, then there is a bijection from the set of leading eigenvectors of the latter matrix to the set of leading nilspace character $g_\chi$ of $f$, such that $g_\chi$ is close in $L^2$ to a scalar multiple of its corresponding eigenvector under this bijection. This is formalized in the following main theorem of this subsection.

We say that a multiset $X$ of complex numbers is \emph{$\delta$-separated} if $|x-y|>\delta$ for every $x\neq y$ in $X$. Note that, in particular, this implies that every element of $X$ has multiplicity 1.

\begin{theorem}\label{thm:HiSpecBiject}
For any $q>0$ and $r>q+1$, there exists  $B_{q,r}\in(0,1)$ such that the following holds. Let $\mc{B}$ be a function $\mb{R}_{>0}\times\mb{N}\to\mb{R}_{>0}$, let $\mc{H}:\mb{R}_{>0}\to\mb{R}_{>0}$, $x\mapsto (x/2)^r$, and let $\rho_0\in (0,B_{q,r}]$. Let $M(\rho_0,\mc{B},\mc{H}),\varepsilon_0(\rho_0,\mc{B},\mc{H})$ be the resulting numbers given by Theorem \ref{thm:algoregul}. Let $\ab$ be a finite abelian group, let $f:\ab\to \mb{C}$ be 1-bounded, and, for an appropriate $\rho\in[\rho_0/2,\rho_0]$, let $f=\sum_{\chi\in S_\rho}g_\chi+f_r+f_e$ be the corresponding decomposition given by Theorem \ref{thm:algoregul}, where $\|f_r\|_{U^3}\le 2\rho^{3/8}$ and $\|f_e\|_2\le (\rho_0/2)^r$. Let $h$ be a function in $\textup{Eigen}_\rho\big(\mc{K}_\varepsilon(f\otimes \overline{f})\big)\subset \mb{C}^{\ab}$, with $\|h\|_2\leq 1$, such that $A:=\mc{K}_\varepsilon(h\otimes \overline{h})$ has spectrum $\textup{Spec}_{\rho^q}(A)$ being $\delta$-separated for some $\delta>\rho^{q}$, and $|\textup{Spec}_{\rho^q}(A)|=|S_\rho|$. Then there is a bijection $J:\textup{Spec}_{\rho^q}(A)\to S_\rho$ such that, for every $\lambda\in\textup{Spec}_{\rho^q}(A)$, letting $\chi=J(\lambda)$ and $v$ be an eigenvector of $A$ \textup{(}with $\|v\|_2=1$\textup{)} corresponding to $\lambda$, we have $\|\langle f,v\rangle v-g_\chi\|_2\le 28\rho^{r-q}$ and $|\langle f,v\rangle|\ge \sqrt{\rho/2}$. 
\end{theorem}
\noindent We can let $B_{q,r}:=\min\big(1, (1/14)^{1/(r-1)}, (1/14)^{1/(r-q)}, (1/56)^{1/(r-q)},(1/452)^{1/(r-q-1)} \big)$. Let us also mention that the specific choice of function $(x/2)^r$ in the theorem is sufficient for our purposes in what follows, but that the result could be extended to a more general function $\mc{H}$.

\medskip 

\begin{remark}\label{rem:motivhthm}
The structured part $P_\rho(f)$ of $f$ obtained in Theorem \ref{thm:algoregul} has $L^2$-norm at most 1 (being an orthogonal projection of the 1-bounded function $f$), and it can then be seen that, setting $h=P_\rho(f)$, the spectra $\textup{Spec}_\rho(\mc{K}_\varepsilon(h\otimes \overline{h}))$ and $\textup{Spec}_\rho(\mc{K}_\varepsilon(f\otimes \overline{f}))$ are roughly equal (this is proved in Lemma \ref{lem:proxispec} below). Therefore, using this choice of $h$, Theorem \ref{thm:HiSpecBiject} implies that, if $\textup{Spec}_\rho\big(\mc{K}_\varepsilon(f\otimes \overline{f})\big)$ is sufficiently separated, then every nilspace character $g_\chi$ in the decomposition of $f$ with $\chi\in S_\rho$ can be recovered directly (up to a small $L^2$-error) as an eigenvector in $\textup{Eigen}_\rho\big(\mc{K}_\varepsilon(h\otimes \overline{h})\big)$. The purpose of the greater generality of Theorem \ref{thm:HiSpecBiject} (enabling us to consider more general functions $h$ in the unit ball of $\textup{Eigen}_\rho\big(\mc{K}_\varepsilon(f\otimes \overline{f})\big)$) is to give us more options for recovering the individual quadratic characters
of $f$, even when the eigenvalues of $\mc{K}_\varepsilon(f\otimes \overline{f})$ are \emph{not} sufficiently separated (a situation that can occur even in the classical Fourier-analytic setting, as we mentioned in Remark \ref{rem:similar-eigen-fourier}). Indeed, the subspace $\textup{Eigen}_\rho\big(\mc{K}_\varepsilon(f\otimes \overline{f})\big)$ has a dimension much smaller than $|\ab|$, and the theorem tells us that it suffices to find a vector (or function) $h$ in this subspace such that the dominant eigenvalues of $\mc{K}_\varepsilon(h\otimes \overline{h})$ are sufficiently separated. In the last part of this subsection, we will see that it is useful to consider a \emph{random} vector $h$ in the unit ball of this subspace, because we can then guarantee that, with high probability, the spectrum $\textup{Spec}_{\rho^q}\big(\mc{K}_\varepsilon(h\otimes \overline{h})\big)$ for such a random $h$ is adequately separated.
\end{remark}

\noindent The proof of Theorem \ref{thm:HiSpecBiject} relies on the following result. This tells us essentially that if $h$ is \emph{any} unit vector that is close to a linear combination of a few quasiorthogonal quadratic characters of unit norm, then the leading eigenvalues of $\mc{K}_\varepsilon(h\otimes \overline{h})$ are close to the squared moduli of the coefficients in the linear combination.
\begin{proposition}\label{prop:evalcorresp}
Let $\varepsilon\in (0,1]$, let $\ab$ be a finite abelian group, let $D$ be a finite set, and let $g_\chi$, $\chi\in D$ be  functions in $\mb{C}^{\ab}$ satisfying $\|g_\chi\|_2\in [\tau_1,1]$, $\|\mc{K}_{\varepsilon^{1/2}}(g_\chi\otimes \overline{g_\chi})-g_\chi\otimes \overline{g_\chi}\|_2\leq \tau_2$, and $\max(|\langle g_\chi,g_{\chi'}\rangle|,\langle g_\chi,g_{\chi'}\rangle_{U^3})\leq\tau_3\leq \tau_1^2/|D|$ for every $\chi'\in D\setminus\{\chi\}$. Let $u_\chi=g_\chi/\|g_\chi\|_2$ for each $\chi\in D$, and suppose that $h\in \mb{C}^{\ab}$ is a function with $\|h\|_2\leq 1$ such that, for some function $h'=\sum_{\chi\in D} c_\chi u_\chi$ with $|c_\chi|\leq 1$, we have $\|h-h'\|_2\leq \tau_4$. Then
\begin{equation}\label{eq:evalcorresp1}
\mc{K}_\varepsilon(h\otimes \overline{h}) = \textstyle\sum_{\chi\in D} |c_\chi|^2 u_\chi \otimes \overline{u_\chi}+E,
\end{equation}
where $\|E\|_2 \le 5\tau_4 + \varepsilon^{-1/2}|D|^2(|D|^{1/2}+3)\tau_3+|D|\max\big(\tau_2/\tau_1^2,2\varepsilon^{1/2})$. In particular, letting $\lambda_1\geq \lambda_2 \geq \cdots \geq \lambda_{|\ab|}$ be the eigenvalues of the matrix $\mc{K}_\varepsilon(h\otimes\overline{h})$, and $\chi_1,\chi_2,\ldots,\chi_{|D|}$ be an ordering such that $c_{\chi_i}\geq c_{\chi_{i+1}}$ for all $i$, and setting $c_{\chi_i}=0$ for $i>|D|$, we have
\begin{equation}\label{eq:evalcorresp2}
\big(\textstyle\sum_{i=1}^{|\ab|} \big|\, |c_{\chi_i}|^2-\lambda_i\big|^2\big)^{\frac{1}{2}}\leq 2|D|C_{|D|}^2\tau_3/\tau_1^2 + \|E\|_2.
\end{equation}
\end{proposition}
To prove Proposition \ref{prop:evalcorresp}, we will combine the refined Gram-Schmidt process (Lemma \ref{lem:GSapp}) with Theorem \ref{thm:H-W}.

\begin{proof}[Proof of Proposition \ref{prop:evalcorresp}]
We first apply Proposition \ref{prop:reglemtool} with the following elements and parameters $\alpha_i$: we take $f=h=\sum_{\chi\in D} c_\chi u_\chi+f_e$, where $f_e:=h-h'$ (so $\alpha_1=\tau_4$), $f_r=0$ (so $\alpha_2=0$), the $f_i$ are the functions $c_\chi u_\chi$, so $\alpha_3=1$ and $\alpha_4=\tau_3$. Finally, we have $\|\mc{K}_{\varepsilon}\big( (c_\chi u_\chi) \otimes \overline{(c_\chi u_\chi)})- (c_\chi u_\chi) \otimes \overline{(c_\chi u_\chi)} \|_2\leq \alpha_5$ where we shall now estimate $\alpha_5$ using Lemma \ref{lem:KepsStructHom}. Indeed, since $\|\mc{K}_{\varepsilon^{1/2}}(g_\chi\otimes \overline{g_\chi})-g_\chi\otimes \overline{g_\chi}\|_2\leq \tau_2$, we deduce by Lemma \ref{lem:KepsStructHom} applied with $c=c_\chi/\|g_\chi\|_2$ (and using that $\|g_\chi\|_2\le1$) that
\[
\|\mc{K}_{\varepsilon}\big( (c_\chi u_\chi) \otimes \overline{(c_\chi u_\chi)})- (c_\chi u_\chi) \otimes \overline{(c_\chi u_\chi)} \|_2\leq \max\big(|c_\chi|^2 \tau_2/\|g_\chi\|_2^2,2\varepsilon^{1/2}) \leq \max\big(\tau_2/\tau_1^2,2\varepsilon^{1/2}),
\]
so we can let $\alpha_5=\max\big(\tau_2/\tau_1^2,2\varepsilon^{1/2})$.

Thus we obtain
\[
\mc{K}_\varepsilon(h\otimes \overline{h}) = \textstyle\sum_{\chi\in D} |c_\chi|^2 u_\chi \otimes \overline{u_\chi} + E,
\]
where $\|E\|_2\leq 5\tau_4 + \varepsilon^{-1/2}|D|^2(|D|^{1/2}+3)\tau_3+|D|\max\big(\tau_2/\tau_1^2,2\varepsilon^{1/2})$.

Now, we apply Theorem \ref{thm:H-W} to deduce that, letting $\lambda_i'$ be the eigenvalues of the self-adjoint matrix $M_u:=\sum_{\chi\in D} |c_\chi|^2 u_\chi \otimes \overline{u_\chi}$ in descending order, and letting $n=|\ab|$, we have
\[
\textstyle\sum_{i\in [n]} |\lambda_i'-\lambda_i|^2\leq \big\|\mc{K}_\varepsilon(h\otimes \overline{h}) - \textstyle\sum_{\chi\in D} |c_\chi|^2 u_\chi \otimes \overline{u_\chi}\big\|_2^2=\|E\|_2^2.
\]
We can thus focus on approximating the eigenvalues $\lambda_i'$. For this we  apply Lemma \ref{lem:GSapp} to the quasi-orthogonal set $\{u_\chi:\chi\in D\}$ with $\tau=C_{|D|}\tau_3/\tau_1^2\leq 1$ (note that, for $\chi\not=\chi'$, $|\langle u_\chi,u_{\chi'}\rangle| = |\langle \frac{g_\chi}{\|g_\chi\|_2},\frac{g_{\chi'}}{\|g_{\chi'}\|_2}\rangle|\le \tau_3/\tau_1^2$),  obtaining an orthonormal set $\{w_\chi:\chi\in D\}$ with $\|u_\chi-w_\chi\|_2\leq C_{|D|}^2\tau_3/\tau_1^2$ for all $\chi$.

It follows that the matrices 
$M_u$ and $M_w:=\sum_{\chi\in D} |c_\chi|^2 w_\chi \otimes \overline{w_\chi}$ satisfy
\[
\|M_u-M_w\|_2 \leq \textstyle\sum_{\chi\in D} |c_\chi|^2\; \|(u_\chi - w_\chi) \otimes \overline{u_\chi}-w_\chi \otimes (\overline{w_\chi}-\overline{u_\chi})\|_2 \leq 2|D|C_{|D|}^2\tau_3/\tau_1^2.
\]
Moreover, the eigenvalues of the matrix $M_w$ are clearly determined: since the $w_\chi$ are orthonormal, it follows that for each $\chi\in D$ we have $M_w w_\chi = |c_\chi|^2 w_\chi$, so each $w_\chi$ is an eigenvector with eigenvalue $|c_\chi|^2$. Applying again Theorem \ref{thm:H-W} to $M_u$, $M_w$, we deduce that $\big(\sum_{i=1}^{|R|} |\lambda_i'-|c_{\chi_i}|^2|^2\big)^{1/2} \leq 2|D|C_{|D|}^2\tau_3/\tau_1^2$. Hence for each $i$ we have $|\lambda_i - |c_{\chi_i}|^2|\leq 2|D|C_{|D|}^2\tau_3/\tau_1^2 + \|E\|_2$ and the result follows.
\end{proof} 
\noindent We can now prove the main result of this section.

\begin{proof}[Proof of Theorem \ref{thm:HiSpecBiject}]
Recall that our assumptions include an application of Theorem \ref{thm:algoregul}. In particular, we have applied Theorem \ref{thm:main-algo-pf} with some particular choices of $\eta$ and $\mc{D}$. For simplicity, instead of applying again Theorem \ref{thm:main-algo-pf} with a new list of constants, we are going to assume that we have repeated the argument of Theorem \ref{thm:algoregul} with two additional constraints specified below, \eqref{it:reg-8} and \eqref{it:reg-9}, on $\varepsilon$ and $\mc{D}$. Recall also that $\mc{H}(\rho_0)=(\rho_0/2)^r$, so in particular $\Upsilon(\rho_0)\leq (\rho_0/2)^r\leq \rho^r$, and thus $\eta=\rho_0^8\Upsilon(\rho_0)/(9\cdot 10^6)\leq \rho^{8+2r}/5^6$.

Let $T_1$ and $T_2$ be defined as in \eqref{eq:T1T2}:
\begin{align}\label{eq:T1T2-7.2}
T_1:=\textstyle\sum_{\chi\in S}g'_\chi\otimes\overline{g'_\chi}\quad\text{and}\quad T_2:=\mc{K}_\varepsilon(f\otimes\overline{f})=\textstyle\sum_{\chi\in S}g_\chi\otimes\overline{g_\chi}+E.
\end{align}
Then, arguing as in the proof of Theorem \ref{thm:algoregul}, we have $d(\mc{P}_{T_1,\rho},\mc{P}_{T_2,\rho})\leq \Upsilon(\rho_0)\leq \rho^r$. Since $\|h\|_2\leq 1$ and $h=\mc{P}_{T_2,\rho}(h)$, it follows that $\|\mc{P}_{T_1,\rho}(h)-h\|_2\leq \rho^r$. Moreover $\mc{P}_{T_1,\rho}(h)=\sum_{\chi\in S_\rho} \langle h,w_\chi\rangle w_\chi$, where $w_\chi=g_\chi'/\|g_\chi'\|_2$ (indeed the orthonormality of the $w_\chi$ implies that the eigenvalues of $T_1$ are precisely the numbers $\|g'_\chi\|_2^2=\|g_\chi\|_2^2$, $\chi\in S$, so $\textup{Eigen}_\rho(T_1)$ is the linear span of $\{w_\chi:\chi\in S_\rho\}$). By Lemma \ref{lem:GSapp}, we have $\|u_\chi-w_\chi\|_2\leq C_s\tau $ where $\tau$ is defined as per \eqref{it:reg1} for $\chi\in S_\rho$ (where $s=|S|$). Combining this with condition \eqref{it:reg2} from the proof of Theorem \ref{thm:algoregul}, and with the fact that $S_\rho\subset S$, we conclude that $\big\|\mc{P}_{T_1,\rho}(h)-\sum_{\chi\in S_\rho} \langle h,w_\chi\rangle u_\chi\big\|_2\leq C_s\tau \,|S_\rho| \le \nsR(\eta,m)C_s\tau\le \eta^{1/2}/2\le \rho^{4+r}/250$. Thus, letting $c_\chi=\langle h,w_\chi\rangle$, we have 
\begin{align}\label{eq:h-h'}
\big\|h-\textstyle\sum_{\chi\in S_\rho}c_\chi u_\chi\big\|_2  \leq 2\rho^r.
\end{align}
We now want to apply Proposition \ref{prop:evalcorresp} to $h$, with $D=S_\rho$ and $h'=\sum_{\chi\in S_\rho}c_\chi u_\chi$. By definition of $S_\rho$ (recall Proposition \ref{prop:regul}), we can let $\tau_1=(\rho_0/2)^{1/2}$. To estimate $\tau_2$, we need to bound $\|\mc{K}_{\varepsilon^{1/2}}(g_\chi\otimes\overline{g_\chi})-(g_\chi\otimes\overline{g_\chi})\|_2 $. Note that $g_\chi=F_\chi\co\phi$ where $F_\chi:\ns\to \mb{C}$ is a Lipschitz function with $\|F_\chi\|_{\textup{sum}}\le 1+m$. By Theorem \ref{thm:nscharsquadchars}, for any $\sigma\in(0,1/2)$ we have that $g_\chi$ is a quadratic character with parameters $(O_{m}(\sigma^{-O_m(1)}),\sigma)$. By Proposition \ref{prop:weak-qua-char-stable-under-operator}, we have $\|\mc{K}_{\varepsilon^{1/2}}\big( (g_\chi) \otimes \overline{(g_\chi)})- (g_\chi) \otimes \overline{(g_\chi)}\|_2\leq 4\sigma+ O_{m}(\varepsilon^{1/8}\sigma^{-O_{m}(1)})$. Setting $\sigma=\eta/4$, we have $\|\mc{K}_{\varepsilon^{1/2}}\big( (g_\chi) \otimes \overline{(g_\chi)})- (g_\chi) \otimes \overline{(g_\chi)}\|_2\leq \eta+O_{m,\eta}(\varepsilon^{1/8})=:\tau_2$. Finally, recall that by Theorem \ref{thm:UpgradedReg} we can let $\tau_3=\mc{D}(\eta,m)$, and since $|S_\rho|\le |S|\le \nsR(\eta,m)$, we have that $\tau_3 |D|\le \tau_3 \nsR(\eta,m) \le \mc{D}(\eta,m)\nsR(\eta,m) $, which is at most $\tau_1^2=\rho_0/2$ by assumption \eqref{it:reg1} from the proof of Theorem \ref{thm:algoregul} and the above bound on $\eta$. Finally, by \eqref{eq:h-h'} we can let $\tau_4=2\rho^r$. Applying now Proposition \ref{prop:evalcorresp} with these parameters $\tau_i$, we obtain \eqref{eq:evalcorresp1} with 
\begin{align}\label{eq:E-from-7-16}
\|E\|_2 & \le 5\tau_4 + \varepsilon^{-1/2}|D|^2(|D|^{1/2}+3)\tau_3+|D|\max\big(\tau_2/\tau_1^2,2\varepsilon^{1/2}) \nonumber\\
& \le 10\rho^r+\varepsilon^{-1/2}|D|^2(|D|^{1/2}+3)\mc{D}(\eta,m)+|D|\max\big( 2(\eta+O_{m,\eta}(\varepsilon^{1/8}))/\rho_0,2\varepsilon^{1/2} \big) \nonumber \\
&\le 11\rho^r+O_{m,\eta}(\varepsilon^{-1/2}\mc{D}(\eta,m))+O_{m,\eta}(\varepsilon^{1/8}),
\end{align}
where for this last inequality we used  that $2 |D|\eta \leq (20/\rho) \eta \leq \rho^r$.

To simplify the upper bound in \eqref{eq:E-from-7-16} further, we impose additional constraints on both $\varepsilon=\varepsilon_{\eta,m}$ and $\mc{D}$. These assumptions are to be thought of as added to the  list at the beginning of the proof Theorem \ref{thm:algoregul}, as the following eighth condition:
\setlength{\leftmargini}{0.9cm}
\begin{enumerate}
  \setcounter{enumi}{7}
  \item\label{it:reg-8} We let $\varepsilon=\varepsilon_{\eta,m}>0$ be such that the term $O_{m,\eta}(\varepsilon^{1/8})$ in \eqref{eq:E-from-7-16} is at most $(\rho_0/2)^r$. Furthermore, we then let $\mc{D}$ be such that the term $O_{m,\eta}(\varepsilon^{-1/2}\mc{D}(\eta,m))$ in \eqref{eq:E-from-7-16} is at most $ (\rho_0/2)^r$. 
\end{enumerate}
Hence $\|E\|_2\le 13\rho^r$. Let $\{\lambda_1\ge \lambda_2\ge \cdots\}$ be the eigenvalues of $\mc{K}_\varepsilon(h\otimes\overline{h})$ ordered decreasingly, let the elements of $S_\rho$ be ordered $\chi_1,\chi_2,\ldots$ so that $|c_{\chi_1}|^2\ge |c_{\chi_2}|^2\ge \cdots$, and also let $c_{\chi_i}:=0$ for $|S_\rho|<i\le |\ab|$. We then have \eqref{eq:evalcorresp2}, and thus, for all $i\in[|\ab|]$, it follows that $||c_{\chi_i}|^2-\lambda_i|\leq \gamma:=2|S_\rho| C_{|S_\rho|}^2\tau_3/\tau_1^2+\|E\|_2$. Similarly, we would like to have a simpler bound for this quantity, so we assume the following additional condition:
\begin{enumerate}
  \setcounter{enumi}{8}
  \item\label{it:reg-9} The function $\mc{D}$ is chosen so that $2\nsR(\eta,m)C_{\nsR(\eta,m)}^2\mc{D}(\eta,m) (\rho_0/2)^{-1}\le (\rho_0/2)^r$.
\end{enumerate}
\noindent Thus we obtain that
\begin{equation}\label{eq:proximity-eigen}
\text{$\forall i\in [|S_\rho|]$, we have } ||c_{\chi_i}|^2-\lambda_i|\leq 14\rho^r \text{ and $\forall i\in[|S_\rho|+1,|\ab|]$, we have }  |\lambda_i|\le 14\rho^r.
\end{equation}
Our aim now is to apply Lemma \ref{lem:clusterapprox} to $A:=\mc{K}_\varepsilon(h\otimes\overline{h})$. Let $\{v_i:i\in |\ab|\}$ be an orthonormal basis of eigenvectors of $A$, where $v_i$ has corresponding eigenvalue $\lambda_i$.

We claim that each $u_\chi$ with $\chi\in S_\rho$ is a $\beta$-pseudoeigenvector of $A$ with pseudoeigenvalue $|c_\chi|^2$, for $\beta=14\rho^r$. Indeed \eqref{eq:evalcorresp1} implies $\mc{K}_\varepsilon(h\otimes\overline{h})u_\chi = \sum_{\chi'\in S_\rho} |c_{\chi'}|^2 u_{\chi'}\langle u_\chi,u_{\chi'}\rangle + Eu_\chi $, so
\begin{align*}
\|\mc{K}_\varepsilon(h\otimes\overline{h})u_\chi -  |c_{\chi}|^2 u_\chi\|_2 &\leq \textstyle\sum_{\chi'\in S_\rho\setminus\{\chi\}} |c_{\chi'}|^2 |\langle u_\chi,u_{\chi'}\rangle| + \|E\|_2 \\ &\leq \tau C_s \nsR(\eta,m) + \|E\|_2 \le 14\rho^r,
\end{align*}
the second inequality holding by \eqref{eq:bnd-normalized-u-chi} and the third by \eqref{it:reg2}, using $\eta\le \rho^r$ and our bound on $\|E\|_2$.

Fix any $\lambda\in \textup{Spec}_{\rho^q}(A)$ and let $v$ be a corresponding unit eigenvector. Since $\textup{Spec}_{\rho^q}(A)$ is $\delta$-separated and $\lambda\ge \rho^q>14\rho^r$, we have $||c_\chi|^2-\lambda|\le 14\rho^r$ for some $\chi\in S_\rho$. Let $C_{\rho^{q}/2}(|c_\chi|^2):=\{i\in[|\ab|]:|\lambda_i-|c_\chi|^2|\le \rho^{q}/2\}$. We claim that $C_{\rho^{q}/2}(|c_\chi|^2) = \{\lambda\}$. First, note that if $i>|S_\rho|$, the second part of \eqref{eq:proximity-eigen} tells us that $\lambda_i\le 14\rho^r$. Since $\rho^q-\rho^q/2-14\rho^r>14\rho^r$, it follows that for $i>|S_\rho|$ we have $|\lambda_i-|c_\chi|^2|>\rho^q/2$. Now if $i\le |S_\rho|$, then, since by assumption $|\textup{Spec}_{\rho^q}(A)|=|S_\rho|$, it follows that $\lambda_i\ge \rho^q$ for all $i\le|S_\rho|$. Let $\lambda_i\not=\lambda$ for $i\ge |S_\rho|$. We want to prove that $|\lambda_i-|c_\chi|^2|>\rho^q$. As $\textup{Spec}_{\rho^q}(A)$ is $\delta$-separated, we have $|\lambda-\lambda_i|\ge \delta\ge \rho^q$. Thus, $|\lambda_i-|c_\chi|^2|\ge |\lambda_i-\lambda|-|\lambda-|c_\chi|^2|\ge \rho^q-14\rho^r>\rho^q/2$. Therefore, we can define a map $J:\textup{Spec}_{\rho^q}(A)\to S_\rho$ that associates $\lambda$ with the unique $\chi\in S_\rho$ such that $|\lambda-|c_\chi|^2|\le \rho^q/2$. As $|\textup{Spec}_{\rho^q}(A)|=|S_\rho|$, $J$ is a bijection. By Lemma \ref{lem:clusterapprox}, we have $\|u_\chi-\langle u_\chi,v\rangle v\|_2\le 28\rho^{r- q}$.

In particular $|1-|\langle u_\chi,v\rangle|\;|\le 28\rho^{r- q}\le 1/2$. Hence $\|u_\chi/\langle u_\chi,v\rangle-v\|_2\le \frac{1}{\langle u_\chi,v\rangle}\beta/\delta\le 56\rho^{r- q}$. Note that $\big|\frac{1}{\langle u_\chi,v\rangle}-\frac{\overline{\langle u_\chi,v\rangle}}{|\langle u_\chi,v\rangle|}\big| = \big|\frac{1}{|\langle u_\chi,v\rangle|}-1\big|=\big|\frac{|\langle u_\chi,v\rangle|-1}{\langle u_\chi,v\rangle}\big|\le 56\rho^{r- q}$ and thus, letting $z:=\frac{\overline{\langle u_\chi,v\rangle}}{|\langle u_\chi,v\rangle|}$ we have that $\left\|v-\frac{zg_\chi}{\|g_\chi\|_2}\right\|_2\le 112\rho^{r- q}$.

With this, we can now recover the component $g_\chi$ of $f$. Indeed, note that 
\begin{align*}
    \textstyle\|g_\chi-\langle f,v\rangle v\|_2\le \|g_\chi-\langle f,\frac{zg_\chi}{\|g_\chi\|_2}\rangle \frac{zg_\chi}{\|g_\chi\|_2}\|_2+\|\langle f,v-\frac{zg_\chi}{\|g_\chi\|_2}\rangle v\|_2+\|\langle f,\frac{zg_\chi}{\|g_\chi\|_2}\rangle (v-\frac{zg_\chi}{\|g_\chi\|_2})\|_2.
\end{align*}
Thus $\|g_\chi-\langle f,v\rangle v\|_2\le \|g_\chi-\langle f,\frac{zg_\chi}{\|g_\chi\|_2}\rangle \frac{zg_\chi}{\|g_\chi\|_2}\|_2+224\rho^{r- q}$. Moreover $\|g_\chi-\langle f,\frac{zg_\chi}{\|g_\chi\|_2}\rangle \frac{zg_\chi}{\|g_\chi\|_2}\|_2 = \frac{1}{\|g_\chi\|_2}|\|g_\chi\|_2^2-\langle f,g_\chi\rangle g_\chi|$. As $\chi\in S_\rho$, we have $\|g_\chi\|_2\ge \sqrt{\rho}$, and by Lemma \ref{lem:correlations} combined with \eqref{it:reg5} and $4\sqrt{\eta}/\sqrt{\rho}\le \rho^r$, we have that $\|g_\chi-\langle f,v\rangle v\|_2\le \rho^r+224\rho^{r- q}\le 225\rho^{r-q}$.

Finally, to see that $v$ indeed correlates with $f$, note that by Lemma \ref{lem:correlations} combined with \eqref{it:reg5}, we have $|\langle f,g_\chi\rangle-\|g_\chi\|_2^2|\le\rho^r$ (similarly as before). Hence $|\langle f, \langle f,v\rangle v\rangle|\ge |\langle f,g_\chi\rangle|-\|g_\chi-\langle f,v\rangle v\|_2\ge \rho-226\rho^{r-q}\ge \rho/2$ where we used that $\|g_\chi\|_2^2\ge \rho$. Therefore $|\langle f,v\rangle|\ge \sqrt{\rho/2}$.
\end{proof}
\noindent The second consequence of Proposition \ref{prop:evalcorresp} is the following confirmation of the fact mentioned at the beginning of Remark \ref{rem:motivhthm}.
\begin{lemma}\label{lem:proxispec}
Under the same assumptions as Theorems \ref{thm:algoregul} and \ref{thm:HiSpecBiject}, assume further that $r>3$ and $\rho<(1/450)^{1/(r-3)}$. Let $h=P_\rho(f)$. Then there is a bijection $B:\textup{Spec}_{\rho}(\mc{K}_\varepsilon(h\otimes \overline{h}))\to \textup{Spec}_\rho(\mc{K}_\varepsilon(f\otimes \overline{f}))$ which preserves multiplicity of eigenvalues and satisfies $|B(\lambda)-\lambda|\leq 15\rho^r$. 
\end{lemma}
\begin{proof} Recall from \eqref{eq:T1T2} the definition of the $\ab$-matrices $T_1$ and $T_2$ and recall that $\|T_1-T_2\|_2\le 24\sqrt{\eta}$. By definition, the $|\ab|$ eigenvalues of $T_1$ are precisely $\|g_{\chi_1}\|_2^2\ge \|g_{\chi_2}\|_2^2\ge \cdots$ (where we complete the list with zeroes if necessary), and the $|\ab|$ eigenvalues of $T_2$ are $\lambda_1\ge \lambda_2\ge \cdots$. By Theorem \ref{thm:H-W} we have $\sum_{i=1}^{|S_\rho|} |\|g_{\chi_i}\|_2^2-\lambda_i|^2\leq \sum_{i=1}^{|\ab|} |\|g_{\chi_i}\|_2^2-\lambda_i|^2\le 24^2 \eta\le \rho^{2r}$.

Now we apply Proposition \ref{prop:evalcorresp} with $h=P_\rho(f)$ and $h'=\sum_{\chi\in S_\rho}g_\chi$, using that $\|P_\rho(f)\|_2\leq 1$ and that, by \eqref{eq:algoregul2}, we have $\|h-h'\|_2=\rho^r$. Let $A:=\mc{K}_\varepsilon(h\otimes \overline{h})$ and let $\textup{Spec}(A)=\{\mu_1\ge \mu_2\ge \cdots\}$. Reusing the bound \eqref{eq:proximity-eigen} from Theorem \ref{thm:HiSpecBiject} we have  $\big(\sum_{i=1}^{|\ab|} ||c_{\chi_i}|^2-\mu_i|^2\big)^{1/2}\le 14\rho^r$ where crucially here $c_\chi=\|g_\chi\|_2$ if $\chi\in S_\rho$ and $c_\chi=0$ otherwise. In particular, $\big(\sum_{i=1}^{|S_\rho|} |\|g_{\chi_i}\|_2^2-\mu_i|^2\big)^{1/2}\le 14\rho^r$. By the triangle inequality, we deduce that $\big(\sum_{i=1}^{|S_\rho|} |\mu_i-\lambda_i|^2\big)^{1/2}\le 15\rho^r$. 

To conclude that this gives the desired bijection, we want to ensure that all the $\mu_i$ are larger than $\rho$. By Remark \ref{rem:separation-rho-from-eigen}, we have $\lambda_{|S_\rho|}\ge \rho+\rho_0^3/30$. By our assumption on $r$ and $\rho$, we know that $\rho_0^3/30\ge \rho^3/30>15\rho^r$. Therefore $\mu_{|S_\rho|}\ge \rho$ and the result follows.\end{proof}

\begin{remark}\label{rem:additional-rem-main-2} It follows from Lemma \ref{lem:pseudoevals} combined with Lemma \ref{lem:clusterapprox} that if an eigenvalue of $\mc{K}_\varepsilon(f\otimes\overline{f})$ is large and sufficiently separated from the rest, then the projection of $f$ to the corresponding eigenspace is very close to one of the nilspace characters $g_\chi$. Moreover, Theorem \ref{thm:HiSpecBiject} could be generalized to be applicable in clusters of  eigenvalues. That is, if $C$ is a submultiset of $\textup{Spec}_{\rho}\big(\mc{K}_\varepsilon(f\otimes\overline{f})\big)$ such that $C$ is included in a ball of small radius separated from the rest of the eigenvalues, then the mentioned generalization of Theorem \ref{thm:HiSpecBiject} would involve a vector $h$ chosen in the eigenspace corresponding to $C$, instead of the whole space $\textup{Eigen}_\rho\big(\mc{K}_\varepsilon(f\otimes\overline{f})\big)$. Another possible extension of Theorem \ref{thm:HiSpecBiject} involves not assuming that $|\textup{Spec}_{\rho^q}(A)|=|S_\rho|$. In this case, we could prove an analogous result but where the map $J:\textup{Spec}_{\rho^q}(A)\to S_\rho$ is only an injection. As we shall not use these extensions in the paper, we omit their proofs.

Finally, note that it follows from  \eqref{eq:proximity-eigen} that, for any $h\in \textup{Eigen}_\rho\big(\mc{K}_\varepsilon(f\otimes \overline{f})\big)$ with $\|h\|_2\le 1$, we have $|\textup{Spec}_{\rho^q}\big(\mc{K}_\varepsilon(h\otimes \overline{h})\big)|\leq |S_\rho|$. Indeed, choosing $h$ in $\textup{Eigen}_\rho\big(\mc{K}_\varepsilon(f\otimes \overline{f})\big)$ generates a gap in the set of the eigenvalues of $\mc{K}_\varepsilon(h\otimes \overline{h})$, i.e., there are at most $|S_\rho|$ such eigenvalues which can be larger than $14\rho^r$ and all the others must be small. In particular, we can combine this fact with Lemma \ref{lem:proxispec} to deduce that for $h=P_\rho(f)$, we have $\textup{Spec}_{\rho^q}\big(\mc{K}_\varepsilon(h\otimes \overline{h})\big)=\textup{Spec}_{\rho}\big(\mc{K}_\varepsilon(h\otimes \overline{h})\big)$.
\end{remark}

\begin{proof}[Proof of Theorem \ref{thm:HiSpecBiject-intro}] This follows from Theorem \ref{thm:HiSpecBiject} applied with $q=7$ and $r=11$. Indeed, this application gives us that each unit eigenvector $v$ has a corresponding 2-step nilspace character $g_\chi$ such that $\|v-g_\chi/\langle f,v\rangle\|_2\leq 56 \rho^{7/2}$. The claimed quadratic character is then $g:=g_\chi/\langle f,v\rangle$. To obtain the complexity and precision bounds for $g$, note that, by Remark \ref{rem:qcharparams}, the function $g_\chi$ is an $(R,\sigma)$-character of order 2 with $R=O_{\ns,\|F\|_{\text{sum}}}(\sigma^{-O_{\ns}(1)})$. Letting $\sigma=c\rho^2$ for a small enough constant $c>0$, we have that $g$ is a $(O_\rho(1),\rho)$-quadratic character and $|\langle f,g\rangle|\ge |\langle f,v\rangle|-\|v-g\|_2\ge \sqrt{\rho/2}-56\rho^{7/2}\ge \sqrt{\rho/4}$.
\end{proof}

\subsubsection{Handling clustered eigenvalues}\label{subsubsec:cluster}\hfill\smallskip\\
Suppose that we are working under the assumptions of Theorem \ref{thm:HiSpecBiject}, and for the given 1-bounded function $f:\ab\to\mb{C}$ we want to recover the individual $g_\chi$, $\chi\in S_\rho$ in the decomposition of $f$ involved in that theorem. As explained earlier (in particular, combining Remark \ref{rem:motivhthm} and Lemma \ref{lem:proxispec}), if the large spectrum of $\mc{K}_\varepsilon(f\otimes\overline{f})$ is sufficiently separated, then the desired recovery is straightforward using this spectrum, by Theorem \ref{thm:HiSpecBiject}. Hence, the last difficulty that we face for this recovery is the possibility that, for $h=P_\rho(f)$, the separation assumptions for the spectrum in Theorem \ref{thm:HiSpecBiject} do not hold. Our goal now is to show that, in fact, if we select another function $h$ \emph{randomly} in $\textup{Eigen}_\rho\big(\mc{K}_\varepsilon(f\otimes \overline{f})\big)$ according to a simple normal distribution (specified below), then with high probability the resulting spectrum $\textup{Spec}_{\rho^q}\big(\mc{K}_\varepsilon(h\otimes \overline{h})\big)$ will satisfy the conditions for Theorem \ref{thm:HiSpecBiject} to be applicable, and thus enable the recovery of the $g_\chi$.

Let us now detail the probabilistic argument. Let $\gamma\in [0,1]$ be some fixed parameter to be determined later. Let $A_1,B_1,A_2,B_2,\ldots,A_t,B_t$ be i.i.d.\ real normal random variables with mean $0$ and variance $\gamma/(2{t})$. Then, for every $j\in [t]$, let $X_j$ be the complex random variable $A_j+iB_j$, and let
\begin{equation}\label{eq:defGV}
X:=(X_1,\ldots,X_{t})
\end{equation}
be the resulting complex Gaussian random vector in $\mb{C}^{t}$ (for background on such vectors see e.g.\ \cite{Goodman}). Note that $\mb{E}\|X\|_{\ell^2}=\mb{E}(|X_1|^2+\cdots+|X_{t}|^2)=\mb{E}(A_1^2+B_1^2+\cdots+A_{t}^2+B_{t}^2)=\gamma$.

Letting $v_1,\ldots,v_{t}$ be an orthonormal basis of complex eigenvectors for $\textup{Eigen}_\rho\big(\mc{K}_\varepsilon(f\otimes \overline{f})\big)$, we then define the random vector
\begin{equation}\label{eq:randvec1}
h:= \textstyle\sum_{j\in [{t}]} X_j v_j.
\end{equation}
Note that $\mb{E}\|h\|_2^2=\gamma$. We want to prove that, with high probability, the large spectrum of $\mc{K}_\varepsilon(h\otimes \overline{h})$ satisfies the properties required in Theorem \ref{thm:HiSpecBiject}. To do so, we plan to use Proposition \ref{prop:evalcorresp}, which requires us to prove that $h$ is sufficiently close to a vector of the form $\sum_{\chi\in S_\rho}c_\chi u_\chi$ such that the set $\{|c_\chi|^2\}_{\chi\in S_\rho}$ is sufficiently separated. Hence, we would like to express $h$ in terms of the $u_{\chi_i}$ rather than the $v_j$ (note that here we are using that ${t}=|S_\rho|=|\textup{Spec}_\rho(f\otimes\overline{f})|$). However, the subspace of $\mb{C}^{\ab}$ generated by $\{u_{\chi_i}\}_{i\in [{t}]}$ may not be the same as the one generated by the $\{v_j\}_{j\in [{t}]}$. Fortunately, these two subspaces are \emph{close} to each other, as we detail in the proof of Theorem \ref{thm:main-random} below. Let $P$ denote the orthogonal projection in $\mb{C}^{\ab}$ onto the subspace spanned by $\{u_\chi:\chi\in S_\rho\}$. Let $W\in \mb{C}^{{t}\times {t}}$ be the matrix defined by 
\begin{equation}\label{eq:defW}
\forall\, j\in[{t}],\; P(v_j) = \sum_{i\in [t]}  W_{i,j} u_{\chi_i}.
\end{equation}
As we shall establish in the proof of Theorem \ref{thm:main-random}, we then have
\begin{equation}\label{eq:cobform}
\forall j\in [t], \; v_j = \textstyle\sum_{i\in [t]}  W_{i,j} u_{\chi_i}+d_j \text{ where } d_j:= v_j-P(v_j)\text{ satisfies }\|d_j\|_2\le \rho^r.
\end{equation}

\begin{remark}\label{rem:normalization}
Note that the matrix $W$ here is not a $\ab$-matrix and we do not use the normalization for it as we do for $\ab$-matrices (recall Definition \ref{def:ZmatOp}). In particular, given a vector $x\in \mb{C}^t$, we use the standard definition $(Wx)_i=\sum_{j\in[t]} W_{i,j}x_j$. Recall that $\|x\|_{\ell^2}^2$ denotes the Euclidean norm of $x\in \mb{C}^t$, that is $\|x\|_{\ell^2}^2=\sum_{i\in[t]}|x_i|^2$. In particular, by \eqref{eq:randvec1} we have $\|h\|_2=\|X\|_{\ell^2}$.
\end{remark}
\noindent Now, by definition of the matrix $W$ above, and \eqref{eq:cobform}, for any vector $x\in \mb{C}^t$ we have 
\begin{equation}\label{eq:approxCoBasis}
\textstyle\sum_j x_j v_j=\textstyle\sum_i(\textstyle\sum_j W_{i,j}x_j)u_{\chi_i}+\textstyle\sum_j x_jd_j = \textstyle\sum_i (Wx)_iu_{\chi_i}+\textstyle\sum_j x_jd_j,
\end{equation}
where all sums and averages here are for $i$ and $j$ in $[{t}]$. This will imply that the matrix $W$ is \emph{nearly unitary}. To justify this we will use the following independent result.

\begin{lemma}\label{lem:quasiunitary}
Let $N\in \mb{N}$, let ${t}\in [N]$, let $\kappa\in(0,1]$, and $\tau\in (0, 1/(2{t}))$. Let $\{v_i\}_{i\in [{t}]}\subset \mb{C}^N$ be an orthonormal set of vectors, let $\{u_i\}_{i\in[{t}]}\subset \mb{C}^N$ be a set of unit vectors such that for every $i\not=j$ we have $|\langle u_i,u_j\rangle|\le \tau$, and let $\{d_i\}_{i\in[{t}]}\subset \mb{C}^N$ be such that for every $i\in[{t}]$ we have $\|d_i\|_2\le \kappa$. Let $W$ be a matrix in $\mb{C}^{{t}\times{t}}$ such that 
for every $j\in [{t}]$ we have $v_j = \textstyle\sum_{i\in [{t}]}  W_{i,j} u_i+d_j$. Then, for every $x\in \mb{C}^{t}$ we have 
\begin{equation}\label{eq:quasiunitary}
\big| \|Wx\|_{\ell^2}^2 - \|x\|_{\ell^2}^2 \big| \leq  (3\kappa {t}+12\tau{t}^2)\|x\|_{\ell^2}^2.
\end{equation}
\end{lemma}
\begin{proof}
By the analogue of \eqref{eq:approxCoBasis} for $v_i,u_i$, we have 
\begin{equation}\label{eq:cob}
\|\textstyle\sum_i (Wx)_i u_i\|_{\ell^2}^2 = \|\textstyle\sum_j x_j (v_j-d_j)\|_{\ell^2}^2.
\end{equation}
We now expand both sides. 

The right side of \eqref{eq:cob} equals $\langle \sum_j x_j (v_j-d_j), \sum_{j'} x_{j'} (v_{j'}-d_{j'})\rangle = \sum_{j,j'} x_j\overline{x_{j'}} (\langle v_j,v_{j'}\rangle-\langle v_j,d_{j'}\rangle-\langle d_j,v_{j'}\rangle+\langle d_j,d_{j'}\rangle)$. As the $v_i$ are orthonormal, we have  $\sum_{j,j'} x_j\overline{x_{j'}} \langle v_j,v_{j'}\rangle=\|x\|_{\ell^2}^2$. By Cauchy-Schwarz, and recalling that $\kappa\le 1$, we have $\big|\|\sum_j x_j (v_j-d_j)\|_2^2-\|x\|_{\ell^2}^2\big|\le \sum_{j,j'}|x_jx_{j'}|3\kappa=3\kappa\big(\sum_j|x_j|\big)^2\le 3\kappa {t}\|x\|_{\ell^2}^2$.

The left-hand side of \eqref{eq:cob} is $\|\sum_i (Wx)_i u_i\|_2^2 = \sum_{i,i'} (Wx)_i\overline{(Wx)_{i'}} \langle u_i,u_{i'}\rangle$, which equals $\|Wx\|_{\ell^2}^2 + \sum_{i\neq i'} (Wx)_i\overline{(Wx)_{i'}}\langle u_i,u_{i'}\rangle$. Hence $\big| \|Wx\|_{\ell^2}^2 -\|\sum_i (Wx)_i u_i\|_2^2\big|\leq \tau (\sum_i |(Wx)_i|)^2$, which is at most $\tau t \sum_i |(Wx)_i|^2 \leq \tau {t} \|W\|_{\ell^2}^2\|x\|_{\ell^2}^2$. 

Finally, let us prove that $\|W\|_{\ell^2}^2\leq 12{t}$. Note that for every $j\in[{t}]$ we have
\begin{equation*}
1=\langle v_j,v_j\rangle = \textstyle\sum_{i,i'} W_{i,j}\overline{W_{i',j}} \langle u_i,u_{i'}\rangle+ \langle\textstyle\sum_{i}  W_{i,j} u_i,d_j\rangle+\langle d_j,\textstyle\sum_{i}  W_{i,j} u_i\rangle+\langle d_j,d_j\rangle.  
\end{equation*}
As $\|\sum_{i}  W_{i,j} u_i\|_2\le \|v_j\|_2+\|d_j\|_2\le 1+\kappa\le 2$, by Cauchy-Schwarz we have $\big|\langle\sum_{i}  W_{i,j} u_i,d_j\rangle+\langle d_j,\sum_{i}  W_{i,j} u_i\rangle+\langle d_j,d_j\rangle\big|\le 4\kappa+\kappa^2\le 5\kappa$. To control the other term, note that 
\begin{equation*}
\textstyle\sum_{i,i'} W_{i,j}\overline{W_{i',j}} \langle u_i,u_{i'}\rangle =  \textstyle\sum_i |W_{i,j}|^2 + \textstyle\sum_{i\neq i'} W_{i,j}\overline{W_{i',j}} \langle u_i,u_{i'}\rangle,
\end{equation*}
so $| 1- \sum_i |W_{i,j}|^2|\leq 5\kappa +\tau (\sum_{i} |W_{i,j}|)^2 \leq 5\kappa+ \tau t\sum_i |W_{i,j}|^2$. We conclude that for every $j\in [t]$ we have $(1-\tau t)\sum_i |W_{i,j}|^2\leq 1+5\kappa$. Summing this over $j$, we deduce $(1-\tau t)\|W\|_{\ell^2}^2\leq (1+5\kappa)t$, whence $\|W\|_{\ell^2}^2\leq \frac{{t}}{1-\tau t}(1+5\kappa)\le 12{t}$, as claimed.

Combining the bounds above, we deduce $\big|\|x\|_{\ell^2}^2-\|Wx\|_{\ell^2}^2\big|\le \big|\|\sum_j x_j (v_j-d_j)\|_2^2-\|x\|_{\ell^2}^2\big|+\big| \|Wx\|_{\ell^2}^2 -\|\sum_i (Wx)_i u_i\|_2^2\big|\le 3\kappa t\|x\|_{\ell^2}^2+12\tau t^2\|x\|^2_{\ell^2}$ and \eqref{eq:quasiunitary} follows.
\end{proof}

\noindent Recall from \eqref{eq:defGV} the definition of the complex Gaussian vector $X=(X_1,\ldots,X_t)$, and from \eqref{eq:randvec1} the random function $h$. Our aim is to apply Theorem \ref{thm:HiSpecBiject} to this $h$. To do so, we need to ensure that the leading eigenvalues of $\mc{K}_\varepsilon(h\otimes\overline{h})$ are sufficiently separated, and for this we need more information on these eigenvalues. Now Proposition \ref{prop:evalcorresp} tells us that if we can write $h=\sum_{\chi\in D} c_\chi u_\chi+\mc{E}$ where $\|\mc{E}\|_2$ is small, then the eigenvalues of $\mc{K}_\varepsilon(h\otimes\overline{h})$ will be approximately the numbers $|c_\chi|^2$. But we \emph{can} write $h$ this way, indeed  we have $h=\sum_jX_jv_j=\sum_i(WX)_iu_{\chi_i}+(h-P(h))$ where $\|h-P(h)\|_2$ will be shown to be small in the proof of Theorem \ref{thm:main-random} (as mentioned already in \eqref{eq:cobform}). Hence the coefficients $c_{\chi_i}$ that we thus obtain are the coordinates of the image of $X$ under the matrix $W$. Let $Y=(Y_1,\ldots,Y_s):=WX$ be this image, so that the aforementioned eigenvalues are approximately the numbers $|Y_i|^2$. Now, if $W$ was unitary, then, by standard results, the vector $Y$ would also be complex Gaussian with the same distribution\footnote{To see this, note that we can view the distribution of the (complex) random variable $X$ as a $2t$-dimensional (real) random multivariate normal $X'=(\textup{Re}(X_i),\textup{Im}(X_i))_i$. Moreover, the action of a unitary $W$ on the complex vector $X$ can be viewed as the action of an orthonormal (real) matrix $W'$ on $X'$ by \cite[Theorem 2.1 and Theorem 2.4]{Goodman}. The fact that the covariance matrix of $X'$ is a multiple of the identity combined with standard results, e.g. \cite[Chapter 5, Theorem 2.2]{Gut}, implies that the distribution of $W'X'$ is the same as that of $X'$.} as $X$, and this would imply (as we will show below in Lemma \ref{lem:prob-estimate}) that with high probability the eigenvalue separation that we need does indeed hold. However, Lemma \ref{lem:quasiunitary} only gives that $W$ is \emph{nearly unitary}. Fortunately, any such matrix can be efficiently approximated by a unitary matrix, as the next result shows.
\begin{lemma}\label{lem:unitarydecomp}
Let ${t}\ge 2$, $\delta\le 1/(24{t})$, and let $W$ be a matrix in $\mb{C}^{{t}\times {t}}$ satisfying
\begin{equation}\label{eq:unitarydecomp}
\forall\, x\in \mb{C}^{t},\quad \big|\,\|Wx\|_{\ell^2}^2-\|x\|_{\ell^2}^2\big|\leq \delta \|x\|_{\ell^2}^2.
\end{equation}
Then there is a unitary matrix $M\in \mb{C}^{{t}\times {t}}$ such that $W=M+E$ where $\|E\|_{\infty} \leq 30 \delta {t}^{1/2}$.
\end{lemma}
Here for a matrix $E\in\mb{C}^{{t}\times {t}}$ recall that $\|E\|_{\infty}:=\max_{i,j}|E_{i,j}| $.
\begin{proof}
Let $\mc{I}_{t}$ denote the identity ${t}\times {t}$ matrix. Firstly we claim that \eqref{eq:unitarydecomp} implies
\begin{equation}\label{eq:unitarydecomp2}
\|W^*W -\mc{I}_{t}\|_{\infty} \leq 6\delta.
\end{equation}
This can be proved by following a standard argument proving that $x\mapsto Ax$ is a Euclidean isometry only if the matrix $A$ is unitary (e.g.\ the proof of $(g)\Rightarrow (a)$ in \cite[Theorem 2.1.4]{Horn&Johnson}). More precisely, let $x=z+w$ for any $z,w$, let $y=Wx$, let $A=W^*W$, and note that\footnote{For vectors $x,y\in \mb{C}^t$ we denote $\sum_{i=1}^t \overline{x_i}y_i$ by $x^*y$. Note that this agrees with the usual convention of regarding vectors $x\in \mb{C}^t$ as column vectors and thus $x^*$ is the row vector whose entries are $(\overline{x_i})_{i\in[t]}$. A similar convention applies to the expression $x^*Ay$ where $A\in\mb{C}^{t\times t}$.}
\begin{align*}
x^*x & =  z^*z+w^*w+2\tRe(z^*w)\\
y^*y & =  z^*Az+w^*Aw+2\tRe(z^*Aw).
\end{align*}
By assumption we have $|x^*x-y^*y|\leq \delta x^*x$, $|z^*z-z^*Az|\leq \delta z^*z$, and $|w^*w-w^*Aw|\leq \delta w^*w$. Using this, we deduce that for every $w,z$ we have
\begin{align}
|\tRe(z^*Aw)-\tRe(z^*w)| & \leq  \tfrac{1}{2}(|x^*x-y^*y|+|z^*z-z^*Az|+|w^*w-w^*Aw|)\nonumber \\
&\leq  \tfrac{\delta}{2}\big(w^* w + x^*x + z^*z) = \delta(w^* w+ z^*z+ \tRe(w^* z)\big).\quad \label{eq:iptool}
\end{align}
Now, applying \eqref{eq:iptool} with $z=e_i$ and $w=e_j$, we deduce that $|\tRe(A_{i,j})-\textbf{1}(i=j)|\leq 3\delta$. Moreover, applying \eqref{eq:iptool} with $z=e_i$ and $w=i e_j$, we deduce that $|\tRe(i A_{i,j})|=|\tIm(A_{i,j})|\leq 3\delta$. The last two inequalities imply \eqref{eq:unitarydecomp2}.  We can now apply \cite[Lemma 7.2]{Aaronson} using that $\delta<1/(24{t})$, thus obtaining a unitary matrix $M$ such that $\|W-M\|_{\infty} \leq 30 \delta {t}^{1/2}$, as claimed. 
\end{proof}
\noindent Recall that $Y=WX$ for $X$ defined in \eqref{eq:defGV}. Our goal now is to prove that with high probability the coordinates $|Y_i|^2$ are sufficiently separated and large. To this end, the decomposition $W=M+E$ provided by Lemma \ref{lem:unitarydecomp} is useful, because  we then have $Y=MX+EX$, where $MX$ has the same distribution as $X$ thanks to the fact that $M$ is unitary, and $EX$ is small. Thus, we have reduced our task to proving the following result, which we will apply to $V=MX$.
\begin{lemma}\label{lem:prob-estimate}
There exists an absolute constant $C>0$ such that the following holds. Let ${t}\in \mb{N}$ and $\delta,\sigma,\gamma>0$. Let $A_1,B_1,\ldots,A_{t},B_{t}$ be i.i.d.\ real normal variables with mean 0 and variance $\gamma/(2{t})$, and let $V$  be the complex Gaussian vector in $\mb{C}^{t}$ with $j$-th coordinate $V_j:=A_j+iB_j$, $j\in [t]$. Let $Q(\delta,\gamma,\sigma,{t})$ denote the following event: we have $\|V\|_{\ell^2}\leq 1$ and for every pair $j\neq k$ in $[{t}]$ we have $||V_j|^2-|V_k|^2|> 2\delta$ and $|V_j|^2>\sigma$. Then
\begin{equation}\label{estimate-prob}
\mb{P}(Q(\delta,\gamma,\sigma,{t}))\geq 1-{t}^2\big(1-e^{-4\delta {t}/(\gamma \pi)}\big)^{1/2}-{t}\big(1-e^{-{t}\sigma/\gamma}\big)-2e^{-C(1-\sqrt{\gamma})^2{t}/\gamma}.
\end{equation}
\end{lemma}
\begin{proof}
Firstly we obtain a small upper bound on the probability that for some pair $j\neq k$ we have $||V_j|^2-|V_k|^2|\leq 2\delta$. By the union bound, this probability is at most $\binom{s}{2}$ times the probability, for any \emph{fixed} pair $j,k$, that $||V_j|^2-|V_k|^2|\leq 2\delta$. Without loss of generality we fix $j=1$, $k=2$.

Since $\big||V_1|^2-|V_2|^2\big|=|A_1^2-A_2^2+B_1^2-B_2^2|$, we want a small upper bound for the probability of the event $E: |A_1^2-A_2^2+B_1^2-B_2^2|\leq 2\delta$. Let $A=A_1^2-A_2^2$ and $B=B_1^2-B_2^2$, and denote by $f_A$, $f_B$ the probability density functions of $A,B$ respectively. Since the probability density function $f_{A+B}$ is the convolution $f_A*f_B$, we have
\[\mb{P}(E) = \textstyle\int_{-2\delta}^{2\delta} f_A*f_B(z)\ud z = \textstyle\int_{-2\delta}^{2\delta} \textstyle\int_{-\infty}^\infty  f_A(x)f_B(z-x)\ud x\ud z =  \textstyle\int_{-\infty}^\infty   f_A(x) \textstyle\int_{[-2\delta,2\delta]-x}f_B(y) \ud y,\]
where the last equality follows from the Fubini--Tonelli theorem \cite[Proposition 5.2.1]{Cohn} and the change of variables $y=z-x$.

Note that $B=(B_1-B_2)(B_1+B_2)$ and that each of $B_1-B_2$, $B_1+B_2$ is a normal variable with mean 0. Hence $B$ follows a \emph{product distribution}, with $f_B(y)$ being the value at $|y|$ of a modified Bessel function of the second kind. From this we just use the consequence that, for every $x$ we have $\int_{[-2\delta,2\delta]-x}f_B(y) \ud y\leq \int_{[-2\delta,2\delta]}f_B(y) \ud y = \mb{P}(|B_1^2-B_2^2|\leq 2\delta)$. As $|B_1^2-B_2^2|=|B_1-B_2|\, |B_1+B_2|$, this product being at most $2\delta$ implies that either $|B_1-B_2|\leq (2\delta)^{1/2}$ or $|B_1+B_2|\leq (2\delta)^{1/2}$. Let us bound $\mb{P}(|B_1-B_2|\leq (2\delta)^{1/2})$ (the other event will be treated similarly). Since $B_1-B_2$ is a normal random variable with mean 0 and variance $\gamma/s$, the variable $Z:=(B_1-B_2)\sqrt{s/\gamma}$ is standard normal, whence $\mb{P}(|B_1-B_2|\leq \sqrt{2\delta})=\mb{P}(|Z|\leq \sqrt{2\delta t/\gamma})=\textup{erf}(\sqrt{\delta t/\gamma})$. By standard estimates for the \emph{error function} erf, we have $\textup{erf}(\sqrt{\delta t/\gamma})\leq \big(1-e^{-4\delta t/(\gamma\pi)}\big)^{1/2}$. The same upper bound applies to $\mb{P}(|B_1+B_2|\leq \sqrt{2\delta})$.

We conclude that the probability that some pair $V_j,V_k$ satisfies $\big| |V_j|^2-|V_k|^2\big|\leq 2\delta$ is at most
\begin{equation}
\textstyle\binom{{t}}{2}\mb{P}(E)\leq \binom{{t}}{2} \int_{-\infty}^\infty f_A(x)\; \mb{P}(|B_1^2-B_2^2|\leq 2\delta)\,\ud x \leq {t}^2 \big(1-e^{-4\delta {t}/(\gamma\pi)}\big)^{1/2}. 
\end{equation}
Next, note that for every $j\in [{t}]$ we have $\mb{P}(|V_j|^2\leq \sigma)=\mb{P}(A_j^2+B_j^2\in [0,\sigma])$. Since $A_j\sqrt{2{t}/\gamma},B_j\sqrt{2{t}/\gamma}$ are both standard normal, the variable $2{t}(A_j^2+B_j^2)/\gamma$ follows a $\chi^2$-distribution with 2 degrees of freedom, so by standard results the cumulative distribution function of this variable is $\mb{P}(2{t}(A_j^2+B_j^2)/\gamma\leq x)=1-e^{-x/2}$. Therefore we have $\mb{P}(|V_j|^2\leq \sigma)=\mb{P}(2{t}(A_j^2+B_j^2)/\gamma\leq 2{t}\sigma/\gamma)=1-e^{-{t}\sigma/\gamma}$, so $\mb{P}(\exists\, j\in [{t}]\textrm{ s.t. }|V_j^2|\leq \sigma)\leq {t}(1-e^{-{t}\sigma/\gamma})$.

Finally, we need to bound the probability that $\|V\|_{\ell^2}>1$. By \cite[Theorem 3.1.1]{Vershynin} we have that $\big\| \|X/\sqrt{\gamma/2{t}}\|_{\ell^2}-\sqrt{2{t}}\big\|_{\psi_2}\le C$ for some absolute constant $C$ (see \cite[Definition 2.5.6]{Vershynin} for the definition of the sub-Gaussian norm $\|\cdot\|_{\psi_2}$). Hence, $\big\| \|X\|_{\ell^2}-\sqrt{\gamma}\big\|_{\psi_2}\le C\sqrt{\gamma/2{t}}$. By \cite[Proposition 2.5.2]{Vershynin} there exists another absolute constant $C'>0$ such that for all $x\ge 0$, $\mb{P}\big(|\|X\|_{\ell^2}-\sqrt{\gamma}|\ge x\big)\le 2e^{-C'x^2{t}/\gamma}$. Letting $x:=1-\sqrt{\gamma}$ we conclude that $\mb{P}\big(\|X\|_{\ell^2}\ge 1\big)\le 2e^{-C'(1-\sqrt{\gamma})^2{t}/\gamma}$.
\end{proof}

Now we have all the necessary ingredients to prove the following result.
\begin{theorem}\label{thm:main-random}
For any $q>6$ and $r>q+3$, there exists  $C_{q,r}>0$ such that the following holds.\footnote{We can let $C_{q,r}=\min(B_{q,r},(1/295200)^{1/(r-2)},(1/(8\cdot10^7))^{1/(r-q-3)})$, for $B_{q,r}$ as defined in Theorem \ref{thm:HiSpecBiject}.} Let $\rho_0\in(0,C_{q,r}]$. Under the assumptions of Theorems \ref{thm:algoregul} and \ref{thm:HiSpecBiject}, recall that, for a 1-bounded function $f:\ab\to \mb{C}$ there exists $\rho\in[\rho_0/2,\rho_0]$ such that we have a decomposition $f=\sum_{\chi\in S_\rho}g_\chi+f_r+f_e$ where $\|f_r\|_{U^3}\le 2\rho^{3/8}$ and $\|f_e\|_2\le (\rho_0/2)^r$. Let $t:=|S_\rho|$ and, for $i\in[t]$, let $A_i,B_i$ be independent normal random variables with mean 0 and variance $\rho/(2t)$. Let $X$ be the complex random vector $X:=(A_1+iB_1,\ldots,A_t+iB_t)$. Let $\{v_j:j\in [t]\}$ be an orthonormal basis of $\textup{Eigen}_\rho\big(\mc{K}_\varepsilon(f\otimes\overline{f})\big)$ and let $h$ be the random function $\sum_{j\in [t]} X_jv_j$. Then, with probability $1-o_{\rho\to0}(1)$, we have that $\|h\|_2\le 1$, that $\textup{Spec}_{\rho^q}(\mc{K}_\varepsilon(h\otimes \overline{h})\big)$ is $\rho^q$-separated,  and that $|\textup{Spec}_{\rho^q}(\mc{K}_\varepsilon(h\otimes \overline{h})\big)|=|S_\rho|$.
\end{theorem}

\begin{proof}
Recall that ${t}:=|S_\rho|$ is equal to $|\textup{Spec}_\rho\big(\mc{K}_\varepsilon(f\otimes\overline{f})\big)|$. We first prove the following claim:
\begin{eqnarray}\label{eq:probargclaim}
&& \textrm{There is a unitary matrix } M\in \mb{C}^{t\times t}\textrm{ such that }\forall\, x=(x_1,\ldots,x_{t})\in \mb{C}^t, \textrm{ we have } \nonumber \\
&& \textstyle\sum_{i}x_iv_i = \sum_{i} (Mx)_iu_{\chi_i}+d\textrm{ where } \|d\|_2\le 4\cdot 10^6\rho^{r-3}\|x\|_{\ell^2}.
\end{eqnarray}
To prove this, let $w=\sum_{i}x_iv_i$ and note that $\|w\|_2=\|x\|_{\ell^2}$. Let $T_1$ and $T_2$ be as defined in \eqref{eq:T1T2}. By \eqref{eq:order-g-chi}, we know that $\mc{P}_{T_1,\rho}$ is precisely the projection to the space spanned by the $u_\chi$ for $\chi\in S_\rho$. By our choice of parameters in Theorems \ref{thm:algoregul} and \ref{thm:HiSpecBiject}, we have that $\|\mc{P}_{T_1,\rho}-\mc{P}_{T_2,\rho}\|\le \rho^r$. As $w\in \textup{Eigen}_\rho\big(\mc{K}_\varepsilon(f\otimes \overline{f})\big)$, we have $\mc{P}_{T_2,\rho}(w)=w$, whence $\|\mc{P}_{T_1,\rho}(w)-w\|_2=\|\mc{P}_{T_1,\rho}(w)-\mc{P}_{T_2,\rho}(w)\|_2\le \rho^r\|w\|_2=\rho^r\|x\|_{\ell^2}$. Let $W\in \mb{C}^{{t}\times {t}}$ be the matrix defined in \eqref{eq:defW}, thus for every $j\in[{t}]$ we have $\mc{P}_{T_1,\rho}(v_j) = \sum_{i\in [t]}  W_{i,j} u_{\chi_i}$. Then, for every $j\in[t]$, we have $v_j = \sum_{i\in [t]}  W_{i,j} u_{\chi_i}+d_j$ where $\|d_j\|_2\le \rho^r$ (as we announced in \eqref{eq:cobform}). Now we want to apply Lemma \ref{lem:quasiunitary} to $\{v_i\}_{i\in[t]}$ and $\{u_\chi\}_{\chi\in S_\rho}$ with $\kappa:=\rho^r$, $\tau:=\rho^{4+r}$, and $t=|S_\rho|$. Since $|S_\rho|\le |S|\le\nsR(\eta,m)$, using condition \eqref{it:reg2} from Theorem \ref{thm:algoregul}, we have that $\tau\le 1/(2t)$. Hence, Lemma \ref{lem:quasiunitary} combined with the bound $t\le 10/\rho$ (given by Theorem \ref{thm:algoregul}) tells us that $\big| \|Wx\|_{\ell^2}^2 - \|x\|_{\ell^2}^2 \big| \leq  1230\rho^{r-1}\|x\|_{\ell^2}^2$. Next, by Lemma \ref{lem:unitarydecomp} (note that we can apply this because by assumption on $\rho_0$ it follows that $24t\delta\le 1$) we have that $W=M+E$ where $M$ is unitary and $\|E\|_\infty\le 30\delta t^{1/2}$ where $\delta=1230\rho^{r-1}$. In particular, note that
\begin{equation*}
\|Ex\|_\infty \leq \max_i \|E_i\|_{\ell^2} \|x\|_{\ell^2}\leq t^{1/2} \|E\|_\infty \|x\|_{\ell^2} \leq 30\delta {t}\|x\|_{\ell^2}.
\end{equation*}
Thus $\|\sum_i (Ex)_iu_{\chi_i}\|_2\le t\|Ex\|_\infty\le 30\delta t^2\|x\|_{\ell^2}$. Letting $d:=(\mc{P}_{T_1,\rho}(w)-\mc{P}_{T_2,\rho}(w))+\sum_i (Ex)_iu_{\chi_i}$ the claim \eqref{eq:probargclaim} follows.

Now let $h$ be the random function $\sum_{i\in[{t}]} X_iv_i$ in the theorem. By \eqref{eq:probargclaim}, we have that $h = \sum_i (MX)_i u_{\chi_i}+d$. Let $V$ denote the vector $MX$. We apply now Lemma \ref{lem:prob-estimate} to $V$ with parameters ${t}=|S_\rho|\le 10/\rho$, $\delta:=\rho^q$, $\sigma:=2\rho^q$, $\gamma:=\rho$. Let us check that under these conditions, the right side of \eqref{estimate-prob} is $1-o_{\rho\to0}(1)$. The first term is ${t}^2\big(1-e^{-4\delta {t}/(\gamma \pi)}\big)^{1/2}\le O\big(\tfrac{1}{\rho^4}(1-e^{-40\rho^q/(\pi \rho^2)}  \big)^{1/2}$. By L'H\^{o}pital's rule it is easily checked that this quantity tends to 0 as $\rho\to 0$ if $q>6$. The second term is $t\big(1-e^{-t\sigma/\gamma}\big)=O\big(\tfrac{1}{\rho}(1-e^{-20\rho^q/\rho^2})\big)$, which tends to 0 as $\rho\to 0$ if $q>3$ (also by L'H\^{o}pital's rule). The third and final term is $2e^{-C(1-\sqrt{\gamma})^2{t}/\gamma}$. Note that we can assume that $t\ge 1$, as otherwise the result is trivial, and we can also assume that $\sqrt{\gamma}=\sqrt{\rho}\le 1/2$. Thus $2e^{-C(1-\sqrt{\gamma})^2{t}/\gamma}\le O\big( e^{-C/(4\rho)} \big)\to 0$ as $\rho\to 0$.

In particular, as $\|V\|_{\ell^2}\le 1$ and $M$ is unitary, we have $\|X\|_{\ell^2}\le 1$, and 
$h = \sum_i V_iu_{\chi_i}+d$
where $\|d\|_2\le 4\cdot 10^6 \rho^{r-3}$.

Finally, we apply Proposition \ref{prop:evalcorresp} to $h$. Note that the constants $\tau_1,\tau_2$, and $\tau_3$ are bounded as in the proof of Theorem \ref{thm:HiSpecBiject}, and we have just proved that we can let $\tau_4=4\cdot 10^6 \rho^{r-3}$. Hence, we have \eqref{eq:evalcorresp1} with $\|E\|_2\le 3\rho^r+20\cdot 10^6 \rho^{r-3}\le 3\cdot10^7\rho^{r-3}$. Let $\lambda_i$ be the eigenvalues of $\mc{K}_\varepsilon(h\otimes \overline{h})$ in descending order. Assume also that $|V_1|^2\ge |V_2|^2\ge\cdots$ (by relabeling if necessary). Then, by \eqref{eq:evalcorresp2} combined with \eqref{it:reg-9}, we have 
$\big(\textstyle\sum_{i=1}^{t} \big|\, |V_i|^2-\lambda_i\big|^2\big)^{\frac{1}{2}}\leq 4\cdot10^7\rho^{r-3}$. In particular, note that, for all $i,j\in[t]$ with $i\not=j$ we have $\big||V_i|^2-|V_j|^2\big|\ge 2\rho^q$. As $4\cdot10^7\rho^{r-3}\le \rho^q/2$, we have $|\lambda_i-\lambda_j|\ge \rho^q$ for all $i,j\in[t]$, $i\not=j$. Similarly, since for all $i\in[t]$ we have $|V_i|^2\ge 2\rho^q$, it follows that $\lambda_i\ge \rho^q$ for all $i\in[t]$. By \eqref{eq:evalcorresp2}, we have $\lambda_i\le 4\cdot10^7\rho^{r-3}<\rho^q$ for all $i>t$, and the result follows.
\end{proof}

\begin{remark}[Uniform spherical distribution]\label{rem:distr-unit-sphere}
In Theorem \ref{thm:main-random} we chose the random vector $X\in \mb{C}^t$ to follow a multivariate (complex) normal distribution and we showed that, with high probability, the function $h=\sum_i X_iv_i$ satisfies the desired properties. However, note that if instead we let $X':=X/\|X\|_{\ell^2}$, then $X'$ follows a uniform distribution in the unit sphere of $\mb{C}^t$ (e.g.\ by \cite[Exercise 3.3.7]{Vershynin}). Moreover, if we let $h':=\sum_i X_i'v_i$ and the event $Q(\rho^q,\rho,2\rho^q,t)$ from Lemma \ref{lem:prob-estimate} holds for $X$, then we also have that $\|h'\|_2\le 1$, $\textup{Spec}_{\rho^q}(\mc{K}_\varepsilon(h'\otimes \overline{h'})\big)$ is $\rho^q$-separated, and $|\textup{Spec}_{\rho^q}(\mc{K}_\varepsilon(h'\otimes \overline{h'})\big)|=|S_\rho|$. To see this, let $V':=MX'$ and note that $V'=V/\|X\|_{\ell^2}$ where $V=MX$ (as in the proof of Theorem \ref{thm:main-random}). Thus, if $h'=\sum_iV'_iu_{\chi_i}+d'$ we have that $h'$ satisfies the same conditions as $h$, i.e., we have $\|d'\|_2\le 4\cdot 10^6 \rho^{r-3}$ (by \eqref{eq:probargclaim}), for all $i\in[t]$ we have $|V'_i|^2=\tfrac{|V_i|^2}{\|X\|_{\ell^2}^2}\ge 2\rho^q$, and for all $i\not=j$ we have $\big||V'_i|^2-|V'_j|^2\big|=\tfrac{||V_i|^2-|V_j|^2|}{\|X\|_{\ell^2}^2}\ge \rho^q$. Therefore, the function $h'$ also satisfies the conclusion of Theorem \ref{thm:main-random} and we could have directly chosen $X$ to be distributed in the unit sphere.
\end{remark}
 
\appendix

\section{Metrics on \textsc{cfr} nilspaces derived from the manifold structure}\label{app:metrics}
\noindent Every $k$-step \textsc{cfr} nilspace $\ns$ is a metric space \cite[Remark 2.1.3 $(ii)$]{Cand:Notes2}. However, there can be different metrics on $\ns$ that generate the same original topology, and some of these metrics are better than others for our purposes in this paper. In particular, we need to work with a family of metrics on $k$-step \textsc{cfr} nilspaces $\ns$ which ensure sufficient control on Lipschitz constants of various functions on $\ns$ (including complex-valued functions but also nilspace translations on $\ns$), and also on Lipschitz constants of compositions of such functions with morphisms between different nilspaces. Such properties are used for instance in the proof of Theorem \ref{thm:nscharsquadchars}.

Below we describe metrics with such properties, eventually obtaining as main results Proposition \ref{prop:eq-metrics-cfr-nil}, Corollary \ref{cor:cpct-trans-unif-lip} and Lemma \ref{lem:lip-bnd-on-factor}, giving the desired control on Lipschitz constants.

Let us first give an example showing some metrics that we want to avoid in this respect.

\begin{example}
Consider the circle $\mb{T}=\mb{R}/\mb{Z}$, which can be regarded as a $1$-step \textsc{cfr} nilspace. If for $x\in \mb{R}$ we let $|x|_{\mb{T}}:=\min_{n\in \mb{Z}}\{|x-n|\}$, then we obtain the well-known metric $d(x+\mb{Z},y+\mb{Z}):=|x-y|_{\mb{T}}$ that generates the usual topology on $\mb{T}$. There are other metrics, like $d_\alpha(x+\mb{Z},y+\mb{Z}):=|x-y|_{\mb{T}}^\alpha$ for $\alpha\in(0,1)$, that generate the same topology. The problem with these metrics is that, for example, the identity map $(\mb{T},d)\to(\mb{T},d_\alpha)$ is not Lipschitz. Hence, if we let  $\ns=\mc{D}_1(\mb{Z}_2\times \mb{T})$ we can define a (somewhat artificial, yet valid) metric $d':\ns\times \ns\to \mb{R}_{\ge 0}$ by setting $d'((a_0,a_1),(b_0,b_1))$ to have value $1$ if $a_0\not=b_0$, value $d(a_1,b_1)$ if $a_0=b_0=0$, and value $d_{1/2}(a_1,b_1)$ if $a_0=b_0=1$. Then, by the previous observation, it is easy to see that the nilspace translation $\alpha\in \tran(\ns)$ defined by $\alpha(a_0,a_1):=(a_0+1,a_1)$ is not Lipschitz. 
\end{example}

\noindent There is a very convenient family of metrics on \textsc{cfr} nilspaces which avoids issues such as those in the above example. This family is naturally suggested by what is already usual practice when working in the more classical setting of compact nilmanifolds (which are particular cases of \textsc{cfr} nilspaces). Indeed, on nilmanifolds we can use  \emph{Riemannian metrics}, i.e.\ metrics induced by Riemannian metric tensors defined on the nilmanifolds. It turns out that we can use a similar construction on \textsc{cfr} nilspaces more generally, because we can endow these spaces with a differentiable structure that makes them \emph{smooth} (i.e.\ $C^\infty$) \emph{manifolds}. In fact, this smooth structure is available, without much added difficulty, on the more general class of (not necessarily compact) \emph{Lie-fibered nilspaces}, i.e.\ topological nilspaces whose structure groups are Lie groups (for the precise definition see \cite[Definition 2.19]{CGSS-dou-cos}).

\begin{theorem}\label{thm:lie-fib-nil-are-manifold}
Let $\ns$ be a $k$-step Lie-fibered \textup{(}resp.\ \textsc{cfr}\textup{)} nilspace. Then $\ns$ can be endowed with a $C^\infty$ manifold (resp. compact manifold) structure such that every cube set $\cu^n(\ns)$ \textup{(}$n\in \mb{N}$\textup{)} is also a $C^\infty$ manifold.
\end{theorem}

\begin{proof}
Our main tool here is a structure theorem for nilspaces, namely \cite[Theorem 1.8]{CGSS-dou-cos}, which tells us that any $k$-step \textsc{cfr} nilspace is isomorphic to a quotient space $F/H$ where $F$ is a $k$-step free nilspace (see \cite[Definition 1.2]{CGSS-dou-cos}) and $H\subset \tran(F)$ is a $k$-th order lattice action (see \cite[Definition 1.7]{CGSS-dou-cos}). Note that this theorem is stated for \textsc{cfr} nilspaces, but it can be checked that its proof works exactly the same for Lie-fibered nilspaces, see \cite[\S 6]{CGSS-dou-cos} (where, of course, we have to drop the fiber-cocompact property).

The advantage of seeing $\ns$ as $F/H$ is that $F$ has now a natural $C^\infty$ structure which induces a uniquely defined $C^\infty$ structure on the quotient $F/H$ (by definition $F=\prod_{i=1}^k \mc{D}_i(\mb{Z}^{a_i}\times \mb{R}^{b_i})$). Let $\pi_H:F\to F/H$ be the quotient morphism. By \cite[p. 51, 5.9.5]{Bourbaki} it suffices to prove that the set $R:=\{(x,gx)\in F^2:x\in F, g\in H\}$ is an embedded closed submanifold of $F^2$ and the projection $p_1:R\to F$, $(x,gx)\mapsto x$ is a submersion. Note that the fact that $R$ is closed follows from \cite[Lemma 5.25]{CGSS-dou-cos}. To prove that $R$ is an embedded submanifold, we are going to prove the following claim: 
\begin{eqnarray}\label{eq:Cinfclaim}
&&\textrm{For any $(x_0,g_0x_0)\in R$, there exists a neighborhood $U\subset F^2$ of $(x_0,g_0x_0)$}\\
&& \textrm{such that $U\cap R = U\cap \{(x,g_0x)\in F^2: x\in F\}\subset F^2$.}\nonumber
\end{eqnarray}
We prove this by induction on the step $k$. 

For $k=0$ the claim is trivially true ($0$-step nilspaces are 1-point spaces). 

Now assume that \eqref{eq:Cinfclaim} holds for step at most $k-1$. Let $\eta_{k-1}:\tran(F)\to \tran(F_{k-1})$ be the homomorphism that forgets the action on the $k$-th order component (see \cite[paragraph before Definition 1.5]{CGSS-dou-cos}). Let $U^{k-1}\subset F_{k-1}^2$ be a neighborhood of $(\pi_{k-1}(x_0),\eta_{k-1}(g_0)(\pi_{k-1}(x_0)))$ satisfying the claim. If $(x,gx)\in R\cap( \pi_{k-1}^2)^{-1}(U^{k-1})$, we have 
\[
(\pi_{k-1}(x),\eta_{k-1}(g)(\pi_{k-1}(x))) = (\pi_{k-1}(x),\eta_{k-1}(g_0)(\pi_{k-1}(x))).
\]
In particular, by the fiber-transitive property (see \cite[Definition 1.4]{CGSS-dou-cos}), there exists $\gamma\in H_{k}$ such that $gx=g_0x+\gamma$. Fix an arbitrary translation-invariant distance in $F$ (e.g., the usual euclidean distance) and let $\delta>0$ be a constant to be fixed later. For any $r\ge 0$ let $B_r(x_0):=\{y\in F:d(y,x_0)<r\}$. As $H_k$ is a discrete lattice in the $k$-th structure group of $F$, there exists $\epsilon>0$ such that if $\gamma'\in B_\epsilon(\underline{0})\cap H_k$ then $\gamma'=0$ (where $\underline{0}\in F$ is the constant $0$ in all coordinates). As $g_0$ is continuous, if $\delta$ is sufficiently small and $x\in B_\delta(x_0)$ then $d(g_0x,gx)<\epsilon/3$. Now in $F^2$ the set $U:=B_\delta(x_0)\times B_{\epsilon/3}(g_0x_0) \cap (\pi_{k-1}^2)^{-1}(U^{k-1})$ is an open neighborhood of $(x_0,g_0x_0)$. Note that if $(x,gx)\in R\cap U$ then there exists $\gamma\in H_k$ such that $(x,gx)=(x,g_0x+\gamma)$. As $d(x,x_0)<\delta$ we have that $d(g_0x,g_0x_0)<\epsilon/3$. Also, by definition $d(g_0x_0,g_0x+\gamma)<\epsilon/3$. Hence $d(\gamma,\underline{0})<2\epsilon/3$ and we have shown that in this case $\gamma=0$. The claim \eqref{eq:Cinfclaim} follows.

Using \eqref{eq:Cinfclaim}, we can prove that $R\subset F^2$ is an embedded submanifold. For any $(x_0,g_0x_0)\in R$ on the neighborhood $U$ given by \eqref{eq:Cinfclaim}, we define the chart $\varphi:(x,y)\in U\mapsto (x,y-g_0x)\in F^2$. As $g_0$ has an expression in terms of polynomials, it is clear that this map is a diffeomorphism. Moreover, by construction of $U$, the elements of $R\cap U$ are exactly those whose images under $\varphi$ have $0\in F$ in the second coordinate. Hence $R$ is a closed embedded submanifold of $F^2$.

The fact that $p_1:R\to F$ is a submersion follows easily using the charts $\varphi$ described above. In fact, it can be checked that the differential is the identity at every point.

A similar argument shows that the sets $\cu^n(\ns)$ are manifolds as well. The key observation is that the fiber-transitive property implies that if $\q\in \cu^n(F)$ and $d\in \cu^n(H_\bullet)$ is such that $\pi_{k-1}^{\db{n}}\co d\co \q= \pi_{k-1}^{\db{n}}\co  \q$, then there exists $d'\in \cu^n(\mc{D}_k(H_k))$ such that $d\co\q=\q+d'$ and $\cu^n(\mc{D}_k(H_k))$ is discrete.
\end{proof}

\begin{remark}\label{rem:quot-map-is-submersion}
It also follows from \cite[p. 51, 5.9.5]{Bourbaki} that the quotient map $\pi_H:F\to F/H$ is a submersion.
\end{remark}

\begin{corollary}\label{cor:hom-sets-are-manifolds}
Let $\ns$ be a Lie-fibered nilspace, $n\ge 0$, and $S\subset \db{n}$ be a subset with the extension property (see \cite[Definition 3.1.3]{Cand:Notes1}). Then $\hom(S,\ns)$ is a $C^\infty$ manifold.
\end{corollary}

\begin{proof}
The proof is essentially the same as the one given above for the cube sets $\cu^n(\ns)$, as for a fiber-transitive action we have that if $\q\in \hom(S,F)$ and $d\in \hom(S,H_\bullet)$ is such that $\pi_{k-1}^{S}\co d\co \q= \pi_{k-1}^{S}\co  \q$ then there exists $d'\in \hom(S,\mc{D}_k(H_k))$ such that $d\co\q=\q+d'$ and $\hom(S,\mc{D}_k(H_k))$ is discrete.
\end{proof}

\noindent Equipping $k$-step nilspaces with this $C^\infty$ manifold structure, we also obtain that the canonical factor maps are $C^\infty$.

\begin{proposition}\label{prop:c-infty-bundle}
Let $\ns$ be a $k$-step Lie-fibered nilspace. Then the action of the last structure group $\ab_k(\ns)$ on $\ns$ is differentiable. Moreover, this makes $\ns$ a fiber-bundle over $\ns_{k-1}$ in the category of $C^\infty$ manifolds.
\end{proposition}

\begin{proof}
Note that if $\ns\cong F/H$ for a fiber-transitive group $H$, we can describe the action of the last structure group as $F/H\times \mc{D}_k(F)/H_k \to F/H$, $(\pi_H(x),z\mod H_k)\mapsto(\pi_H(x+z))$. By \cite[Theorem 4.29]{lee} to prove that this map is $C^\infty$ it suffices to check that the action $F\times \mc{D}_k(F)\to F$ is $C^\infty$, which is trivial as this is just addition on the $k$-th factor of $F$. As the action is proper (note that the map $F/H\times \mc{D}_k(F)\to F/H\times F/H$ is a homeomorphism onto its image) we have that by standard results (e.g., \cite[Exercise 21-6]{lee}), $\ns\cong F/H$ is a fiber-bundle over $\ns_{k-1}$.
\end{proof}

Moreover, morphisms are automatically $C^\infty$ maps.

\begin{theorem}[Automatic differentiability of morphisms]\label{thm:auto-c-infty}
Let $\ns$ and $\nss$ be $k$-step Lie-fibered nilspaces and let $\varphi:\ns\to \nss$ be a continuous morphism. Then $\varphi$ is $C^\infty$. Moreover, if $\varphi$ is a fibration then it is a submersion.
\end{theorem}

To prove Theorem \ref{thm:auto-c-infty} we shall use the following result.

\begin{proposition}[Free nilspaces are projective in the category of Lie-fibered nilspaces]\label{prop:free-are-proj}
Let $\ns,\nss$ be $k$-step Lie-fibered nilspaces, and let $F$ be a free nilspace. Suppose that we have a morphism $\phi:F\to \ns$ and a fibration $\varphi:\nss\to\ns$. Then there exists a morphism $\psi:F\to\nss$ such that $\phi=\varphi\co\psi$.
\end{proposition}

\begin{proof}
By \cite[Theorem 4.4]{CGSS-dou-cos} there exists a free nilspace $F'$ and a fibration $\varphi:F'\to \nss$. Thus, without loss of generality we can assume that $\nss$ is also a free nilspace.

Consider now the fiber product $F\times_{\ns} \nss:=\{(f,y)\in F\times \nss:\phi(f)=\varphi(y)\}$ with cube sets $\cu^n(F\times_{\ns} \nss):=\{(\q,\q')\in \cu^n(F)\times \cu^n(\nss):\phi\co\q=\varphi\co\q'\}$. This construction has appeared many times in nilspace theory, see  \cite[Lemma 4.2]{CGSS} and \cite[Lemma 2.18]{CGSS-dou-cos}. We leave as an exercise for the reader to check that, even though in such results both $\phi$ and $\varphi$ are required to be fibrations, it suffices that one of them (in this case $\varphi$) is a fibration for all the arguments to work. Using these results we have the following commutative diagram: \vspace{-0.2cm}
\begin{center}
\begin{tikzpicture}
  \matrix (m) [matrix of math nodes,row sep=2em,column sep=4em,minimum width=2em]
  {
     F\times_{\ns} \nss & \nss  \\
     F & \ns, \\};
  \path[-stealth]
    (m-1-2) edge node [right] {$\varphi$} (m-2-2)
    (m-2-1) edge node [above] {$\phi$} (m-2-2)
    (m-1-1) edge node [above] {$p_2$} (m-1-2)
    (m-1-1) edge node [right] {$p_1$} (m-2-1);
\end{tikzpicture}
\end{center}\vspace{-0.2cm}
where it is not difficult to prove that $p_1:F\times_{\ns} \nss\to F$ is a fibration. 

Next, we claim that $F\times_{\ns} \nss$ is a free nilspace. We already know by \cite[Lemma 4.2]{CGSS} and \cite[Lemma 2.18]{CGSS-dou-cos} that it is a \textsc{lch} $k$-step nilspace, so we just need to check that the structure groups are Lie groups. We prove this just for the $k$-th structure group, as the claim for the other structure groups follows similarly. By \cite[Proposition A.20]{CGSS-p-hom} the $k$-structure group of $F\times_{\ns} \nss$, i.e.\ $\ab_k(F\times_{\ns} \nss)$, is isomorphic to $\{(z_1,z_2)\in \ab_k(F)\times \ab_k(\nss):\gamma_k(z_1)=\gamma_k'(z_2)\}$ where $\gamma_k$ and $\gamma_k'$ are the $k$-th structure morphisms of $\phi$ and $\varphi$ respectively. Note that we have the following short exact sequence $0\to \ker(\varphi)\to \ab_k(F\times_{\ns} \nss)\to \ab_k(F)\to 0$. As both $\ker(\varphi)$ and $\ab_k(F)$ are Lie groups, so is $\ab_k(F\times_{\ns} \nss)$. Moreover, as $F$ is a free nilspace, we have $\ab_k(F)\cong \mb{R}^a\times \mb{Z}^b$ for some $a,b\in \mb{Z}_{\ge 0}$, and since this group is projective in the category of abelian Lie groups, we have $\ab_k(F\times_{\ns} \nss)\cong \ab_k(F)\times \ker(\varphi)$. Moreover, as $\ker(\varphi)$ is a closed subgroup of $\ab_k(\nss)\cong \mb{R}^{a'}\times \mb{Z}^{b'}$ for some $a',b'\in \mb{Z}_{\ge 0}$, we have that $\ker(\varphi)$ is also itself a product of the form $\mb{R}^{a''}\times \mb{Z}^{b''}$ for some $a'',b''\in \mb{Z}_{\ge 0}$. Hence $\ab_k(F\times_{\ns}\nss)$ is of the desired form.

To conclude, we apply \cite[Lemma 8.9]{CGSS-dou-cos} to the fibration $p_1:F\times_{\ns}\nss\to F$. Hence, there exists a cross-section $s:F\to F\times_{\ns}\nss$ (note that such lemma gives a free nilspace $F'$ and a nilspace isomorphism $\tau:F\times F'\to F\times_{\ns}\nss$ and thus we may let $s$ be $\tau$ composed with the obvious inclusion of $F$ into $F\times F'$) with $p_1\co s=\id_{F}$, which is a morphism. Thus, we may take $\psi=p_2\co s$ and the result follows.
\end{proof}

\begin{proof}[Proof of Theorem \ref{thm:auto-c-infty}]
Let $\ns$ and $\nss$ be Lie-fibered $k$-step nilspaces, and let $\varphi:\ns\to\nss$ be a morphism. By \cite[Lemma 4.4]{CGSS-dou-cos} there exist free nilspaces $F,F'$ and fibrations $\phi:F\to \ns$ and $\phi':F'\to \nss$. We now apply Proposition \ref{prop:free-are-proj} to $\varphi\co\phi$ and $\phi'$, obtaining a morphism $\psi:F\to F'$ which makes the following diagram commutative: \vspace{-0.2cm}
\begin{center}
\begin{tikzpicture}
  \matrix (m) [matrix of math nodes,row sep=2em,column sep=4em,minimum width=2em]
  {
     F& F'  \\
     \ns & \nss. \\};
  \path[-stealth]
    (m-1-2) edge node [right] {$\phi'$} (m-2-2)
    (m-2-1) edge node [above] {$\varphi$} (m-2-2)
    (m-1-1) edge node [right] {$\phi$} (m-2-1);
    \path[dashed,->]
    (m-1-1) edge node [above] {$\psi$} (m-1-2);
\end{tikzpicture}
\end{center}\vspace{-0.2cm}
\noindent By \cite[\S 3.1]{CGSS-dou-cos} morphisms between $k$-step free nilspaces are basically polynomials of degree at most $k$, so they are $C^\infty$ maps. By \cite[Theorem 4.29]{lee} it then follows that $\varphi$ is also $C^\infty$.

Now if $\varphi$ is a fibration, then firstly, similarly as in the proof of Proposition \ref{prop:free-are-proj}, we can construct the fiber-product of $F\times_{\nss}F'$. Then, the fact that $\varphi$ is a fibration implies that $p_1:F\times_{\nss}F'\to F$ and $p_2:F\times_{\nss}F'\to F'$ are also fibrations. By \cite[Lemma 8.9]{CGSS-dou-cos}, such maps are basically projections, i.e., $F\times_{\nss}F'\cong F\times N$ and $F\times_{\nss}F'\cong F'\times N'$ for some free nilspaces $N,N'$. In particular, $p_1,p_2$ are submersions. Hence $\varphi\co\phi\co p_1 = \phi'\co p_2$ where both $\phi'\co p_2$ and $\phi\co p_1$ are submersions. Hence, as $\phi\co p_1$ is also surjective, $\varphi$ is also a submersion.
\end{proof}

\begin{corollary}\label{cor:tran-gr-c-infty}
Let $\ns$ be a $k$-step \textsc{cfr} nilspace. Then the action of the translation group $\ns$, i.e., $\ns\times \tran(\ns)\to \ns$, $(x,\alpha)\mapsto \alpha(x)$, is $C^\infty$.
\end{corollary}

\begin{proof}
Recall that, by \cite[Theorem 2.9.10]{Cand:Notes2}, the translation group $\tran(\ns)$ is a $k$-step nilpotent filtered Lie group, so in particular it is a $k$-step Lie-fibered nilspace. Then the map $\ns\times \tran(\ns)\to \ns$, $(x,\alpha)\mapsto \alpha(x)$ is a continuous morphism, so the result follows by Theorem \ref{thm:auto-c-infty}.
\end{proof}

\noindent We can now also prove the following result, used for instance in the proof of Proposition \ref{prop:quant-stone-weierstrass}.

\begin{lemma}\label{lem:cor-compl-is-diff}
Let $\ns$ be a $k$-step Lie-fibered nilspace. Then the corner completion map $K:\cor^{k+1}(\ns)\to\ns$ is $C^\infty$.
\end{lemma}

\begin{proof}
By \cite[Theorem 1.8]{CGSS-dou-cos}, there is a fibration $\varphi:F\to\ns$ where $F$ is a $k$-step free nilspace and $\ns\cong F/H$ for some fiber-transitive fiber-discrete group $H$. By Corollary \ref{cor:hom-sets-are-manifolds}, we have that $\cor^{k+1}(\ns)$ is a $C^\infty$ manifold and we have the following commutative diagram:\vspace{-0.2cm}
\begin{center}
\begin{tikzpicture}
  \matrix (m) [matrix of math nodes,row sep=2em,column sep=4em,minimum width=2em]
  {
     \cor^{k+1}(F)& F  \\
     \cor^{k+1}(\ns) & \ns. \\};
  \path[-stealth]
    (m-1-2) edge node [right] {$\varphi$} (m-2-2)
    (m-2-1) edge node [above] {$K$} (m-2-2)
    (m-1-1) edge node [right] {$\varphi^{\db{k+1}\backslash\{0^{k+1}\}}$} (m-2-1)
    (m-1-1) edge node [above] {$K'$} (m-1-2);
\end{tikzpicture}
\end{center}\vspace{-0.2cm}
Here $K'$ is the completion function of $F$. Note that $\varphi$ and $\varphi^{\db{k+1}\backslash\{0^{k+1}\}}$ are submersions by Remark \ref{rem:quot-map-is-submersion}, and $K'$ is $C^{\infty}$ (this last claim can be checked by hand, as $K'$ has an explicit polynomial expression). Hence $K$ is also a $C^\infty$ map.\end{proof}

We can now define the desired metrics that we shall work with.

\begin{lemma}\label{lem:inv-riem-metric} Let $\ns$ be a $k$-step \textsc{cfr} nilspace endowed with the $C^\infty$ structure given by Theorem \ref{thm:lie-fib-nil-are-manifold}. Then there exists a Riemannian metric tensor on $\ns$ which is $\ab_k(\ns)$-invariant.
\end{lemma}

\begin{proof}
Any smooth manifold admits a Riemannian metric tensor (see \cite[Proposition 13.3]{lee}). If $g$ is such a tensor, it is easy to see that $\int_{\ab_k(\ns)}g(\cdot+z,\cdot+z)\ud_{\ab_k(\ns)}(z)$ is also a Riemannian metric tensor which is ${\ab_k(\ns)}$-invariant.\end{proof}

\noindent Note that there may be many such Riemannian tensors, as the previous result does not guarantee uniqueness. However, this will not be a problem in what follows.

Recall that in a connected Riemannian manifold, the geodesic distance is defined as the length of the shortest path between two points, and this distance agrees with the topology of the manifold, see \cite[Proposition 8.19]{lee}.  This is the distance function that we shall use. Let us record it as follows.

\begin{defn}[Riemannian metric on a \textsc{cfr} $k$-step nilspace]\label{def:metric-cfr-nil}
Let $\ns$ be a \textsc{cfr} $k$-step nilspace endowed with a Riemannian metric tensor $g$. We define the distance $d(x,y)$ for $x,y\in \ns$ to be $1$ if $x$ and $y$ are in different connected components, and $d(x,y)=\inf_{\gamma_{x,y}}\int \sqrt{g(\nabla \gamma,\nabla \gamma)}$ otherwise, where the infimum is on the $C^\infty$ maps $\gamma:[0,1]\to \ns$ with $\gamma(0)=x$ and $\gamma(1)=y$.
\end{defn}

A particular case is that of 1-step \textsc{cfr} nilspaces. These are in fact compact abelian Lie groups. For those, we will always use a metric defined as follows.

\begin{defn}[Riemannain metric on compact abelian Lie groups]\label{def:metric-cpct-lie-group}
Let $G$ be a compact abelian Lie group, thus $G\cong \ab\times \mb{T}^n$ where $\ab$ is a finite abelian group and $n\in \mb{Z}_{\geq 0}$. We define a Riemannian metric tensor on $G$ by giving a Riemannian metric tensor on each of its connected components: on $\{a\}\times \mb{T}^n$ for $a\in \ab$, we define a Riemannian metric tensor as the $n$-fold product of the usual metric tensor on $\mb{T}$ obtained as the induced Riemannian metric when $\mb{T}$ is embedded in $\mb{R}^2$ the usual way, see \cite[Chapter 1, Example 2.7]{docarmo}. On $G$, we define the distance $d_G(x,y)=1$ if $x,y\in G$ are in different connected components, and as the usual Riemannian distance otherwise (see Definition \ref{def:metric-cfr-nil}). We call this metric on $G$ the \emph{usual} or \emph{flat} metric on $G$.
\end{defn}

\noindent We can now reap the benefits of using this Riemannian distance, by proving the main results concerning Lipschitz properties, announced at the beginning of the appendix. 

We start with the following.

\begin{proposition}\label{prop:eq-metrics-cfr-nil}
Let $\ns,\nss$ be $k$-step \textsc{cfr}  nilspaces, and let $\varphi:\ns\to\nss$ be a (continuous) morphism. Then, relative to any two Riemannian metrics on $\ns$ and $\nss$, the map $\varphi$ is Lipschitz.
\end{proposition}

\begin{proof}
Let $g_{\ns}$ and $g_{\nss}$ be the metric tensors of $\ns$ and $\nss$ respectively. Consider the unit tangent bundle $\ut(\ns)$ on $\ns$ and the map $D:\ut(\ns)\to \mb{R}_{\ge 0}$ defined as
\[
(x,W_x)\mapsto \frac{g_{\nss}(\nabla \varphi(W_x),\nabla \varphi(W_x))}{g_{\ns}(W_x,W_x)}.
\]
This map is always well-defined as the denominator is constantly 1. Moreover, it is a continuous map and $\ut(\ns)$ is compact. Thus, it attains a maximum $M$. Hence $g_{\nss}(\nabla \varphi(W),\nabla \varphi(W))\le Mg_{\ns}(W,W)$ for any tangent vector $W$.

Then, given any two points $x,y\in \ns$, if they are in different connected components then we always have $d_{\nss}(\varphi(x),\varphi(y))\le Ld_{\ns}(x,y)$ for any constant $L\ge \diam(\nss)$ so there is nothing to prove. On the other hand, if they are in the same connected component, let $\gamma:[0,1]\to\ns$ be such that $\gamma(0)=x$, $\gamma(1)=y$, and $d(x,y)=\int \sqrt{g_{\ns}(\nabla \gamma,\nabla\gamma)}$. Then $d(\varphi(x),\varphi(y)) \le \int \sqrt{g_{\nss}(\nabla\varphi\co\nabla\gamma,\nabla\varphi\co\nabla\gamma)}\le M\int \sqrt{g_{\ns}(\nabla\gamma,\nabla\gamma)} = Md(x,y)$. Hence, the morphism $\varphi$ is $\max(\diam(\nss),M)$-Lipschitz.
\end{proof}

\begin{remark}
It is important to note that the Lipschitz constant depends on the metrics chosen. But once we fix such metrics, all continuous morphisms are  Lipschitz maps. We will endow every $k$-step \textsc{cfr} nilspace with one such Riemannian metric tensor and a corresponding metric.
\end{remark}

\noindent The idea of the previous result can be extended slightly, to show that if we have a compact family of translations acting on a nilspace, then their Lipschitz constants are bounded uniformly.

\begin{corollary}\label{cor:cpct-trans-unif-lip}
Let $\ns$ be a \textsc{cfr} $k$-step nilspace where we have fixed some Riemannian metric $g_{\ns}$. Let $C\subset \tran(\ns)$ be compact. Then $\sup_{\alpha\in C}\|\alpha\|_L<\infty$.
\end{corollary}

\begin{proof}
The idea of the proof is the same as before. The only difference is to consider the map $\ut(\ns)\times C\to \mb{R}$ defined as $(x,W_x,\alpha)\mapsto \frac{g_{\ns}(\nabla \alpha(W_x),\nabla \alpha(W_x))}{g_{\ns}(W_x,W_x)}$, which is continuous by Corollary \ref{cor:tran-gr-c-infty} (in fact, $C^\infty$). In particular, it attains a maximum and thus, repeating the same argument from the proof of  Proposition \ref{prop:eq-metrics-cfr-nil}, the result follows.
\end{proof}

We need the next technical result that follows from the fact that $\pi_{k-1}:\ns\to\ns_{k-1}$ is a submersion by Theorem \ref{thm:auto-c-infty} (we omit its proof).

\begin{lemma}\label{lem:lip-local-cross-sec}
Let $\ns$ be a \textsc{cfr} $k$-step nilspace and assume that we have fixed Riemannian metrics on $\ns$ and $\ns_{k-1}$. Then for every point $\pi_{k-1}(x)\in \ns_{k-1}$ there exists a neighborhood $U_x\subset \ns_{k-1}$ and a Lipschitz cross-section $s_x:U_x\to \ns$. In particular, by compactness, there exists a finite set $\{\pi_{k-1}(x_1),\ldots,\pi_{k-1}(x_s)\}\subset \ns_{k-1}$  such that $\ns_{k-1}=\bigcup_{i=1}^s U_{x_i}$.
\end{lemma}

\begin{lemma}\label{lem:lip-bnd-on-factor}
Let $\nss$ be a $k$-step \textsc{cfr} nilspace and assume that we have fixed Riemannian metrics on $\nss$ and $\nss_{k-1}$. Let $F:\nss\to \mb{C}$ be Lipschitz with respect to the metric on $\nss$ and suppose that $F=F'\co\pi_{k-1}$ for some function $F':\nss_1\to\mb{C}$. Then $F':\nss_{k-1}\to \mb{C}$ is also Lipschitz, with $\|F'\|_{\max}$ bounded in terms of $\|F\|_{\max}$ and the chosen metrics.
\end{lemma}

\begin{proof}
Consider a covering of $\ns_{k-1}$ by open sets such as the one given in Lemma \ref{lem:lip-local-cross-sec}. By the Lebesgue number lemma there exists $\delta>0$ such that for any $\pi_{k-1}(x),\pi_{k-1}(y)\in \ns_{k-1}$ if $d(\pi_{k-1}(x),\pi_{k-1}(y))<\delta$ then $\pi_{k-1}(x),\pi_{k-1}(y)\in U_{x_i}$ for some $i\in[s]$. 

Let $\pi_{k-1}(x),\pi_{k-1}(y)$ be any two points in $\ns_{k-1}$. If $d(\pi_{k-1}(x),\pi_{k-1}(y))\ge \delta$ then clearly $|F'(\pi_{k-1}(x))-F'(\pi_{k-1}(y))|\le 2\|F\|_\infty \le \frac{2\|F\|_\infty}{\delta}d(\pi_{k-1}(x),\pi_{k-1}(y))$. On the other hand, if $d(\pi_{k-1}(x),\pi_{k-1}(y))<\delta$ then we have $|F'(\pi_{k-1}(x))-F'(\pi_{k-1}(y))| = |F(s_i(\pi_{k-1}(x)))-F(s_i(\pi_{k-1}(x)))|\le \|F\|_Ld(s_i(\pi_{k-1}(x)),s_i(\pi_{k-1}(x)))\le \|F\|_L\|s_i\|_Ld(\pi_{k-1}(x),\pi_{k-1}(y))$. Thus, we have $\|F\|_L\le \max\left(\frac{2\|F\|_\infty}{\delta},\max_{i\in[s]}\|F\|_L\|s_i\|_L\right)$.
\end{proof}

\end{document}